\documentclass[11pt,reqno]{amsart}
\usepackage{xifthen,setspace}
\usepackage{bm,bbm}
\usepackage{esvect}
\usepackage{amsmath,mathabx} 
\usepackage{style}
\usepackage{relsize} 
\usepackage{enumitem}
\usepackage{xcolor}
\newtheorem{assumption}{Assumption}
\newtheorem{example}{Example}
\newtheorem{defi}{Definition}

\newcommand{\wbtu}{{\tilde{\beta}_N}}

\newcommand{\p}{{\mathbb{P}}}
\newcommand{\e}{{\mathbb{E}}}

\newcommand{\bs}{{\vec{\sigma}}}
\newcommand{\os}{{\overline{{\sigma}}}}

\newcommand{\bt}{{\vec{\tau}}}

\definecolor{acnew}{rgb}{0.1, 0.4, 0.7}

\usepackage[numbers, comma, sort]{natbib}

\allowdisplaybreaks 

\title[Ising Models on Inhomogeneous Random Graphs]{Ising Models on Inhomogeneous Random Graphs: Inference, Local Asymptotic Minimaxity, and Limit of Experiments}

\author[Mukherjee]{Somabha Mukherjee} 
\address{Department of Statistics and Data Science, National University of Singapore, {\tt somabha@nus.edu.sg}}

\author[Bhowal]{Sanchayan Bhowal}
\address{Department of Statistics, Stanford University, {\tt sbhowal@stanford.edu}}

\author[Chatterjee]{Anirban Chatterjee}
\address{Department of Statistics, University of Chicago, {\tt anirbanc@uchicago.edu}}

\author[Bhattacharya]{Bhaswar B. Bhattacharya} 
\address{Department of Statistics and Data Science, University of Pennsylvania, {\tt bhaswar@wharton.upenn.edu}}

\begin{document}

\begin{abstract} 
In this paper, we develop an inferential framework with sharp asymptotic optimality guarantees for Ising models on inhomogeneous random graphs in the subcritical parameter regime. We begin by characterizing the asymptotic distribution of the maximum likelihood (ML) estimate of the natural parameter, based on a single sample from the underlying model, covering both sparse and dense network regimes. Next, to overcome the computational intractability of the ML method, we propose a simple closed-form estimate obtained from a one-step approximation to the likelihood equation. We show that this estimate attains the same asymptotic distribution and variance as the ML estimate, thereby yielding a computationally efficient and asymptotically valid confidence interval for the natural parameter. We complement these inferential results by establishing a H\'ajek--Le Cam-type local asymptotic minimax theorem, showing that the proposed estimate achieves the smallest possible asymptotic maximum risk, both in rate and in leading constant, over shrinking neighborhoods of the true parameter. We also derive the corresponding limit of experiments. Specifically, we show that in the sparse regime the limiting experiment is Gaussian, while the dense regime leads to a non-Gaussian limiting experiment governed by spectral properties of the underlying graphon that generates the inhomogeneous random graph. To the best of our knowledge, these are among the first sharp asymptotic optimality results for network-dependent data. Finally, we study goodness-of-fit testing for the natural parameter, deriving the local power of the likelihood ratio test and minimax detection rates. A highlight of our method is that the proposed tests adapt to the unknown graphon,  through a direct plug-in estimation of the rejection threshold in the sparse regime and a multiplier-bootstrap calibration in the dense regime. Our analysis relies on new fluctuation results for the sufficient statistic (Hamiltonian) and for the random partition function of Ising models on inhomogeneous random graphs, which are of independent interest. 
\end{abstract}


\maketitle



\section{Introduction}

The proliferation of dependent network data in recent decades has propelled the development of statistical models that can accommodate complex dependence structures while remaining mathematically tractable for rigorous inference. Markov random fields, specifically the Ising model, provide a natural framework for modeling dependencies with an underlying network structure. Originally introduced in statistical physics to model ferromagnetism \cite{ising1925}, the Ising model has since become a fundamental modeling tool across a wide range of disciplines, including social networks, image processing, computer vision, pattern recognition, spatial statistics, disease mapping, computational biology, neural networks, and finance, among others \cite{montanari,isingimage7,geman,patrec,sudipto,green,hopfield,isingfin}. This can be viewed as a discrete exponential family with binary outcomes, where the sufficient statistic is a quadratic form designed to capture correlations arising from pairwise interactions.  Formally, given an interaction matrix $J_N := (( J_N(i, j) ))_{1 \leq i, j \leq N}$ and a binary vector $\bs=(\sigma_1, \sigma_2, \cdots, \sigma_N) \in \{-1, 1\}^N$, the Ising model with parameter $\beta$ encodes the dependence among the coordinates of $\bs$ as follows: 
	\begin{equation}
		\mathbb P_{\beta}(\bs)= \frac{1}{2^N Z_N(\beta)} e^{ - \beta H_N(\bs)  } , 
		\label{model_def}
	\end{equation} 
	where $H_N(\bs) = - \sum_{1 \leq i , j \leq N} J_N(i, j) \sigma_i \sigma_j$ is the sufficient statistic (Hamiltonian) and $ Z_N(\beta)$ is the partition function (normalizing constant), which is determined by the condition $\sum_{\bs \in \{-1, 1\}^N}\mathbb P_{\beta}(\bs)=1$. Specifically, 	\begin{align}\label{eq:ZN}
Z_N(\beta) = \frac{1}{2^N}\sum_{\bs \in \{-1,+1\}^N} e^{-\beta H_N(\bs)} . 
\end{align}	
  The parameter $\beta \geq 0$ controls the strength of dependence between neighboring nodes, with larger values corresponding to stronger alignment or interaction. 
The interaction matrix $J_N$ is typically chosen as an appropriately scaled adjacency matrix of the underlying network, which encodes the pairwise interactions.

The central statistical task in the Ising model is to estimate the parameter $\beta$ given a single sample $\bs$ from the distribution in \eqref{model_def}. This problem has a long history, going back to Pickard's foundational works on asymptotic inference for lattice Ising models \cite{pickard1976,pickard1977,pickard1979}, as well as classical results on consistency and optimality of the maximum likelihood (ML) estimates in spatial lattice settings \cite{gidas1988consistency, comets1992consistency,pickard1987inference}. However, parameter estimation by the ML method is computationally challenging, in general, because the partition function $Z_N(\beta)$ appearing in the likelihood involves a sum over exponentially many configurations $\bs \in \{-1, 1\}^N$. To circumvent this difficulty, Chatterjee~\cite{chspin} proposed using the maximum pseudolikelihood (MPL) estimate \cite{besag1974spatial, besag1975statistical}, a computationally efficient alternative to maximum likelihood based on conditional distributions. Specifically, \cite[Theorem 1.1]{chspin} gives general conditions on the interaction matrix under which the MPL estimate of $\beta$ is $\sqrt{N}$-consistent. This result was later extended by Bhattacharya and Mukherjee \cite{bhattacharya2018inference} to obtain estimation rates of the MPL estimate for Ising models on general weighted graphs. This revealed the following phase-transition phenomenon: For many common choices of $J_N$, there exists a critical parameter value at which the convergence rate of the MPL estimate changes sharply as $\beta$ crosses the threshold. The parameter range below the critical threshold is usually referred to as the {\it subcritical} (high-temperature) regime, while the range above the threshold is known as the {\it supercritical} (low-temperature) regime. The above results, however, establish only convergence rates for the estimates and therefore cannot be used for inferential tasks, such as constructing confidence intervals or conducting hypothesis tests. These require deriving the limiting distribution of the estimates, which is known only in a handful of cases \cite{cometsgidas,xu2023inference}. 
In particular, the asymptotic distribution of the ML estimate is known for all parameter values only for the Curie-Weiss model \cite{cometsgidas}, where $J_N$ is the scaled adjacency matrix of the complete graph, and for dense regular graphs \cite{xu2023inference}. The latter work also derives experimental limits (in the sense of Le Cam \cite{lecam1972limits}) for Ising models on dense regular graphs.

A natural question that emerges from the aforementioned results is whether one can develop an inferential framework for Ising models with optimality guarantees paralleling those of the classical independent-sampling paradigm. In this paper, we obtain a series of results in this direction for Ising models on inhomogeneous random graphs in the subcritical regime. We begin by recalling the relevant definitions and by formally introducing the model.

\subsection{Ising Models on Inhomogeneous Random Graphs} 

Inhomogeneous random graphs provide a canonical framework for modeling network data with latent node heterogeneity \cite{bickel2009nonparametric,bickel2011method,crane2018probabilistic,bollobas2007phase,hoff2002latent}. A natural way to sample such a network is through a graphon \cite{Borgs2008,lovasz2012large}, a symmetric measurable function $W : [0,1]^2 \to [0,1]$ (here, symmetry means $W(x,y) = W(y,x)$, for all $x,y \in [0,1]$). Given a graphon $W$ and a sparsity parameter  $\theta_N$ controlling the edge density, an inhomogeneous random graph can be sampled as follows:

\begin{defn}[Sparse graphon model] Given a sparsity parameter $\theta_N \in (0, 1)$ and a graphon $W$ with $\int_{[0, 1]^2} W(x, y) \mathrm d x \mathrm dy > 0$, the {\it $W$-random graph with sparsity $\theta_N$} on the vertex set $[N] := \{1, 2, \ldots, N\}$, hereafter denoted by $G(N, \theta_N, W)$, is obtained by connecting the vertices $i$ and $j$ with probability $\theta_N W(\frac{i}{N}, \frac{j}{N})$, independently for all $1\leq i \leq j \leq N$. 
\label{defn:W} 
\end{defn}

The model above encompasses many well-known network models, including the classical Erd\H{o}s--R\'{e}nyi random graph model (where $W\equiv 1$ is a constant function), the stochastic block model \citep{bickel2009nonparametric,holland1983stochastic} (where 
$W$ is a block function), and the $\beta$-model \cite{chatterjee2011random}, among others. Over the years, the sparse graphon model has emerged as a fundamental tool in modern network analysis, with applications in community detection, subgraph count statistics, and nonparametric estimation of graph parameters, among others. Consequently, it is a natural candidate for modeling networks underlying Markov random fields, particularly in the Ising model \cite{Borgs2012,basak2017universality,bhattacharya2018inference}. In this case, the underlying network $G_N \sim G(N, \theta_N, W)$ is generated from the sparse graphon model (as in Definition \ref{defn:W}) and the interaction matrix in \eqref{model_def} is chosen to be $J_N = \frac{1}{2N\theta_N} A_{G_N}$, where $A_{G_N} = \{A_{G_N}(i, j)\}_{1 \leq i, j \leq N}$ is the adjacency matrix of the graph $G_N$. In this case, the sufficient statistic (Hamiltonian) in the model \eqref{model_def} can be expressed as: 
\begin{align}\label{eq:sufficientstatistics}
H_N(\bs) := -\frac{1}{2N\theta_N} \sum_{1\le i,j \le N} A_{G_N}\left(i,j\right) \sigma_i \sigma_j ,
\end{align}
where $\{A_{G_N}(i,j)\}_{1\le i\le j\le N}$ are independent Bernoulli variables with parameters $\{ \theta_N W(\frac{i}{N},\frac{j}{N})\}_{1\le i\le j\le N}$, respectively. Observe that the normalization of $A_{G_N}$ by $N\theta_N$ ensures that each vertex of the network has $O(1)$ contribution to the sufficient statistic, with high probability.

In the context of parameter estimation, Ising model on inhomogeneous random graphs was first considered in Bhattacharya and Mukherjee \cite{bhattacharya2018inference} in the dense regime, where $\theta_N \asymp 1$ (see Section \ref{sec:aNbN} for the asymptotic notation). From the results in \cite{bhattacharya2018inference} it is evident that the model \eqref{model_def} with sufficient statistic as in \eqref{eq:sufficientstatistics} undergoes a phase transition at
$\beta = \frac{1}{\|W\|_{\mathrm{op}}}$, where $\|W\|_{\mathrm{op}}$ is the operator norm of the graphon $W$ (see \eqref{eq:Wmaxeigenvalue} for the definition).  Specifically, in the subcritical regime $0 < \beta < \frac{1}{\|W\|_{\mathrm{op}}}$, the estimation rate for $\beta$ is $1/\sqrt{\theta_N}$, while in the supercritical regime $\beta > \frac{1}{\|W\|_{\mathrm{op}}}$, the estimation rate is $\sqrt{N}$. Thus, whenever $N \theta_N \gg 1$,  the estimation rate exhibits an abrupt transition, as $\beta$ crosses the critical value, from a rate slower than the usual $\sqrt{N}$ parametric rate to the parametric rate itself. The slow rate in the subcritical regime is particularly interesting, and a longstanding gap in the literature has been to develop inferential results in this setting. In this paper, we fill this gap and augment the resulting inferential theory with sharp asymptotic optimality results, which we summarize in the next section.

\subsection{Summary of Results}   

In this section, we provide an overview of our results. Throughout, we assume that $\beta$ is in the subcritical regime and $\bs$ is a sample from the model \eqref{model_def} with $H_N(\bs)$ as in \eqref{eq:sufficientstatistics}.

\begin{itemize}
   
    \item \textit{Asymptotic Distribution of the ML Estimate}: We begin by deriving the asymptotic distribution of the ML estimate of $\beta$, under an appropriate sparsity condition (Section \ref{sec:mlestimate}). Our results reveal two different fluctuation regimes: In the sparse regime $(\theta_N \ll 1)$, the ML estimate has fluctuations of order $\sqrt{\theta_N}$ and a Gaussian limit. On the other hand, in the dense regime $(\theta_N \asymp 1)$, the ML estimate has $O(1)$ fluctuations and a `defective' limit. Specifically, it has a non-Gaussian limiting distribution with positive probability bounded away from one, while placing the remaining mass at infinity. This, in particular, means that the ML estimate is inconsistent in the dense regime. Indeed, in this regime, consistent estimation of $\beta$ is impossible (see Remark \ref{estimation}).

    \item \textit{Efficient Inference}: Next, in Section \ref{sec:estimate} we propose a simple closed-form estimate of $\beta$, based on a one-step approximation to the likelihood equation, that has the same limiting distribution (with the same asymptotic variance) as the ML estimate. Using this result, in Section \ref{sec:CN}, we construct a confidence interval for $\beta$ that has valid asymptotic coverage. Thus, whenever consistent estimation is possible, we have a closed-form method for valid asymptotic inference that attains the same statistical benchmark as the ML estimate. In contrast, as already mentioned, the ML estimate is computationally infeasible in general, while the MPL estimate, although efficiently computable, still requires solving an equation numerically.

\item \textit{Local Asymptotic Minimaxity}:  One of the sharpest optimality results for the ML estimate in the i.i.d. setting is the H\'ajek--Le Cam local asymptotic minimax theorem \cite[Chapter 8]{Vaart_1998}, which states that the ML estimate attains the smallest possible asymptotic maximum risk, both in rate and leading constant, over shrinking neighborhoods of the true parameter. In Section \ref{sec:minmaxl2}, we prove an analogue of this theorem for estimating $\beta$ in the Ising model \eqref{model_def} on an inhomogeneous random graph, in the sparse regime. Specifically, we show that our proposed estimate of $\beta$ is locally asymptotically minimax over neighborhoods of size $\sqrt{\theta_N}$ throughout the subcritical regime.

  \item \textit{Limit of Experiments}: In Section \ref{sec:limexp}, we complement the preceding results by deriving the limit of experiments for Ising models on inhomogeneous random graphs. This establishes a sharp asymptotic standard, in the sense that any procedure in the original sequence of experiments is asymptotically constrained by the optimal risk achievable in the limiting experiment. Specifically, we show that in the sparse regime $(\theta_N \ll 1)$, the limiting experiment is Gaussian. In contrast, in the dense regime $(\theta_N \asymp 1)$, the limiting experiment is non-Gaussian: its log-likelihood ratio decomposes into two independent components, one Gaussian and the other a possibly infinite linear combination of independent $\chi_1^2$ random variables. To the best of our knowledge, this is the first instance of a non-Gaussian experimental limit in a setting where the limiting likelihood ratio does not admit an explicit density.

    \item \textit{Minimax Hypothesis Testing}: In Section \ref{sec:gof} we consider the problem of goodness-of-fit testing for the parameter $\beta$. To begin with, we consider the likelihood ratio test for $\beta=\beta_0$, for any $\beta_0$ in the subcritical regime, versus $\beta \ne \beta_0$, which reduces to a test based on the sufficient statistic $H_N(\bs)$. We establish consistency and derive the limiting local power of this test at the scale $\sqrt{\theta_N}$ in both the sparse and dense regimes. A distinctive feature of our test is that it can be implemented without prior knowledge of the graphon $W$. This is important because, although $H_N(\bs)$ is observable, its distribution, and consequently the rejection threshold of the resulting test, depends on $W$, which is generally unknown in practice. We show that this threshold can nevertheless be consistently estimated from the observed network $G_N$, yielding a test that adapts to the unknown choice of $W$ (which plays the role of a nuisance parameter in this case). In the sparse regime, where the limiting distribution is Gaussian, this can be achieved by direct plug-in estimation of the asymptotic variance (Section \ref{sec:sparseUMP}). The dense regime is more challenging, since the limiting distribution is non-Gaussian. In this case, we develop a multiplier bootstrap procedure that combines external randomness with spectral properties of $G_N$ to obtain a consistent estimate of the required quantile (Section \ref{sec:densealpha}). Finally, in Section \ref{sec:minimax} we show that the detection rate of the likelihood ratio test is minimax optimal, that is, all tests are asymptotically powerless in detecting separations which are of smaller order than $\sqrt{\theta_N}$.

\end{itemize}

The key technical ingredient behind the above results is the limiting distribution of the sufficient statistic (Hamiltonian) $H_N(\bs)$. Establishing this limit requires, in turn, a precise understanding of the fluctuations of the (random) partition function $Z_N(\beta)$. These results are presented in Sections \ref{sec:fluct_ham6} and \ref{sec:fluctpart78}, respectively, and are both of independent interest. In particular, our results for the Hamiltonian are new even for the Erd\H{o}s--R\'enyi model (which corresponds to $W\equiv 1$), while our results for the partition function extend the recent results of Kabluchko et al.~\cite{kabluchko2021fluctuations} for the Erd\H{o}s--R\'enyi model to the inhomogeneous random graph setting.

\subsection{Related Work} 

Following the seminal work of Chatterjee \cite{chspin}, the problem of parameter estimation in Ising models has been extended in several directions. One such direction is the inclusion of an additional external-field parameter $h \in \mathbb{R}$ for capturing global node-level effects. For this model, Ghosal and Mukherjee~\cite{ghosal2020joint} obtained general conditions for $\sqrt{N}$-consistent joint MPL estimation of $(\beta,h)$. This was recently extended by Chen et al.~\cite{skjointestimation} to the spin-glass setting, specifically, the Sherrington--Kirkpatrick model, where the entries of the interaction matrix are i.i.d. Gaussian. Very recently, Deb \cite{deb2025pivotal} developed a general technique for proving fluctuation results for conditionally centered functionals of Markov random fields. In particular, this yields central limit theorems for the MPL estimate of $\beta$ (at the $\sqrt{N}$-scale) for a broad class of interaction matrices. Detection problems for Ising models with node-specific external fields were studied in \cite{mukherjee2018global,deb2024detecting}. The practical scope of the Ising model can be further expanded by considering regression formulations that incorporate node-level covariates, which can also be viewed as logistic regression models for network-dependent observations (see \cite{daskalakis2019regression,daskalakis2020logistic,mukherjee2024high,miles2026inference} and the references therein). A variational Bayes procedure based on the pseudolikelihood for estimation in a two-parameter Ising model was also recently developed by Kim et al.~\cite{kim2024statistically}.

Understanding asymptotic properties of Ising models on Erd\H{o}s-R\'enyi random graphs also has a long history in statistical physics, beginning with work of Bovier and Gayrard \cite{BovierGayrard1993}. They showed that the model exhibits a phase transition at $\beta = 1$ (note that in this case $\|W\|_{\mathrm{op}} = 1$) and established the concentration of the magnetization $\os$.\footnote{Technically, \cite{BovierGayrard1993} considered the directed Erd\H{o}s--R\'enyi model, while we consider the undirected model. The differences between the two formulations are minor, and results for one model can be easily transferred to the other. } 
Later, Kabluchko et al.~\cite{kabluchko2019fluctuations} established central limit theorems for $\os$ in the subcritical regime $\beta<1$, under the condition $N^{\frac{2}{3}}\theta_N \gg 1$. This sparsity condition was relaxed to include the full diverging degree regime $N\theta_N \gg 1$ in \cite{kabluchko2020fluctuations}, which also obtained the fluctuations of $\os$ at the critical point $\beta=1$ (see also \cite{KabluchkoLoeweSchubert2022,apostel2026fine} for related results). Kabluchko et al.~\cite{kabluchko2021fluctuations} also derived fluctuations of the partition function $Z_N(\beta)$ in the subcritical regime under the condition $N\theta_N \gg 1$, which is a key technical ingredient for our work. For general non-negative interaction matrices (not necessarily arising from random graph models), one of the sharpest available results is due to Deb and Mukherjee~\cite{debmukh}, which establishes fluctuations of the magnetization for Ising models on $d$-regular graphs, when $d \gg \sqrt{N}$ in the entire ferromagnetic regime and when $d \gg N^{\frac{1}{3}}$ in the subcritical regime. These techniques have since found a range of applications, including to fluctuations of random-field Ising models \cite{lee2025fluctuations} and high-dimensional Bayesian linear regression \cite{lee2025clt}. Very recently, fluctuations of the partition function for the Ising model on Erd\H{o}s--R\'enyi graphs have also been obtained in the bounded-degree regime $N\theta_N \asymp 1$ \cite{coja2026fluctuations,prodromidis2026distribution}.

There is also a parallel line of work in machine learning and theoretical computer science that focuses on learning the underlying graphical structure of an Ising model, or more generally a Markov random field, from multiple i.i.d. samples. This is often referred to as the structure learning problem, for which a variety of efficient algorithms and information-theoretic lower bounds have been developed over the years (see, for example, \cite{AnandkumarTanHuangWillsky2012, RavikumarWainwrightLafferty2010, Bresler2015, VuffrayMisraLokhovChertkov2016, SanthanamWainwright2012, HamiltonKoehlerMoitra2017, KlivansMeka2017, dagan2021learning} and references therein).  Related problems of goodness-of-fit and independence testing for Ising models from multiple samples were studied by Daskalakis et al. \cite{DaskalakisDikkalaKamath2019}. More recently, Neykov and Liu \cite{NeykovLiu2019} and Cao et al. \cite{CaoNeykovLiu2022} investigated testing graph properties, such as connectivity and the presence of cycles or cliques, from multiple samples drawn from an Ising model on the graph. In a different direction, Berthet et al. \cite{BerthetRigolletSrivastava2019} proposed an algorithm for recovering block structure in a mean-field Ising model using multiple observations. In contrast to this literature, we assume that the graph structure is observed and focus on parameter estimation from a single network snapshot. This setting is motivated by applications such as disease mapping, elections, and social network analysis, where multiple independent samples are rarely available.

\subsection{Asymptotic Notations}
\label{sec:aNbN}

Throughout the paper we will use the following asymptotic notations: For two sequences $a_N$ and $b_N$ we will write $a_N \lesssim b_N$ if for all $N$ large enough $a_N \leq C _1 b_N$, for some constant $C_1 > 0$. Similarly, $a_N \gtrsim b_N$ will mean  $a_N \geq C_2 b_N$ and $a_N \asymp b_N$ will mean $C_2 b_N \leq a_N \leq C_1 b_N$, for $N$ large enough and constants $C_1, C_2 > 0$. 
Subscripts in the above notation, for example $\lesssim_{\square}$, $\gtrsim_{\square}$, and $\asymp_{\square}$ denote that the hidden constants may depend on the subscripted parameters. Moreover, $a_N \ll b_N$, $a_N\gg b_N$, and $a_N \sim b_N$ will mean $a_N/b_N \rightarrow 0$, $a_N/b_N \rightarrow \infty$, and $a_N/b_N \rightarrow 1$, as $N \rightarrow \infty$, respectively.

\section{Estimation, Inference, and Local Asymptotic Minimaxity } 
\label{sec:estimation}

In this section we describe our main results on parameter estimation and their optimality properties.

\subsection{Preliminaries }

We begin by recalling the relevant spectral properties of a graphon. To this end, denote by $L^{2}[0,1] $ the collection of all square integrable functions on $[0, 1]$, that is, 
$$L^{2}[0,1] := \left\{f: [0, 1] \rightarrow \R: \|f \|_2^2 := \int_0^1 f(x)^2 \mathrm d x < \infty \right\} . $$
A graphon $W: [0, 1]^2 \rightarrow [0, 1]$ defines an operator $T_{W}: L^{2}[0,1]\rightarrow L^{2}[0, 1]$ as follows:  
\begin{equation}\label{eq:TW}
(T_{W}f)(x)=\int_0^1 W (x, y)f(y) \mathrm d y, 
\end{equation} 
for each $f\in L^{2}[0,1]$ (see \cite[Chapter 7.5]{lovasz2012large}). Note that $T_{W}$ is a symmetric Hilbert--Schmidt operator, 
hence, it  is compact and has a discrete spectrum. In particular, it has a countable
multiset of non-zero real eigenvalues contained in the interval $[-1, 1]$, which we denote by 
$\mathrm{Spec}(W)$, such that
\begin{align}\label{eq:eigenvalue_l2_sum}
    \sum_{\lambda \in \mathrm{Spec}(W)} \lambda^2=\int_{[0, 1]^2 } W(x,y)^2 \mathrm d x \mathrm d y := \| W \|_2^2 <\infty.
\end{align}
The operator norm of $W$ is defined as: 
\begin{align}\label{eq:Wmaxeigenvalue}
\|W\|_{\mathrm{op}} = \sup_{f: \|f\|_2 = 1} \| T_W f\|_2 = \lambda_{\max}(W) , 
\end{align} 
where $\lambda_{\max}(W)$ is the largest eigenvalue of $T_W$ (for the second equality in \eqref{eq:Wmaxeigenvalue}, see \cite[Exercise 7.19]{lovasz2012large}).

Given a graphon $W$ and a finite graph $F=(V(F), E(F))$, the homomorphism density of $F$ in $W$ is defined as:
    \begin{align}\label{eq:tHW}
        t(F, W):=\int_{[0,1]^{|V(F)|}} \prod_{(i, j) \in E(F)} W(x_i,x_j) \prod_{i \in V(F)} \mathrm{d}x_i. 
    \end{align}
    The homomorphism density can be interpreted as the probability that a $W$-random graph with $\theta_N=1$ on $|V(F)|$ vertices contains the graph $F$. For our results, the homomorphism densities of cycles, which can also be expressed as moments of the spectrum of $W$, will play a key role. To this end, for $r \geq 3$, denote by $C_r$ the cycle with $r$ vertices. Then \eqref{eq:tHW} and \cite[Chapter 7.5]{lovasz2012large} give, 
    \begin{align}\label{eq:cycle} 
    t(C_r, W) = \int_{[0, 1]^r} W(x_1, x_2) W(x_2, x_3) \cdots W(x_{r}, x_1) \mathrm d x_1 \mathrm d x_2 \cdots \mathrm d x_r = \sum_{\lambda \in \mathrm{Spec}(W)} 
    \lambda^r, 
    \end{align}
for $r \geq 3$.

Throughout we will assume that $\beta$ is in the subcritical regime, that is, $0 < \beta < \frac{1}{\|W\|_{\mathrm{op}}}$. For ease of reference, we list this and the other standing assumptions below. 

\begin{assumption} Throughout we will assume the following on the parameter $\beta$, the graphon $W$, and the sparsity parameter $\theta_N$. 

\begin{enumerate} 

\item  $\beta \in (0, \frac{1}{\|W\|_{\mathrm{op}}})$, where $\|W\|_{\mathrm{op}}$ is the operator norm of the graphon $W$ (as defined in \eqref{eq:Wmaxeigenvalue}). 

\item The graphon $W: [0, 1]^2 \rightarrow [0, 1]$ and its diagonal map $\mathrm{diag}_W : [0, 1] \rightarrow [0, 1]$, defined as $\mathrm{diag}_W(x) := W(x, x)$, are both Riemann integrable. Moreover, $\int_{[0, 1]^2} W(x, y) \mathrm d x \mathrm d y > 0$. 
 
\item $N^{\frac{2}{3}}\theta_N \gtrsim 1$ and $\theta_N \rightarrow \theta \in [0, 1]$. 

\end{enumerate} 
\label{assumption}
\end{assumption}


Assumption~\ref{assumption} (2) ensures that the limiting edge density of $G_N \sim G(N, \theta_N, W)$ is well defined. To see this, recall that the edge-density of $G_N$ is given by 
$$\rho(G_N) := \frac{1}{N^2} \sum_{1 \leq i, j \leq N} A_{G_N}(i, j), $$ 
 where $A_{G_N}$ is the adjacency matrix of $G_N$. Consequently, 
 $$\E \rho(G_N) := \frac{\theta_N}{N^2} \sum_{1 \leq i \ne j \leq N} W \left( \frac{i}{N}, \frac{j}{N} \right) + O\left(\frac{1}{N}\right).$$ 
The Riemann integrability of $W$ implies that $\frac{1}{N^2} \sum_{1 \leq i \ne j \leq N} W ( \frac{i}{N}, \frac{j}{N} ) \rightarrow \int_{[0, 1]^2} W(x, y) \mathrm d x \mathrm dy$. Hence, under Assumption~\ref{assumption} (2),  
$$\frac{1}{\theta_N} \E \rho(G_N) \rightarrow \int_{[0, 1]^2} W(x, y) \mathrm d x \mathrm dy > 0. $$
Similarly, the Riemann integrability of the diagonal map $\mathrm{diag}_W$ implies that $\frac{1}{N} \sum_{i=1}^N W(\frac{i}{N}, \frac{i}{N}) \rightarrow \int_0^1 W(x, x) \mathrm d x$, which ensures the expected density of self-loops has a well-defined limit.

Assumption~\ref{assumption} (3) is a technical requirement on the sparsity of the underlying network. It is needed to derive the asymptotic distribution of the sufficient statistic~\eqref{eq:sufficientstatistics}, which is a key technical ingredient in our analysis. Finally, the assumption $\theta_N \to \theta \in [0,1]$ ensures that both the dense and sparse regimes are incorporated within the same notational framework. Specifically, $\theta = 0$ corresponds to the {\it sparse} regime (where $\theta_N \ll 1$ and the number of edges in $G_N$ is $o(N^2)$) and $\theta \in (0,1]$ corresponds to the {\it dense} regime (where $\theta_N \asymp 1$ and the number of edges is $\asymp N^2$). In particular, recall that in this parametrization taking $W \equiv 1$ gives the Erd\H{o}s-R\'enyi model $G(N, \theta_N)$.

\subsection{Asymptotic Distribution of the Maximum Likelihood Estimate }  
\label{sec:mlestimate}

Given a sample $\bs = (\sigma_1, \sigma_2, \ldots, \sigma_N)$ from the model \eqref{model_def} with sufficient statistic as in \eqref{eq:sufficientstatistics}, the maximum likelihood (ML) estimate $\hat{\beta}_N$ of $\beta$ is defined as:  
\begin{align}\label{eq:betaML}
\hat{\beta}_N := \arg\min_{\beta \in (0, \infty) } \left\{ \beta H_N(\sigma) + \psi_N(\beta) \right\} , 
\end{align}
where $ \psi_N(\beta) := \log Z_N(\beta)$, if it exists. Note that 
\begin{align}\label{eq:HNderivative}
\psi_N'(\beta) = - \E_\beta [H_N(\bs)| A_{G_N}] \quad \text{ and } \quad \psi_N''(\beta) = \mathrm{Var}_\beta [H_N(\bs)| A_{G_N}] > 0. 
\end{align} 
Hence, the function $\ell_N(\beta) := \beta H_N(\sigma) + \psi_N(\beta)$ is strictly convex, which means if $\ell_N(\beta)$ has a minimizer in $(0, \infty)$, then it is unique and $\hat{\beta}_N$ in \eqref{eq:betaML} is well-defined. Moreover, in this case, the ML estimate $\hat{\beta}_N$ can be obtained by solving the gradient equation:
\begin{align}\label{eq:mleeq7}
    \ell_N'(\beta) =0 \Longleftrightarrow H_N(\bs) = \E_\beta [H_N(\bs)| A_{G_N}] . 
\end{align}
On the other hand, if $\ell_N(\beta)$  does not have a minimizer in $(0,\infty)$, then by convention, we define $\hat{\beta}_N := +\infty$. In the following theorem, we derive the limiting distribution of $\hat{\beta}_N$ in the subcritical regime.

\begin{thm}
        \label{mle}
        Suppose Assumption \ref{assumption} holds. Then, for $\hat{\beta}_N$ as defined above the following hold: 
        \begin{enumerate} 
        
                    \item[$(1)$]  (Sparse regime) If $\theta = 0$, 
                  \begin{align*} 
                      \frac{1}{\sqrt{\theta_N}}( \hat{\beta}_N-\beta )\xrightarrow{D} \cN\left(0,\frac{2}{\int_{[0,1]^2} W(x,y)\mathrm{d}x\mathrm{d}y}\right).
                  \end{align*}
                  
            \item[$(2)$] (Dense regime)  If $\theta \in (0,1]$, then
                \begin{align*}
       \hat{\beta}_N \xrightarrow{D} \begin{cases}
                     F^{-1}(V_\beta^+)  & \text{ with probability } \P(V_\beta>0), \\
            \infty &  \text{ with probability } \P(V_\beta < 0) ,
             \end{cases}
    \end{align*}    
                  where
                  \begin{align}\label{eq:Vbeta}
                      V_{\beta}:=\beta\sigma_W^2 + \sigma_WZ_0+\frac{1}{2} \sum_{\lambda \in \mathrm{Spec}(W)} \lambda \left(\frac{Z_{\lambda}^2}{1-\beta \lambda}-1\right),
                  \end{align}
                  with 
                  \begin{itemize} 
                  \item $Z_0, \{Z_\lambda\}_{\lambda \in \mathrm{Spec}(W)}$ i.i.d. $\mathcal N(0, 1)$ random variables, 
                  \item $\sigma_W^2:=\frac{1}{2\theta} \int_{[0,1]^2} W(x,y)(1-\theta W(x,y))\mathrm{d}x\mathrm{d}y$,  
                  \item $F : [0,\frac{1}{\|W\|_{\mathrm{op}}}) \rightarrow [0,\infty)$ given by $F(x) := \E V_x = x\sigma_W^2 + \frac{1}{2} \sum_{ \lambda \in \mathrm{Spec}(W) } \lambda (\frac{1}{1-x \lambda }-1 )$, which is a strictly increasing bijection,
                  \item $y^+ := \max\{y,0\}$ for a real number $y$.
            \end{itemize} 

        \end{enumerate}
    \end{thm}

The proof of Theorem \ref{mle} is given in Appendix \ref{sec:mlepr8}. 
The first step in the proof is to use the gradient equation \eqref{eq:mleeq7} and the monotonicity of the function $\psi_N'$ to express the distribution of $\hat{\beta}_N$ in terms of $H_N(\bs)$. Specifically, for any positive sequence $\{a_N\}_{N \geq 1}$, 
\begin{align*} 
\P_\beta ( \hat{\beta}_N \leq a_N ) = \P_\beta ( \psi_N'(\hat{\beta}_N) \leq \psi_N' ( a_N ) ) = \P_\beta\left(H_N(\bs)\geq \E_{a_N}\left[H_N(\bs) | A_{G_N} \right]\right) , 
 \end{align*} 
since $\psi_N'$ is an increasing function and $\psi_N'(\hat{\beta}_N) = - H_N(\bs)$, whenever $\hat{\beta}_N$ is finite (recall \eqref{eq:mleeq7}). Hence, to derive the asymptotic distribution of $\hat{\beta}_N$, one needs to analyze the fluctuations of $H_N(\bs)$ under $\P_\beta$ and the perturbed measure $\P_{a_N}$, for a suitable choice of $a_N$ (which depends on whether we are in the sparse or the dense regime). To this end, we invoke our results on the asymptotic distribution of $H_N(\bs)$ (presented in Section~\ref{sec:fluct_ham6}), combined with a change of measure argument which enables us to transfer distributional results from $\P_\beta$ to the perturbed measure $\P_{a_N}$.

\begin{remark}\label{estimation}
Theorem \ref{mle} shows that the distribution of the ML estimate $\hat{\beta}_N$ is very different depending on whether we are in the sparse or the dense regime. 
\begin{itemize}

\item In the sparse regime (where $\theta_N \ll 1$), $\hat{\beta}_N$ converges to the true value $\beta$ at rate $\sqrt{\theta_N}$ and the limiting distribution is Gaussian. This, in particular, means that the ML estimate is finite with high probability and is $1/\sqrt{\theta_N}$-consistent. 

\item In contrast, in the dense regime (where $\theta_N \asymp 1$), the ML estimate is finite only with positive probability, and $\hat{\beta}_N$ itself (without any centering or scaling) converges to a non-Gaussian limiting distribution. In particular, this implies that the ML estimate is inconsistent in this regime. This aligns with the result in \cite[Theorem~3.3]{bhattacharya2018inference}, which shows that no estimate of $\beta$ can be consistent in the subcritical phase for Ising models on dense graphs. 
\end{itemize} 
 To the best of our knowledge, prior to Theorem \ref{mle} the only cases where the asymptotic distribution of the ML estimate is known are for the Curie-Weiss model  \cite{cometsgidas} (where $G_N$ is the complete graph on $N$ vertices) and Ising model on dense regular graphs \cite{xu2023inference}. Even in these cases, there is a technical issue regarding the finiteness of the ML estimate that appears to have gone unnoticed. We discuss this further in Example \ref{example1_26} below. 
\end{remark}

To illustrate the results in Theorem \ref{mle} we now compute the limiting distribution of the ML estimate in two examples (1) when the underlying graph is generated from a 2-block stochastic block model (SBM) and (2) when the underlying graph is generated from a rank one graphon model. These include various well-known graphon models commonly used in practice and will be our running examples throughout the paper.

\begin{example}\label{example1_26} (Block models) Ising models with an underlying block-type structure have been studied in a series of recent papers (see \cite{bianchi2026ising} and the references therein). Here, for illustration we consider the Ising model \eqref{model_def} with the underlying network generated from the 2-block graphon:
\begin{align}\label{eq:Wpq} 
W(x, y) =   \begin{cases}
                     p  & \text{ for } (x, y) \in [0, \frac{1}{2}]^2 \bigcup [\frac{1}{2}, 1]^2, \\
                     q  & \text{ for } (x, y) \in [0, \frac{1}{2}] \times [\frac{1}{2}, 1] \bigcup [\frac{1}{2}, 1] \times [0, \frac{1}{2} ], 
                         \end{cases} 
\end{align}  
where $p, q \in [0, 1]$. This is the graphon corresponding to balanced stochastic block model with within and between block connection probabilities $p$ and $q$, respectively. In this case, $\int_{[0,1]^2} W(x,y) \mathrm dx \mathrm dy = \frac{1}{2}(p+q)$. Moreover, a direct computation shows that the operator $T_W$ (recall \eqref{eq:TW}) has only 2 non-zero eigenvalues which are given by $\lambda_{\pm} =
\frac{1}{2} (p \pm q)$. Suppose $G_N \sim G(N, \theta_N, W)$ is a random graph sampled from this model and $\bs$ is a sample from the Ising model \eqref{model_def} on $G_N$ with $0 < \beta < \frac{2}{(p+q)}$ and $N^{\frac{2}{3}}\theta_N \gg 1$. Then, from Theorem \ref{mle} we have the following: 
    \begin{itemize}
                \item If $\theta = 0$, then
                  \begin{align}\label{eq:sparseexample}
                      \frac{1}{\sqrt{\theta_N}} (\hat{\beta}_N-\beta)\xrightarrow{D} \cN\left(0, \frac{4}{p+q} \right).
                  \end{align} 
            \item If $\theta \in (0,1]$,
                \begin{align}\label{eq:dexample}
       \hat{\beta}_N \xrightarrow{D} \begin{cases}
                     F^{-1}(V_\beta^+)  & \text{ with probability } \P(V_\beta>0), \\
            \infty &  \text{ with probability } \P(V_\beta < 0) ,
             \end{cases}
    \end{align}    
                  where
                  \begin{align*}
                      V_{\beta} =\beta\sigma_W^2 + \sigma_W Z_0+ \frac{\lambda_{+}}{2}   \left(\frac{Z_{\lambda_{+}}^2}{1-\beta \lambda_{+}}-1\right) + \frac{\lambda_{-} }{2}  \left(\frac{Z_{\lambda_{-}}^2}{1-\beta \lambda_{-}}-1\right),
                  \end{align*}
                  where
                 $Z_0, Z_{\lambda_{+}}, Z_{\lambda_{-}}$ are i.i.d. $\mathcal N(0, 1)$,                   $\sigma_W^2 =\frac{1}{4 \theta} \left( p(1-\theta p) + q(1-\theta q)\right) $, and   $F(x) =  \E V_x$.
                                  \end{itemize} 
The following are some important special cases of this example: 
\begin{itemize} 
\item {\it Erd\H{o}s-R\'enyi model $G(N, \theta_N)$}: In this case $W \equiv 1$, which corresponds to taking $p=q=1$ in \eqref{eq:Wpq}. Then in the sparse case (that is, $\theta=0$),  \eqref{eq:sparseexample} simplifies to $\frac{1}{\sqrt{\theta_N}} (\hat{\beta}_N-\beta)\xrightarrow{D} \cN\left(0, 2 \right)$. Moreover, in the dense case (that is, $\theta \in (0, 1]$) \eqref{eq:dexample} holds with 
                  \begin{align}\label{eq:dexampleGN}
                      V_{\beta} =\frac{\beta(1-\theta)}{2\theta} + \sqrt{\frac{1-\theta}{2\theta}}Z_0+ \frac{1}{2}\left(\frac{Z_1^2}{1-\beta}-1\right) , 
                  \end{align}
                  where $Z_0 , Z_1$ are i.i.d. $\cN(0, 1)$ and $F(x) = \frac{x (1-(1-\theta)x)}{2\theta(1-x)}$.

\item {\it Curie-Weiss Model}: This is the Ising model on the complete graph, which corresponds to taking $W \equiv 1$ and $\theta=1$. In this case, \eqref{eq:dexampleGN} simplifies to $$V_{\beta} = \frac{1}{2} \left(\frac{Z_1^2}{1-\beta}-1 \right),$$ where $Z_1 \sim \cN(0, 1)$ and  $F(x) = \frac{x }{2(1-x)}$. Note that $F^{-1}$ is invertible on $[0,\infty)$ and $F^{-1}(y) = \frac{2y}{1+2y}$. Hence, 
$$F^{-1}(V_\beta^+) = \left(1-\frac{1-\beta}{Z_1^2}\right)\bm 1\left\{ Z_1^2>1-\beta\right\} .$$  
Then by \eqref{eq:dexample}, 
\begin{align}\label{eq:MLdistribution} 
\hat{\beta}_N \xrightarrow{D}  \p(V_\beta >0) \left(1-\frac{1-\beta}{Z_1^2}\right)\bm 1\left\{ Z_1^2>1-\beta \right\}  + \p(V_\beta <0) ~\delta_{\infty}. 
\end{align}  
Recall that the asymptotic distribution of the ML estimate in the Curie--Weiss model is classically known from the work of \cite{cometsgidas}. However, in \cite[Theorem 1.4(b)]{cometsgidas}, only the first term in \eqref{eq:MLdistribution} (without the indicator) appears in the limiting distribution. There are two reasons for this discrepancy: First, in our setting the ML estimate is constrained to be nonnegative, which introduces the additional indicator term $\mathbf{1}\{Z_1^2>1-\beta\}$ in the first term of the limiting distribution. Second, as already noted in \cite{ellis1992}, the issue of nonexistence of the ML estimate is not addressed in \cite{cometsgidas}, an issue that appears to have gone unnoticed in several subsequent papers as well. This explains the absence of the mass at infinity in the result of \cite{cometsgidas}. In contrast, the actual limiting distribution does assign positive mass to infinity, because the function $\ell_N(\beta)$ fails to have a minimizer in $(0,\infty)$ with positive probability.

\item {\it Random bipartite graph}:  This corresponds to taking $p=0$ and $q=1$ in \eqref{eq:Wpq}. Hence, in the sparse case \eqref{eq:sparseexample} simplifies to $\frac{1}{\sqrt{\theta_N}} (\hat{\beta}_N-\beta)\xrightarrow{D} \cN\left(0, 4 \right)$. Moreover, since $\lambda_{+} = \frac{1}{2}$ and $\lambda_{-} = -\frac{1}{2}$, in the dense case
\eqref{eq:dexample} holds with 
                  \begin{align*}
                      V_{\beta} =\frac{\beta(1-\theta)}{4 \theta} + \sqrt{\frac{1-\theta}{4 \theta}}Z_0 + \frac{Z_1^2}{4 -  2  \beta } - \frac{Z_2^2}{4 + 2 \beta}  , 
                  \end{align*}
                  where $Z_0 , Z_1, Z_2$ are i.i.d. $\cN(0, 1)$ and $F(x) = \frac{x(1-\theta)}{4 \theta} + \frac{x}{4 -  x^2 }$. 
\end{itemize} 
\end{example}

\begin{example}[Rank one graphon model] 
\label{rankOne}
Consider the rank one graphon: $W(x, y) = f(x) f(y)$, where $f: [0, 1] \rightarrow [0, 1]$ is continuous almost everywhere. This may be viewed as a graphon analogue of a Chung--Lu-type random graph model \cite{chung2002connected}. In this case, the operator $T_W$ has rank 1 with 
$\lambda_{\mathrm{max}}(W) = \mu_2:= \| f \|_2^2$.  Suppose $G_N \sim G(N, \theta_N, W)$ is a random graph sampled from this model and $\bs$ be a sample from the Ising model \eqref{model_def} on $G_N$ with $0 < \beta < \frac{1}{\mu_2}$ and $N^{\frac{2}{3}}\theta_N \gg 1$. Then, from Theorem \ref{mle} we have the following: 

\begin{itemize}
    \item If $\theta=0$,
    \begin{equation*}
        \frac{1}{\sqrt{\theta_N}}( \hat{\beta}_N-\beta )\xrightarrow{D} \cN\left(0,\frac{2}{\mu_1^2}\right) , 
    \end{equation*} 
    where $\mu_1:=\int_0^1 f(x) \mathrm dx$. 
    \item If $\theta \in (0,1]$, 
                \begin{align*}
       \hat{\beta}_N \xrightarrow{D} \begin{cases}
                     F^{-1}(V_\beta^+)  & \text{ with probability } \P(V_\beta>0), \\
            \infty &  \text{ with probability } \P(V_\beta < 0) ,
             \end{cases}
    \end{align*}    
                  where
\begin{equation*}
V_\beta=\frac\beta{2\theta}(\mu_1^2-\theta \mu_2^2)+Z_0\sqrt{\frac1{2\theta}(\mu_1^2-\theta \mu_2^2)} +\frac{\mu_2}{2}\left(\frac{Z_1^2}{1-\beta \mu_2}-1\right),
\end{equation*}
where $F(x)=\frac{x \mu_1^2}{2\theta}+\frac{x^2 \mu_2^3}{2(1-x\mu_2)}$ and $Z_0,Z_1$ are i.i.d. $\cN(0,1)$.
\end{itemize}
\end{example}

\subsection{A Simple Efficient Estimate} 
\label{sec:estimate}

Although the maximum likelihood (ML) estimate possesses several desirable theoretical properties, for the Ising model it is generally computationally intractable. This intractability arises from the presence of the normalizing constant $Z_N$ in the likelihood function, which entails a summation over an exponentially large number of terms. In this section, we show that a one-step approximation to the likelihood equation leads to a simple closed-form estimate of $\beta$ that attains the same asymptotic variance as the ML estimate in the entire subcritical regime.  To motivate the construction of our estimate, first recall the definition of $\psi_N'$ from \eqref{eq:HNderivative}. By a first-order Taylor expansion (ignoring error terms), 
$$- \E_\beta [H_N(\bs)| A_{G_N}] = \psi_N'(\beta) \approx \psi_N'(0) + \beta \psi_N''(0) = \frac{1}{2N\theta_N} \sum_{i=1}^N A_{G_N}\left(i,i \right) + \frac{\beta}{2N^2 \theta_N^2}\sum_{1 \leq i \ne j \leq N} A_{G_N}(i,j) , 
$$
since $\psi_N'(0) = \frac{1}{2N\theta_N} \sum_{i=1}^N A_{G_N}\left(i,i \right) $ and $\psi_N''(0) = \mathrm{Var}_0 [H_N(\bs)| A_{G_N}] = \frac{1}{2N^2 \theta_N^2}\sum_{1 \leq i \ne j \leq N} A_{G_N}(i,j)$.  Now, ignoring the lower-order contributions from the diagonal terms $\sum_{i=1}^N A_{G_N}\left(i,i \right)$ and substituting the above approximation in the RHS of \eqref{eq:mleeq7} leads to the following simple estimate of $\beta$: 
\begin{align}\label{newestimate}
    \wbtu := - \frac{2N^2\theta_N^2}{\sum_{1 \leq i,j \leq N}A_{G_N}(i,j)} H_N(\bs). 
\end{align}

The following theorem shows that the above estimate has the same asymptotic distribution as the ML estimate throughout the subcritical regime $G_N$ is sparse (that is, $\theta_N \ll 1$). 
Recall (from Theorem \ref{mle} (2) and Remark \ref{estimation}) that consistent estimation of $\beta$ is only possible when $\theta_N \ll 1$. Hence, the asymptotic performance of \eqref{newestimate} matches that of the ML estimate, that is,  the asymptotic relative efficiency of $\wbtu$ with respect to the ML estimate $\hat{\beta}_N$ is 1, whenever consistent estimation is possible.

\begin{thm}\label{clttilde}
Suppose Assumption \ref{assumption} holds and $\theta_N \ll 1$. Then,
     \begin{align}\label{eq:cltestimate}
                      \frac{1}{\sqrt{\theta_N}}(\wbtu-\beta)\xrightarrow{D} \cN\left(0,\frac{2}{\int_{[0,1]^2} W(x,y)\mathrm{d}x\mathrm{d}y}\right).
     \end{align} 
\end{thm}

The proof of Theorem \ref{clttilde} is given in Appendix \ref{sec:clttildepf}. The key step in the proof is the following asymptotic expansion of the sufficient statistic $H_N(\bs)$, for $\beta \in (0, \frac{1}{\|W\|_{\mathrm{op}}})$ and $\theta_N \ll 1$: 
   \begin{equation*} 
     \theta_N H_N(\bs) = - \frac{\beta}{2N^2 }\sum_{1 \leq i,j \leq N}W\left(\frac{i}{N},\frac{j}{N}\right) + O_P ( \sqrt{\theta_N} ) , 
              \end{equation*}
which we prove in Theorem \ref{thm:HNsigma} (where Gaussian fluctuations of $H_N(\bs)$ are established). Consequently, recalling \eqref{newestimate} and replacing $\sum_{1 \leq i,j \leq N} W(\frac{i}{N},\frac{j}{N})$ with its empirical analogue based on $G_N$, one can obtain the result in Theorem \ref{clttilde}.

To illustrate the finite-sample performance of $\wbtu$, we compare its simulated histogram with the limiting normal density from Theorem \ref{clttilde} for the following choices of $W$:
\begin{itemize}

\item[(a)] $W \equiv 1$ (which corresponds to the Erd\H{o}s-R\'enyi model), 

\item[(b)] $W(x,y) = \sin(|x-y|)$, 

\item[(c)] $W$ is the graphon corresponding to the random bipartite graph (which corresponds to taking $p=0$ and $q=1$ in \eqref{eq:Wpq}),  

\item[(d)] the rank one graphon $W(x,y) = x y$. 
\end{itemize} 
Note that for (a) and (d), $\|W\|_{\mathrm{op}} = 1$, for (b) $\|W\|_{\mathrm{op}} \le \|W\|_2 \leq 1$, and for (c) $\|W\|_{\mathrm{op}} = \frac{1}{2}$.  Figure \ref{fig:beta_tilde_comparisonI} shows the histogram of $\wbtu$ (computed over 1000 iterations) when $G_N \sim G(N, \theta_N, W)$ for the above four choices of $W$, with $N = 500$ and $\theta_N = N^{-0.6}$. The value $\beta$ is chosen to be $0.5$ for (a), (b), and (d) and equal to 0.25 for (c). This ensures the models are in the subcritical regime in all the four cases. The limiting normal densities as predicted by Theorem \ref{clttilde} are overlaid in red in the figures. We observe that, in all four settings, the empirical distributions align closely with the asymptotic theory.

\begin{figure}[htbp]
   \centering
    \begin{subfigure}[b]{0.48\textwidth}
        \centering
        \resizebox{\linewidth}{!}{
\begin{tikzpicture}[x = 1pt, y = 1pt]
\definecolor{fillColor}{RGB}{255,255,255}
\path[use as bounding box,fill=fillColor,fill opacity=0.00] (0,0) rectangle (252.94,252.94);
\begin{scope}
\path[clip] ( 34.16, 30.69) rectangle (247.44,230.29);
\definecolor{drawColor}{gray}{0.92}

\path[draw=drawColor,line width= 0.3pt,line join=round] ( 34.16, 71.93) --
	(247.44, 71.93);

\path[draw=drawColor,line width= 0.3pt,line join=round] ( 34.16,136.28) --
	(247.44,136.28);

\path[draw=drawColor,line width= 0.3pt,line join=round] ( 34.16,200.63) --
	(247.44,200.63);

\path[draw=drawColor,line width= 0.3pt,line join=round] ( 53.22, 30.69) --
	( 53.22,230.29);

\path[draw=drawColor,line width= 0.3pt,line join=round] (102.99, 30.69) --
	(102.99,230.29);

\path[draw=drawColor,line width= 0.3pt,line join=round] (152.76, 30.69) --
	(152.76,230.29);

\path[draw=drawColor,line width= 0.3pt,line join=round] (202.53, 30.69) --
	(202.53,230.29);

\path[draw=drawColor,line width= 0.6pt,line join=round] ( 34.16, 39.76) --
	(247.44, 39.76);

\path[draw=drawColor,line width= 0.6pt,line join=round] ( 34.16,104.11) --
	(247.44,104.11);

\path[draw=drawColor,line width= 0.6pt,line join=round] ( 34.16,168.45) --
	(247.44,168.45);

\path[draw=drawColor,line width= 0.6pt,line join=round] ( 78.11, 30.69) --
	( 78.11,230.29);

\path[draw=drawColor,line width= 0.6pt,line join=round] (127.87, 30.69) --
	(127.87,230.29);

\path[draw=drawColor,line width= 0.6pt,line join=round] (177.64, 30.69) --
	(177.64,230.29);

\path[draw=drawColor,line width= 0.6pt,line join=round] (227.41, 30.69) --
	(227.41,230.29);
\definecolor{drawColor}{RGB}{0,0,0}
\definecolor{fillColor}{RGB}{173,216,230}

\path[draw=drawColor,line width= 0.6pt,fill=fillColor,fill opacity=0.70] ( 43.85, 39.76) rectangle ( 56.78, 47.69);

\path[draw=drawColor,line width= 0.6pt,fill=fillColor,fill opacity=0.70] ( 56.78, 39.76) rectangle ( 69.70, 52.64);

\path[draw=drawColor,line width= 0.6pt,fill=fillColor,fill opacity=0.70] ( 69.70, 39.76) rectangle ( 82.63, 77.41);

\path[draw=drawColor,line width= 0.6pt,fill=fillColor,fill opacity=0.70] ( 82.63, 39.76) rectangle ( 95.56, 97.23);

\path[draw=drawColor,line width= 0.6pt,fill=fillColor,fill opacity=0.70] ( 95.56, 39.76) rectangle (108.48,140.84);

\path[draw=drawColor,line width= 0.6pt,fill=fillColor,fill opacity=0.70] (108.48, 39.76) rectangle (121.41,195.34);

\path[draw=drawColor,line width= 0.6pt,fill=fillColor,fill opacity=0.70] (121.41, 39.76) rectangle (134.34,207.23);

\path[draw=drawColor,line width= 0.6pt,fill=fillColor,fill opacity=0.70] (134.34, 39.76) rectangle (147.26,216.15);

\path[draw=drawColor,line width= 0.6pt,fill=fillColor,fill opacity=0.70] (147.26, 39.76) rectangle (160.19,168.58);

\path[draw=drawColor,line width= 0.6pt,fill=fillColor,fill opacity=0.70] (160.19, 39.76) rectangle (173.12,121.02);

\path[draw=drawColor,line width= 0.6pt,fill=fillColor,fill opacity=0.70] (173.12, 39.76) rectangle (186.04, 74.44);

\path[draw=drawColor,line width= 0.6pt,fill=fillColor,fill opacity=0.70] (186.04, 39.76) rectangle (198.97, 58.59);

\path[draw=drawColor,line width= 0.6pt,fill=fillColor,fill opacity=0.70] (198.97, 39.76) rectangle (211.90, 47.69);

\path[draw=drawColor,line width= 0.6pt,fill=fillColor,fill opacity=0.70] (211.90, 39.76) rectangle (224.82, 41.74);

\path[draw=drawColor,line width= 0.6pt,fill=fillColor,fill opacity=0.70] (224.82, 39.76) rectangle (237.75, 40.75);
\definecolor{drawColor}{RGB}{255,0,0}

\path[draw=drawColor,line width= 1.4pt,line join=round] ( 45.37, 42.24) --
	( 47.18, 42.74) --
	( 48.99, 43.34) --
	( 50.80, 44.04) --
	( 52.61, 44.85) --
	( 54.42, 45.79) --
	( 56.23, 46.88) --
	( 58.04, 48.13) --
	( 59.85, 49.56) --
	( 61.66, 51.18) --
	( 63.47, 53.01) --
	( 65.28, 55.08) --
	( 67.09, 57.40) --
	( 68.90, 59.99) --
	( 70.71, 62.85) --
	( 72.51, 66.02) --
	( 74.32, 69.50) --
	( 76.13, 73.29) --
	( 77.94, 77.42) --
	( 79.75, 81.88) --
	( 81.56, 86.68) --
	( 83.37, 91.80) --
	( 85.18, 97.25) --
	( 86.99,103.01) --
	( 88.80,109.05) --
	( 90.61,115.36) --
	( 92.42,121.90) --
	( 94.23,128.65) --
	( 96.04,135.54) --
	( 97.85,142.55) --
	( 99.66,149.62) --
	(101.47,156.69) --
	(103.28,163.70) --
	(105.09,170.59) --
	(106.90,177.29) --
	(108.71,183.74) --
	(110.52,189.87) --
	(112.33,195.61) --
	(114.14,200.91) --
	(115.95,205.70) --
	(117.76,209.93) --
	(119.57,213.55) --
	(121.38,216.51) --
	(123.19,218.78) --
	(125.00,220.33) --
	(126.81,221.15) --
	(128.62,221.21) --
	(130.43,220.53) --
	(132.24,219.11) --
	(134.05,216.97) --
	(135.86,214.13) --
	(137.66,210.63) --
	(139.47,206.50) --
	(141.28,201.81) --
	(143.09,196.60) --
	(144.90,190.93) --
	(146.71,184.86) --
	(148.52,178.47) --
	(150.33,171.81) --
	(152.14,164.95) --
	(153.95,157.95) --
	(155.76,150.89) --
	(157.57,143.82) --
	(159.38,136.80) --
	(161.19,129.87) --
	(163.00,123.10) --
	(164.81,116.52) --
	(166.62,110.17) --
	(168.43,104.07) --
	(170.24, 98.26) --
	(172.05, 92.76) --
	(173.86, 87.57) --
	(175.67, 82.72) --
	(177.48, 78.20) --
	(179.29, 74.01) --
	(181.10, 70.15) --
	(182.91, 66.62) --
	(184.72, 63.40) --
	(186.53, 60.48) --
	(188.34, 57.84) --
	(190.15, 55.48) --
	(191.96, 53.37) --
	(193.77, 51.49) --
	(195.58, 49.83) --
	(197.39, 48.37) --
	(199.20, 47.09) --
	(201.01, 45.98) --
	(202.81, 45.01) --
	(204.62, 44.17) --
	(206.43, 43.46) --
	(208.24, 42.84) --
	(210.05, 42.32) --
	(211.86, 41.88) --
	(213.67, 41.50) --
	(215.48, 41.19) --
	(217.29, 40.93) --
	(219.10, 40.71) --
	(220.91, 40.53) --
	(222.72, 40.38) --
	(224.53, 40.26) --
	(226.34, 40.16);
\end{scope}
\begin{scope}
\path[clip] (  0.00,  0.00) rectangle (252.94,252.94);
\definecolor{drawColor}{gray}{0.30}

\node[text=drawColor,anchor=base east,inner sep=0pt, outer sep=0pt, scale=  0.88] at ( 29.21, 36.73) {0.0};

\node[text=drawColor,anchor=base east,inner sep=0pt, outer sep=0pt, scale=  0.88] at ( 29.21,101.08) {0.1};

\node[text=drawColor,anchor=base east,inner sep=0pt, outer sep=0pt, scale=  0.88] at ( 29.21,165.42) {0.2};
\end{scope}
\begin{scope}
\path[clip] (  0.00,  0.00) rectangle (252.94,252.94);
\definecolor{drawColor}{gray}{0.30}

\node[text=drawColor,anchor=base,inner sep=0pt, outer sep=0pt, scale=  0.88] at ( 78.11, 25.68) {-2.5};

\node[text=drawColor,anchor=base,inner sep=0pt, outer sep=0pt, scale=  0.88] at (127.87, 25.68) {0.0};

\node[text=drawColor,anchor=base,inner sep=0pt, outer sep=0pt, scale=  0.88] at (177.64, 25.68) {2.5};

\node[text=drawColor,anchor=base,inner sep=0pt, outer sep=0pt, scale=  0.88] at (227.41, 25.68) {5.0};
\end{scope}

\begin{scope}
\path[clip] (  0.00,  0.00) rectangle (252.94,252.94);
\definecolor{drawColor}{RGB}{0,0,0}

\node[text=drawColor,rotate= 90.00,anchor=base,inner sep=0pt, outer sep=0pt, scale=  1.10] at ( 13.08,130.49) {Density};
\end{scope}
\begin{scope}
\path[clip] (  0.00,  0.00) rectangle (252.94,252.94);
\definecolor{drawColor}{RGB}{0,0,0}

\node[text=drawColor,anchor=base,inner sep=0pt, outer sep=0pt, scale=  1.32] at (140.80,238.35) {\small{Sparse Erd\H{o}s-R\'enyi}};
\end{scope}
\end{tikzpicture}}
        \small{(a)}
    \end{subfigure}
    \hfill
    \begin{subfigure}[b]{0.48\textwidth}
        \centering
        \resizebox{\linewidth}{!}{
\begin{tikzpicture}[x=1pt,y=1pt]
\definecolor{fillColor}{RGB}{255,255,255}
\path[use as bounding box,fill=fillColor,fill opacity=0.00] (0,0) rectangle (252.94,252.94);
\begin{scope}
\path[clip] ( 38.56, 30.69) rectangle (247.44,230.29);
\definecolor{drawColor}{gray}{0.92}

\path[draw=drawColor,line width= 0.3pt,line join=round] ( 38.56, 66.74) --
	(247.44, 66.74);

\path[draw=drawColor,line width= 0.3pt,line join=round] ( 38.56,120.69) --
	(247.44,120.69);

\path[draw=drawColor,line width= 0.3pt,line join=round] ( 38.56,174.64) --
	(247.44,174.64);

\path[draw=drawColor,line width= 0.3pt,line join=round] ( 38.56,228.60) --
	(247.44,228.60);

\path[draw=drawColor,line width= 0.3pt,line join=round] ( 60.70, 30.69) --
	( 60.70,230.29);

\path[draw=drawColor,line width= 0.3pt,line join=round] (115.57, 30.69) --
	(115.57,230.29);

\path[draw=drawColor,line width= 0.3pt,line join=round] (170.43, 30.69) --
	(170.43,230.29);

\path[draw=drawColor,line width= 0.3pt,line join=round] (225.30, 30.69) --
	(225.30,230.29);

\path[draw=drawColor,line width= 0.6pt,line join=round] ( 38.56, 39.76) --
	(247.44, 39.76);

\path[draw=drawColor,line width= 0.6pt,line join=round] ( 38.56, 93.71) --
	(247.44, 93.71);

\path[draw=drawColor,line width= 0.6pt,line join=round] ( 38.56,147.67) --
	(247.44,147.67);

\path[draw=drawColor,line width= 0.6pt,line join=round] ( 38.56,201.62) --
	(247.44,201.62);

\path[draw=drawColor,line width= 0.6pt,line join=round] ( 88.13, 30.69) --
	( 88.13,230.29);

\path[draw=drawColor,line width= 0.6pt,line join=round] (143.00, 30.69) --
	(143.00,230.29);

\path[draw=drawColor,line width= 0.6pt,line join=round] (197.87, 30.69) --
	(197.87,230.29);
\definecolor{drawColor}{RGB}{0,0,0}
\definecolor{fillColor}{RGB}{173,216,230}

\path[draw=drawColor,line width= 0.6pt,fill=fillColor,fill opacity=0.70] ( 48.05, 39.76) rectangle ( 60.71, 41.63);

\path[draw=drawColor,line width= 0.6pt,fill=fillColor,fill opacity=0.70] ( 60.71, 39.76) rectangle ( 73.37, 48.18);

\path[draw=drawColor,line width= 0.6pt,fill=fillColor,fill opacity=0.70] ( 73.37, 39.76) rectangle ( 86.03, 53.79);

\path[draw=drawColor,line width= 0.6pt,fill=fillColor,fill opacity=0.70] ( 86.03, 39.76) rectangle ( 98.69, 78.11);

\path[draw=drawColor,line width= 0.6pt,fill=fillColor,fill opacity=0.70] ( 98.69, 39.76) rectangle (111.35,105.23);

\path[draw=drawColor,line width= 0.6pt,fill=fillColor,fill opacity=0.70] (111.35, 39.76) rectangle (124.01,156.68);

\path[draw=drawColor,line width= 0.6pt,fill=fillColor,fill opacity=0.70] (124.01, 39.76) rectangle (136.67,186.61);

\path[draw=drawColor,line width= 0.6pt,fill=fillColor,fill opacity=0.70] (136.67, 39.76) rectangle (149.33,221.21);

\path[draw=drawColor,line width= 0.6pt,fill=fillColor,fill opacity=0.70] (149.33, 39.76) rectangle (161.99,181.93);

\path[draw=drawColor,line width= 0.6pt,fill=fillColor,fill opacity=0.70] (161.99, 39.76) rectangle (174.65,153.87);

\path[draw=drawColor,line width= 0.6pt,fill=fillColor,fill opacity=0.70] (174.65, 39.76) rectangle (187.31,104.30);

\path[draw=drawColor,line width= 0.6pt,fill=fillColor,fill opacity=0.70] (187.31, 39.76) rectangle (199.97, 67.82);

\path[draw=drawColor,line width= 0.6pt,fill=fillColor,fill opacity=0.70] (199.97, 39.76) rectangle (212.63, 48.18);

\path[draw=drawColor,line width= 0.6pt,fill=fillColor,fill opacity=0.70] (212.63, 39.76) rectangle (225.29, 40.69);

\path[draw=drawColor,line width= 0.6pt,fill=fillColor,fill opacity=0.70] (225.29, 39.76) rectangle (237.95, 43.50);
\definecolor{drawColor}{RGB}{255,0,0}

\path[draw=drawColor,line width= 1.4pt,line join=round] ( 57.78, 41.20) --
	( 59.56, 41.51) --
	( 61.33, 41.88) --
	( 63.10, 42.32) --
	( 64.87, 42.84) --
	( 66.65, 43.45) --
	( 68.42, 44.16) --
	( 70.19, 44.99) --
	( 71.96, 45.95) --
	( 73.74, 47.05) --
	( 75.51, 48.31) --
	( 77.28, 49.74) --
	( 79.05, 51.37) --
	( 80.83, 53.22) --
	( 82.60, 55.29) --
	( 84.37, 57.60) --
	( 86.14, 60.17) --
	( 87.92, 63.02) --
	( 89.69, 66.15) --
	( 91.46, 69.59) --
	( 93.23, 73.33) --
	( 95.01, 77.39) --
	( 96.78, 81.76) --
	( 98.55, 86.45) --
	(100.32, 91.44) --
	(102.09, 96.74) --
	(103.87,102.31) --
	(105.64,108.15) --
	(107.41,114.23) --
	(109.18,120.51) --
	(110.96,126.96) --
	(112.73,133.53) --
	(114.50,140.19) --
	(116.27,146.87) --
	(118.05,153.53) --
	(119.82,160.10) --
	(121.59,166.52) --
	(123.36,172.74) --
	(125.14,178.69) --
	(126.91,184.30) --
	(128.68,189.52) --
	(130.45,194.29) --
	(132.23,198.55) --
	(134.00,202.26) --
	(135.77,205.37) --
	(137.54,207.84) --
	(139.32,209.64) --
	(141.09,210.75) --
	(142.86,211.16) --
	(144.63,210.86) --
	(146.40,209.86) --
	(148.18,208.16) --
	(149.95,205.80) --
	(151.72,202.79) --
	(153.49,199.18) --
	(155.27,195.00) --
	(157.04,190.31) --
	(158.81,185.15) --
	(160.58,179.60) --
	(162.36,173.70) --
	(164.13,167.52) --
	(165.90,161.12) --
	(167.67,154.57) --
	(169.45,147.92) --
	(171.22,141.24) --
	(172.99,134.58) --
	(174.76,127.99) --
	(176.54,121.51) --
	(178.31,115.20) --
	(180.08,109.09) --
	(181.85,103.22) --
	(183.63, 97.60) --
	(185.40, 92.26) --
	(187.17, 87.21) --
	(188.94, 82.48) --
	(190.71, 78.06) --
	(192.49, 73.95) --
	(194.26, 70.16) --
	(196.03, 66.68) --
	(197.80, 63.49) --
	(199.58, 60.60) --
	(201.35, 57.99) --
	(203.12, 55.63) --
	(204.89, 53.53) --
	(206.67, 51.65) --
	(208.44, 49.99) --
	(210.21, 48.52) --
	(211.98, 47.23) --
	(213.76, 46.11) --
	(215.53, 45.13) --
	(217.30, 44.29) --
	(219.07, 43.56) --
	(220.85, 42.93) --
	(222.62, 42.40) --
	(224.39, 41.95) --
	(226.16, 41.57) --
	(227.94, 41.24) --
	(229.71, 40.97) --
	(231.48, 40.75) --
	(233.25, 40.56) --
	(235.02, 40.41);
\end{scope}
\begin{scope}
\path[clip] (  0.00,  0.00) rectangle (252.94,252.94);
\definecolor{drawColor}{gray}{0.30}

\node[text=drawColor,anchor=base east,inner sep=0pt, outer sep=0pt, scale=  0.88] at ( 33.61, 36.73) {0.00};

\node[text=drawColor,anchor=base east,inner sep=0pt, outer sep=0pt, scale=  0.88] at ( 33.61, 90.68) {0.05};

\node[text=drawColor,anchor=base east,inner sep=0pt, outer sep=0pt, scale=  0.88] at ( 33.61,144.64) {0.10};

\node[text=drawColor,anchor=base east,inner sep=0pt, outer sep=0pt, scale=  0.88] at ( 33.61,198.59) {0.15};
\end{scope}
\begin{scope}
\path[clip] (  0.00,  0.00) rectangle (252.94,252.94);
\definecolor{drawColor}{gray}{0.30}

\node[text=drawColor,anchor=base,inner sep=0pt, outer sep=0pt, scale=  0.88] at ( 88.13, 19.68) {-5};

\node[text=drawColor,anchor=base,inner sep=0pt, outer sep=0pt, scale=  0.88] at (143.00, 19.68) {0};

\node[text=drawColor,anchor=base,inner sep=0pt, outer sep=0pt, scale=  0.88] at (197.87, 19.68) {5};
\end{scope}

\begin{scope}
\path[clip] (  0.00,  0.00) rectangle (252.94,252.94);
\definecolor{drawColor}{RGB}{0,0,0}

\node[text=drawColor,rotate= 90.00,anchor=base,inner sep=0pt, outer sep=0pt, scale=  1.10] at ( 13.08,130.49) {Density};
\end{scope}
\begin{scope}
\path[clip] (  0.00,  0.00) rectangle (252.94,252.94);
\definecolor{drawColor}{RGB}{0,0,0}

\node[text=drawColor,anchor=base,inner sep=0pt, outer sep=0pt, scale=  1.32] at (143.00,238.35) {\small{$W(x,y)=\sin(|x-y|)$}};
\end{scope}
\end{tikzpicture}}
        \small{(b)}
    \end{subfigure}
        \hfill
    \begin{subfigure}[b]{0.48\textwidth}
        \centering
        \resizebox{\linewidth}{!}
        {
\begin{tikzpicture}[x=1pt,y=1pt]
\definecolor{fillColor}{RGB}{255,255,255}
\path[use as bounding box,fill=fillColor,fill opacity=0.00] (0,0) rectangle (252.94,252.94);
\begin{scope}
\path[clip] ( 38.56, 30.69) rectangle (247.44,230.29);
\definecolor{drawColor}{gray}{0.92}

\path[draw=drawColor,line width= 0.3pt,line join=round] ( 38.56, 61.98) --
	(247.44, 61.98);

\path[draw=drawColor,line width= 0.3pt,line join=round] ( 38.56,106.41) --
	(247.44,106.41);

\path[draw=drawColor,line width= 0.3pt,line join=round] ( 38.56,150.85) --
	(247.44,150.85);

\path[draw=drawColor,line width= 0.3pt,line join=round] ( 38.56,195.29) --
	(247.44,195.29);

\path[draw=drawColor,line width= 0.3pt,line join=round] ( 74.88, 30.69) --
	( 74.88,230.29);

\path[draw=drawColor,line width= 0.3pt,line join=round] (128.73, 30.69) --
	(128.73,230.29);

\path[draw=drawColor,line width= 0.3pt,line join=round] (182.59, 30.69) --
	(182.59,230.29);

\path[draw=drawColor,line width= 0.3pt,line join=round] (236.44, 30.69) --
	(236.44,230.29);

\path[draw=drawColor,line width= 0.6pt,line join=round] ( 38.56, 39.76) --
	(247.44, 39.76);

\path[draw=drawColor,line width= 0.6pt,line join=round] ( 38.56, 84.20) --
	(247.44, 84.20);

\path[draw=drawColor,line width= 0.6pt,line join=round] ( 38.56,128.63) --
	(247.44,128.63);

\path[draw=drawColor,line width= 0.6pt,line join=round] ( 38.56,173.07) --
	(247.44,173.07);

\path[draw=drawColor,line width= 0.6pt,line join=round] ( 38.56,217.51) --
	(247.44,217.51);

\path[draw=drawColor,line width= 0.6pt,line join=round] ( 47.96, 30.69) --
	( 47.96,230.29);

\path[draw=drawColor,line width= 0.6pt,line join=round] (101.81, 30.69) --
	(101.81,230.29);

\path[draw=drawColor,line width= 0.6pt,line join=round] (155.66, 30.69) --
	(155.66,230.29);

\path[draw=drawColor,line width= 0.6pt,line join=round] (209.51, 30.69) --
	(209.51,230.29);
\definecolor{drawColor}{RGB}{0,0,0}
\definecolor{fillColor}{RGB}{173,216,230}

\path[draw=drawColor,line width= 0.6pt,fill=fillColor,fill opacity=0.70] ( 48.05, 39.76) rectangle ( 60.71, 40.70);

\path[draw=drawColor,line width= 0.6pt,fill=fillColor,fill opacity=0.70] ( 60.71, 39.76) rectangle ( 73.37, 41.65);

\path[draw=drawColor,line width= 0.6pt,fill=fillColor,fill opacity=0.70] ( 73.37, 39.76) rectangle ( 86.03, 41.65);

\path[draw=drawColor,line width= 0.6pt,fill=fillColor,fill opacity=0.70] ( 86.03, 39.76) rectangle ( 98.69, 50.15);

\path[draw=drawColor,line width= 0.6pt,fill=fillColor,fill opacity=0.70] ( 98.69, 39.76) rectangle (111.35, 73.78);

\path[draw=drawColor,line width= 0.6pt,fill=fillColor,fill opacity=0.70] (111.35, 39.76) rectangle (124.01, 98.35);

\path[draw=drawColor,line width= 0.6pt,fill=fillColor,fill opacity=0.70] (124.01, 39.76) rectangle (136.67,139.94);

\path[draw=drawColor,line width= 0.6pt,fill=fillColor,fill opacity=0.70] (136.67, 39.76) rectangle (149.33,192.86);

\path[draw=drawColor,line width= 0.6pt,fill=fillColor,fill opacity=0.70] (149.33, 39.76) rectangle (161.99,221.21);

\path[draw=drawColor,line width= 0.6pt,fill=fillColor,fill opacity=0.70] (161.99, 39.76) rectangle (174.65,191.92);

\path[draw=drawColor,line width= 0.6pt,fill=fillColor,fill opacity=0.70] (174.65, 39.76) rectangle (187.31,161.67);

\path[draw=drawColor,line width= 0.6pt,fill=fillColor,fill opacity=0.70] (187.31, 39.76) rectangle (199.97,117.26);

\path[draw=drawColor,line width= 0.6pt,fill=fillColor,fill opacity=0.70] (199.97, 39.76) rectangle (212.63, 73.78);

\path[draw=drawColor,line width= 0.6pt,fill=fillColor,fill opacity=0.70] (212.63, 39.76) rectangle (225.29, 50.15);

\path[draw=drawColor,line width= 0.6pt,fill=fillColor,fill opacity=0.70] (225.29, 39.76) rectangle (237.95, 46.37);
\definecolor{drawColor}{RGB}{255,0,0}

\path[draw=drawColor,line width= 1.4pt,line join=round] ( 57.41, 39.99) --
	( 59.18, 40.05) --
	( 60.95, 40.12) --
	( 62.73, 40.22) --
	( 64.50, 40.33) --
	( 66.27, 40.48) --
	( 68.04, 40.65) --
	( 69.82, 40.86) --
	( 71.59, 41.11) --
	( 73.36, 41.42) --
	( 75.13, 41.78) --
	( 76.91, 42.22) --
	( 78.68, 42.73) --
	( 80.45, 43.34) --
	( 82.22, 44.06) --
	( 84.00, 44.89) --
	( 85.77, 45.86) --
	( 87.54, 46.98) --
	( 89.31, 48.27) --
	( 91.09, 49.75) --
	( 92.86, 51.43) --
	( 94.63, 53.34) --
	( 96.40, 55.49) --
	( 98.17, 57.91) --
	( 99.95, 60.60) --
	(101.72, 63.59) --
	(103.49, 66.89) --
	(105.26, 70.52) --
	(107.04, 74.47) --
	(108.81, 78.77) --
	(110.58, 83.41) --
	(112.35, 88.39) --
	(114.13, 93.70) --
	(115.90, 99.34) --
	(117.67,105.28) --
	(119.44,111.50) --
	(121.22,117.98) --
	(122.99,124.66) --
	(124.76,131.52) --
	(126.53,138.51) --
	(128.31,145.57) --
	(130.08,152.64) --
	(131.85,159.67) --
	(133.62,166.58) --
	(135.40,173.31) --
	(137.17,179.79) --
	(138.94,185.95) --
	(140.71,191.72) --
	(142.48,197.03) --
	(144.26,201.83) --
	(146.03,206.05) --
	(147.80,209.65) --
	(149.57,212.57) --
	(151.35,214.78) --
	(153.12,216.25) --
	(154.89,216.96) --
	(156.66,216.91) --
	(158.44,216.10) --
	(160.21,214.52) --
	(161.98,212.22) --
	(163.75,209.21) --
	(165.53,205.53) --
	(167.30,201.23) --
	(169.07,196.36) --
	(170.84,190.98) --
	(172.62,185.15) --
	(174.39,178.95) --
	(176.16,172.43) --
	(177.93,165.67) --
	(179.71,158.74) --
	(181.48,151.70) --
	(183.25,144.63) --
	(185.02,137.58) --
	(186.79,130.60) --
	(188.57,123.76) --
	(190.34,117.10) --
	(192.11,110.66) --
	(193.88,104.48) --
	(195.66, 98.57) --
	(197.43, 92.98) --
	(199.20, 87.71) --
	(200.97, 82.77) --
	(202.75, 78.18) --
	(204.52, 73.93) --
	(206.29, 70.02) --
	(208.06, 66.43) --
	(209.84, 63.18) --
	(211.61, 60.23) --
	(213.38, 57.57) --
	(215.15, 55.19) --
	(216.93, 53.08) --
	(218.70, 51.20) --
	(220.47, 49.54) --
	(222.24, 48.09) --
	(224.02, 46.82) --
	(225.79, 45.72) --
	(227.56, 44.77) --
	(229.33, 43.96) --
	(231.10, 43.26) --
	(232.88, 42.66) --
	(234.65, 42.16);
\end{scope}
\begin{scope}
\path[clip] (  0.00,  0.00) rectangle (252.94,252.94);
\definecolor{drawColor}{gray}{0.30}

\node[text=drawColor,anchor=base east,inner sep=0pt, outer sep=0pt, scale=  0.88] at ( 33.61, 36.73) {0.00};

\node[text=drawColor,anchor=base east,inner sep=0pt, outer sep=0pt, scale=  0.88] at ( 33.61, 81.17) {0.05};

\node[text=drawColor,anchor=base east,inner sep=0pt, outer sep=0pt, scale=  0.88] at ( 33.61,125.60) {0.10};

\node[text=drawColor,anchor=base east,inner sep=0pt, outer sep=0pt, scale=  0.88] at ( 33.61,170.04) {0.15};

\node[text=drawColor,anchor=base east,inner sep=0pt, outer sep=0pt, scale=  0.88] at ( 33.61,214.48) {0.20};
\end{scope}
\begin{scope}
\path[clip] (  0.00,  0.00) rectangle (252.94,252.94);
\definecolor{drawColor}{gray}{0.30}

\node[text=drawColor,anchor=base,inner sep=0pt, outer sep=0pt, scale=  0.88] at ( 47.96, 19.68) {-8};

\node[text=drawColor,anchor=base,inner sep=0pt, outer sep=0pt, scale=  0.88] at (101.81, 19.68) {-4};

\node[text=drawColor,anchor=base,inner sep=0pt, outer sep=0pt, scale=  0.88] at (155.66, 19.68) {0};

\node[text=drawColor,anchor=base,inner sep=0pt, outer sep=0pt, scale=  0.88] at (209.51, 19.68) {4};
\end{scope}

\begin{scope}
\path[clip] (  0.00,  0.00) rectangle (252.94,252.94);
\definecolor{drawColor}{RGB}{0,0,0}

\node[text=drawColor,rotate= 90.00,anchor=base,inner sep=0pt, outer sep=0pt, scale=  1.10] at ( 13.08,130.49) {Density};
\end{scope}
\begin{scope}
\path[clip] (  0.00,  0.00) rectangle (252.94,252.94);
\definecolor{drawColor}{RGB}{0,0,0}

\node[text=drawColor,anchor=base,inner sep=0pt, outer sep=0pt, scale=  1.32] at (143.00,238.35) {\small{Random bipartite}};
\end{scope}
\end{tikzpicture}}
        \small{(c)}
    \end{subfigure}
        \hfill
    \begin{subfigure}[b]{0.48\textwidth}
        \centering
        \resizebox{\linewidth}{!}
        {
\begin{tikzpicture}[x=1pt,y=1pt]
\definecolor{fillColor}{RGB}{255,255,255}
\path[use as bounding box,fill=fillColor,fill opacity=0.00] (0,0) rectangle (252.94,252.94);
\begin{scope}
\path[clip] ( 38.56, 30.69) rectangle (247.44,230.29);
\definecolor{drawColor}{gray}{0.92}

\path[draw=drawColor,line width= 0.3pt,line join=round] ( 38.56, 71.92) --
	(247.44, 71.92);

\path[draw=drawColor,line width= 0.3pt,line join=round] ( 38.56,136.25) --
	(247.44,136.25);

\path[draw=drawColor,line width= 0.3pt,line join=round] ( 38.56,200.58) --
	(247.44,200.58);

\path[draw=drawColor,line width= 0.3pt,line join=round] ( 60.38, 30.69) --
	( 60.38,230.29);

\path[draw=drawColor,line width= 0.3pt,line join=round] (115.46, 30.69) --
	(115.46,230.29);

\path[draw=drawColor,line width= 0.3pt,line join=round] (170.54, 30.69) --
	(170.54,230.29);

\path[draw=drawColor,line width= 0.3pt,line join=round] (225.62, 30.69) --
	(225.62,230.29);

\path[draw=drawColor,line width= 0.6pt,line join=round] ( 38.56, 39.76) --
	(247.44, 39.76);

\path[draw=drawColor,line width= 0.6pt,line join=round] ( 38.56,104.09) --
	(247.44,104.09);

\path[draw=drawColor,line width= 0.6pt,line join=round] ( 38.56,168.42) --
	(247.44,168.42);

\path[draw=drawColor,line width= 0.6pt,line join=round] ( 87.92, 30.69) --
	( 87.92,230.29);

\path[draw=drawColor,line width= 0.6pt,line join=round] (143.00, 30.69) --
	(143.00,230.29);

\path[draw=drawColor,line width= 0.6pt,line join=round] (198.08, 30.69) --
	(198.08,230.29);
\definecolor{drawColor}{RGB}{0,0,0}
\definecolor{fillColor}{RGB}{173,216,230}

\path[draw=drawColor,line width= 0.6pt,fill=fillColor,fill opacity=0.70] ( 48.05, 39.76) rectangle ( 60.71, 42.00);

\path[draw=drawColor,line width= 0.6pt,fill=fillColor,fill opacity=0.70] ( 60.71, 39.76) rectangle ( 73.37, 47.59);

\path[draw=drawColor,line width= 0.6pt,fill=fillColor,fill opacity=0.70] ( 73.37, 39.76) rectangle ( 86.03, 61.03);

\path[draw=drawColor,line width= 0.6pt,fill=fillColor,fill opacity=0.70] ( 86.03, 39.76) rectangle ( 98.69,102.45);

\path[draw=drawColor,line width= 0.6pt,fill=fillColor,fill opacity=0.70] ( 98.69, 39.76) rectangle (111.35,131.56);

\path[draw=drawColor,line width= 0.6pt,fill=fillColor,fill opacity=0.70] (111.35, 39.76) rectangle (124.01,175.21);

\path[draw=drawColor,line width= 0.6pt,fill=fillColor,fill opacity=0.70] (124.01, 39.76) rectangle (136.67,207.68);

\path[draw=drawColor,line width= 0.6pt,fill=fillColor,fill opacity=0.70] (136.67, 39.76) rectangle (149.33,216.64);

\path[draw=drawColor,line width= 0.6pt,fill=fillColor,fill opacity=0.70] (149.33, 39.76) rectangle (161.99,200.96);

\path[draw=drawColor,line width= 0.6pt,fill=fillColor,fill opacity=0.70] (161.99, 39.76) rectangle (174.65,157.30);

\path[draw=drawColor,line width= 0.6pt,fill=fillColor,fill opacity=0.70] (174.65, 39.76) rectangle (187.31,117.00);

\path[draw=drawColor,line width= 0.6pt,fill=fillColor,fill opacity=0.70] (187.31, 39.76) rectangle (199.97, 94.61);

\path[draw=drawColor,line width= 0.6pt,fill=fillColor,fill opacity=0.70] (199.97, 39.76) rectangle (212.63, 71.10);

\path[draw=drawColor,line width= 0.6pt,fill=fillColor,fill opacity=0.70] (212.63, 39.76) rectangle (225.29, 49.83);

\path[draw=drawColor,line width= 0.6pt,fill=fillColor,fill opacity=0.70] (225.29, 39.76) rectangle (237.95, 40.88);
\definecolor{drawColor}{RGB}{255,0,0}

\path[draw=drawColor,line width= 1.4pt,line join=round] ( 54.74, 43.04) --
	( 56.51, 43.61) --
	( 58.29, 44.26) --
	( 60.06, 45.01) --
	( 61.83, 45.85) --
	( 63.60, 46.82) --
	( 65.37, 47.90) --
	( 67.15, 49.13) --
	( 68.92, 50.50) --
	( 70.69, 52.04) --
	( 72.46, 53.75) --
	( 74.24, 55.65) --
	( 76.01, 57.74) --
	( 77.78, 60.05) --
	( 79.55, 62.58) --
	( 81.33, 65.34) --
	( 83.10, 68.35) --
	( 84.87, 71.60) --
	( 86.64, 75.11) --
	( 88.42, 78.87) --
	( 90.19, 82.90) --
	( 91.96, 87.19) --
	( 93.73, 91.74) --
	( 95.51, 96.54) --
	( 97.28,101.59) --
	( 99.05,106.86) --
	(100.82,112.35) --
	(102.60,118.03) --
	(104.37,123.89) --
	(106.14,129.89) --
	(107.91,136.01) --
	(109.68,142.21) --
	(111.46,148.46) --
	(113.23,154.72) --
	(115.00,160.95) --
	(116.77,167.10) --
	(118.55,173.12) --
	(120.32,178.99) --
	(122.09,184.64) --
	(123.86,190.03) --
	(125.64,195.13) --
	(127.41,199.87) --
	(129.18,204.23) --
	(130.95,208.16) --
	(132.73,211.62) --
	(134.50,214.60) --
	(136.27,217.04) --
	(138.04,218.95) --
	(139.82,220.28) --
	(141.59,221.04) --
	(143.36,221.21) --
	(145.13,220.80) --
	(146.91,219.81) --
	(148.68,218.24) --
	(150.45,216.11) --
	(152.22,213.45) --
	(153.99,210.27) --
	(155.77,206.61) --
	(157.54,202.51) --
	(159.31,197.99) --
	(161.08,193.09) --
	(162.86,187.87) --
	(164.63,182.37) --
	(166.40,176.63) --
	(168.17,170.69) --
	(169.95,164.61) --
	(171.72,158.42) --
	(173.49,152.18) --
	(175.26,145.92) --
	(177.04,139.68) --
	(178.81,133.51) --
	(180.58,127.43) --
	(182.35,121.49) --
	(184.13,115.70) --
	(185.90,110.09) --
	(187.67,104.69) --
	(189.44, 99.51) --
	(191.22, 94.56) --
	(192.99, 89.86) --
	(194.76, 85.42) --
	(196.53, 81.23) --
	(198.30, 77.31) --
	(200.08, 73.65) --
	(201.85, 70.24) --
	(203.62, 67.09) --
	(205.39, 64.19) --
	(207.17, 61.53) --
	(208.94, 59.09) --
	(210.71, 56.87) --
	(212.48, 54.85) --
	(214.26, 53.03) --
	(216.03, 51.40) --
	(217.80, 49.93) --
	(219.57, 48.61) --
	(221.35, 47.45) --
	(223.12, 46.41) --
	(224.89, 45.50) --
	(226.66, 44.69) --
	(228.44, 43.99) --
	(230.21, 43.37) --
	(231.98, 42.83);
\end{scope}
\begin{scope}
\path[clip] (  0.00,  0.00) rectangle (252.94,252.94);
\definecolor{drawColor}{gray}{0.30}

\node[text=drawColor,anchor=base east,inner sep=0pt, outer sep=0pt, scale=  0.88] at ( 33.61, 36.73) {0.00};

\node[text=drawColor,anchor=base east,inner sep=0pt, outer sep=0pt, scale=  0.88] at ( 33.61,101.06) {0.05};

\node[text=drawColor,anchor=base east,inner sep=0pt, outer sep=0pt, scale=  0.88] at ( 33.61,165.39) {0.10};
\end{scope}
\begin{scope}
\path[clip] (  0.00,  0.00) rectangle (252.94,252.94);
\definecolor{drawColor}{gray}{0.30}

\node[text=drawColor,anchor=base,inner sep=0pt, outer sep=0pt, scale=  0.88] at ( 87.92, 19.68) {-5};

\node[text=drawColor,anchor=base,inner sep=0pt, outer sep=0pt, scale=  0.88] at (143.00, 19.68) {0};

\node[text=drawColor,anchor=base,inner sep=0pt, outer sep=0pt, scale=  0.88] at (198.08, 19.68) {5};
\end{scope}

\begin{scope}
\path[clip] (  0.00,  0.00) rectangle (252.94,252.94);
\definecolor{drawColor}{RGB}{0,0,0}

\node[text=drawColor,rotate= 90.00,anchor=base,inner sep=0pt, outer sep=0pt, scale=  1.10] at ( 13.08,130.49) {Density};
\end{scope}
\begin{scope}
\path[clip] (  0.00,  0.00) rectangle (252.94,252.94);
\definecolor{drawColor}{RGB}{0,0,0}

\node[text=drawColor,anchor=base,inner sep=0pt, outer sep=0pt, scale=  1.32] at (143.00,238.35) {\small{Rank one graphon}};
\end{scope}
\end{tikzpicture}}
        \small{(d)}
    \end{subfigure}
    \caption{\small{Histograms of $(\wbtu-\beta)/\sqrt{\theta_N}$ when $G_N$ is generated from (a) an Erd\H{o}s-R\'enyi model, (b) the graphon $W(x,y) = \sin(|x-y|)$, (c) $W$ as in \eqref{eq:Wpq} with $p=0$ and $q=1$ (which corresponds to a random bipartite graph), and (d) the rank one graphon $W(x,y) = x y$. The red curves represent the limiting normal densities from Theorem \ref{clttilde}. } }
    \label{fig:beta_tilde_comparisonI}
\end{figure}
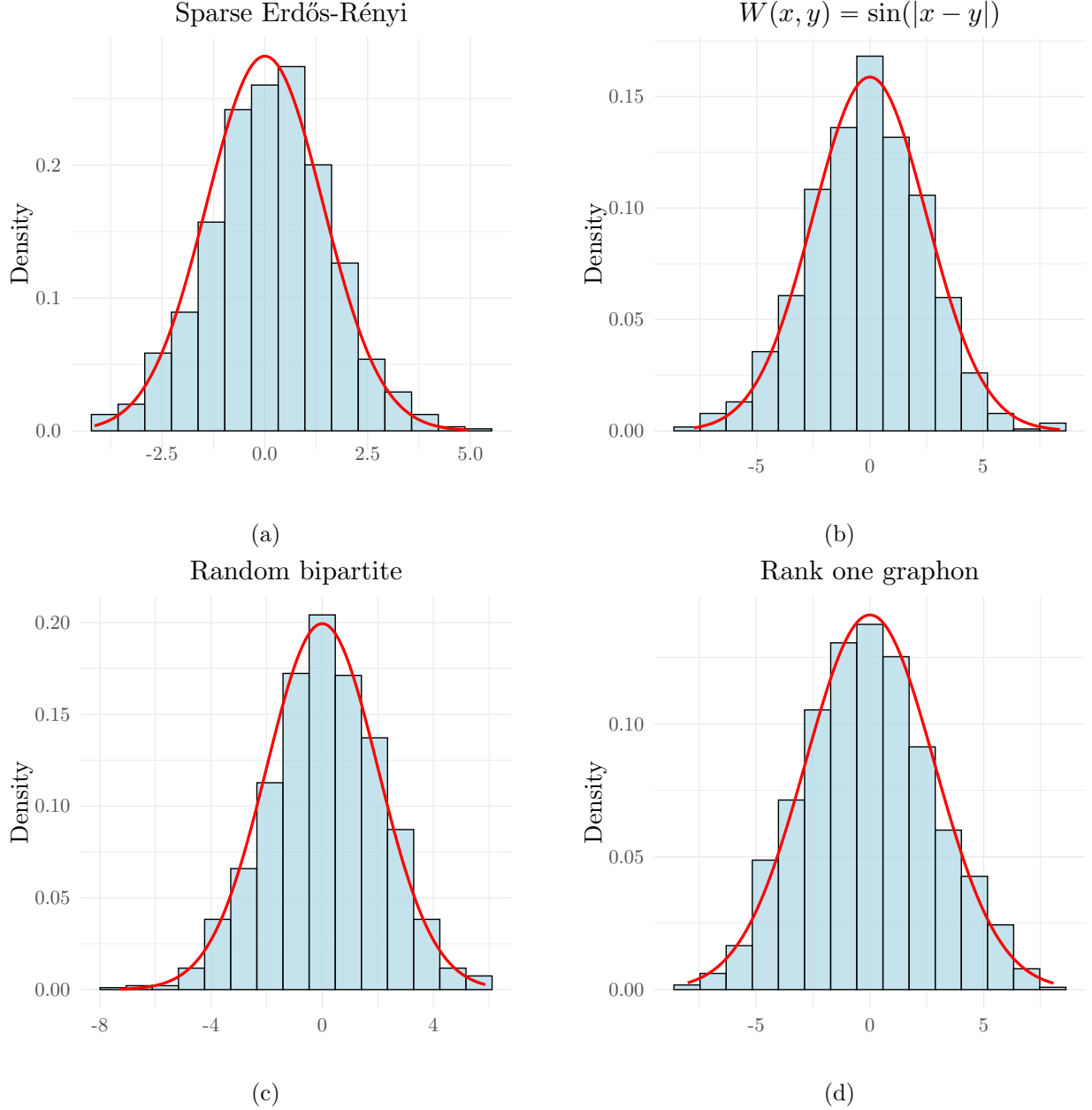

\subsection{Asymptotic Confidence Intervals} 
\label{sec:CN}

Using Theorem \ref{clttilde} we can readily construct a confidence interval for $\beta$, by replacing the asymptotic variance in \eqref{eq:cltestimate} by its empirical estimate. To this end, observe that 
$$\frac{\sum_{1 \leq i,j \leq N} A_{G_N}(i, j)}{ N^2 \theta_N } \pto \int_{[0,1]^2} W(x,y)\mathrm{d}x\mathrm{d}y , $$ 
when $N^2 \theta_N \gg 1$. This means, under the assumptions of Theorem \ref{clttilde}, 
$$\frac{1}{\theta_N}(\wbtu-\beta) \sqrt{\frac{\sum_{1 \leq i,j \leq N} A_{G_N}(i,j)}{2N^2}} \xrightarrow{D} \cN(0,1) . $$
Hence,   
    \begin{align}\label{wbtuconf}
       \mathcal{C}_{N}:=\left[~ \wbtu - z_{\alpha/2}\theta_N\sqrt{\frac{2N^2}{\sum_{1 \leq i,j \leq N} A_{G_N}(i,j)}}, ~\wbtu + z_{\alpha/2} \theta_N \sqrt{\frac{2N^2}{\sum_{1 \leq i,j \leq N} A_{G_N}(i,j)} }~\right]
    \end{align}
is an asymptotic level $1-\alpha$ confidence interval for $\beta \in (0, \frac{1}{\|W\|_{\mathrm{op}}})$ and $\theta_N \ll 1$,  where $z_\alpha$ denotes the $(1-\alpha)$-th quantile of $\cN(0,1)$, for $\alpha \in (0, 1)$.

Figure \ref{fig:conInt_comparison} shows $100$ instances of $95\%$ confidence intervals (computed as in \eqref{wbtuconf}) for the four graphons as in the previous section. The nominal level is set to $\alpha= 0.05$. The other experimental parameters remain the same as before. 
The intervals not containing the true parameter are shown in red.  The empirical coverage in the four cases are $95\%$, $92\%$, $93\%$ and $98\%$, respectively, which closely align  with the asymptotic result.

\begin{figure}[htbp]
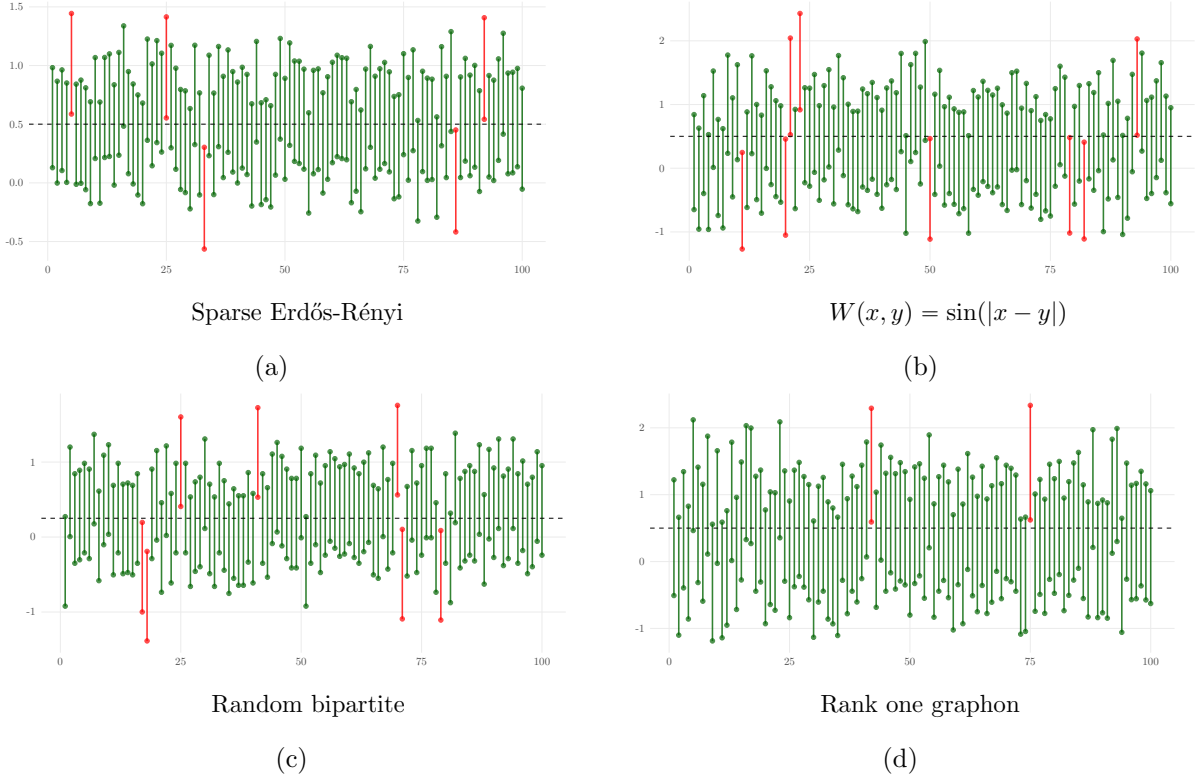

    \centering
    \begin{subfigure}[b]{0.48\textwidth}
        \centering
    \begin{adjustbox}{width=0.95\linewidth}
        \input{ci_beta0.5N500_ER.tex}
    \end{adjustbox}
        \caption*{(a)}
    \end{subfigure}
    \hfill
    \begin{subfigure}[b]{0.48\textwidth}
        \centering
        \begin{adjustbox}{width=0.95\linewidth}
        \input{ci_beta0.5N500_sinabsx-y.tex}
        \end{adjustbox}
        \caption*{(b)}
    \end{subfigure}
       \begin{subfigure}[b]{0.48\textwidth}
        \centering
    \begin{adjustbox}{width=0.95\linewidth}
\input{ci_beta0.25N500_Bipartite.tex}
    \end{adjustbox}
        \caption*{(c)}
    \end{subfigure}
       \begin{subfigure}[b]{0.48\textwidth}
        \centering
    \begin{adjustbox}{width=0.95\linewidth}
        \input{ci_beta0.5N500_XY.tex}
    \end{adjustbox}
        \caption*{(d)}
    \end{subfigure} 
    \caption{\small{ 100 instances of 95\% confidence intervals, where $G_N$ is generated from (a) an Erd\H{o}s-R\'enyi model, (b) the graphon $W(x,y) = \sin(|x-y|)$, (c) a random bipartite graph, and (d) the rank one graphon $W(x,y) = x y$. Here, $N = 500$ and  $\theta_N = N^{-0.6}$.  } }  
    \label{fig:conInt_comparison} 
\end{figure}

\subsection{Local Asymptotic Minimax Estimation}\label{sec:minmaxl2}

In the classical independent sampling paradigm it is well known that the asymptotic variance of the ML estimate is the best possible. Specifically, under appropriate regularity conditions, the ML estimate based on $N$ i.i.d. samples is {\it locally asymptotically minimax}, that is, it achieves the smallest possible asymptotic maximum risk among all estimates in shrinking neighborhoods (typically of order $1/\sqrt N$) of the true parameter \cite[Chapter 8]{Vaart_1998}. In this context, it is natural to ask whether similar optimality results can be established for Ising models. Classical techniques based on the central limit theorem break down in this case because of the dependence between the coordinates of $\bs$. Nonetheless, using our results on the asymptotic distributions of the sufficient statistic and the log-partition function, we show in the next theorem that both the ML estimate $\hat{\beta}_N$ and our proposed estimate $\wbtu$  are locally asymptotically minimax (over neighborhoods of size $\sqrt{\theta_N}$) throughout the entire subcritical regime.

\begin{thm}\label{locminest} 
Suppose Assumption \ref{assumption} holds and $\theta_N \ll 1$. Then, 
    \begin{align}\label{eq:locmin1}
         \sup_{I\in \mathscr{F}}~ \liminf_{N\rightarrow \infty}~ \sup_{h\in I}~\frac{1}{{\theta_N}}\E_{\beta+h\sqrt{\theta_N}} \left[\left( \bar{\beta}_N - \left(\beta + h\sqrt{\theta_N}\right)\right)^2 \right] \ge \frac{2}{\int_{[0,1]^2} W(x,y) \mathrm{d}x \mathrm{d}y} , 
    \end{align}  
    where $\bar{\beta}_N$ is any sequence of estimates in the model $\P_{\beta}$, and $\mathscr{F}$ denotes the collection of finite subsets of $\mathbb{R}$. Moreover, our proposed estimate $\wbtu$ satisfies:
    \begin{align}\label{eq:locmin2}
        \sup_{h\in I}~\frac{1}{{\theta_N}}\E_{\beta+h\sqrt{\theta_N}} \left[\left(\wbtu - \left(\beta + h\sqrt{\theta_N}\right)\right)^2 \Bigg| A_{G_N}\right]\xrightarrow{P} \frac{2}{\int_{[0,1]^2} W(x,y) \mathrm{d}x \mathrm{d}y} , 
    \end{align}
 for all $I \in \mathscr{F}$.
\end{thm}

The proof of Theorem \ref{locminest} is given in Appendix \ref{sec:proof_locmintest}. To the best of our knowledge, Theorem \ref{locminest}  is the first instance of a sharp asymptotic optimality result in the context of Ising models (more generally,  for Markov random fields). The only related prior result is \cite[Theorem 3.1]{MukherjeeRay2022}, where a minimax lower bound is established for estimating $\beta$ in Ising models with general interaction matrices. In the setting of the sparse graphon model, their result implies a $1/\sqrt{\theta_N}$ lower bound on the rate of estimation (not the sharp constant) in a subregion of the subcritical regime. In contrast, Theorem \ref{locminest} shows that the ML estimate $\hat{\beta}_N$ and our proposed estimate $\wbtu$ are optimal both in terms of rate and asymptotic variance in the entire subcritical regime.

\section{Goodness-of-Fit Testing: Local Power and Minimax Detection Rates}\label{sec:gof}

In this section we consider the problem of goodness-of-fit testing for Ising models. In particular, given $\beta \in (0, \frac{1}{\|W\|_{\mathrm{op}}})$ and a sample $\bs \sim \P_{\beta}$ (as in \eqref{model_def}), our goal is to test the following hypothesis: 
\begin{align}\label{null_alt7}
    H_0: \beta =\beta_0\quad \text{versus}\quad H_1: \beta \ne \beta_0,
\end{align}
for some fixed $\beta_0 \in (0, \frac{1}{\|W\|_{\mathrm{op}}})$. This section is organized as follows: In Section \ref{sec:localpower} we derive the asymptotic local power of the likelihood ratio test and in Section \ref{sec:minimax} we establish its minimax optimality.

\subsection{Asymptotic Local Power}
\label{sec:localpower}

Note that for $0 < \beta_0 < \beta_1 < \frac{1}{\|W\|_{\mathrm{op}}}$, by the Neyman-Pearson lemma, the most powerful test for $H_0: \beta =\beta_0$ versus $H_1: \beta =\beta_1$, rejects $H_0$ for large values of the likelihood ratio:
\begin{align*}
    \frac{\P_{\beta_1}(\bs)}{\P_{\beta_0}(\bs)} &= \frac{Z_N(\beta_0)}{Z_N(\beta_1)} e^{ -(\beta_1-\beta_0)H_N(\bs)}\\ 
    &= (1+o_P(1)) C(\beta_0,\beta_1,W) e^{ -(\beta_1-\beta_0) H_N(\bs) + \frac{\beta_1^2-\beta_0^2}{4N^2\theta_N}\sum_{1 \leq i,j \leq N} W\left(\frac{i}{N},\frac{j}{N}\right) } , 
\end{align*}
for some constant $C(\beta_0,\beta_1,W)$, that depends on $\beta_0,\beta_1,W$.  
Hence, the test which rejects $H_0$ for small values of $H_N(\bs)$ is uniformly most powerful (UMP) for $H_0$ in \eqref{null_alt7} against the alternative $\beta > \beta_0$, and similarly, the test which rejects $H_0$ for large values of $H_N(\bs)$ is UMP for $H_0$ against the alternative $\beta < \beta_0$. Using our results on the sufficient statistic we can now derive the asymptotic properties of this test. We present the results in the sparse ($\theta_N \ll 1$) and dense  ($\theta_N \asymp 1$) cases separately in the following subsections:

\subsubsection{Sparse Regime}
\label{sec:sparseUMP}

Note that since $\wbtu$ is a function of $H_N(\bs)$, the likelihood ratio test with asymptotic level $\alpha$ in the sparse regime can be obtained by rejecting $H_0$ in \eqref{null_alt7} when $\beta_0 \notin \mathcal{C}_N$, where $\mathcal{C}_N$ is the $1-\alpha$ confidence interval in \eqref{wbtuconf}. This is equivalent to the test function: 
\begin{align}\label{eq:testcd} 
\phi_N^{\alpha} = \bm 1 \left\{ H_N(\bs) < \hat{c}_{\alpha} \text{ or } H_N(\bs) > \hat{d}_{\alpha}\right \},
\end{align}
 where  
$$\hat{c}_\alpha := -z_{\alpha/2}\sqrt{\frac{1}{2N^2\theta_N^2}\sum_{1 \leq i,j \leq N} A_{G_N}(i,j)} - \frac{\beta_0}{2N^2 \theta_N^2} \sum_{1 \leq i,j \leq N}A_{G_N}(i,j)$$
and 
$$\hat{d}_\alpha := z_{\alpha/2}\sqrt{\frac{1}{2N^2\theta_N^2}\sum_{1 \leq i,j \leq N} A_{G_N}(i,j)} - \frac{\beta_0}{2N^2 \theta_N^2} \sum_{1 \leq i,j \leq N}A_{G_N}(i,j).$$ 
Define the power function of the test by:  $$\rho_N(\beta) = \P_{\beta}(H_N(\bs) < \hat{c}_\alpha) + \P_{\beta}(H_N(\bs) > \hat{d}_\alpha).$$ The following result is an immediate consequence of the above discussion and Theorem \ref{thm:HNsigma}.

\begin{cor}\label{31cor266} Suppose Assumption \ref{assumption} holds and $\theta_N \ll 1$. Then the following hold: 
\begin{itemize} 

\item (Asymptotic level $\alpha$) $\lim_{N\rightarrow\infty} \rho_N(\beta_0)= \alpha$. 

\item (Consistent for fixed alternatives) For any $\beta \ne \beta_0$, $\lim_{N\rightarrow\infty} \rho_N(\beta)= 1$.  
\end{itemize}
\end{cor}

The above result shows that the likelihood ratio test can consistently detect any fixed alternative. In fact, since the fluctuations of the sufficient statistic are $\asymp \sqrt{\theta_N}$, the likelihood ratio test can consistently detect not only fixed alternatives, but also any sequence of local alternatives approaching $H_0$ at a rate slower than $\sqrt{\theta_N}$. Moreover, using the asymptotic distribution of the sufficient statistic established in Section \ref{sec:fluct_ham6} (see Theorem \ref{thm:HNsigma}), we derive in the following theorem the limiting local power for alternatives of order $\sqrt{\theta_N}$. The proof is given in Appendix \ref{sec:lppf}.

\begin{thm}\label{locpower}
Suppose Assumption \ref{assumption} holds and $\theta_N \ll 1$. Fix $t \in \R$ and consider testing the hypotheses:
    $$H_0: \beta = \beta_0\quad\text{versus} \quad H_{1,N}: \beta = \beta_0+t\sqrt{\theta_N}.$$ Then, 
\begin{align*} 
& \lim_{N\rightarrow \infty} \rho_{N}(\beta_0+t\sqrt{\theta_N}) \nonumber \\ 
& = \Phi\left(-z_{\alpha/2} - t\sqrt{\frac{1}{2} \int_{[0,1]^2} W(x,y) \mathrm{d}x \mathrm{d}y}\right) + \Phi\left(-z_{\alpha/2} + t\sqrt{\frac{1}{2} \int_{[0,1]^2} W(x,y) \mathrm{d}x \mathrm{d}y}\right).
\end{align*}
\end{thm}

\subsubsection{Dense Regime} 
\label{sec:densealpha}

Recall from Remark \ref{estimation} that the ML estimate is inconsistent in the dense regime (where $\theta_N \asymp 1$). In fact, \cite[Theorem~3.3]{bhattacharya2018inference} shows that no estimate of $\beta$ can be consistent in the subcritical phase for Ising models on dense graphs. This is a consequence of the fact that no test can consistently detect fixed alternatives in the subcritical phase for Ising models on dense graphs. Nonetheless, the test which rejects for small values of $H_N(\sigma)$ is still the UMP for testing the null hypothesis in \eqref{null_alt7} against one-sided alternatives. To choose the cutoff with the desired asymptotic level, recall from Theorem \ref{thm:HNsigma} (2), the asymptotic distribution of $H_N(\bs)$ in the dense regime. Specifically, if Assumption \ref{assumption} holds and $\theta_N \rightarrow \theta \in (0, 1]$, then 
 \begin{align}\label{eq:def_Ubeta}
    -H_N(\bs) \dto U_\beta:=\beta \sigma_W^2 + \frac{1}{2}\int_{0}^{1}W(x,x)\mathrm{d}x + \sigma_WZ_0+\frac{1}{2} \sum_{\lambda \in \mathrm{Spec}(W)} \lambda \left(\frac{Z_{\lambda}^2}{1-\beta \lambda}-1\right) , 
\end{align}
where $\sigma_W^2:=\frac{1}{2\theta} \int_{[0,1]^2} W(x,y)(1-\theta W(x,y))\mathrm{d}x\mathrm{d}y$. Hence, if the graphon $W$ is assumed to be known, the test that rejects $H_0$ when $-H_N(\bs)\leq q_{1-\alpha/2, \beta_0, W}$ or $-H_N(\bs) > q_{\alpha/2,\beta_0,W}$ is asymptotically level $\alpha$, where $q_{\zeta,\beta_0,W}$ is the $1-\zeta$-th quantile of $U_{\beta_0}$. Unlike in the sparse regime, the power of this test does not converge to $1$ at a fixed alternative, instead, it converges to a non-trivial limit (between $0$ and $1$), which can also be derived from \eqref{eq:def_Ubeta} (similar to \cite[Corollary~3.5]{bhattacharya2018inference}).

In practice, the underlying graphon is unknown, and we only observe the sampled graph $G_N$. Consequently, the procedure described above is no longer valid, since, in particular, the quantiles $q_{\alpha/2,\beta_0, W}$ and $q_{1-\alpha/2,\beta_0,W}$ depend on $W$. In this section,  we propose a multiplier bootstrap-based approach, a method for estimating
the quantiles of the limiting distribution $U_{\beta}$ (recall \eqref{eq:def_Ubeta}), based on the observed network $G_N$ itself and additional external randomness. To this end, for $\beta \in (0, \frac{1}{\|W\|_{\mathrm{op}}})$ define the following quantity: 
\begin{align}\label{eq:Uestimate}
    \hat U_\beta := \frac{\beta}{2N^2\theta_N^2}\sum_{1 \leq i \ne j \leq N} A_{G_N}(i,j) & +  \frac{1}{2N\theta_N}\sum_{i=1}^N A_{G_N}(i,i) \nonumber \\ 
    & + \frac{\beta^2}{2}\sum_{i=1}^{N}\frac{\hat\lambda_i(G_N)^3}{1-\beta\hat\lambda_i(G_N)} + \frac{1}{2}\sum_{i=1}^{N}\frac{\hat\lambda_i(G_N)}{1-\beta\hat\lambda_i(G_N)}\left(\eta_i^2-1\right) , 
\end{align}
where $\{\hat\lambda_i(G_N):1\leq i\leq N\}$ are the eigenvalues of $\frac{1}{N \theta_N} A_{G_N}$ arranged in non-increasing order and $\{\eta_i\}_{i\geq 1}$ are i.i.d. $\mathcal N(0, 1)$ random variables (independent of $G_N$). Note that \eqref{eq:Uestimate} depends only on the observed graph $G_N$ and the Gaussian multipliers $\eta_1, \eta_2, \ldots, \eta_N$, but not on the graphon $W$. In the next theorem we show that the distribution of $\hat U_\beta$ conditional on $G_N$ converges to $U_\beta$ (which is the limiting distribution of $H_N(\bs)$) in the dense regime.

\begin{thm}\label{thm:mb} Suppose Assumption \ref{assumption} holds and $\theta_N\rightarrow\theta\in(0,1]$. Then, as $N\rightarrow\infty$,
    \begin{align}\label{eq:Ubetaconvergence}
        \hat U_\beta\mid A_{G_N}\overset{d}{\rightarrow}U_\beta\text{ almost surely, }
    \end{align}
    where $\hat U_\beta$ and $U_\beta$ are defined in \eqref{eq:def_Ubeta} and \eqref{eq:Uestimate}, respectively. 
    \end{thm}

The proof of Theorem \ref{thm:mb} is given in Appendix \ref{sec:mbpf}. It shows that the asymptotic distribution of
$ \hat U_\beta\mid A_{G_N}$ is the same as that of the sufficient statistic $H_N(\bs)$, in the dense regime. Hence, we can use the distribution $\hat U_\beta\mid A_{G_N}$, which depends only on the observed graph $G_N$, to approximate the quantiles of the limiting distribution $H_N(\bs)$. Specifically, for $\beta \in (0, \frac{1}{\| W \|_\mathrm{op}})$, denote by $\hat q_{\zeta,\beta, G_N}$, the $(1-\zeta)$-th quantile of $\hat U_{\beta}|A_{G_N}$ and consider the test which rejects $H_0$ when 
\begin{align}\label{eq:densec}
\phi_N^+= \bm 1 \{-H_N(\bs) \leq \hat q_{1-\alpha/2,\beta_0, G_N} \text{ or }-H_N(\bs) > \hat q_{\alpha/2,\beta_0, G_N}\}. 
\end{align} 
The following corollary readily follows from Theorem \ref{thm:mb}, but for the sake of completeness, we prove it in Appendix \ref{sec:prooftestgraphon}.
\begin{cor}\label{cor:testgraphon}
Suppose Assumption \ref{assumption} holds and $\theta_N\rightarrow\theta\in(0,1]$. 
Then, for $\phi_N^+$ as in \eqref{eq:densec}, the following hold: 
 \begin{itemize} 
\item $\lim_{N\rightarrow\infty} \E_{H_0}(\phi_N^+)= \alpha$. 
\item For any $\beta \ne \beta_0$, $\lim_{N\rightarrow\infty} \E_{\beta}(\phi_N^+)= 1 + F_\beta(q_{1-\alpha/2, \beta_0, W}) - F_\beta(q_{\alpha/2, \beta_0, W})$, where $F_{\beta}$ denotes the cumulative distribution function of $U_\beta$. 
\end{itemize}
\end{cor}

Corollary \ref{cor:testgraphon} shows that the test $\phi_N^+$ (where the quantiles of the limiting distribution are estimated from the observed graph $G_N$) attains the same limiting power as the `oracle' test where $W$ is assumed to be known. In other words, $\phi_N$ adapts to the choice of $W$.

\subsection{Minimax Detection Rates} 
\label{sec:minimax}

We know from the results in Section \ref{sec:sparseUMP} that in the sparse regime, the likelihood ratio test can consistently detect deviations from $H_0$ which are of larger order than $\sqrt{\theta_N}$. In this section, we show this is optimal in a minimax sense, that is, no test can consistently detect deviations from $H_0$ which are of smaller order than $\sqrt{\theta_N}$. Specifically, we consider the testing problem: 
\begin{align}\label{perttest8}
  H_0 : \beta = \beta_0\quad\text{versus} \quad H_{1}: |\beta - \beta_{0}| \geq \delta_N , 
\end{align} 
for a sequence $\delta_N \ll 1$. For any test function $\phi_N :  \{-1,1\}^N \rightarrow \{0, 1\}$, its conditional worst-case risk (given the graph $G_N$) is defined as: 
\begin{align}\label{eq:RphiN}
\mathcal{R}(\phi_N) := \P_{H_0}(\phi_N(\bs) = 1|A_{G_N}) + \sup_{ \beta \in \Theta(\beta_0, \delta_N) } \p_{\beta}(\phi_N(\bs) = 0|A_{G_N}) , 
\end{align}   
where $\Theta(\beta_0, \delta_N):= \{ \beta \in (0, \frac{1}{\|W\|_{\mathrm{op}}}) : |\beta-\beta_0| \geq \delta_N \}$. (Note that in our setting $\mathcal{R}(\phi_N)$ is random, as it depends on the random graph $G_N$).

\begin{thm}\label{testminimax} Suppose Assumption \ref{assumption} holds and $\theta_N \ll 1$. Then, the following hold: 
    \begin{enumerate}
        \item[$(1)$] If $\delta_N \ll \sqrt{\theta_N}$, then $\inf_{\phi_N} \mathcal{R}(\phi_N) \pto 1.$ 
        \item[$(2)$] If $\delta_N \gg \sqrt{\theta_N}$, then for every fixed $\alpha \in (0,1)$, the likelihood ratio test $\phi_N^\alpha$  defined in \eqref{eq:testcd} satisfies: $\mathcal{R}(\phi_N^\alpha) \pto \alpha$. Consequently,
$\inf_{\phi_N} \mathcal{R}(\phi_N) \pto 0$, where the infimum is taken over all tests not necessarily of fixed asymptotic level.
    \end{enumerate}  
\end{thm}

Theorem \ref{testminimax} shows that the minimax detection threshold for the goodness-of-fit test in the subcritical regime is $\sqrt{\theta_N}$. The proof is given in Appendix \ref{proofminimax}. The proof of (2) uses the asymptotic distribution of $\tilde{\beta}_N$ from \eqref{eq:cltestimate}, combined with a monotone likelihood ratio property. The proof of (1) relies on a second-moment computation that uses the results on the sufficient statistic and the log-partition function from Section \ref{sec:ZH}.

\section{Limit of Experiments}\label{sec:limexp}

The notion of limit of experiments, introduced by Le Cam \cite{lecam1972limits}, is a fundamental concept in his asymptotic theory of statistical decision making. It describes how a sequence of statistical models, under suitable scaling, converges to a simpler canonical experiment that retains the essential information for asymptotic inference. It provides an asymptotic benchmark, in the sense that no sequence of procedures can asymptotically outperform the optimal procedure in the limit experiment. Below, we give the rigorous definition of limit of experiments.

\begin{defi}\label{definition:experiment}
    Suppose that for each $1\le N \le \infty$, $\{Q_{N,h}: h\in H\}$ is a family of probability measures on the measurable space $(\Omega_N,\mathscr{F}_N)$. We say that $\{Q_{N,h}: h\in H\}$ converges to $\{Q_{\infty,h}: h\in H\}$ as $N\rightarrow \infty$, and write $$Q_{N,h}\xrightarrow{\mathrm{Exp}} Q_{\infty,h},$$
    if for every $h_0 \in H$ and every finite $I\subseteq H$, we have:
    $$\left(\frac{\mathrm{d} Q_{N,h}}{\mathrm{d} Q_{N,h_0}}(X_N)\right)_{h\in I} \xrightarrow[X_N\sim Q_{N,h_0}]{d} \left(\frac{\mathrm{d} Q_{\infty,h}}{\mathrm{d} Q_{\infty,h_0}}(X)\right)_{h\in I}$$
    where $X \sim Q_{\infty,h_0}$. In particular, if $H \subseteq \mathbb{R}$ is an open set, and $Q_{\infty,h} = \cN(h,\sigma^2)$ for some $\sigma>0$, then we say that the family $\{Q_{N,h}: h\in H\}$ of experiments is \textit{locally asymptotically normal} (LAN). Note that if $X_{\infty,h}\sim Q_{\infty,h}$, then, as a matter of notational convenience, we shall sometimes write
$Q_{N,h}\xrightarrow{\mathrm{Exp}} X_{\infty,h}$
to mean that
$Q_{N,h}\xrightarrow{\mathrm{Exp}} Q_{\infty,h}.$
\end{defi}

Although the literature on limit of experiments for i.i.d. models is quite extensive (see, for example, \cite[Chapter 9]{Vaart_1998}), the corresponding theory for models with dependent outcomes remains scarce. The first attempt to establish limit experiments for Markov random fields was made in \cite{xu2023inference}, where the authors considered Ising models on dense degree-regular graphs in the supercritical and critical regimes. 
In the following theorem, we derive the limiting experiment for Ising models on sparse inhomogeneous random graphs, that allows for both sparse graphs and degree-heterogeneity, in the subcritical regime.

\begin{thm}\label{th:limexp} 
Suppose Assumption \ref{assumption} holds. Then the following hold: 
    \begin{enumerate}
        \item[$(1)$] (Sparse regime) If $\theta=0$, then $$\{\P_{\beta_0 + h\sqrt{\theta_N}}\}_{h\in \mathbb{R}} \xrightarrow{\mathrm{Exp}} \left\{\cN\left(h, \frac{2}{\int_{[0,1]^2} W(x,y) \mathrm{d}x \mathrm{d}y} \right)\right\}_{h\in \mathbb{R}}. $$  
        
        \item[$(2)$] (Dense regime) If $\theta>0$, then 
        $$\{\P_{\beta}\}_{\beta\in \left(0,\frac{1}{\|W\|_{\mathrm{op}}}\right)} \xrightarrow{\mathrm{Exp}} \left\{ V_{\beta} \right\}_{\beta\in \left(0,\frac{1}{\|W\|_{\mathrm{op}}}\right)} , $$
        where $V_{\beta}$ is as defined in \eqref{eq:Vbeta}. 
    \end{enumerate}
\end{thm}

The proof of Theorem \ref{th:limexp} is given in Appendix \ref{proof_limexp8}. The result shows that the limiting experiment is LAN in the sparse setting ($\theta_N\ll 1$). In contrast, in the dense setting ($\theta_N \asymp 1$), the limiting experiment is non-Gaussian. To the best of our knowledge, the latter is the first example of an experiment limit in the context of Markov random fields, where the limiting likelihood ratio does not admit a closed-form density.

\begin{remark}
    An immediate consequence of Theorem \ref{th:limexp} is that in the sparse regime, the measures $\P_{\beta+h\sqrt{\theta_N}}$ and $\P_{\beta+h'\sqrt{\theta_N}}$ are mutually contiguous for any $0<\beta<\frac{1}{\|W\|_{\mathrm{op}}}$ and $h, h'\in \mathbb{R}$. This follows from Le Cam's first lemma, since the likelihood ratio of two non-degenerate Gaussians is always positive with probability $1$. Hence, no consistent test exists for any simple versus simple hypotheses, where the separation between the null and the alternative is proportional to $\sqrt{\theta_N}$. This aligns with Theorem \ref{locpower}, which shows that the limiting local power of the likelihood ratio test for such hypotheses is strictly less than $1$. A similar conclusion also holds for hypotheses of the form \eqref{perttest8} in the dense regime, which gives an alternative proof of Theorem 3.3 (a) in \cite{bhattacharya2018inference}. 
\end{remark}

For illustration, as before, we now compute the limiting experiment in our running examples:

\begin{example}[Block models] 
Let $W$ be the 2-block graphon as in \eqref{eq:Wpq}. Suppose $G_N \sim G(N, \theta_N, W)$ and $\bs$ be a sample from the Ising model \eqref{model_def} on $G_N$ with $0 < \beta < \frac{2}{(p+q)}$ and $N^{\frac{2}{3}}\theta_N \gg 1$. Then from Theorem \ref{th:limexp}, we obtain the following:

\begin{itemize}
\item If $\theta=0$, then
\begin{align}\label{eq:sparseexample-limexp}
\{\P_{\beta_0+h\sqrt{\theta_N}}\}_{h\in\mathbb R}
\xrightarrow{\mathrm{Exp}}
\left\{\cN\left(h,\frac{4}{p+q}\right)\right\}_{h\in\mathbb R}.
\end{align}

\item If $\theta\in(0,1]$, then
\begin{align}\label{eq:denseexample-limexp}
\{\P_{\beta}\}_{\beta\in\left(0,\frac{2}{p+q}\right)}
\xrightarrow{\mathrm{Exp}}
\{V_\beta\}_{\beta\in\left(0,\frac{2}{p+q}\right)},
\end{align}
where
\begin{align*}
V_{\beta}
&= \beta\sigma_W^2 + \sigma_W Z_0 
+ \frac{\lambda_{+}}{2} \left(\frac{Z_{\lambda_{+}}^2}{1-\beta \lambda_{+}}-1\right)
+ \frac{\lambda_{-}}{2} \left(\frac{Z_{\lambda_{-}}^2}{1-\beta \lambda_{-}}-1\right),
\end{align*}
with $Z_0,Z_{\lambda_{+}},Z_{\lambda_{-}}$ i.i.d.\ $\cN(0,1)$, and
\begin{align*}
\sigma_W^2
= \frac{1}{4\theta} \left(p(1-\theta p) + q(1-\theta q)\right).
\end{align*}
\end{itemize}
As in Example \ref{example1_26}, we consider the following important special cases.
\begin{itemize}

\item {\it Erd\H{o}s-R\'enyi model $G(N,\theta_N)$:}
This corresponds to $W\equiv 1$, that is, $p=q=1$ in \eqref{eq:Wpq}. In this case,
$\lambda_{+}=1$, $\lambda_{-}=0$, and $\|W\|_{\mathrm{op}}=1$.
Hence, in the sparse regime \eqref{eq:sparseexample-limexp} simplifies to
\begin{align*}
\{\P_{\beta_0+h\sqrt{\theta_N}}\}_{h\in\mathbb R}
\xrightarrow{\mathrm{Exp}}
\left\{\cN(h,2)\right\}_{h\in\mathbb R}.
\end{align*}
Moreover, in the dense regime \eqref{eq:denseexample-limexp} holds with
\begin{align}\label{eq:dexampleGN-limexp}
V_{\beta}
= \frac{\beta(1-\theta)}{2\theta}
+ \sqrt{\frac{1-\theta}{2\theta}}\,Z_0
+ \frac{1}{2}\left(\frac{Z_1^2}{1-\beta}-1\right),
\end{align}
where $Z_0,Z_1$ are i.i.d.\ $\cN(0,1)$.

\item {\it Curie-Weiss model:}
This corresponds to $W\equiv 1$ and $\theta=1$. In this case,
\eqref{eq:dexampleGN-limexp} simplifies to
\begin{align*}
V_{\beta}
= \frac{1}{2}\left(\frac{Z_1^2}{1-\beta}-1\right),
\end{align*}
where $Z_1\sim \cN(0,1)$. Hence,
\begin{align*}
\{\P_{\beta}\}_{\beta\in(0,1)}
\xrightarrow{\mathrm{Exp}}
\left\{
\frac{1}{2}\left(\frac{Z_1^2}{1-\beta}-1\right)
\right\}_{\beta\in(0,1)}.
\end{align*}

\item {\it Random bipartite graph:}
This corresponds to taking $p=0$ and $q=1$ in \eqref{eq:Wpq}. In this case,
$\lambda_{+}=\frac{1}{2}$, $\lambda_{-}=-\frac{1}{2}$, and $\|W\|_{\mathrm{op}}=\frac{1}{2}$.
Hence, in the sparse regime \eqref{eq:sparseexample-limexp} simplifies to
\begin{align*}
\{\P_{\beta_0+h\sqrt{\theta_N}}\}_{h\in\mathbb R}
\xrightarrow{\mathrm{Exp}}
\left\{\cN(h,4)\right\}_{h\in\mathbb R}.
\end{align*}
Moreover, in the dense regime \eqref{eq:denseexample-limexp} holds with
\begin{align*}
V_{\beta}
= \frac{\beta(1-\theta)}{4\theta}
+ \sqrt{\frac{1-\theta}{4\theta}}\,Z_0
+ \frac{Z_1^2}{4-2\beta}
- \frac{Z_2^2}{4+2\beta},
\end{align*}
where $Z_0,Z_1,Z_2$ are i.i.d.\ $\cN(0,1)$.
\end{itemize}
\end{example}

\begin{example}[Rank one graphon model]
Consider the rank one graphon $W(x, y) = f(x) f(y)$, where $f: [0, 1] \rightarrow [0, 1]$ is continuous almost everywhere, as in Example \ref{rankOne}. Suppose $G_N \sim G(N, \theta_N, W)$ and $\bs$ is a sample from the Ising model \eqref{model_def} on $G_N$ with $0 < \beta < \frac{1}{\mu_2}$, where $\mu_2 = \| f \|_2^2$,  and $N^{\frac{2}{3}}\theta_N \gg 1$. Then from Theorem \ref{th:limexp}, we obtain the following:
\begin{itemize}
\item If $\theta=0$, 
$$\{\P_{\beta_0 + h\sqrt{\theta_N}}\}_{h\in \mathbb{R}} \xrightarrow{\mathrm{Exp}} \left\{\cN\left(h,\frac{2}{\mu_1^2} \right)\right\}_{h\in \mathbb{R}}, $$
where $\mu_1 = \int_0^1 f(x) \mathrm d x$, 
\item If $\theta \in (0,1]$, 
\begin{align} 
& \{\P_{\beta}\}_{\beta\in \left(0,\frac{1}{\mu_2}\right)}  \nonumber \\ 
& \xrightarrow{\mathrm{Exp}} \left\{\frac\beta{2\theta}(\mu_1^2-\theta \mu_2^2)+Z_0\sqrt{\frac1{2\theta}(\mu_1^2-\theta \mu_2^2)} +\frac{\mu_2}{2}\left(\frac{Z_1^2}{1-\beta \mu_2}-1\right)\right\}_{\beta\in (0,\frac{1}{\mu_2})} , \nonumber 
\end{align} 
where $Z_0, Z_1$ are i.i.d. $\cN(0, 1)$.  
\end{itemize} 
\end{example}

\section{Fluctuations of the Hamiltonian and the Partition Function}
\label{sec:ZH}

At the core of our inferential results lie the fluctuations of the sufficient statistic (Hamiltonian) $H_N(\bs)$ (recall \eqref{eq:sufficientstatistics}), which is the sufficient statistic of the model \eqref{model_def}. These are, in turn, closely tied to the fluctuations of the logarithm of the partition function $Z_N$. To state the results, we recall the definition of $\alpha$- H\"{o}lder continuity (in the context of graphons): A graphon $W: [0, 1]^2 \rightarrow [0, 1] $ is said to be $\alpha$-{\it H\"{o}lder continuous}, for some $\alpha \in (0, 1]$, if there exists a constant $C>0$ such that $$|W(x, y)-W(x', y')| \leq C \left(\sqrt{ (x-x')^2 + (y-y')^2} \right)^{\alpha} , $$ for all $(x, y), (x', y') \in [0, 1]^2$.

\subsection{Fluctuations of the Hamiltonian}\label{sec:fluct_ham6}

In this section we present our results on the fluctuations of the Hamiltonian $H_N(\vec \sigma)$ (recall \eqref{eq:sufficientstatistics}). Throughout, the convergence holds jointly over the randomness of $\bs$ and the graph $G_N$.

\begin{thm}[Fluctuations of the Hamiltonian]
    \label{thm:HNsigma} 
    Suppose Assumption \ref{assumption} holds. Then the following hold: 
    \begin{enumerate} 
    
            \item[$(1)$] (Sparse regime) If $\theta = 0$, then
              \begin{equation}\label{eq:HNsparse}
                   \hspace{0.2in}  \frac{1}{\sqrt{\theta_N}}\left(\theta_N H_N(\bs)+\frac{\beta}{2N^2}\sum_{1 \leq i,j \leq N}W\left(\frac{i}{N},\frac{j}{N}\right)\right)\xrightarrow{D}\cN\left(0,\frac{1}{2}\int_{[0,1]^2} W(x,y)\mathrm{d}x\mathrm{d}y\right).
              \end{equation}
               Consequently, if $W$ is $\alpha$-H\"{o}lder continuous with $\alpha > \frac{1}{3}$, then:
\begin{align}\label{holder112}
                  \frac{1}{\sqrt{\theta_N}}\left(\theta_N H_N(\bs)+\frac{\beta}{2}\int_{[0,1]^2} W(x,y) \mathrm{d}x \mathrm{d}y\right)\xrightarrow{D}\cN\left(0,\frac{1}{2}\int_{[0,1]^2} W(x,y)\mathrm{d}x\mathrm{d}y\right).
              \end{align}

        \item[$(2)$] (Dense regime) If $\theta \in (0,1]$, then
              \begin{equation}\label{eq:HNdense}
                 \hspace{0.2in}  -H_N(\bs) \xrightarrow{D} \beta \sigma_W^2 + \frac{1}{2}\int_{0}^{1}W(x,x)\mathrm{d}x + \sigma_WZ_0 + \frac{1}{2} \sum_{\lambda \in \mathrm{Spec}(W)} \lambda \left(\frac{Z_{\lambda}^2}{1-\beta \lambda}-1\right),
              \end{equation}
              where $\sigma_W^2:=\frac{1}{2\theta} \int_{[0,1]^2} W(x,y)(1-\theta W(x,y))\mathrm{d}x\mathrm{d}y$
              and $Z_0, Z_1,\ldots$ are i.i.d. standard normal random variables.

    \end{enumerate}
\end{thm}

The proof of Theorem \ref{thm:HNsigma} is given in Appendix \ref{corhamlproof}. Prior to the above result, the limiting distribution of the Hamiltonian was only known for the Curie-Weiss model (where $G_N$ is the complete graph and hence, the Hamiltonian is a function (polynomial) of just the sample mean) \cite{ellis1985entropy, ellis1980limit} (also see \cite{cometsgidas}) and for Ising models on dense graphs in the supercritical regime \cite{bhattacharya2018inference}. The above result gives the fluctuations of the Hamiltonian for Ising models on general inhomogeneous random graphs. In the following we compute the limit explicitly in our two running examples.

\begin{example} (Block models) Let $W$ be the 2-block graphon as in \eqref{eq:Wpq}. Suppose $G_N \sim G(N, \theta_N, W)$ and $\bs$ be a sample from the Ising model \eqref{model_def} on $G_N$ with $0 < \beta < \frac{2}{(p+q)}$ and $N^{\frac{2}{3}}\theta_N \gg 1$. Then, from Theorem \ref{thm:HNsigma}, we have the following: 
\begin{itemize}

\item If $\theta=0$, then
\begin{align}\label{eq:HNsparse-2block}
\frac{1}{\sqrt{\theta_N}}\left(\theta_N H_N(\bs)+\frac{\beta}{4}(p+q)\right)
\xrightarrow{D}
\cN\left(0,\frac{p+q}{4}\right).
\end{align}

\item If $\theta\in(0,1]$, then
\begin{align}\label{eq:HNdense-2block}
-H_N(\bs) \xrightarrow{D}
\beta\sigma_W^2 + \frac{p}{2}
+ \sigma_W Z_0
+ \frac{\lambda_{+}}{2}\left(\frac{Z_{\lambda_{+}}^2}{1-\beta \lambda_{+}}-1\right)
+ \frac{\lambda_{-}}{2}\left(\frac{Z_{\lambda_{-}}^2}{1-\beta \lambda_{-}}-1\right),
\end{align}
where $Z_0,Z_{\lambda_{+}},Z_{\lambda_{-}}$ are i.i.d.\ $\cN(0,1)$ and $\sigma_W^2
= \frac{1}{4\theta}\left(p(1-\theta p)+q(1-\theta q)\right)$. 

\end{itemize}

Once again, as in Example \ref{example1_26}, we consider the following  important special cases:

\begin{itemize}

\item {\it Erd\H{o}s-R\'enyi model $G(N,\theta_N)$:}
This corresponds to $W\equiv 1$ (that is, $p=q=1$ in \eqref{eq:Wpq}) and consequently, in the sparse regime \eqref{eq:HNsparse-2block} simplifies to
\begin{align*}
\frac{1}{\sqrt{\theta_N}}\left(\theta_N H_N(\bs)+\frac{\beta}{2}\right)
\xrightarrow{D}
\cN\left(0,\frac{1}{2}\right).
\end{align*}
Moreover, in the dense regime, we have the following from \eqref{eq:HNdense-2block}:
\begin{align}\label{eq:HNdense-GN}
-H_N(\bs) \xrightarrow{D}
\frac{\beta(1-\theta)}{2\theta}
+ \sqrt{\frac{1-\theta}{2\theta}}\,Z_0
+ \frac{Z_1^2}{2(1-\beta)},
\end{align}
where $Z_0,Z_1$ are i.i.d.\ $\cN(0,1)$.

\item {\it Curie-Weiss model:}
This corresponds to $W\equiv 1$ and $\theta=1$. In this case, $\sigma_W^2=0$, and \eqref{eq:HNdense-GN} simplifies to $-H_N(\bs) \xrightarrow{D} \frac{1}{2(1-\beta)} Z_1^2$.

\item {\it Random bipartite graph:}
This corresponds to taking $p=0$ and $q=1$. 
Hence, in the sparse regime \eqref{eq:HNsparse-2block} simplifies to
\begin{align*}
\frac{1}{\sqrt{\theta_N}}\left(\theta_N H_N(\bs)+\frac{\beta}{4}\right)
\xrightarrow{D}
\cN\left(0,\frac{1}{4}\right).
\end{align*}
Moreover, in the dense regime \eqref{eq:HNdense-2block} simplifies to 
\begin{align*}
-H_N(\bs) \xrightarrow{D}
\frac{\beta(1-\theta)}{4\theta}
+ \sqrt{\frac{1-\theta}{4\theta}}\,Z_0
+ \frac{Z_1^2}{4-2\beta}
- \frac{Z_2^2}{4+2\beta},
\end{align*}
where $Z_0,Z_1,Z_2$ are i.i.d.\ $\cN(0,1)$.

\end{itemize}
\end{example}

\begin{example}[Rank one graphon model]
  Consider the rank one graphon: $W(x, y) = f(x) f(y)$, where $f: [0, 1] \rightarrow [0, 1]$ is continuous almost everywhere, as in Example \ref{rankOne}. Suppose $G_N \sim G(N, \theta_N, W)$ and $\bs$ be a sample from the Ising model \eqref{model_def} on $G_N$ with $0 < \beta < \frac{1}{\mu_2}$, where $\mu_2 = \| f \|_2^2$,  and $N^{\frac{2}{3}}\theta_N \gg 1$. 
  Then, from Theorem \ref{thm:HNsigma} we have the following.
    \begin{enumerate}
        \item If $\theta \in (0,1]$, then
              \begin{align*}
                  -H_N(\bs) \xrightarrow{D} \frac\beta{2\theta}(\mu_1^2-\theta \mu_2^2)+Z_0\sqrt{\frac1{2\theta}(\mu_1^2-\theta \mu_2^2)} +\frac{\mu_2 Z_1^2}{2(1-\beta \mu_2)},
              \end{align*}
              where $\mu_1 = \int_0^1 f(x) \mathrm d x$ and $Z_0, Z_1$ are i.i.d. $\mathcal N(0, 1)$.    
        \item If $\theta = 0$, then
              \begin{align*}
                  \frac{1}{\sqrt{\theta_N}}\left(\theta_N H_N(\bs)+\frac{\beta}{2N^2}\left(\sum_{i=1}^N f\left(\frac{i}{N}\right)\right)^2\right)\xrightarrow{D}\cN\left(0,\frac{\mu_1^2}{2}\right).
              \end{align*}
    \end{enumerate} 
\end{example}

\subsection{Fluctuations of the Partition Function}\label{sec:fluctpart78}

In this section we present results on the asymptotic distribution of the log-partition function $\log Z_N$, where $Z_N$ is as defined in \eqref{eq:ZN}.

\begin{thm}[Fluctuations of the log-partition function]\label{thm:ZN} 
Suppose $W$ satisfies the conditions in Assumption \ref{assumption} (2). Fix $0<\beta < \frac{1}{\|W\|_{\mathrm{op}}}$. Then for $G_N \sim G(N, \theta_N, W)$, the following hold: 
    \begin{enumerate}
        \item[$(1)$] If $\theta_N \gg N^{-\frac{1}{2}}$ and $\theta_N \rightarrow \theta \in [0, 1]$, then 
              \begin{align}\label{eq:ZN12}      
               \sqrt{N\theta_N}\left(\log Z_N(\beta) - \log\E Z_N(\beta)\right) \xrightarrow{D} \cN\left(0,\frac{\beta^2}{4}\int_0^1 W(x,x)\left(1-\theta W(x,x)\right)\mathrm{d}x\right).
              \end{align}
              
        \item[$(2)$] If $\theta_N \sim c N^{-\frac{1}{2}}$ for some $c\in (0,\infty)$, then 
              \begin{align}\label{eq:ZN23}  
                  N \theta_N^{\frac{3}{2}} & \left(\log Z_N(\beta)  -\log\E Z_N(\beta)\right) \nonumber \\ 
                  & \to \mathcal{N}\left(0,\frac{\beta^4}{8} \int_{[0,1]^2} W(x,y) \mathrm{d}x \mathrm{d}y + \frac{\beta^2c^2}{4} \int_0^1 W(x,x)\mathrm{d}x\right).
              \end{align} 
              
                     \item[$(3)$] If $N^{-\frac{2}{3}} \ll \theta_N \ll N^{-\frac{1}{2}}$, then 
              \begin{align}\label{eq:ZNtheta12}  
                  N \theta_N^{\frac{3}{2}} & \left(\log Z_N(\beta)  -\log\E Z_N(\beta)\right) \to \mathcal{N}\left(0,\frac{\beta^4}{8} \int_{[0,1]^2} W(x,y) \mathrm{d}x \mathrm{d}y \right).
              \end{align}

        \item[$(4)$]  If $\theta_N \sim c N^{-\frac{2}{3}}$, for some $c \in (0,\infty)$, then 
              \begin{align}\label{eq:23threshold}  
                  \log Z_N(\beta) - \frac{\beta^2}{4N^2\theta_N}\sum_{1 \leq i,j \leq N}W\left(\frac{i}{N},\frac{j}{N}\right) \xrightarrow{D} \mathcal{N}\left(\mu_W,\frac{\beta^4}{8c^{3}} \int_{[0,1]^2} W(x,y) \mathrm{d}x \mathrm{d}y\right),
              \end{align}
              with 
              \begin{align*}  
                  \mu_W :=\frac{\beta}{2}\int_0^1 W(x,x)\mathrm{d}x & - \frac{\beta^2}{4}\int_{[0,1]^2} W(x,y)\left(\frac{\beta^2}{6c^3}+W(x,y)\right) \mathrm{d}x \mathrm{d}y \nonumber \\ 
                  & - \frac{1}{2}\sum_{\lambda \in \mathrm{Spec}(W)} \left(\log(1-\beta \lambda) + \beta\lambda \right) . 
              \end{align*}
 Consequently, if $W$ is $\alpha$-H\"{o}lder continuous with $\alpha > {\frac{2}{3}}$, then 
              \begin{align}\label{eq:23thresholdW}
               \log Z_N(\beta) - \frac{\beta^2}{4\theta_N}\int_{[0,1]^2} W(x,y) \mathrm{d}x \mathrm{d}y \xrightarrow{D} \mathcal{N}\left(\mu_W,\frac{\beta^4}{8c^{3}} \int_{[0,1]^2}  W(x,y) \mathrm{d}x \mathrm{d}y\right) . 
              \end{align}

        \item[$(5)$] If $N^{-1} \ll \theta_N \ll N^{-\frac{2}{3}}$, then 
              \begin{align}\label{eq:23bthreshold}
                  N\theta_N^{\frac{3}{2}} & \left(\log Z_N(\beta) - \theta_N \log\cosh\left(\frac{\beta}{N\theta_N}\right)\sum_{1 \leq i < j \leq N}W\left(\frac{i}{N},\frac{j}{N}\right)\right) \nonumber \\ 
                  & \xrightarrow{D} \mathcal{N}\left(0,\frac{\beta^4}{8} \int_{[0,1]^2} W(x,y) \mathrm{d}x \mathrm{d}y\right) . 
              \end{align}
              Consequently, if $W$ is $\alpha$-H\"{o}lder continuous with $\alpha \ge {\frac{2}{3}}$, then 
              \begin{align}\label{eq:23bthresholdW}
                  N\theta_N^{\frac{3}{2}} & \left(\log Z_N(\beta) - \frac{N^2\theta_N}{2}\log\cosh\left(\frac{\beta}{N\theta_N}\right)  \int_{[0,1]^2} W(x,y) \mathrm{d}x \mathrm{d}y\right) \nonumber \\ 
                  & \xrightarrow{D} \mathcal{N}\left(0,\frac{\beta^4}{8} \int_{[0,1]^2} W(x,y) \mathrm{d}x \mathrm{d}y\right).
              \end{align}
    \end{enumerate}
\end{thm}

The proof of Theorem \ref{thm:ZN} is given in Appendix \ref{proof_thm:ZN}. The proof builds on the techniques developed in \cite{kabluchko2021fluctuations}, where fluctuations of the partition function were established for the case in which $G_N$ is sampled from a homogeneous Erd\H{o}s--R\'{e}nyi model (that is, $W \equiv 1$). Theorem \ref{thm:ZN} extends this result to the setting of inhomogeneous random graphs. Depending on the asymptotic regime, several interesting features emerge from the results of Theorem \ref{thm:ZN}: 
\begin{itemize} 
\item  The weak-limit behavior of $Z_N(\beta)$ exhibits a transition from self-averaging to non-self-averaging across the threshold $N^{-\frac{2}{3}}$. Specifically, for  $\theta_N \gg N^{-\frac{2}{3}}$ (which includes cases (1), (2), and (3)) $Z_N(\beta)/\E Z_N(\beta)$ converges to $1$. On the other hand, in the regime $\theta_N \lesssim N^{-\frac{2}{3}}$ (which includes cases (4) and (5)), this property no longer holds. In each regime, $\log Z_N(\beta)$ has Gaussian fluctuations after appropriate centering and scaling.  
\item The different cases of Theorem \ref{thm:ZN} jointly cover the sparsity regime $N^{-1} \ll \theta_N \lesssim 1$, which exhausts the entire diverging expected degree regime. In contrast, our results for the Hamiltonian, and hence also those for inference, require the stronger sparsity condition $\theta_N \gtrsim N^{-\frac{2}{3}}$ (see Assumption \ref{assumption}). We expect the Hamiltonian fluctuations to hold throughout the entire diverging expected-degree regime as well, although current techniques are insufficient to establish this.
\item The behavior of the model \eqref{model_def} is fundamentally different in the constant expected-degree regime $\theta_N\sim N^{-1}$, where the graph is locally tree-like. For the homogeneous Erd\H{o}s--R\'enyi model, fluctuations of the partition function in this regime were obtained recently in~\cite{coja2026fluctuations}. Related critical-regime fluctuations of the magnetization for sparse Erd\H{o}s--R\'enyi and random regular graphs were obtained in~\cite{prodromidis2026distribution}. Extending these results to inhomogeneous random graphs is a natural direction for future work.
\item Finally, note that in the regime $\theta_N \gg N^{-\frac{1}{2}}$, the variance of the limiting normal distribution depends only on the diagonal part of $W$ (see \eqref{eq:ZN12}). Hence, if the model has no self-loops, the limiting distribution of $Z_N(\beta)$ obtained in \eqref{eq:ZN12} is degenerate. In this case, a finer analysis is required to obtain a non-degenerate fluctuation limit for $Z_N(\beta)$. In fact, to the best of our knowledge this issue remains unresolved even for the Erd\H{o}s--R\'{e}nyi model. This degeneracy, however, has no impact on our inferential results, since the limiting distribution of the Hamiltonian is always non-degenerate, irrespective of the behavior of the diagonal part of $W$ (recall Theorem \ref{thm:HNsigma}).
\end{itemize}

\small

\subsection*{Acknowledgement} S. Bhowal thanks Yogeshwaran D for helpful discussions during the early stages of this project. S. Mukherjee was supported by the AcRF Tier 1 grants A-8001449-00-00 and A-8002932-00-00. B. B. Bhattacharya was supported by NSF CAREER grant DMS 2046393 and a Sloan Research Fellowship.

\bibliographystyle{plain}  
\bibliography{refs,graphonbib}

\normalsize

\appendix

%
%
%
%
%


\section{Proof of Theorem \ref{thm:ZN}}\label{proof_thm:ZN}

Recall the definition of the partition function $Z_N$ of the model from \eqref{eq:ZN}. To begin with, define 
\begin{align}\label{eq:ANij}
    \vartheta_N := \sum_{i=1}^N A_{G_N}\left(i,i\right)
    \;\;
    \text{ and }
    \;\;
    \eta_N := \sum_{1 \leq i, j \leq N}  A_{G_N}\left(i,j\right).\end{align}
Let $\gamma:= \frac{\beta}{2 N\theta_N}$. For each $\bs\in \{-1,+1\}^N$, define the random variable
\begin{align}\label{Tstatdef}
    T\left(\bs\right):= e^{\textstyle  \gamma  \sum_{1 \leq i, j \leq N}  A_{G_N}\left(i,j\right)\sigma_i \sigma_j - \gamma \vartheta_N
    -\gamma^2\eta_N } .
\end{align}
Then the \textit{modified partition function} is defined as: 
\begin{align}\label{modpart27}
    \hat Z_N\left(\beta\right) := \frac{1}{2^N} \sum_{\bs\in \{-1,+1\}^N} T\left(\bs\right)
    = Z_N\left(\beta\right) e^{\textstyle  - \gamma \vartheta_N
    -\gamma^2\eta_N } .
\end{align}
In this notation, 
\begin{align}\label{eq:Z_N_hat_Z_N}
    Z_N\left(\beta\right)
    =
    \hat Z_N\left(\beta\right) e^{\textstyle  \gamma \vartheta_N
    +\gamma^2\eta_N	} .
\end{align}
The following lemma shows that the modified partition function has a self-averaging property, that is, it concentrates around its expected value. The proof is given in Appendix \ref{proof6p2}. 

\begin{lemma}\label{var_sum}
 Suppose $\theta_N \gtrsim N^{-\frac{2}{3}}$. Then 
    $$
        \var \left[\frac{\hat Z_N\left(\beta\right)}{\E \hat Z_N\left(\beta\right)}\right]  \ll \frac 1 {N\theta_N} .
    $$
\end{lemma}

From Lemma~\ref{var_sum} it follows that 
\begin{equation*}
  J_N(\beta) -1:= \frac{\hat Z_N(\beta)}{\E \hat Z_N(\beta)} - 1 = o_{L^2}\left(\frac 1 {\sqrt{N\theta_N}}\right) .    
\end{equation*}
Hence, from~\eqref{eq:Z_N_hat_Z_N}, 
\begin{align}    \label{prep_proof_th1.2}
    \frac{Z_N(\beta)}{\E \hat Z_N(\beta)}
    = J_N(\beta) e^{\textstyle  \gamma \vartheta_N + \gamma^2\eta_N }  =
    \left(1 + o_{L^2}\left(\frac 1 {\sqrt{N\theta_N}}\right)\right)e^{\textstyle  \gamma \vartheta_N + \gamma^2\eta_N } . 
\end{align}
Now, introduce the centered random variables
\begin{align*}
    \overline{\vartheta}_N := \sum_{i=1}^N \overline{A_{G_N}(i,i)},
    \;\;\;
    \text{ and }
    \;\;\;
    \overline{\eta}_N := \sum_{1 \leq i, j \leq N}  \overline{A_{G_N}(i,j)},
\end{align*} 
where
$\overline{A_{G_N}(i,j)}:= A_{G_N}(i,j) - \E A_{G_N}(i,j)  = A_{G_N}(i,j) - \theta_N W(\frac{i}{N},\frac{j}{N})$.
With this notation, we can rewrite~\eqref{prep_proof_th1.2} as follows:
\begin{align}\label{prep_proof_th1.2_rewr}
    & \frac{Z_N(\beta)}{e^{\textstyle \frac{\beta}{2N}\sum_{i=1}^{N}W\left(\frac{i}{N},\frac{i}{N}\right) + \frac{\beta^2}{4N^2\theta_N}\sum_{1 \leq i,j \leq N}W\left(\frac{i}{N},\frac{j}{N}\right)} \E \hat Z_N(\beta)} \nonumber \\ 
    & = J_N(\beta)
    e^{\textstyle \gamma \overline{\vartheta}_N + \gamma^2\overline{\eta}_N}  \nonumber \\ 
    & = \left(1 + o_{L^2}\left(\frac 1 {\sqrt{N\theta_N}}\right)\right)
    e^{\textstyle \gamma \overline{\vartheta}_N + \gamma^2\overline{\eta}_N} . 
\end{align}
Using this representation, we now present the proof of Theorem~\ref{thm:ZN} in the different regimes. The rest of the section is organized as follows: the proofs of Theorem~\ref{thm:ZN} (1), (2)--(3), (4), and (5) are given in Appendices~\ref{sec:thm:ZNpf1}, \ref{sec:thm:ZNpf2}, \ref{sec:thm:ZNpf3}, and \ref{sec:thm:ZNpf4}, respectively. These proofs rely on certain technical lemmas, which are proved in Appendix~\ref{sec:analysis_partition}.

\subsection{Proof of Theorem \ref{thm:ZN} (1)}
\label{sec:thm:ZNpf1}

The idea in this case is to replace the exponential term in the RHS of \eqref{prep_proof_th1.2_rewr} by its linearization and then analyze the fluctuations of the linear term. This is formalized in the following lemma. The proof of the lemma is given in Appendix \ref{proofdeltaN1}.

\begin{lemma}\label{lem:clt_for_delta_N1}
Suppose $N\theta_N^2 \gg 1$. Then for any $p\ge 1$,
    \begin{align}\label{eq:linear11}
        e^{\textstyle \gamma \overline{\vartheta}_N + \gamma^2 \overline{\eta}_N}
        =
        1 + \gamma \overline{\vartheta}_N + \gamma^2 \overline{\eta}_N + o_{L^p}\left(\frac{1}{\sqrt{N\theta_N}}\right).
    \end{align}
   Furthermore, if $N\theta_N^2 \gg 1$ such that $\theta_N \to \theta$, then 
     \begin{align}\label{eq:thetaN}
        \sqrt{N\theta_N} \left(\gamma \overline{\vartheta}_N + \gamma^2\overline{\eta}_N\right) \xrightarrow{D} \mathcal{N}\left(0,\frac{\beta^2}{4} \int_0^1 W(x,x)(1-\theta W(x,x)) \mathrm{d}x\right) . 
    \end{align}
\end{lemma}

Using the above  lemma we can now complete the proof of Theorem~\ref{thm:ZN} (1) as follows: To begin with, using~\eqref{prep_proof_th1.2_rewr} and~\eqref{eq:linear11} gives, 
\begin{align}\label{eq:L2_tech1}
    &\frac{Z_N(\beta)}{e^{\textstyle \frac{\beta}{2N}\sum_{i=1}^{N}W\left(\frac{i}{N},\frac{i}{N}\right) + \frac{\beta^2}{4N^2\theta_N}\sum_{1 \leq i,j \leq N}W\left(\frac{i}{N},\frac{j}{N}\right)} \E \hat Z_N(\beta)}
    \notag\\ 
    & =
    \left(1 + o_{L^2}\left(\frac 1 {\sqrt{N\theta_N}}\right)\right)
    \left(1 + \gamma \overline{\vartheta}_N + \gamma^2\overline{\eta}_N + o_{L^2}\left(\frac{1}{\sqrt{N\theta_N}}\right)\right) \notag\\ 
    &=
    1 + \gamma \overline{\vartheta}_N + \gamma^2\overline{\eta}_N + o_{L^1}\left(\frac 1 {\sqrt{N\theta_N}}\right).
\end{align}
Multiplying \eqref{eq:L2_tech1} throughout by $\sqrt{N\theta_N}$ and applying Lemma \ref{lem:clt_for_delta_N1} gives, 
\begin{align}\label{eq:CLT_wrong_norming1}
    &\sqrt{N \theta_N} \left(\frac{Z_N(\beta)}{e^{\textstyle \frac{\beta}{2N}\sum_{i=1}^{N}W\left(\frac{i}{N},\frac{i}{N}\right) + \frac{\beta^2}{4N^2\theta_N} \sum_{1 \leq i,j \leq N} W\left(\frac{i}{N},\frac{j}{N}\right)}\E \hat{Z}_N(\beta)} - 1\right)
    \nonumber\\
    & =
    \sqrt{N\theta_N}(\gamma \overline{\vartheta}_N + \gamma^2\overline{\eta}_N) + o_{L^1}(1) \nonumber \\ 
    & \xrightarrow{D}
    \mathcal{N}\left(0,\frac{\beta^2}{4} \int_0^1 W(x,x)(1-\theta W(x,x)) \mathrm{d}x\right) . 
\end{align}
The $L^1$-convergence in the first equality of \eqref{eq:CLT_wrong_norming1} further implies that 
\begin{align*}
    \varepsilon_N := \sqrt{N\theta_N} \left( \frac{\E Z_N(\beta)}{e^{\textstyle \frac{\beta}{2N}\sum_{i=1}^{N}W\left(\frac{i}{N},\frac{i}{N}\right) + \frac{\beta^2}{4N^2\theta_N}\sum_{1 \leq i,j \leq N}W\left(\frac{i}{N},\frac{j}{N}\right)} \E \hat Z_N(\beta)} -1\right) \ll 1.
\end{align*}
Hence, 
\begin{align*}
    & \sqrt{N \theta_N} \left(\frac{Z_N(\beta)}{\E Z_N(\beta)} - 1\right) \nonumber \\ 
     & =
    \sqrt{N \theta_N} \left(Z_N(\beta)\cdot \frac{1+ \frac{\varepsilon_N}{\sqrt{N\theta_N}}}{\E Z_N(\beta)} - 1\right) - \varepsilon_N \frac{Z_N(\beta)}{\E Z_N(\beta)}                                                                                                                           \\
     & =
    \sqrt{N \theta_N} \left(\frac{Z_N(\beta)}{e^{\textstyle \frac{\beta}{2N}\sum_{i=1}^{N}W\left(\frac{i}{N},\frac{i}{N}\right) + \frac{\beta^2}{4N^2\theta_N}\sum_{1 \leq i,j \leq N}W\left(\frac{i}{N},\frac{j}{N}\right)} \E \hat Z_N(\beta)} - 1\right) - \varepsilon_N \frac{Z_N(\beta)}{\E Z_N(\beta)} \\
     & \xrightarrow{D}
    \mathcal{N}\left(0,\frac{\beta^2}{4} \int_0^1 W(x,x)(1-\theta W(x,x)) \mathrm{d}x\right),
\end{align*} 
where the last step uses~\eqref{eq:CLT_wrong_norming1}, together with the facts that $\varepsilon_N \ll 1$ and $\frac{Z_N(\beta)}{\E Z_N(\beta)} = O_P(1)$. The result in \eqref{eq:ZN12} then follows by applying the delta method.

\subsection{Proof of Theorem \ref{thm:ZN} (2) and (3)}
\label{sec:thm:ZNpf2}

Throughout this section we will assume $\theta_N \sim c N^{-\frac{1}{2}}$ for some $c\in [0,\infty)$. This includes both case (2) (where $c \in (0, \infty)$) and case (3) (where $c = 0$). As in the previous section, the proof in this case relies on the following linearization result, which we prove in Appendix \ref{proofdeltaN28}. 

\begin{lemma}\label{lem:clt_for_delta_N2}
    Suppose $N\theta_N^2 \lesssim 1$ and $N^2\theta_N^{3} \gg 1$. Then for any $p\ge 1$, 
    \begin{align}\label{eq:linear12}
        e^{\textstyle \gamma \overline{\vartheta}_N + \gamma^2 \overline{\eta}_N}
        =
        1 + \gamma \overline{\vartheta}_N + \gamma^2 \overline{\eta}_N + o_{L^p}\left(\frac{1}{N\theta_N^{\frac{3}{2}}}\right).
    \end{align} 
    Furthermore, if $N\theta_N^2 \to c^2$ and $N^2\theta_N^3 \to\infty$, then 
    \begin{align}\label{eq:thetaNZN} 
        N\theta_N^{\frac{3}{2}} \left(\gamma \overline{\vartheta}_N + \gamma^2\overline{\eta}_N\right) \xrightarrow{D} \mathcal{N}\left(0,\frac{\beta^4}{8} \int_{[0,1]^2} W (x, y) \mathrm{d}x \mathrm{d} y + \frac{\beta^2c^2}{4} \int_0^1 W(x,x)\mathrm{d}x \right).
    \end{align}
\end{lemma}

Using the above lemma we can now complete the proof of Theorem~\ref{thm:ZN} (2) and (3) as follows: To begin with, using~\eqref{prep_proof_th1.2_rewr} and~\eqref{eq:linear12} gives: 
\begin{align}\label{eq:L2_tech2}
    &\frac{Z_N(\beta)}{e^{\textstyle \frac{\beta}{2N}\sum_{i=1}^{N}W\left(\frac{i}{N},\frac{i}{N}\right) + \frac{\beta^2}{4N^2\theta_N}\sum_{1 \leq i,j \leq N}W\left(\frac{i}{N},\frac{j}{N}\right)} \E \hat Z_N(\beta)}\notag\\
    &=
    \left(1 + o_{L^2}\left(\frac 1 {\sqrt{N\theta_N}}\right)\right)
    \left(1 + \gamma \overline{\vartheta}_N + \gamma^2\overline{\eta}_N + o_{L^2}\left(\frac{1}{N\theta_N^{\frac{3}{2}}}\right)\right)\notag \\
    &=
    1 + \gamma \overline{\vartheta}_N + \gamma^2\overline{\eta}_N + o_{L^1}\left(\frac 1 {N\theta_N^{\frac{3}{2}}}\right).
\end{align}
Multiplying \eqref{eq:L2_tech2} throughout by $N\theta_N^{\frac{3}{2}}$ and applying Lemma \ref{lem:clt_for_delta_N2} gives, 
\begin{align}
    \label{eq:L2_tech_multiplied2}
    & N \theta_N^{\frac{3}{2}} \left(\frac{Z_N(\beta)}{e^{\textstyle \frac{\beta}{2N}\sum_{i=1}^{N}W\left(\frac{i}{N},\frac{i}{N}\right) + \frac{\beta^2}{4N^2\theta_N} \sum_{1 \leq i,j \leq N} W\left(\frac{i}{N},\frac{j}{N}\right)}\E \hat{Z}_N(\beta)} - 1\right)\notag\\
    &=
    N\theta_N^{\frac{3}{2}}(\gamma \overline{\vartheta}_N + \gamma^2\overline{\eta}_N) + o_{L^1}(1) \nonumber \\ 
    &\xrightarrow{D} \mathcal{N}\left(0,\frac{\beta^4}{8} \int_{[0,1]^2} W + \frac{\beta^2c^2}{4} \int_0^1 W(x,x)\mathrm{d}x \right) .
\end{align}
The $L^1$-convergence in the first equality in \eqref{eq:L2_tech_multiplied2} further implies that 
\begin{align*}
    \varepsilon_N := N\theta_N^{\frac{3}{2}} \left( \frac{\E Z_N(\beta)}{e^{\textstyle \frac{\beta}{2N}\sum_{i=1}^{N}W\left(\frac{i}{N},\frac{i}{N}\right) + \frac{\beta^2}{4N^2\theta_N}\sum_{1 \leq i,j \leq N}W\left(\frac{i}{N},\frac{j}{N}\right)} \E \hat Z_N(\beta)} -1\right) \ll 1.
\end{align*}
Hence, 
\begin{align*}
    & N \theta_N^{\frac{3}{2}} \left(\frac{Z_N(\beta)}{\E Z_N(\beta)} - 1\right) \\ 
     & =
    N \theta_N^{\frac{3}{2}} \left(Z_N(\beta)\cdot \frac{1+ \frac{\varepsilon_N}{\sqrt{N^{2}\theta_N^{3}}}}{\E Z_N(\beta)} - 1\right) - \varepsilon_N \frac{Z_N(\beta)}{\E Z_N(\beta)}     \\
     & =
    N \theta_N^{\frac{3}{2}} \left(\frac{Z_N(\beta)}{e^{\textstyle \frac{\beta}{2N}\sum_{i=1}^{N}W\left(\frac{i}{N},\frac{i}{N}\right) + \frac{\beta^2}{4N^2\theta_N}\sum_{1 \leq i,j \leq N}W\left(\frac{i}{N},\frac{j}{N}\right)} \E \hat Z_N(\beta)} - 1\right) - \varepsilon_N \frac{Z_N(\beta)}{\E Z_N(\beta)} \\
     & \to
    \mathcal{N}\left(0,\frac{\beta^4}{8} \int_{[0,1]^2} W(x,y)\mathrm{d}x\mathrm{d}y + \frac{\beta^2c^2}{4}\int_0^1 W(x,x)\mathrm{d}x\right)
\end{align*}
where the last step uses \eqref{eq:L2_tech_multiplied2}, together with the facts that $\varepsilon_N \ll 1$ and $\frac{Z_N(\beta)}{\E Z_N(\beta)} = O_P (1)$. The results in \eqref{eq:ZN23} (where $c \in (0, \infty)$) and \eqref{eq:ZNtheta12} (where $c = 0$)  then follow from the above by applying the delta method.

\subsection{Proof of Theorem \ref{thm:ZN} (4)}
\label{sec:thm:ZNpf3}

First, note from \eqref{prep_proof_th1.2_rewr} that 
\begin{multline*}
    \log Z_N(\beta)-\frac{\beta}{2N}\sum_{i=1}^{N}W\left(\frac{i}{N},\frac{i}{N}\right) - \frac{\beta^2}{4N^2\theta_N}\sum_{1 \leq i,j \leq N}W\left(\frac{i}{N},\frac{j}{N}\right) - \log \E \hat Z_N(\beta)\\
    =
    \log \left(1 + o_{L^2}\left(\frac 1 {\sqrt{N\theta_N}}\right)\right)
    + \gamma \overline{\vartheta}_N + \gamma^2\overline{\eta}_N . 
\end{multline*}
Hence, 
\begin{align}\label{logzntoclt5}
    \log Z_N(\beta)-\frac{\beta}{2N}\sum_{i=1}^{N}W\left(\frac{i}{N},\frac{i}{N}\right) - \frac{\beta^2}{4N^2\theta_N}\sum_{1 \leq i,j \leq N}W\left(\frac{i}{N},\frac{j}{N}\right) & - \log \E \hat Z_N(\beta)
    - \gamma \overline{\vartheta}_N - \gamma^2\overline{\eta}_N \nonumber \\ 
    & \xrightarrow{P} 0 . 
\end{align}
Next, define 
\begin{align}\label{eq:gammatheta}
    \overline{\Delta}_N & :=\gamma \overline{\vartheta}_N + \gamma^2\overline{\eta}_N \nonumber \\ 
    & =
    \sum_{1 \leq i < j \leq N} \frac{\beta^2}{2(N\theta_N)^2} \overline{A_{G_N}(i,j)}
    +
    \sum_{i=1}^N
    \left(\frac{\beta^2}{4(N\theta_N)^2} + \frac {\beta}{2N \theta_N}\right) \overline{A_{G_N}(i,i)}.
\end{align} 
Since $N\theta_N^2 \ll 1$ in the current regime,
\begin{align}\label{pr81}
    \var\left[N\theta_N^{\frac{3}{2}} \sum_{i=1}^N
    \left(\frac{\beta^2}{4(N\theta_N)^2} + \frac {\beta}{2N \theta_N}\right) \overline{A_{G_N}(i,i)}\right] \ll 1.
\end{align} 
Furthermore, 
\begin{align}\label{pr82}
    \var\left[N\theta_N^{\frac{3}{2}} \sum_{1 \leq i < j \leq N} \frac{\beta^2}{2(N\theta_N)^2} \overline{A_{G_N}(i,j)}\right] \to \frac{\beta^4}{8} \int_{[0,1]^2} W (x, y) \mathrm{d}x \mathrm{d} y . 
\end{align} 
Combining \eqref{pr81} and \eqref{pr82} gives, 
 $$\var\left[{N\theta_N^{\frac{3}{2}}} \overline{\Delta}_N\right] \rightarrow \frac{\beta^4}{8}\int_{[0,1]^2} W (x, y) \mathrm{d}x \mathrm{d} y . $$
Moreover, 
$$N\theta_N^{\frac{3}{2}}\overline{\Delta}_N =  \sum_{1 \leq i < j \leq N} \frac{\beta^2}{2N \sqrt{\theta_N}} \overline{A_{G_N}(i,j)}
    +
    \sum_{i=1}^N
    \left(\frac{\beta^2}{4N \sqrt{\theta_N}} + \frac {\beta\sqrt{\theta_N}}{2}\right) \overline{A_{G_N}(i,i)}.$$
Note that each of the independent summands in the above expansion converges uniformly to $0$, as $N\rightarrow\infty$. Hence, by the Lindeberg-Feller CLT, \begin{align}\label{intmclt7}
    N\theta_N^{\frac{3}{2}}\overline{\Delta}_N \xrightarrow{D} \mathcal{N}\left(0,\frac{\beta^4}{8} \int_{[0,1]^2} W (x, y) \mathrm{d}x \mathrm{d} y \right).
\end{align}
Combining \eqref{logzntoclt5} (on noting that $N\theta_N^{\frac{3}{2}} = O(1)$) and \eqref{intmclt7}, we have:
\begin{align}\label{nthetazn2}
    &N\theta_N^{\frac{3}{2}} \left(\log Z_N(\beta)-\frac{\beta}{2N}\sum_{i=1}^{N}W\left(\frac{i}{N},\frac{i}{N}\right) - \frac{\beta^2}{4N^2\theta_N}\sum_{1 \leq i,j \leq N}W\left(\frac{i}{N},\frac{j}{N}\right) - \log \E \hat Z_N(\beta)\right)\notag\\ 
    &\xrightarrow{D} \mathcal{N}\left(0,\frac{\beta^4}{8} \int_{[0,1]^2} W (x, y) \mathrm{d}x \mathrm{d} y\right).
\end{align} 
Recalling that $N\theta_N^\frac{3}{2} \rightarrow c^\frac{3}{2}$, we thus have the following from \eqref{nthetazn2}: 
\begin{align}\label{nthetazn282}
    & \log Z_N(\beta) - \frac{\beta^2}{4N^2\theta_N}\sum_{1 \leq i,j \leq N}W\left(\frac{i}{N},\frac{j}{N}\right) - \log \E \hat Z_N(\beta)\notag\\ 
    &\xrightarrow{D} \mathcal{N}\left(\frac{\beta}{2}  \int_0^1 W(x,x) dx~,~\frac{\beta^4}{8c^3} \int_{[0,1]^2} W (x, y) \mathrm{d}x \mathrm{d} y\right).
\end{align}
The result in \eqref{eq:23threshold} now follows from \eqref{nthetazn282} and part (2) of the following lemma, which is proved in Appendix \ref{proof6p246}.  

Next, note that if $W$ is $\alpha$-H\"older continuous, 
\begin{align}\label{holdapprox62}
    \left|\frac{1}{N^2}\sum_{1\le i,j\le N} W\left(\frac{i}{N}, \frac{j}{N}\right) - \int_{[0,1]^2} W(x,y)~\mathrm{d}x \mathrm{d}y\right|  &\le \sum_{i,j=1}^N \int_{\frac{i-1}{N}}^\frac{i}{N} \int_{\frac{j-1}{N}}^\frac{j}{N} \left| W\left(\frac{i}{N}, \frac{j}{N}\right) - W(x,y)\right|~\mathrm{d}x \mathrm{d}y\notag\\ 
    &\lesssim N^{-\alpha} \sum_{i,j=1}^N \int_{\frac{i-1}{N}}^\frac{i}{N} \int_{\frac{j-1}{N}}^\frac{j}{N}1 ~\mathrm{d}x \mathrm{d}y = N^{-\alpha}.
\end{align}
It follows from \eqref{holdapprox62} that if $\alpha > \frac{2}{3}$, then
\begin{equation}\label{compholdap44}
    \left|\frac{\beta^2}{4 N^2 \theta_N}\sum_{1\le i,j\le N} W\left(\frac{i}{N}, \frac{j}{N}\right) - \frac{\beta^2}{4\theta_N}\int_{[0,1]^2} W(x,y)~\mathrm{d}x \mathrm{d}y\right| \lesssim \frac{\beta^2}{4\theta_N N^\alpha} \sim \frac{\beta^2}{4c N^{\alpha - \frac{2}{3}}} = o(1).
\end{equation}
The result in \eqref{eq:23thresholdW} now follows from \eqref{eq:23threshold} and \eqref{compholdap44}, thereby completing the proof of part (4). \hfill $\Box$

\begin{lemma} \label{lem:expect_sum}
    Suppose $0< \beta <\frac{1}{\|W\|_{\mathrm{op}}}$. Then the following hold: 
    \begin{enumerate}
        \item[$(1)$] If $\theta_N \gg N^{-\frac{2}{3}}$, then
              $$
                  \E \hat Z_N\left(\beta\right)
                  \to e^{\textstyle -\frac{\beta^2}{4}\int_{[0,1]^2} W(x,y)^2 \mathrm{d}x \mathrm{d}y } \prod_{\lambda \in \mathrm{Spec}(W)} \frac{ e^{\textstyle  -\frac{\beta \lambda}{2}}}{\sqrt{1-\beta \lambda}} . 
              $$
        \item[$(2)$] If $\theta_N \sim c N^{-\frac{2}{3}}$ for some $c \in (0,\infty)$, then
              \begin{align*}
                  & \log \E \hat{Z}_N(\beta) \nonumber \\ 
                  & \to -\frac{\beta^2}{4}\int_{[0,1]^2} W(x,y)\left(\frac{\beta^2}{6c^3}+W(x,y)\right) \mathrm{d}x \mathrm{d}y-\frac{1}{2}\sum_{\lambda \in \mathrm{Spec}(W)} \left(\log(1-\beta \lambda)+\beta\lambda \right). 
              \end{align*}
    \end{enumerate}
\end{lemma}


\subsection{Proof of Theorem \ref{thm:ZN} (5)}
\label{sec:thm:ZNpf4}

To begin with, note that
\begin{align*}
    Z_N(\beta) = & \frac{1}{2^N} \sum_{\bs\in \{-1,+1\}^N} e^{\textstyle   \frac{\beta}{2 N\theta_N } \sum_{1 \leq i, j \leq N}  A_{G_N}\left(i, j\right)\sigma_i \sigma_j } \\
    = & \E_{\mu_N}\left[\prod_{1 \leq i, j \leq N} e^{\textstyle   \frac{\beta}{2 N\theta_N } A_{G_N}\left(i, j\right)\sigma_i \sigma_j }  \right] \\     
    = & e^{\textstyle \frac{\beta}{2N\theta_N}\vartheta_N} \E_{\mu_N}\left[\prod_{1 \leq i \ne j \leq N} e^{\textstyle   \frac{\beta}{2 N\theta_N } A_{G_N}\left(i, j\right)\sigma_i \sigma_j }  \right] \\ 
     = & e^{\textstyle \frac{\beta}{2N\theta_N}\vartheta_N} \E_{\mu_N}\left[\prod_{1 \leq i < j \leq N} \left(\cosh\left(\frac{\beta}{N\theta_N} A_{G_N}(i,j)\right)+\sigma_i\sigma_j \sinh\left(\frac{\beta}{N\theta_N} A_{G_N}(i,j)\right)\right)\right]                            \\
    =            & e^{\textstyle \frac{\beta}{2N\theta_N}\vartheta_N} \prod_{1 \leq i < j \leq N} \left(\cosh\left(\frac{\beta}{N\theta_N} A_{G_N}(i,j)\right)\right) \E_{\mu_N}\left[\prod_{1 \leq i < j \leq N} \left(1+\sigma_i\sigma_j \tanh\left(\frac{\beta}{N\theta_N} A_{G_N}(i,j)\right)\right)\right] . 
\end{align*} 
For ease of notation, define 
\begin{align}
    \label{eq:defZ_Ntilde}
    \tilde{Z}_N(\beta) := \E_{\mu_N}\left[\prod_{1 \leq i < j \leq N} \left(1+\sigma_i\sigma_j \tanh\left(\frac{\beta}{N\theta_N} A_{G_N}(i,j)\right)\right)\right].
\end{align}
In this notation, 
\begin{align}\label{exprlogzn7}
    \log Z_N(\beta) = & \log \tilde{Z}_N(\beta) + \sum_{1 \leq i < j \leq N}\log\cosh\left(\frac{\beta}{N\theta_N} A_{G_N}(i,j)\right)+\frac{\beta}{2N\theta_N}\vartheta_N \nonumber \\
    =                 & \frac{\beta}{2N\theta_N}\vartheta_N + \log\cosh\left(\frac{\beta}{N\theta_N}\right)\sum_{1 \leq i < j \leq N}A_{G_N}(i,j)+ \log \tilde{Z}_N(\beta) . 
\end{align}
Now, we will show that the sequence $\log \tilde{Z}_N(\beta)$ is tight.
We begin with the following representation of $ \tilde{Z}_N(\beta)$. Here, for two graphs $G$ and $H$ we write $H \sqsubseteq G$, when $H$ is a subgraph of $G$. 

\begin{lemma}\label{lm:ZNbetaexpansion}
Let $\tilde{Z}_N(\beta)$ be as defined in \eqref{eq:defZ_Ntilde}. Then 
\begin{align*}
    \tilde{Z}_N(\beta) & =
    \E_{\mu_N}\left[\prod_{1 \leq i < j \leq N} \left(1+\sigma_i\sigma_j A_{G_N}(i,j)\tanh\left(\frac{\beta}{N\theta_N} \right)\right)\right] \nonumber \\ 
    & =                     \sum_{\Gamma \in \cE_N} \left(\tanh\left(\frac{\beta}{N\theta_N} \right)\right)^{|E(\Gamma)|}\one\left\{\Gamma \sqsubseteq G_{N}\right\},
\end{align*}
where $\cE_N$ denotes the collection of all spanning subgraphs of the complete graph on $N$ vertices in which every vertex has even degree and $G_N \sim G(N, \theta_N, W)$ denotes the random graph with adjacency $A_{G_N}$.  
\end{lemma}

\begin{proof} 
Let $t:=\tanh(\frac{\beta}{N\theta_N})$, and let $G_N$ be a fixed simple graph on vertex set $[N]:=\{1,\dots,N\}$. Then expanding the product gives, 
\begin{align*}
\prod_{1\le i<j\le N}\left(1+\sigma_i\sigma_j A_{G_N}(i,j)t\right) &= \sum_{ F\subseteq \binom{[N]}{2} } t^{|F|} \left( \prod_{(i, j)\in F} A_{G_N}(i,j) \right) \left( \prod_{(i, j)\in F} \sigma_i\sigma_j\right)\\
&=\sum_{\Gamma\subseteq K_N} t^{|E(\Gamma)|}\mathbf 1\{\Gamma\subseteq G_N\}
\prod_{(i, j)\in E(\Gamma)}\sigma_i\sigma_j ,  
\end{align*}
where the notation $\Gamma \subseteq K_N$ denotes that $\Gamma$ is a spanning subgraph of the complete graph $K_N$ on $N$ vertices, and $$\binom{[N]}{2} := \{(i, j): 1\le i < j \le N\}.$$ 
Now taking expectation under $\mu_N$ gives, 
\begin{equation}
\label{clusterexpand}    
\mathbb E_{\mu_N}\Bigg[\prod_{1\le i<j\le N}\Big(1+\sigma_i\sigma_j A_{G_N}(i,j)t\Big)\Bigg]
=
\sum_{\Gamma\subseteq K_N}
t^{|E(\Gamma)|}\mathbf 1\{\Gamma\subseteq G_N\}
\mathbb E_{\mu_N}\Bigg[\prod_{(i, j)\in E(\Gamma)}\sigma_i\sigma_j\Bigg].
\end{equation}
For a fixed graph $\Gamma$, let $d_\Gamma(v)$ be the degree of vertex $v$ in $\Gamma$. Then
\[
\prod_{(i, j)\in E(\Gamma)}\sigma_i\sigma_j
=
\prod_{v=1}^N \sigma_v^{\,d_\Gamma(v)}.
\]
This is because each edge incident to $v$ contributes one factor of $\sigma_v$, and this happens exactly $d_\Gamma(v)$ times. Hence,
\begin{align*}
\mathbb E_{\mu_N}\Bigg[\prod_{v=1}^N \sigma_v^{\,d_\Gamma(v)}\Bigg] =
\prod_{v=1}^N \mathbb E_{\mu_N}\big[\sigma_v^{\,d_\Gamma(v)}\big] =
\begin{cases}
1,& d_\Gamma(v)\text{ is even for every }v,\\
0,& \text{otherwise}.
\end{cases}
\end{align*}
Hence, $\mathbb E_{\mu_N} \left[ \prod_{(i, j)\in E(\Gamma)}\sigma_i\sigma_j \right]=
\mathbf 1\{\Gamma\in \mathcal E_N\}$. Substituting this into \eqref{clusterexpand} gives, 
\begin{equation}\label{prclaimdel}
    \mathbb E_{\mu_N}\Bigg[\prod_{1\le i<j\le N}\Big(1+\sigma_i\sigma_j A_{G_N}(i,j)t\Big)\Bigg]
=
\sum_{\Gamma\subseteq K_N}
t^{|E(\Gamma)|}\mathbf 1\{\Gamma\subseteq G_N\}\mathbf 1\{\Gamma\in\mathcal E_N\}.
\end{equation}
Lemma \ref{lm:ZNbetaexpansion} now follows from \eqref{prclaimdel}.
\end{proof}

Since the empty spanning subgraph belongs to $\mathcal E_N$,
Lemma~\ref{lm:ZNbetaexpansion} gives $\tilde Z_N(\beta)\ge1$.
Moreover, every nonempty graph in which all vertices have even degree
can be decomposed into an edge-disjoint union of cycles
\cite[Theorem 2.17]{bondy2008graph}. Hence, using
Lemma~\ref{lm:ZNbetaexpansion},
\[
\tilde Z_N(\beta)
\le
\prod_{\gamma\in\mathcal C_N}
\left(
1+
\left(\tanh\left(\frac{\beta}{N\theta_N}\right)\right)^{|E(\gamma)|}
\mathbf 1\{\gamma\sqsubseteq G_N\}
\right).
\]

Hence,
\begin{align}\label{eq:ZNcycle}
    \log \tilde{Z}_N(\beta)
    &\leq \sum_{\gamma \in\cC_N}\log \left(1+\left(\tanh\left(\frac{\beta}{N\theta_N} \right)\right)^{|E(\gamma)|} \one\left\{\gamma \sqsubseteq G_{N}\right\} \right) \nonumber \\
    &\leq \sum_{\gamma \in\cC_N}\left(\tanh\left(\frac{\beta}{N\theta_N} \right)\right)^{|E(\gamma)|}\one\left\{\gamma \sqsubseteq G_{N} \right\} \nonumber  \\
    &\leq \sum_{\gamma \in\cC_N}\left(\frac{\beta}{N\theta_N} \right)^{|E(\gamma)|}\one\left\{\gamma \sqsubseteq G_{N} \right\}.
\end{align}

Next, define
\begin{align*}
    U_N :=\sum_{\gamma \in\cC_N}\left(\frac{\beta}{N\theta_N} \right)^{|E(\gamma)|}\one\left\{\gamma \sqsubseteq G_{N} \right\} . 
\end{align*}
Then from \eqref{eq:ZNcycle} we get, 
\begin{align}\label{eq:ANBN}
0\le \log \tilde{Z}_N(\beta) \le U_N.
\end{align}
Hence, to show $\log \tilde{Z}_N(\beta)$ is tight, it suffices to show that $U_N$ is tight. Towards this, define:
    \begin{align*}
        \bm{W}_N(i,j)=\begin{cases}
            W\left(\frac{i}{N}, \frac{j}{N}\right) & \text{ if } i\neq j \\
            0                                      & \text{ otherwise}.
        \end{cases}
    \end{align*}

Note that if $\{i_1, i_2, \ldots, i_s\}$ is a cycle of length $s \geq 3$, then 
$$\P\left(\text{the cycle}~i_1i_2\ldots i_si_1~\text{ is contained in } G_N\right) = \theta_N^s \prod_{a=1}^{s} W\left(\frac{i_a}{N}, \frac{i_{a+1}}{N}\right),$$
where $i_{s+1} = i_1$.  Note that $\frac{1}{N} \lambda_1(\bm W_N) \to \lambda_1(W)$. Since $\beta \lambda_1(W)<1$, we can choose $\varepsilon >0$ such that $0 < \frac{\beta}{N} \lambda_1(\bm W_N) <1-\varepsilon$ for large enough $N$. Further, $\bm W_N$ being symmetric with nonnegative entries, Perron--Frobenius
gives $\max_i|\lambda_i(\bm W_N)|=\lambda_1(\bm W_N)$ thereby implying that $\frac{1}{N} \beta |\lambda_i(\bm W_N)| <1-\varepsilon$. Now, we have: 
\begin{align*} 
\E U_N & = \sum_{s=3}^\infty \left(\frac{\beta}{N\theta_N} \right)^{s} \frac{1}{2s}\sum_{1\leq i_1 \ne i_2 \ne \ldots \ne i_s \leq N}\P\left(\text{the cycle}~i_1i_2\ldots i_si_1~\text{ is contained in } G_N\right) \nonumber \\ 
& = \sum_{s=3}^\infty \left(\frac{\beta}{N} \right)^{s} \frac{1}{2s}\sum_{1\leq i_1 \ne i_2 \ne \ldots \ne i_s \leq N} \prod_{a=1}^{s} W\left(\frac{i_a}{N}, \frac{i_{a+1}}{N}\right)\nonumber \\ 
& \le \sum_{s=3}^\infty  \left(\frac{\beta}{N} \right)^{s} \frac{\tr(\bm W_N^s)}{2s}    \nonumber                              \\
& \leq \sum_{s=3}^\infty \sum_{i=1}^N \left(\frac{\beta |\lambda_i(\bm W_N)|}{N}\right)^s    \nonumber                              \\
& = \sum_{i=1}^N\left(\frac{\beta |\lambda_i(\bm W_N)|}{N}\right)^3\frac{1}{1-\frac{\beta |\lambda_i(\bm W_N)|}{N}}   \nonumber                              \\
& \leq \sum_{i=1}^N\left(\frac{\beta \lambda_i(\bm W_N)}{N}\right)^2  \frac{1}{\varepsilon} \nonumber \\
& = \frac{\beta^2}{\varepsilon N^2}\tr (\bm W_N^2)\\
& = \frac{\beta^2}{\varepsilon N^2}\sum_{1\leq i\neq j \leq N} W\left(\frac{i}{N},\frac{j}{N}\right)^2\\
&\leq \frac{\beta^2}{\varepsilon}.
\end{align*}
Therefore, $\E U_N \lesssim \frac{\beta^2}{\varepsilon}$, for large enough $N$. This shows that $U_N$ is tight. 

The above argument combined with \eqref{eq:ANBN} shows that $\log \tilde{Z}_N(\beta)$ is tight. This implies, in the regime $\theta_N \ll N^{-\frac{2}{3}}$,
\begin{align}\label{eq:expry65}
    N\theta_N^{\frac{3}{2}} \log \tilde{Z}_N(\beta) = o_P(1).
\end{align}
Now, recalling \eqref{eq:ANij} gives, 
$$N\theta_N^{\frac{3}{2}} \E \left[\frac{\beta}{2N\theta_N}\vartheta_N \right] = \frac{\beta\theta_N^{\frac{3}{2}}}{2}\sum_{i=1}^N W\left(\frac{i}{N}, \frac{i}{N}\right)$$
and 
\begin{align*}
    N^2\theta_N^{3} \mathrm{Var}\left[\frac{\beta}{2N\theta_N}\vartheta_N\right] \lesssim \frac{1}{N^{\frac{1}{3}}},
\end{align*}
 in the regime $\theta_N \ll N^{-\frac{2}{3}}$. Hence,
\begin{align}\label{ol2}
    N\theta_N^{\frac{3}{2}} \left( \frac{\beta}{2N\theta_N}\vartheta_N \right) = \frac{\beta\theta_N^{\frac{3}{2}}}{2}\sum_{i=1}^N W\left(\frac{i}{N}, \frac{i}{N}\right) + o_{L^2}(1).
\end{align}
Finally, note that
\begin{align}\label{twodecom89}
    N\theta_N^{\frac{3}{2}} \log\cosh\left(\frac{\beta}{N\theta_N}\right)\sum_{1 \leq i < j \leq N}A_{G_N}(i,j) &= N\theta_N^{\frac{3}{2}} \log\cosh\left(\frac{\beta}{N\theta_N}\right)\sum_{1 \leq i < j \leq N}\overline{A_{G_N}(i,j)} \\ 
    & \hspace{0.45in} + N\theta_N^{\frac{5}{2}} \log\cosh\left(\frac{\beta}{N\theta_N}\right)\sum_{1 \leq i < j \leq N}W\left(\frac{i}{N}, \frac{j}{N}\right) .  \nonumber  
\end{align}
By the Lindberg--Feller CLT, 
\begin{align*}
    \frac{\sum_{1 \leq i < j \leq N}\overline{A_{G_N}(i,j)}}{\sqrt{\frac{N^2 \theta_N}{2} \int_{[0, 1]^2} W(x,y)\mathrm{d}x\mathrm{d}y}} \xrightarrow{D} \cN(0,1) . 
\end{align*}
This, together with the observation $N\theta_N^{\frac{3}{2}}\log\cosh\left(\beta/(N\theta_N)\right) \sim \beta^2/(2\sqrt{N^2\theta_N})$ in the regime of interest, implies that:
\begin{align}\label{logcosh67}
    N\theta_N^{\frac{3}{2}} \log\cosh\left(\frac{\beta}{N\theta_N}\right)\sum_{1 \leq i < j \leq N}\overline{A_{G_N}(i,j)} \xrightarrow{D} \mathcal{N}\left(0,\frac{\beta^4}{8} \int_{[0, 1]^2} W(x,y)\mathrm{d}x\mathrm{d}y \right).
\end{align} 
The proof of \eqref{eq:23bthreshold} then follows from \eqref{exprlogzn7}, \eqref{eq:expry65}, \eqref{ol2}, \eqref{twodecom89} and \eqref{logcosh67}. 
Now, recall from \eqref{holdapprox62} that if $W$ is $\alpha$-H\"older continuous, then
$$\left|\frac{1}{N^2}\sum_{1\le i,j\le N} W\left(\frac{i}{N}, \frac{j}{N}\right) - \int_{[0,1]^2} W(x,y)~\mathrm{d}x \mathrm{d}y\right| \le N^{-\alpha}.$$
Therefore, for $\theta_N \ll N^{-\frac{2}{3}}$ and $\alpha \ge \frac{2}{3}$, 
\begin{align}\label{wholdtrn32}
   & N\theta_N^\frac{3}{2}
   \left|\theta_N \log\cosh\left(\frac{\beta}{N\theta_N}\right)
   \sum_{1 \leq i < j \leq N}W\left(\frac{i}{N},\frac{j}{N}\right)
   - \frac{N^2\theta_N}{2}\log\cosh\left(\frac{\beta}{N\theta_N}\right)
   \int_{[0,1]^2} W(x,y) \mathrm{d}x \mathrm{d}y\right| \nonumber \\
   &\leq
   \frac{N^3\theta_N^\frac{5}{2}}{2}
   \log \cosh \left(\frac{\beta}{N\theta_N}\right)
   \left|\frac{1}{N^2}\sum_{1\le i,j\le N}
   W\left(\frac{i}{N}, \frac{j}{N}\right)
   - \int_{[0,1]^2} W(x,y)~\mathrm{d}x \mathrm{d}y\right| \nonumber \\
   &\quad
   +\frac{N\theta_N^\frac{5}{2}}{2}
   \log \cosh \left(\frac{\beta}{N\theta_N}\right)
   \sum_{i=1}^N W\left(\frac{i}{N},\frac{i}{N}\right) \nonumber \\
   &=
   O(N^{1-\alpha}\sqrt{\theta_N}) + O(\sqrt{\theta_N}) \nonumber \\
   &= o(N^{\frac{2}{3}-\alpha})+o(1)
   =o(1),
\end{align} 

The result in \eqref{eq:23bthresholdW} now follows from \eqref{eq:23bthreshold} and 
\eqref{wholdtrn32}. 
\hfill $\Box$

\subsection{Proofs of Technical Lemmas for Theorem \ref{thm:ZN} }\label{sec:analysis_partition}

This section is organized as follows: We begin by proving  Lemma \ref{var_sum} in Appendix \ref{proof6p2}. Lemma \ref{lem:clt_for_delta_N1}, \ref{lem:clt_for_delta_N2}, and \ref{lem:expect_sum} are proved in Appendix \ref{proofdeltaN1}, \ref{proofdeltaN28}, \ref{proof6p246}, respectively.

\subsubsection{Proof of Lemma \ref{var_sum}}\label{proof6p2}

For $\bs \in \{-1, 1\}^N$, define 
    \begin{align}\label{eq:QN}
        Q_N(\bs):=\frac{1}{N^3 \theta_N^2}\left|\sum_{1 \leq i < j \leq N}W\left(\frac{i}{N},\frac{j}{N}\right)\sigma_i \sigma_j\right|.
    \end{align}
    Also, for $\bs, \bt \in \{-1, 1\}^N$ define 
    \begin{align}\label{eq:RN}
        R_N(\bs,\bt):=\frac{\beta^2}{\theta_N N^2} \sum_{1 \leq i < j \leq N}W\left(\frac{i}{N},\frac{j}{N}\right)\left(1-\theta_N W\left(\frac{i}{N},\frac{j}{N}\right)\right)\sigma_i\sigma_j\tau_i\tau_j.
    \end{align}
Now, recall the definition of $\hat Z_N(\beta)$ from \eqref{modpart27}. Then from Lemma \ref{lem:expvariance} it follows that 
\begin{align}\label{eq:ZHbetasigmatau}
        & \var \hat Z_N(\beta)  \nonumber \\ 
        & = \frac{1}{4^N} \sum_{\bs, \bt \in \{-1, 1\}^N} \cov \left[T(\bs),T(\bt)\right]              \\  
        & =   \frac{1}{4^N} \sum_{\bs, \bt \in \{-1, 1\}^N } \E T(\bs) \E T(\bt) \left(e^{\textstyle R_N(\bs,\bt)+O(Q_N(\bs))+O(Q_N(\bt))+O\left(\frac{1}{N\theta_N}Q_N(\bs\bt)\right)}-1\right) , \nonumber 
\end{align} 
where $\bs \bt = (\sigma_1 \tau_1, \sigma_2 \tau_2, \ldots \sigma_N \tau_N)$. The next lemma shows that in the asymptotic regime of Lemma \ref{var_sum}, the exponential term in \eqref{eq:ZHbetasigmatau} can be linearized. As before, $\mu_N$ will denote the uniform measure $\{-1,+1\}^{N}$.

\begin{lemma}
    \label{lem:taylor_expRn} 
    For $\bs, \bt \in \{-1, 1\}^N$ and $L_0, L_1$ define 
      $$\Lambda_N :=R_N(\bs,\bt)+L_0\left(Q_N(\bs)+Q_N(\bt)\right)+ \frac{L_1 }{N \theta_N} Q_N(\bs \bt). $$
    Then 
    \begin{align*}
       \E_{\mu_N}\left[e^{ \Lambda_N} -( 1 + \Lambda_N) \right] \ll \frac{1}{N^2\theta_N^2}  ,  
    \end{align*}
    where the expectation is taken over independent choices of $\bs,\bt$ from $\mu_N
$. 
\end{lemma}

\begin{proof}
    By a second order Taylor expansion, 
    $$
        |e^{ \Lambda_N} - 1 - \Lambda_N| \leq \tfrac12 e^{\textstyle |\Lambda_N|} \Lambda_N^2. 
    $$
    Hence, 
    $$
        N^2 \theta_N^2 \E_{\mu_N} \left[\left(e^{ \Lambda_N} - 1 - \Lambda_N\right)^2\right] \leq \tfrac14 N^2 \theta_N^2 \E_{\mu_N}\left[e^{  2|\Lambda_N|} \Lambda_N^4\right].
    $$
    By the Cauchy-Schwarz inequality, to prove Lemma \ref{lem:taylor_expRn} it suffices to show that
    \begin{align}\label{eq:expansionpf}
        \E_{\mu_N}\left[e^{  4|\Lambda_N|}\right] \lesssim 1 
        \quad \text{ and } \quad
        N^4 \theta_N^4 \E_{\mu_N}[\Lambda_N^8] \ll 1 .
    \end{align}
From Proposition \ref{quadJointAsym}, we know that $Z_N:=N\theta_N \Lambda_N$ has a non-degenerate limiting distribution. This means, $N^4 \theta_N^4 \Lambda_N^8 = \frac{1}{N^4\theta_N^4} Z_N^8$ converges in probability to zero, since $N\theta_N \gg 1$.  Consequently, if we can show that $\E (N\theta_N\Lambda_N)^{16} \lesssim 1$, then uniform integrability will give us the second condition in \eqref{eq:expansionpf}. Hence, to prove Lemma \ref{lem:taylor_expRn} it suffices to check the following:    
\begin{align}\label{eq:uniform_integr_lambda}
        \E_{\mu_N}\left[e^{  4|\Lambda_N|}\right] \lesssim 1 
        \quad \text{ and } \quad
        \E_{\mu_N}\left[(N\theta_N\Lambda_N)^{16}\right] \lesssim 1.
    \end{align}
    
    First, we will verify the second condition in \eqref{eq:uniform_integr_lambda}. To this end, define
    $$\bm{W}_N :=\left(W\left(\frac{i}{N},\frac{j}{N}\right) \bm{1}\{i\ne j\}\right)_{1\leq i,j\leq N}.$$ 
    Also, define $$\tilde{\bm{W}}_N :=\left(W\left(\frac{i}{N},\frac{j}{N}\right)\left(1-\theta_N W\left(\frac{i}{N},\frac{j}{N}\right) \right)\bm{1}\{i\ne j\}\right)_{1\leq i,j\leq N}.$$
    Then using a standard union bound gives, 
    \begin{align}
        \label{eq:union_bound}
            \P_{\mu_N}\left(N \theta_{N}\left|\Lambda_{N}\right| \geq t\right) & \leq \sum_{i=1}^4 \P_{\mu_N} \left( S_i \geq \frac{t}{4} \right) , 
            \end{align}
            where $S_1 := N \theta_{N}\left|R_{N}(\bs, \bt)\right|$, $S_2 := L_{0} N \theta_{N} Q_{N}(\bs)$, $S_3 :=L_{0} N \theta_{N} Q_{N}(\bt)$, and $S_4  := L_{1}  Q_{N}(\bs\bt)$. Recalling \eqref{eq:RN} note that 
$$S_1 = \frac{\beta^2}{2 N}\left|(\bs \bt)^{\top} \tilde{\bm{W}}_N (\bs \bt)\right|. $$
Since the entries of $\tilde{\bm{W}_N}$ are bounded by $1$, we have $\max\{\|\tilde{\bm{W}}_N\|_{2},\|\tilde{\bm{W}}_N\|_{F}\}\leq N$. Hence, by the Hanson-Wright inequality \cite{Rudelson2013HansonWrightIA}, 
$$\P_{\mu_N}\left( S_1 \geq \frac{t}{4} \right) \leq 2e^{  -C\min\{t^2,t\}},$$ for some constant $C>0$ and for all $t\ge 0$ and sufficiently large $N$. Next, recalling \eqref{eq:QN} gives, 
$$S_2 = \frac{L_{0}}{2 N^{2} \theta_{N}}\left|\bs^{\top} \bm{W}_N \bs\right| $$
Again, since the entries of $\bm{W}$ are bounded by $1$, we have $\max\{\|\bm{W}_N\|_{2}, \|\bm{W}_N\|_F\}\leq N$. Hence, as before, by the Hanson-Wright inequality \cite{Rudelson2013HansonWrightIA}, we have for large $N$: 
$$\P_{\mu_N}\left( S_2 \geq \frac{t}{4} \right) \leq \P_{\mu_N} \left(\frac{L_{0}}{2 N}\left|\bs^{\top} \bm{W}_N \bs\right| \ge \frac{t}{4}\right)\leq 2e^{  -C\min\{t^2,t\}}, $$ 
since $N\theta_N \to\infty$. Similarly, it can be 
shown that 
$$\P_{\mu_N}\left( S_3 \geq \frac{t}{4} \right) \leq 2e^{  -C\min\{t^2,t\}} \quad \text{ and } \quad \P_{\mu_N}\left( S_4 \geq \frac{t}{4} \right) \leq 2e^{  -C\min\{t^2,t\}}.$$ 
Combining the above tail bounds with \eqref{eq:union_bound} gives 
\begin{align}\label{eq:lambda}
\P_{\mu_N}\left(N \theta_{N}\left|\Lambda_{N}\right| \geq t\right) \leq 8 e^{  -C\min\{t^2,t\}}.
\end{align}  
Hence, 
    \begin{align*}
        \E_{\mu_N}\left[(N\theta_N\Lambda_N)^{16} \right]  = \int_0^\infty \P_{\mu_N}(N\theta_N|\Lambda_N|>t^{\frac{1}{16}}) \mathrm{d}t \le 8\int_0^\infty e^{  -C\min\{t^{\frac{1}{8}},t^{\frac{1}{16}}\}} \mathrm{d}t \lesssim 1 . 
    \end{align*}
    This proves the second claim in \eqref{eq:uniform_integr_lambda}. 

Next, we will prove the first condition in \eqref{eq:uniform_integr_lambda}. Note that, since $N\theta_N\to\infty$, for any fixed $\alpha > 0$ we have $N\theta_N > \alpha$, for sufficiently large $N$. Hence, from \eqref{eq:lambda}, 
    \begin{equation}\label{jjjke1}
\P_{\mu_N}\left(\left|\Lambda_N\right| \geq t \right) \leq \P_{\mu_N}\left(N\theta_N \left|\Lambda_N\right| \geq N\theta_N t \right) \leq \P_{\mu_N}\left(N\theta_N \left|\Lambda_N\right| \geq \alpha t \right) \leq 8 e^{  - C\min\{\alpha^2 t^2, \alpha t\}} , 
    \end{equation} for any $\alpha>0$ and sufficiently large $N$. This implies, 
    \begin{align}\label{jjjke2}
        \E_{\mu_N}\left[e^{  4|\Lambda_N|}\right]= \int_0^\infty \P_{\mu_N}\left(|\Lambda_N| \ge \frac{1}{4}\log t\right) \mathrm{d}t \le e^{  4/\alpha} + 8\int_{e^{  4/\alpha}}^\infty t^{- \frac{\alpha C}{4} } ~\mathrm{d}t \lesssim 1 , 
    \end{align} 
    if $\alpha > \frac{4}{C}$. This completes the proof of Lemma \ref{lem:taylor_expRn}. 
\end{proof}

The technique used in the proof above can also be applied to derive general moment bounds for $Q_N(\bs)$. We state this result in the following lemma, which will be used later.

\begin{lemma}\label{lem:Qnlp} Let $Q_N(\bs)$ be as defined in \eqref{eq:QN}. Then for any $1\leq r <\infty$, 
    \begin{align*}
        \E_{\mu_N} |Q_N(\bs)|^r \lesssim \frac{1}{N^{2r}\theta_N^{2r} } .
    \end{align*}
\end{lemma}

\begin{proof} Note that 
$$(N\theta_N)^2 Q_N(\bs) = \frac{1}{2 N}\left|\bs^{\top} \bm{W}_N \bs\right|.$$
Hence, by applying the Hanson-Wright inequality as in the proof of Lemma \ref{lem:taylor_expRn}, we get 
    \begin{align*}
        \P_{\mu_N}\left((N\theta_N)^2\left|Q_N(\bs)\right| \geq t\right) \leq 2e^{  -C\min\{4t^2, 2t\}}, 
    \end{align*} 
    for some constant $C > 0$ and $N$ large enough. 
    Therefore, 
        \begin{align*}
        \E_{\mu_N}\left[(N\theta_N)^{2r} |Q_N(\bs)|^r \right] \le \int_0^\infty \P_{\mu_N}\left(N^2\theta_N^2 |Q_N(\bs)| \ge t^{\frac{1}{r}}\right) ~\mathrm{d}t \le 2 \int_0^\infty e^{  -C\min\{4t^{\frac{2}{r}}, 2t^{\frac{1}{r}}\}}~\mathrm{d}t \lesssim 1,
    \end{align*}
    which completes the proof of Lemma \ref{lem:Qnlp}. 
\end{proof}


We now return to the proof of Lemma \ref{var_sum}. Applying Lemma \ref{lem:taylor_expRn} to \eqref{eq:ZHbetasigmatau} gives,  
\begin{align}\label{vznshep}
    \var \hat Z_N(\beta) & = \E_{\mu_N}\Bigg[ \E T(\bs) \E T(\bt)\Bigg(R_N(\bs,\bt)+O(Q_N(\bs))+O(Q_N(\bt))+O\left(\frac{1}{N\theta_N}Q_N(\bs\bt)\right)\nonumber\\ 
    & \hspace{0.85in} + o_{L^2}\left(\frac{1}{N\theta_N}\right)\Bigg) \Bigg].
\end{align}
Also, from Lemma \ref{lem:expect_sum} we know that $\E \hat{Z}_N(\beta) \gtrsim 1$. Hence, in order to prove Lemma \ref{var_sum}, it suffices to establish the following claims: 
\begin{enumerate}[label=$(\mathbf{C\arabic*})$]

\item\label{itm:C1} $\E_{\mu_N}\left[ R_N(\bs,\bt) \E T(\bs) \E T(\bt)  \right] \ll \frac{1}{N\theta_N} $. 

\item\label{itm:C2} $\E_{\mu_N} \left[  \left( Q_N(\bs) + Q_N(\bt) \right) \E T(\bs) \E T(\bt) \right] \ll \frac{1}{N\theta_N} $.  
 
\item\label{itm:C3} $\E_{\mu_N} \left[ Q_N(\bs \bt) \E T(\bs) \E T(\bt) \right] \ll 1$. 

\end{enumerate}

\subsubsection*{Proof of \ref{itm:C1}} By Lemma \ref{lem:expvariance} we know that 
 \begin{align*}
        \E T\left(\bs\right) = e^{\textstyle  \mathcal{E}_N(\beta, W) + \frac{\beta}{N}\sum_{1 \leq i < j \leq N}W\left(\frac{i}{N},\frac{j}{N}\right)\sigma_i \sigma_j +o\left(\frac 1{N\theta_N}\right) + O\left( Q_N(\bs) \right) } .
    \end{align*}
   where $\mathcal{E}_N(\beta, W)$ is defined in \eqref{eq:JN}. Note that, since $N^2 \theta_N^3 \gtrsim 1$ and $|W| \leq 1$, we have $\mathcal{E}_N(\beta, W) \lesssim_{\beta} 1$. Hence, there exists a constant $L_0 > 0$ such that, 
\begin{align}\label{eq:C1pf1}
& N\theta_N  \E_{\mu_N} \big[ R_N(\bs,\bt)\E T(\bs) \E T(\bt) \big]  \nonumber \\ 
& \lesssim_{\beta} N\theta_N \E_{\mu_N} \left[R_N(\bs,\bt) e^{\textstyle  \frac{\beta}{N}\sum_{1 \leq i < j \leq N}W\left(\frac{i}{N},\frac{j}{N}\right)(\sigma_i \sigma_j + \tau_i \tau_j) + L_0 \left( Q_N(\bs) +  Q_N(\bt) \right) } \right]\nonumber \\ 
& = \E_{\mu_N} [U_N] + \E_{\mu_N} [V_N], 
\end{align} 
where 
\begin{align}\label{eq:U}
\begin{aligned} 
        U_N & :=  N\theta_N R_N(\bs,\bt) e^{\textstyle \frac{\beta}{N}\sum_{1 \leq i < j \leq N}W\left(\frac{i}{N},\frac{j}{N}\right)(\sigma_i\sigma_j+\tau_i\tau_j)} , \\
        V_N & :=  N\theta_N R_N(\bs,\bt) \left( e^{\textstyle L_0( Q_N(\bs) + Q_N(\bt)) }-1\right) e^{\textstyle  \frac{\beta}{N}\sum_{1 \leq i < j \leq N}W\left(\frac{i}{N},\frac{j}{N}\right)(\sigma_i\sigma_j+\tau_i\tau_j) }  . 
        \end{aligned} 
\end{align} 
Now, we can write
\begin{align*}
    N \theta_N R_N(\bs,\bt)=\tilde{R}_N(\bs,\bt)-\bar{R}_N(\bs,\bt),
\end{align*}
where
\begin{align}\label{eq:RN12}
\begin{aligned} 
        \tilde{R}_N(\bs,\bt) & =\frac{\beta^2}{N}\sum_{1 \leq i< j \leq N}W\left(\frac{i}{N},\frac{j}{N}\right)\sigma_i\sigma_j\tau_i\tau_j  , \\
        \bar{R}_N(\bs,\bt) &= \frac{\theta_N\beta^2}{N}\sum_{1 \leq i< j \leq N}W\left(\frac{i}{N},\frac{j}{N}\right)^2\sigma_i\sigma_j\tau_i\tau_j . 
        \end{aligned} 
\end{align}

\begin{lemma}
    \label{sigtauUniInt}
    Let $\tilde{R}_N(\bs,\bt)$ and $\bar{R}_N(\bs,\bt)$ be as defined above. 
For $\bs,\bt \in \{-1, 1\}^N$, define 
    $$\tilde X_N:= \tilde{R}_N(\bs,\bt) e^{\textstyle  \frac{\beta}{N}\sum_{1 \leq i < j \leq N}W\left(\frac{i}{N},\frac{j}{N}\right)(\sigma_i\sigma_j+\tau_i\tau_j)}$$ and $$\bar{X}_N:= \bar{R}_N(\bs,\bt) e^{\textstyle  \frac{\beta}{N}\sum_{1 \leq i < j \leq N}W\left(\frac{i}{N},\frac{j}{N}\right)(\sigma_i\sigma_j+\tau_i\tau_j)}.$$ 
    Then there exists $\delta> 0$ such that 
    $$\limsup_{N \to \infty}\E |\tilde X_N |^{1+\delta} < \infty \quad \text{ and } \quad \limsup_{N \to \infty}\E |\bar{X}_N|^{1+\delta}  < \infty,$$
where the expectation is taken over independent choices of $\bs,\bt$ from $\mu_N$.      \end{lemma}
     
\begin{proof} Fix $\delta > 0$. Then applying H\"older's inequality gives, 
    \begin{align*}
      &\E |\tilde X_N|^{1+\delta} \\ & = \E\left[|\tilde{R}_N(\bs,\bt)|^{1+\delta}  e^{\textstyle \frac{\beta(1+\delta)}{2 N} \sum_{1 \leq i,j \leq N} W\left(\frac{i}{N}, \frac{j}{N}\right) \left(\sigma_{i} \sigma_{j}+\tau_i \tau_j\right) } \right]  \nonumber \\
            & =  e^{\textstyle \frac{\beta(1+\delta)}{N} \sum_{i} W\left(\frac{i}{N}, \frac{i}{N}\right)} \E\left[|\tilde{R}_N(\bs,\bt)|^{1+\delta} e^{\textstyle \frac{\beta(1+\delta)}{2 N} \sum_{1 \leq i \neq j \leq N} W\left(\frac{i}{N}, \frac{j}{N}\right) \left(\sigma_{i} \sigma_{j}+\tau_{i} \tau_j\right)  } \right] \nonumber \\ 
            &  \lesssim_{\beta, \delta}  \E\left[|\tilde{R}_N(\bs,\bt)|^{1+\delta} e^{\textstyle \frac{\beta(1+\delta)}{2 N} \sum_{1 \leq i \neq j \leq N} W\left(\frac{i}{N}, \frac{j}{N}\right) \left(\sigma_{i} \sigma_{j}+\tau_{i} \tau_j\right)  } \right] \nonumber                                                                                                    \\
        & \leq \left( \E \left[ e^{\textstyle  \frac{\beta(1+\delta)(1+ \tilde{\delta})}{2 N} \sum_{1 \leq i \neq j \leq N} W\left(\frac{i}{N}, \frac{j}{N}\right) \left(\sigma_{i} \sigma_{j}+\tau_{i} \tau_{j}\right) } \right] \right)^{\frac{1}{1+\tilde{\delta}}} \left( \E\left[|\tilde{R}_N(\bs,\bt)|^{(1+\delta)\frac{1+\tilde{\delta}}{\tilde{\delta}}}\right] \right)^{\frac{\tilde{\delta}}{1+\tilde{\delta}}} . 
    \end{align*}
    From Lemma \ref{expUniInt}, we know that for $\delta, \tilde\delta$ small enough, 
    \begin{align*}
        \limsup_{N \to \infty} \E e^{\textstyle  \frac{\beta(1+\delta)(1+\tilde\delta)}{2 N} \sum_{1 \leq i \neq j \leq N}\left(\sigma_{i} \sigma_{j}+\tau_{i} \tau_{j}\right) W\left(\frac{i}{N}, \frac{j}{N}\right)} <\infty.
    \end{align*}
    Also, $|\tilde{R}_N(\bs,\bt)| = (\beta N\theta_N)^2 Q_N(\bs \bt)$, for $Q_N$ as defined in \eqref{eq:QN}. Hence, by Lemma \ref{lem:Qnlp}, all its moments are bounded. This shows, $\limsup_{N \to \infty} \E|\tilde X_N|^{1+\delta}  < \infty$, for $\delta>0$ small enough. By a similar argument, we also have:
\begin{align}\label{xnbar6684}
    & \E |\bar X_N|^{1+\delta}  \\ 
    &\lesssim_{\beta,\delta} \left( \E \left[ e^{\textstyle  \frac{\beta(1+\delta)(1+ \tilde{\delta})}{2 N} \sum_{1 \leq i \neq j \leq N} W\left(\frac{i}{N}, \frac{j}{N}\right) \left(\sigma_{i} \sigma_{j}+\tau_{i} \tau_{j}\right) } \right] \right)^{\frac{1}{1+\tilde{\delta}}} \left( \E\left[|\bar{R}_N(\bs,\bt)|^{(1+\delta)\frac{1+\tilde{\delta}}{\tilde{\delta}}}\right] \right)^{\frac{\tilde{\delta}}{1+\tilde{\delta}}} .  \nonumber  
\end{align} 
Once again, by Lemma \ref{expUniInt}, the first term in \eqref{xnbar6684} remains bounded as $N\rightarrow \infty$. On the other hand,
$$|\bar{R}_N(\bs,\bt)| = \frac{\beta^2 \theta_N}{N} \left|(\bs \bt)^\top \bm W_N^{(2)} \bs\bt\right|$$ 
where  $$\bm{W}_N^{(2)} :=\left(W^2\left(\frac{i}{N},\frac{j}{N}\right) \bm{1}\{i\ne j\}\right)_{1\leq i,j\leq N}$$ 
Since $\max\{\|\bm W_N^{(2)}\|_2, \|\bm W_N^{(2)}\|_F\} \le \max\{\|\bm W_N\|_2, \|\bm W_N\|_F\} \le N$, one can apply the Hanson-Wright inequality as in the proof of Lemma \ref{lem:Qnlp} to conclude that all moments of $N^{-1}(\bs \bt )^\top W_N^{(2)} \bs \bt$ are bounded. Since $\theta_N \le 1$, this implies that all moments of $|\bar{R}_N(\bs,\bt)| $ are bounded too, thereby showing that $\limsup_{N \to \infty} \E|\bar X_N|^{1+\delta}  < \infty$ for $\delta >0$ small enough, completing the proof. 
\end{proof}

Now, recalling \eqref{eq:RN12}, $U_N$ (recall \eqref{eq:U}) can be written in the above notation as $U_N = \tilde{X}_N - \bar{X}_N$. By Proposition \ref{quadJointAsym} and the continuous mapping theorem $U_N\xrightarrow{D} U$ for some random variable $U$ with mean $0$, under $\mu_N$. Also, by Lemma \ref{sigtauUniInt} we know that $U_N$ is uniformly integrable. Hence, 
\begin{align} \label{eq:C1pf2} 
\E_{\mu_N} U_N \ll 1. 
\end{align}
Hence, by the uniform integrability of $U_N$ from Lemma \ref{sigtauUniInt}, $\E_{\mu_N} U_N \ll 1$. 
Next, recalling the definition of $V_N$ from \eqref{eq:U} gives, 
$$V_N = U_N  \left( e^{\textstyle L_0( Q_N(\bs) + Q_N(\bt)) }-1\right) . $$
Since $U_N = \tilde{X}_N + \bar{X}_N$, Lemma \ref{sigtauUniInt} implies that $\E_{\mu_N} |U_N|^{1+\delta} \lesssim 1$, for some $\delta > 0$. Therefore, by H\"older's inequality,
\begin{align}\label{LN_Holder}
    \E_{\mu_N} |V_N| &\le \left(\E_{\mu_N} |U_N|^{1+\delta}\right)^\frac{1}{1+\delta} \left(\E_{\mu_N}\left[e^{\textstyle L_0(Q_N(\bs)+Q_N(\bt))}(Q_N(\bs)+Q_N(\bt))\right]^{\frac{(1+\delta)}{\delta}}\right)^{\frac{\delta}{1+\delta}}\nonumber\\&\lesssim  \left(\E_{\mu_N}\left[e^{\textstyle L_0(Q_N(\bs)+Q_N(\bt))}(Q_N(\bs)+Q_N(\bt))\right]^{\frac{(1+\delta)}{\delta}}\right)^{\frac{\delta}{1+\delta}} .
\end{align} 

Next, for any constant $L>0$ with $S := L N\theta_N (Q_N(\bs) + Q_N(\bt))$, we have from the proof of Lemma \ref{lem:taylor_expRn}:
$$\p_{\mu_N}\left(S\ge \frac{t}{2}\right) \le 4 e^{-C\min\{t^2,t\}}$$ from which it follows by an argument similar to \eqref{jjjke1} and \eqref{jjjke2}, that 
\begin{align}\label{qneqn7}
    \E_{\mu_N} e^{\textstyle L (Q_N(\bs) + Q_N(\bt))} \lesssim 1. 
\end{align} 
Hence, applying the Cauchy-Schwarz inequality on \eqref{LN_Holder} and applying Lemma \ref{lem:Qnlp} gives, 
\begin{align}\label{eq:C1pf3}
\E_{\mu_N} |V_N| \lesssim \left(\E_{\mu_N}\left[Q_N(\bs)+Q_N(\bt)\right]^{\frac{2(1+\delta)}{\delta}}\right)^{\frac{\delta}{2(1+\delta)}}  \lesssim \frac{1}{N^2\theta_N^2} \ll 1.
\end{align} 
Combining \eqref{eq:C1pf1}, \eqref{eq:C1pf2}, and \eqref{eq:C1pf3} the result in \ref{itm:C1} follows. \hfill $\Box$

\subsubsection*{Proof of \ref{itm:C2}} For any $\delta > 0$, by H\"older's inequality, 
\begin{align*} 
& N \theta_N \E_{\mu_N} \left[  \left( Q_N(\bs) + Q_N(\bt) \right) \E T(\bs) \E T(\bt) \right] \nonumber \\ 
& \leq \E_{\mu_N}\left[\left(\E T(\bs) \E T(\bt)\right)^{1+\delta}\right]^{\frac{1}{1+\delta}} \E_{\mu_N}\left[\left(N \theta_N \left(Q_N(\bs)+Q_N(\bt)\right)\right)^{\frac{(1+\delta)}{\delta}}\right]^{\frac{\delta}{1+\delta}} . 
\end{align*} 
By Lemma \ref{expUniInt} $(\E T(\bs) \E T(\bt))^{1+\delta}= O_{\mu_N}(1)$ for some $\delta>0$.\footnote{For two sequences of random variables $X_N$ and $Y_N$, we say that $X_N = O_{\mu_N}(Y_N)$, if $\frac{X_N}{Y_N}$ is tight under the probability measure $\mu_N$, that is, given every $\varepsilon > 0$, there exists a constant $M >0$ such that $\sup_{N\ge 1} \mu_N(|X_N| \ge M|Y_N|) < \varepsilon$.} Then applying Lemma \ref{lem:Qnlp} gives, 
$$N \theta_N \E_{\mu_N} \left[  \left( Q_N(\bs) + Q_N(\bt) \right) \E T(\bs) \E T(\bt) \right]  \lesssim \frac{1}{N\theta_N}  \ll 1.
$$
This completes the proof of \ref{itm:C2}. \hfill $\Box$

\subsubsection*{Proof of \ref{itm:C3}} As before, there exists $\delta>0$ such that 
\begin{align*}
    \E_{\mu_N} \left[ Q_N(\bs \bt) \E T(\bs) \E T(\bt) \right] 
   & \lesssim \E_{\mu_N}\left[\left(\E T(\bs) \E T(\bt)\right)^{1+\delta}\right]^{\frac{1}{1+\delta}} \E_{\mu_N}\left[Q_N(\bs\bt)^{\frac{(1+\delta)}{\delta}}\right]^{\frac{\delta}{1+\delta}} \nonumber \\ 
   & \lesssim \frac{1}{N^2\theta_N^2} \ll 1.
\end{align*} 
This completes the proof of \ref{itm:C3}.  \hfill $\Box$ \\

Combining \ref{itm:C1},  \ref{itm:C2}, and  \ref{itm:C3} with \eqref{vznshep}, shows that 
$\var \hat Z_N(\beta) \ll \frac{1}{N \theta_N}$. Since $\E \hat{Z}_N(\beta) \gtrsim 1$ by  from Lemma \ref{lem:expect_sum}, the proof of Lemma \ref{var_sum} is now complete. \hfill $\Box$

\subsubsection{Proof of Lemma \ref{lem:clt_for_delta_N1}}\label{proofdeltaN1}     
   We will first prove \eqref{eq:linear11}. Recalling \eqref{eq:gammatheta} gives, 
    \begin{align}\label{eq:tech2}
        \sqrt{N\theta_N} \overline{\Delta}_N
        =
        \sum_{1 \leq i < j \leq N} \frac{\beta^2}{2(N\theta_N)^{\frac{3}{2}}} \overline{A_{G_N}(i,j)}
        +
        \sum_{i=1}^N
        \left(\frac{\beta^2}{4(N\theta_N)^{\frac{3}{2}}} + \frac {\beta}{2\sqrt{N \theta_N}}\right) \overline{A_{G_N}(i,i)} . 
    \end{align}
    The RHS of \eqref{eq:tech2} is a sum of independent random variables. Hence, for $N\theta_N^2 \gg 1$,
    \begin{align}
        \var [\sqrt{N\theta_N} \overline{\Delta}_N]
        & =    \sum_{1 \leq i\neq j \leq N} \frac{\beta^4}{8(N\theta_N)^3} \theta_N W\left(\frac{i}{N},\frac{j}{N}\right)\left(1-\theta_N W\left(\frac{i}{N},\frac{j}{N}\right)\right)\notag                                                             \\
        \label{eq:var_asympt1}
         & \hspace{0.75in} +    \sum_{i=1}^N \left(\frac{\beta^2}{4(N\theta_N)^{\frac{3}{2}}} + \frac {\beta}{2\sqrt{N\theta_N}}\right)^2 \theta_N W\left(\frac{i}{N},\frac{i}{N}\right)\left(1-\theta_N W\left(\frac{i}{N},\frac{i}{N}\right)\right)\notag \\
        & \to  \frac{\beta^2}{4} \int_0^1 W(x,x)(1-\theta W(x,x))\mathrm{d}x . 
    \end{align}
    The result in \eqref{eq:thetaN} now follows from the Lindeberg-Feller CLT, since each summand in the RHS of \eqref{eq:tech2} converges uniformly to $0$, as $N\rightarrow\infty$.

We now prove \eqref{eq:linear11}. By a second-order Taylor expansion, we have:
\[
\bigl|e^{\overline{\Delta}_N}-1-\overline{\Delta}_N\bigr|
\le \frac12 e^{|\overline{\Delta}_N|}\,\overline{\Delta}_N^2.
\]
Therefore, for any $r \ge 1$, we have:
\[
(N\theta_N)^{\frac{r}{2}} 
\mathbb E \Bigl[\bigl(e^{\overline{\Delta}_N}-1-\overline{\Delta}_N\bigr)^r\Bigr]
\le \frac{1}{2^r} (N\theta_N)^{\frac{r}{2}} 
\mathbb E \Bigl[e^{r|\overline{\Delta}_N|}\,\overline{\Delta}_N^{2r}\Bigr].
\]
To show \eqref{eq:linear11}, it is enough to prove that the RHS of the above inequality goes to $0$ as $N \rightarrow \infty$. By the Cauchy--Schwarz inequality, it suffices to show that
\[
\mathbb E \bigl[e^{2r|\overline{\Delta}_N|}\bigr] \lesssim 1
\quad \text{ and } \quad
(N\theta_N)^r \mathbb E \bigl[\overline{\Delta}_N^{4r}\bigr] \ll 1.
\]
To this end, note that
\[
(N\theta_N)^r \overline{\Delta}_N^{4r}
= \frac{(\sqrt{N\theta_N}\,\overline{\Delta}_N)^{4r}}{(N\theta_N)^r}
\ll 1 \quad \text{in probability},
\]
since $N\theta_N \to \infty$ and $\sqrt{N\theta_N}\,\overline{\Delta}_N$ is tight. Hence, to establish the second condition above, it suffices to show that the corresponding expectation converges to $0$. Therefore, it suffices to check that
\begin{align}\label{eq:uniform_integr_linear11}
\mathbb E \bigl[e^{2r|\overline{\Delta}_N|}\bigr] \lesssim 1
\quad \text{ and } \quad
\mathbb E \Bigl[(\sqrt{N\theta_N}\,\overline{\Delta}_N)^{8r}\Bigr] \lesssim 1.
\end{align}

To prove both assertions, we use Bernstein's inequality (see~\cite[Theorem 2.8.4]{vershynin2018high}) which states that, given a collection of independent, zero-mean random variables $X_1,\ldots,X_m$, with $|X_i |<K$ for all $1 \leq i \leq m$, we have
    \begin{equation}\label{bernmod6}
        \P\left(\left|\sum_{i=1}^m X_i \right| \geq t \right) \leq 2 e^{  -\frac {t^2}{2 \mathrm{Var}[\sum_{i=1}^m X_i] + \frac{2}{3} Kt } }    
    \end{equation}
    for all $t\geq 0$.  Writing
\[
\sqrt{N\theta_N}\,\overline{\Delta}_N = \sum_{k} X_k,
\]
where $X_k$ are independent, centered random variables corresponding to the summands in \eqref{eq:tech2}, we note that $|X_k| \le K_N := C_1/\sqrt{N\theta_N}$ for some constant $C_1>0$. Moreover, by \eqref{eq:var_asympt1},
\[
\var\!\left(\sqrt{N\theta_N}\,\overline{\Delta}_N\right) \le C_2,
\]
for some constant $C_2>0$. Therefore, Bernstein's inequality yields
\begin{align*}
\mathbb P\Bigl(\bigl|\sqrt{N\theta_N}\,\overline{\Delta}_N\bigr| \ge t\Bigr)
\le 2 \exp\!\left(
-\frac{t^2}{2C_2 + \frac{2}{3}K_N t}
\right)
\le C_3 e^{-C_4 t},
\end{align*}
for all $t>0$, where $C_3, C_4 > 0$ are constants independent of $N$.
Hence, we have
\begin{align*}
\mathbb E \Bigl[(\sqrt{N\theta_N}\,\overline{\Delta}_N)^{8r}\Bigr]
&= \int_0^\infty 
\mathbb P\Bigl(\bigl|\sqrt{N\theta_N}\,\overline{\Delta}_N\bigr| > t^{\frac{1}{8r}}\Bigr)\,dt \\
&\le C_3 \int_0^\infty e^{-C_4 t^{\frac{1}{8r}}}\,dt < \infty,
\end{align*}
thereby proving the second claim in~\eqref{eq:uniform_integr_linear11}.

To prove the first claim in~\eqref{eq:uniform_integr_linear11}, note that $N\theta_N \to \infty$, and hence for sufficiently large $N$,
\begin{align*}
\mathbb E \Bigl[ e^{2r|\overline{\Delta}_N|}\mathbf 1\{|\overline{\Delta}_N|\ge 1\} \Bigr]
&= \int_{e^{2r}}^\infty 
\mathbb P\Bigl(|\overline{\Delta}_N| \ge \tfrac{1}{2r}\log t\Bigr)\,dt \nonumber \\
&= \int_{e^{2r}}^\infty 
\mathbb P\Bigl(\bigl|\sqrt{N\theta_N}\,\overline{\Delta}_N\bigr| 
\ge \tfrac{\sqrt{N\theta_N}}{2r}\log t\Bigr)\,dt \nonumber \\
&\le C_3 \int_{e^{2r}}^\infty 
\exp\!\left(-C_4 \tfrac{\sqrt{N\theta_N}}{2r}\log t\right)\,dt \nonumber \\
&= C_3 \int_{e^{2r}}^\infty 
t^{-\frac{C_4 \sqrt{N\theta_N}}{2r}}\,dt \lesssim 1.
\end{align*}
Together with the trivial bound $e^{2r}$ on $\{|\overline{\Delta}_N|<1\}$, this proves the first claim in~\eqref{eq:uniform_integr_linear11} and completes the proof of \eqref{eq:linear11}. \hfill $\Box$

\subsubsection{Proof of Lemma \ref{lem:clt_for_delta_N2}}\label{proofdeltaN28}
        We will first prove \eqref{eq:thetaNZN}. As in the previous case, recalling \eqref{eq:gammatheta} gives,  
    \begin{align}\label{eq:tech1}
        \overline{\Delta}_N
        =
        \sum_{1 \leq i < j \leq N} \frac{\beta^2}{2N^2 \theta_N^2 } \overline{A_{G_N}(i,j)}
        +
        \sum_{i=1}^N
        \left(\frac{\beta^2}{4N^2 \theta_N^2 } + \frac {\beta }{2 N \theta_N }\right) \overline{A_{G_N}(i,i)}.
    \end{align} 
The RHS of \eqref{eq:tech1} is a sum of independent random variables. Hence, for $N \theta_N^2 \to c^2 \in [0,\infty)$,     
\begin{align}
        \var (N\theta_N^{\frac{3}{2}} \overline{\Delta}_N)
        & =  \sum_{1 \leq i\neq j \leq N} \frac{\beta^4}{8N^2} W\left(\frac{i}{N},\frac{j}{N}\right)\left(1-\theta_N W\left(\frac{i}{N},\frac{j}{N}\right)\right)\notag                                                                      \\
        \label{eq:var_asympt2}
       & \hspace{0.75in} +  \sum_{i=1}^N \left(\frac{\beta^2}{4N \sqrt{\theta_N}} + \frac {\beta\sqrt \theta_N}{2}\right)^2 \theta_N W\left(\frac{i}{N},\frac{i}{N}\right)\left(1-\theta_N W\left(\frac{i}{N},\frac{i}{N}\right)\right)\notag \\
        & \to \frac{\beta^4}{8} \int_{[0,1]^2} W (x, y) \mathrm{d}x \mathrm{d}y + \frac{\beta^2c^2}{4} \int_0^1 W(x,x)\mathrm{d}x . 
    \end{align} 
    The result in \eqref{eq:thetaNZN} now follows from the Lindeberg-Feller CLT, since each summand in the right-hand side of \eqref{eq:tech1} converges uniformly to $0$, as $N\rightarrow\infty$.

Next, we prove \eqref{eq:linear12}. By a second-order Taylor expansion, 
    $$
        |e^{  \overline{\Delta}_N} - 1 - \overline{\Delta}_N| \le \frac 12 e^{  |\overline{\Delta}_N|} \overline{\Delta}_N^2 .
    $$
    Hence, for any $r \geq 1$, 
    $$
        N^r \theta_N^{\frac{3r}{2}} \E \left[(e^{  \overline{\Delta}_N} - 1 - \overline{\Delta}_N)^r\right] \leq \frac{1}{2^r} N^r \theta_N^{\frac{3r}{2}} \E \left[e^{  r |\overline{\Delta}_N|} \overline{\Delta}_N^{2r}\right].
    $$
    To show \eqref{eq:linear12} it suffices to show that the RHS of the above inequality goes to $0$. In view of the Cauchy-Schwarz inequality, it suffices to show that:
    $$
        \E \left[e^{  2r |\overline{\Delta}_N|} \right] \lesssim 1 
        \quad \text{ and } \quad
        N^{2r} \theta_N^{3r} \E \left[ \overline{\Delta}_N^{4 r} \right] \ll 1 .
    $$
   To this end, note that $N^{2r} \theta_N^{3r} \overline{\Delta}_N^{4r}= (N\theta_N^{\frac{3}{2}}\overline{\Delta}_N)^{4r} / (N\theta_N^{\frac{3}{2}})^{2r} \ll 1$ in probability, because the numerator converges in distribution to the $(4r)$-th power of a normal distribution by Lemma~\ref{lem:clt_for_delta_N2}. Hence, to establish the second condition above, we need to show that the corresponding expectation converges as well. To show  \eqref{eq:linear12} it thus suffices to check that
    \begin{align}\label{eq:uniform_integr}
        \E \left[e^{  2r |\overline{\Delta}_N|} \right] \lesssim 1 
        \quad \text{ and } \quad
        \E \left[(N\theta_N^{\frac{3}{2}}\overline{\Delta}_N)^{8r} \right] \lesssim 1.
    \end{align}
    To prove both assertions, we once again use the Bernstein inequality \eqref{bernmod6}. 
    Applying \eqref{bernmod6} to the independent and centered summands on the RHS of~\eqref{eq:tech1} with $K := 3C_1\sqrt \theta_N$, for some large constant $C_1$,     
     gives,    
    \begin{align*}
        \P\left(\left|N\theta_N^{\frac{3}{2}} \overline{\Delta}_N\right| \geq t \right) \leq 2 e^{  -\frac {t^2/2}{C_2 + C_1 t \sqrt \theta_N}}        \leq C_3 e^{  -C_4 t } , 
    \end{align*} 
    for all $t>0$, where $C_1, C_2, C_3, C_4$ are positive constants. (Note that we are using $\mathrm{Var}[N\theta_N^{\frac{3}{2}} \overline{\Delta}_N] \leq C_2$, for some constant $C_2 > 0$, by \eqref{eq:var_asympt2}.) Hence, we have:
    \begin{align*}
        \E \left[(N\theta_N^{\frac{3}{2}}\overline{\Delta}_N)^{8r} \right] = \int_0^\infty \p \left(|N\theta_N^{\frac{3}{2}}\overline{\Delta}_N| > t^{\frac{1}{8r}}\right)~\mathrm{d}t \le C_3 \int_0^\infty e^{-C_4 t^{\frac{1}{8r}}}~\mathrm{d}t <\infty.
    \end{align*}
    thereby proving the second claim in~\eqref{eq:uniform_integr}. 
    
    To prove the first claim in~\eqref{eq:uniform_integr}, recall that $N\theta_N^{\frac{3}{2}}\to\infty$, and hence for sufficiently large $N$, 
    \begin{align}\label{eq:DeltaNmoment}
        \E \left[ e^{  2r |\overline{\Delta}_N|}\bm{1} \{ |\overline{\Delta}_N|\ge 1 \} \right]  
        &= \int_{e^{  2r}}^\infty \P\left(|\overline{\Delta}_N| \ge \frac{1}{2r}\log t\right)~\mathrm{d}t \nonumber \\ 
        &\le C_3\int_{e^{  2r}}^\infty e^{  -C_4\frac{N\theta_N^{\frac{3}{2}}}{2r}\log t } ~\mathrm{d}t \nonumber \\ 
        &=  C_3 \int_{e^4}^\infty t^{-\frac{C_4 N\theta_N^{\frac{3}{2}}}{2r}}~\mathrm{d}t \lesssim 1 ,
    \end{align}
    which yields the first claim in~\eqref{eq:uniform_integr}. This completes the proof of  \eqref{eq:linear12}. \hfill $\Box$

\subsubsection{Proof of Lemma \ref{lem:expect_sum}}\label{proof6p246}

Recalling the definition of $\hat{Z}_N(\beta)$ from \eqref{modpart27} we have $\E \hat{Z}_N(\beta) = \E_{\mu_N} \E T(\bs)$, where $\mu_N$ denotes the uniform measure on $\{-1,+1\}^N$.  By Lemma \ref{lem:expvariance} we know that 
 \begin{align}\label{eq:ETsigma}
        \E T\left(\bs\right) = e^{\textstyle  \mathcal{E}_N(\beta, W) + \frac{\beta}{N}\sum_{1 \leq i < j \leq N}W\left(\frac{i}{N},\frac{j}{N}\right)\sigma_i \sigma_j +o\left(\frac 1{N\theta_N}\right) + O\left( Q_N(\bs) \right) } , 
    \end{align} 
    where 
    $$\mathcal{E}_N(\beta, W) = -\frac{\beta^2}{4N^2\theta_N}\sum_{i=1}^{N}W\left(\frac{i}{N},\frac{i}{N}\right) - \frac{\beta^2}{2N^2}\sum_{1 \leq i < j \leq N}W^2\left(\frac{i}{N},\frac{j}{N}\right)-\frac{\beta^4}{12 N^4\theta_N^3}\sum_{1 \leq i < j \leq N} W\left(\frac{i}{N},\frac{j}{N}\right)$$ is as defined in \eqref{eq:JN}. 
    From Lemma \ref{lem:Qnlp}, $Q_N(\bs) \xrightarrow{P} 0$, under $\mu_N$. Also, from Proposition \ref{quadJointAsym}, under $\mu_N$, 
$$\frac{\beta}{N}\sum_{1 \leq i < j \leq N}W\left(\frac{i}{N},\frac{j}{N}\right)\sigma_i \sigma_j \xrightarrow{D} \frac{\beta}{2}\sum_{\lambda \in \mathrm{Spec}(W)} \lambda (X_\lambda^2-1),$$
where $\{X_i\}_{i=1}^\infty$ is a collection of i.i.d. $\mathcal N(0, 1)$ random variables. Therefore,
$$Y_N := e^{\textstyle \frac{\beta}{N}\sum_{1 \leq i < j \leq N}W\left(\frac{i}{N},\frac{j}{N}\right)\sigma_i \sigma_j + O(Q_N(\bs))}\xrightarrow{D} e^{\textstyle  \frac{\beta}{2}\sum_{\lambda \in \mathrm{Spec}(W)} \lambda (X_\lambda^2-1)} . $$
Furthermore, Lemma \ref{expUniInt}, \eqref{qneqn7}, and a simple application of H\"older's inequality imply that $\E_{\mu_N} Y_N^{1+\gamma} \lesssim 1$, for some $\gamma > 0$, which shows the uniform integrability of $Y_N$ under $\mu_N$. Hence,
\begin{align}\label{sigmapart7}
    \E_{\mu_N} Y_N \rightarrow \E e^{\textstyle \frac{\beta}{2}\sum_{\lambda \in \mathrm{Spec}(W)} \lambda (X_\lambda^2-1)} = \prod_{\lambda \in \mathrm{Spec}(W)} \frac{ e^{\textstyle -\frac{\beta \lambda}{2} } }{\sqrt{1-\beta \lambda}} , 
\end{align} 
where the last step follows from \cite[Proposition 7.1]{bhattacharya2017universal}. Next, note that when $\theta_N \gg \frac{1}{N}$,  
\begin{align}\label{constant76}
    & -\frac{\beta^2}{4N^2\theta_N}\sum_{i=1}^{N}W\left(\frac{i}{N},\frac{i}{N}\right)
    -\frac{\beta^2}{2N^2}\sum_{1 \leq i < j \leq N}W^2\left(\frac{i}{N},\frac{j}{N}\right)+o\left(\frac 1{N\theta_N}\right) \nonumber \\ 
    & = -\frac{\beta^2}{4}\int_{[0,1]^2} W(x,y)^2 \mathrm{d}x\mathrm{d}y + o(1).
\end{align}
Finally, 
\begin{align}\label{variable98part}
    -\frac{\beta^4}{12 N^4\theta_N^3}\sum_{1 \leq i < j \leq N} W\left(\frac{i}{N},\frac{j}{N}\right)=
    \begin{cases}
        o(1)                                               & \text{if } \theta_N \gg N^{-\frac{2}{3}} ,  \\
        -\frac{\beta^4}{24c^3}\int_{[0,1]^2} W(x,y) \mathrm{d}x \mathrm{d}y & \text{if } \theta_N \sim c N^{-\frac{2}{3}} . 
    \end{cases}
\end{align}
Combining \eqref{sigmapart7}, \eqref{constant76}, and \eqref{variable98part}, Lemma \ref{lem:expect_sum} follows. \hfill  $\Box$

\section{Proof of Theorem \ref{thm:HNsigma}}\label{corhamlproof} 
 We will prove Theorem \ref{thm:HNsigma} by computing the limit of the conditional moment-generating function (MGF) of $-H_N(\bs)$. To this end, choose $\varepsilon>0$ such that $[\beta-\varepsilon,\beta+\varepsilon]\subset (0, \frac{1}{\|W\|_{\mathrm{op}}})$. Then, for any $t\in(-\varepsilon,\varepsilon)$, 
\begin{align}\label{eq:mgfHN}
    \phi_{A_{G_N}}(t) & :=  \E\left[ e^{ \textstyle \frac{t}{2N\theta_N}\bs^\top A_{G_N}\bs } \bigg | A_{G_N}\right]             \nonumber \\                                                  
     & =        \frac{1}{2^N Z_N(\beta)} \sum_{ \sigma \in \{-1, 1\}^N } e^{\textstyle \frac{t}{2N\theta_N}\bs^\top A_{G_N} \bs+\frac{\beta}{2N\theta_N}\bs^\top A_{G_N} \bs} \nonumber \\ 
    & =          \frac{Z_N(\beta+t)}{Z_N(\beta)} , 
\end{align}
where the first expectation is with respect to the Ising model \eqref{model_def}, conditional on $A_{G_N}$. From \eqref{prep_proof_th1.2_rewr}, we have
\begin{align*}
    \frac{Z_N(\beta)}{e^{\textstyle \frac{\beta}{2N}\sum_{i=1}^{N}W\left(\frac{i}{N},\frac{i}{N}\right) + \frac{\beta^2}{4N^2\theta_N}\sum_{1 \leq i,j \leq N}W\left(\frac{i}{N},\frac{j}{N}\right)} \E \hat Z_N(\beta)}
    =
    J_N(\beta)
    e^{\textstyle \gamma \overline{\vartheta}_N + \gamma^2\overline{\eta}_N}.
\end{align*}
Hence,
\begin{align}\label{eq:ZNbetat}
    \frac{Z_N(\beta+t)}{Z_N(\beta)} = \frac{J_N(\beta+t)}{J_N(\beta)}\cdot \frac{\E\hat{Z}_N(\beta+t)}{\E\hat{Z}_N(\beta)} e^{ \textstyle \frac{t}{2N\theta_N}\vartheta_N+\frac{t^2+2\beta t}{4N^2\theta_N^2}\eta_N } . 
\end{align}
Now, by \eqref{eq:linear11} and \eqref{eq:linear12}, 
\begin{align*}
    e^{ \textstyle \frac{t}{2N\theta_N}\overline{\vartheta}_N+\frac{t^2+2\beta t}{4N^2\theta_N^2}\overline{\eta}_N } = 1+o_{L^p}\left(1\right) , 
\end{align*}
for every $p\ge 1$. Also, by Lemma \ref{var_sum}, $\frac{J_N(\beta+t)}{J_N(\beta)} = 1+o_P(1)$. Hence, for every $p\ge 1$, 
\begin{align}
    \label{eq:Z_Nratio}
    \frac{Z_N(\beta+t)}{Z_N(\beta)} = \left(1 + o_{P}\left(1\right)\right)\frac{\E\hat{Z}_N(\beta+t)}{\E\hat{Z}_N(\beta)} e^{ \textstyle \frac{t}{2N}\sum_{i=1}^{N}W\left(\frac{i}{N},\frac{i}{N}\right) + \frac{t^2+2\beta t}{4N^2\theta_N}\sum_{1 \leq i,j \leq N}W\left(\frac{i}{N},\frac{j}{N}\right)} . 
\end{align} 
We now consider the two cases of Theorem \ref{thm:HNsigma} separately in the next subsections.

\subsection{Proof of Theorem \ref{thm:HNsigma} (1)}
Let 
$$S_N=-H_N(\bs)-\frac{\beta}{2N^2\theta_N}\sum_{1 \leq i,j \leq N}W\left(\frac{i}{N},\frac{j}{N}\right).$$ 
Then,
\begin{align}\label{eq:SN}
\E \left[e^{\textstyle t\sqrt{\theta_N} S_N}\Big| A_{G_N}\right] = \frac{{Z}_N(\beta+t\sqrt{\theta_N})}{{Z}_N(\beta)} e^{ \textstyle -\frac{\beta t}{2N^2\sqrt{\theta_N}}\sum_{1 \leq i,j \leq N}W\left(\frac{i}{N},\frac{j}{N}\right) } .  
\end{align} 
As in \eqref{eq:ZNbetat}, 
\begin{align*}
 \frac{Z_N(\beta+ t \sqrt \theta_N )}{Z_N(\beta)} = \frac{J_N(\beta+t \sqrt \theta_N )}{J_N(\beta)}\cdot \frac{\E\hat{Z}_N(\beta+t \sqrt \theta_N )}{\E\hat{Z}_N(\beta)} e^{ \textstyle \frac{t \sqrt \theta_N }{2N\theta_N}\vartheta_N+\frac{t^2+2\beta t \sqrt \theta_N }{4N^2\theta_N^2}\eta_N } . 
\end{align*}
By Lemma \ref{var_sum}, $\frac{J_N(\beta+t\sqrt{\theta_N})}{J_N(\beta)} = 1+o_P(1)$. Hence, for every $p\ge 1$, 
\begin{align}\label{eq:ZNbetatN}
    \frac{Z_N(\beta+ t \sqrt \theta_N )}{Z_N(\beta)} = \frac{\E\hat{Z}_N(\beta+t \sqrt \theta_N )}{\E\hat{Z}_N(\beta)} e^{ \textstyle \frac{t \sqrt \theta_N }{2N\theta_N}\vartheta_N+\frac{t^2+2\beta t \sqrt \theta_N }{4N^2\theta_N^2}\eta_N } . 
\end{align}
The next lemma shows that the term in the exponent above can be asymptotically  replaced in its expected value.

\begin{lemma}\label{lem:linearization3}
    Suppose $ \theta_N \gtrsim N^{-\frac{2}{3}}$. Then for all $t \in \mathbb{R}$ and $p\ge 1$, 
    \begin{align*}
        e^{\textstyle \frac{t \sqrt \theta_N}{2N\theta_N}\overline{\vartheta}_N+\frac{\theta_N t^2+2\beta t \sqrt \theta_N}{4N^2\theta_N^2}\overline{\eta}_N } = 1+o_{L^p}(1).
    \end{align*}
\end{lemma}

\begin{proof}
    To begin with, define 
    $$\overline{\zeta}_N := \frac{t \sqrt \theta_N}{2N\theta_N}\overline{\vartheta}_N+\frac{\theta_N t^2+2\beta t \sqrt \theta_N}{4N^2\theta_N^2}\overline{\eta}_N. $$
    Then, by a first-order Taylor expansion, $|e^{ \overline{\zeta}_N} - 1| \le e^{|\overline{\zeta}_N|} |\overline{\zeta}_N| .$ Hence, to prove Lemma \ref{lem:linearization3}, by Cauchy-Schwarz inequality, it suffices to show that 
    \begin{align}\label{eq:zetaNmoment}
        \E e^{ 2p|\overline{\zeta}_N|} \lesssim 1 
        \quad \text{ and } \quad
        \E \overline{\zeta}_N^{2p} \ll 1 .
    \end{align}
    Now, it is easy to check that
    \begin{align*}
        \E \overline{\zeta}_N^2 =  \var[ \overline{\zeta}_N ] = O\left(\frac{1}{N}\right) + O\left(\frac{1}{N^2\theta_N^2}\right) \lesssim \theta_N.
    \end{align*}
    Hence, $\mathrm{Var}[N\theta_N \overline{\zeta}_N] \lesssim N^2\theta_N^3 \lesssim 1$ and  applying the Bernstein inequality gives, 
    $$
        \P\left(\left|N\theta_N \overline{\zeta}_N\right| \geq t \right) \leq 2 e^{-\frac {t^2/2}{C_2 + C_1 t \sqrt \theta_N} } \leq C_3 e^{-C_4 t } , 
    $$
    for all $t>0$, where $C_1, C_2, C_3, C_4$ are some positive constants. The bounds in \eqref{eq:zetaNmoment} can now be proved by arguments similar to \eqref{eq:DeltaNmoment}.  
\end{proof}

Applying Lemma \ref{lem:linearization3} to \eqref{eq:ZNbetatN} gives, 
\begin{align}\label{eq:ZNthetaN}
    & \frac{Z_N(\beta+t\sqrt{\theta_N})}{Z_N(\beta)}  \\ 
    & = \left(1 + o_{P}\left(1\right)\right)\frac{\E\hat{Z}_N(\beta+t\sqrt{\theta_N})}{\E\hat{Z}_N(\beta)} e^{ \textstyle \frac{t\sqrt{\theta_N}}{2N}\sum_{i=1}^{N}W\left(\frac{i}{N},\frac{i}{N}\right) + \left(\frac{t^2}{4N^2} + \frac{\beta t}{2N^2\sqrt{\theta_N}} \right) \sum_{1 \leq i,j \leq N}W\left(\frac{i}{N},\frac{j}{N}\right)  } . \nonumber
\end{align}
The next lemma gives the asymptotics for $\frac{\E\hat{Z}_N(\beta+t\sqrt{\theta_N})}{\E\hat{Z}_N(\beta)}$. 

\begin{lemma}\label{thetanperturb6}
    If $\theta_N \ll 1$ and $\theta_N \gtrsim N^{-\frac{2}{3}} $, then for all $t\in \R$, 
    $$\lim_{N\rightarrow\infty}\frac{\E\hat{Z}_N(\beta+t\sqrt{\theta_N})}{\E\hat{Z}_N(\beta)}=1.$$
\end{lemma}
\begin{proof}
    We closely follow the proof of Lemma \ref{lem:expect_sum}. Define $\beta_N := \beta + t\sqrt{\theta_N}$. Then, recalling \eqref{eq:ETsigma}, gives 
        $$\E \hat{Z}_N(\beta_N) = \E_{\mu_N} \left[ e^{\textstyle  \mathcal{E}_N(\beta_N, W) + \frac{\beta_N}{N}\sum_{1 \leq i < j \leq N}W\left(\frac{i}{N},\frac{j}{N}\right)\sigma_i \sigma_j +o\left(\frac 1{N\theta_N}\right) + O\left( Q_N(\bs) \right) } \right] , $$ 
    where $$\mathcal{E}_N(\beta_N, W) = -\frac{\beta_N^2}{4N^2\theta_N}\sum_{i=1}^{N}W\left(\frac{i}{N},\frac{i}{N}\right)  - \frac{\beta_N^2}{2N^2}\sum_{1 \leq i < j \leq N}W^2\left(\frac{i}{N},\frac{j}{N}\right)-\frac{\beta_N^4}{12 N^4\theta_N^3}\sum_{1 \leq i < j \leq N} W\left(\frac{i}{N},\frac{j}{N}\right),$$ is as defined in \eqref{eq:JN}. Note that in the regime $\theta_N \gtrsim N^{-\frac{2}{3}} $, 
    \begin{align*}  
        \mathcal{E}_N(\beta_N, W) &= -\frac{\beta^2}{4N^2\theta_N}\sum_{i=1}^{N}W\left(\frac{i}{N},\frac{i}{N}\right)
        -\frac{\beta^2}{2N^2}\sum_{1 \leq i < j \leq N}W^2\left(\frac{i}{N},\frac{j}{N}\right) -\frac{\beta^4}{12 N^4\theta_N^3}\sum_{1 \leq i < j \leq N} W\left(\frac{i}{N},\frac{j}{N}\right)\\ 
        & = -\frac{\beta^2}{4}\int_{[0,1]^2} W(x,y)\left(\frac{\beta^2}{6c^3} + W(x,y)\right) \mathrm{d}x\mathrm{d}y +  o(1),
    \end{align*}
    where $\theta_N \sim cN^{-\frac{2}{3}}$ for some constant $c \in (0, \infty]$.  
    Now, define 
    $$Y_N^{(t)} := e^{\textstyle \frac{\beta_N}{N}\sum_{1 \leq i < j \leq N}W\left(\frac{i}{N},\frac{j}{N}\right)\sigma_i \sigma_j + O(Q_N(\bs)) } . $$
      By Lemma \ref{lem:Qnlp}, $Q_N(\bs) \xrightarrow{P} 0$, under $\mu_N$. Also, 
   by Proposition \ref{quadJointAsym},  $\frac{t\sqrt{\theta}_N}{N} \sum_{1 \leq i < j \leq N}W\left(
        \frac{i}{N}, \frac{j}{N}\right)\sigma_i\sigma_j \xrightarrow{P} 0$ 
        and hence,  
    $$Y_N^{(t)} \xrightarrow{D} e^{\textstyle \frac{\beta}{2}\sum_{\lambda \in \mathrm{Spec}(W)} \lambda (X_i^2-1) } , $$
 as $N \rightarrow \infty$. Further, since $\theta_N \ll 1$, it follows from Lemma \ref{expUniInt} and \eqref{qneqn7} that $Y_{N}$ is uniformly integrable, and hence, from \eqref{sigmapart7}
       $$\E_{\mu_N} Y_N^{(t)} \rightarrow \E e^{\textstyle \frac{\beta}{2}\sum_{\lambda \in \mathrm{Spec}(W)} \lambda (X_\lambda^2-1)} = \prod_{i=1}^{\infty} \frac{ e^{\textstyle -\frac{\beta \lambda }{2} } }{\sqrt{1-\beta \lambda }}, $$ 
       as $N \rightarrow \infty$. 
Combining the above shows that $\E \hat{Z}_N(\beta_N)$ has the same limit as $\E \hat{Z}_N(\beta)$ as in Lemma \ref{lem:expect_sum}. Hence, their ratio converges to 1,  completing the proof of Lemma \ref{thetanperturb6}. 
\end{proof}

Applying Lemma \ref{thetanperturb6} and the observation 
\begin{align*}
    e^{\textstyle \frac{t\sqrt{\theta_N}}{2N}\sum_{i=1}^{N}W\left(\frac{i}{N},\frac{i}{N}\right) + \frac{t^2}{4N^2}\sum_{1 \leq i,j \leq N}W\left(\frac{i}{N},\frac{j}{N}\right) }  \to e^{\textstyle \frac{t^2}{4} \int_{[0, 1]^2} W(x,y)\mathrm{d}x\mathrm{d}y} , 
\end{align*} 
in \eqref{eq:ZNthetaN} gives, 
\begin{align}\label{ref2707}
    \frac{Z_N(\beta+t\sqrt{\theta_N})}{Z_N(\beta)} =\left(1 + o_{P}\left(1\right)\right)e^{\textstyle \frac{t^2}{4} \int_{[0, 1]^2} W(x,y)\mathrm{d}x\mathrm{d}y }  e^{\textstyle  \frac{\beta t}{2N^2\sqrt{\theta_N}}\sum_{1 \leq i,j \leq N}W\left(\frac{i}{N},\frac{j}{N}\right)} . 
 \end{align}
Hence, 
$$\frac{Z_N(\beta+t\sqrt{\theta_N})}{Z_N(\beta)} e^{\textstyle - \frac{\beta t}{2N^2\sqrt{\theta_N}}\sum_{1 \leq i,j \leq N}W\left(\frac{i}{N},\frac{j}{N}\right) } \xrightarrow{P} e^{\textstyle \frac{t^2}{4} \int_{[0, 1]^2} W(x,y)\mathrm{d}x\mathrm{d}y } .$$
This implies, recalling \eqref{eq:SN}, 
\begin{align*} 
\E \left[e^{\textstyle t\sqrt{\theta_N} S_N}\Big| A_{G_N}\right]  \xrightarrow{P} e^{\textstyle \frac{t^2}{4} \int_{[0, 1]^2} W(x,y)\mathrm{d}x\mathrm{d}y } ,
\end{align*} 
for all $t  \in \R$. Hence, by Lemma \ref{l2ratio762}, 
\begin{align*}
    \sqrt{\theta_N} S_N \xrightarrow{D} \cN\left(0,\frac{1}{2}\int_{[0, 1]^2} W(x,y)\mathrm{d}x\mathrm{d}y \right) .  
\end{align*}
This completes the proof of \eqref{eq:HNsparse}.

Finally, recall from \eqref{holdapprox62} that if $W$ is $\alpha$-H\"older continuous, then
$$\left|\frac{1}{N^2}\sum_{1\le i,j\le N} W\left(\frac{i}{N}, \frac{j}{N}\right) - \int_{[0,1]^2} W(x,y)~\mathrm{d}x \mathrm{d}y\right| \le N^{-\alpha}.$$
Therefore, $$\frac{1}{\sqrt{\theta_N}} \left(\frac{\beta}{2N^2}\sum_{1\le i,j\le N} W\left(\frac{i}{N}, \frac{j}{N}\right) - \frac{\beta}{2}\int_{[0,1]^2} W(x,y)~\mathrm{d}x \mathrm{d}y\right) = O\left((\theta_N N^{2\alpha})^{-1/2}\right)$$
and the RHS above is $o(1)$ by Assumption \ref{assumption} (3), since $\alpha > \frac{1}{3}$. This proves \eqref{holder112}.

\subsection{Proof of Theorem \ref{thm:HNsigma} (2)} In this case, 
\begin{align}\label{newlab:align6}
    & e^{ \textstyle \frac{t}{2N}\sum_{i=1}^{N}W\left(\frac{i}{N},\frac{i}{N}\right) + \frac{t^2+2\beta t}{4N^2\theta_N}\sum_{1 \leq i,j \leq N}W\left(\frac{i}{N},\frac{j}{N}\right) }  \nonumber \\ 
    & \to e^{ \textstyle \frac{t^2+2t \beta}{4\theta} \int_{[0, 1]^2} W(x,y)\mathrm{d}x\mathrm{d}y + \frac{t}{2}\int_0^1 W(x,x)\mathrm{d}x } . 
\end{align} 
We also have the following from Lemma \ref{lem:expect_sum}: 
\begin{align}
    \label{eq:EZ_Nratio}
    \frac{\E \hat{Z}_N(\beta+t)}{\E \hat{Z}_N(\beta)} \to e^{ \textstyle -\frac{t^2+2t \beta}{4} \int_{[0,1]^2} W(x,y)^2 \mathrm{d}x\mathrm{d}y } \prod_{\lambda \in \mathrm{Spec}(W)} e^{ \textstyle -\frac{t \lambda }{2} } \sqrt{\frac{1-\beta \lambda }{1-(\beta+t) \lambda }} . 
\end{align} 
Combining \eqref{newlab:align6} with \eqref{eq:mgfHN}, \eqref{eq:Z_Nratio} and \eqref{eq:EZ_Nratio}, we get $\phi_{A_{G_N}}(t)\xrightarrow{P}\phi(t)$, where 
\begin{align}\label{eq:mgfU}
    \phi(t)& :=  e^{ \textstyle \frac{(t^2+2t \beta) \sigma_W^2}{2} +\frac{t}{2}\int_0^1 W(x,x)\mathrm{d}x} \prod_{\lambda \in \mathrm{Spec}(W)} e^{\textstyle -\frac{t \lambda}{2} } \sqrt{\frac{1-\beta \lambda}{1-(\beta+t) \lambda }} , 
\end{align} 
where $\sigma_W^2=\frac{1}{2\theta} \int_{[0, 1]^2} W(x,y)(1-\theta W(x,y))\mathrm{d}x\mathrm{d}y$ is as defined in Proposition \ref{thm:HNsigma}. Now, recalling the definition of $V_{\beta}$ from \eqref{eq:Vbeta}, let 
\begin{align}\label{eq:UV}
    U_\beta & :=\frac{1}{2}\int_{0}^{1}W(x,x)\mathrm{d}x + V_{\beta} \nonumber \\
    & = \beta \sigma_W^2 + \frac{1}{2}\int_{0}^{1}W(x,x)\mathrm{d}x + \sigma_WZ_0+\frac{1}{2} \sum_{\lambda \in \mathrm{Spec}(W)} \lambda \left(\frac{Z_{\lambda}^2}{1-\beta \lambda}-1\right),
\end{align}
where $Z_0, Z_1,\ldots$ are i.i.d. $\mathcal N(0, 1)$ random variables. Lemma \ref{lm:HNmgf} below shows that the MGF of $U_\beta$ is $\phi(\cdot)$. This means that, recalling \eqref{eq:mgfHN}, 
$$\E\left[ e^{ \textstyle \frac{t}{2N\theta_N}\bs^\top A_{G_N}\bs } \bigg | A_{G_N}\right] \pto \E[e^{t U_\beta}],$$
for $t$ in a neighborhood of zero. Hence, by Lemma \ref{l2ratio762}, we conclude that $-H_N(\bs) \xrightarrow{D} U_\beta$, completing the proof of \eqref{eq:HNdense}. \hfill $\Box$

\begin{lemma}\label{lm:HNmgf}
There exists $\tau>0$ such that for all $|t|<\tau$, $\E[e^{t U_\beta}] = \phi(t)$, where $\phi(t)$ is defined in \eqref{eq:mgfU}. 
\end{lemma}

\begin{proof} 
Denote $M_{\beta}(t) := \E\left[ e^{tU_\beta} \right]$. Then,
\begin{align}\label{eq:logMGFUbeta}
    \log M_\beta(t) 
    & = t\left[\beta\sigma_W^2 + \frac{1}{2}\int_0^1 W(x,x)\mathrm{d} x + \frac{\beta}{2}\sum_{\lambda \in \mathrm{Spec}(W)}\frac{\lambda^2}{1-\beta\lambda}\right] + \frac{\sigma_W^2t^2}{2}\nonumber\\
    & \hspace{150pt}+ \log\E\left[ e^{ \textstyle \frac{t}{2}\sum_{\lambda \in \mathrm{Spec}(W)} \frac{\lambda}{1-\beta\lambda}(Z_\lambda^2-1) } \right].
\end{align}
The above expression is well defined since $\sum_{\lambda \in \mathrm{Spec}(W) }\lambda ^2 = \|W\|_2^2<\infty$ and $(1-\beta\lambda )^{-1}\leq (1-\beta\|W\|_{\mathrm{op}})^{-1}<\infty$, for $\lambda \in \mathrm{Spec}(W)$. 
Then by \cite[Proposition 7.1]{bhattacharya2017universal} we know,
\begin{align*} 
\log\E\left[ e^{ \textstyle \frac{t}{2}\sum_{\lambda \in \mathrm{Spec}(W)}\frac{\lambda }{1-\beta\lambda }(Z_\lambda^2-1) } \right]  = \frac{1}{2}\sum_{\lambda \in \mathrm{Spec}(W)}\sum_{s=2}^{\infty} \frac{\lambda ^s}{(1-\beta\lambda )^s} \frac{t^s}{s}, \text{ for all }|t|< \frac{1}{8\|W\|_2}.
\end{align*}
Now, using the absolute convergence from Lemma \ref{lemma:eigenvalue_summability} (a) and Fubini's theorem gives,
\begin{align}
& \log\E\left[ e^{ \textstyle \frac{t}{2}\sum_{\lambda \in \mathrm{Spec}(W)}\frac{\lambda }{1-\beta\lambda }(Z_\lambda^2-1) } \right]  \nonumber \\ 
& = \frac{1}{2}\sum_{s=2}^{\infty}\sum_{\lambda \in \mathrm{Spec}(W)} \frac{\lambda ^s}{(1-\beta\lambda )^s} \frac{t^s}{s} \nonumber\\
& = \frac{t^2}{4}\sum_{\lambda \in \mathrm{Spec}(W)} \frac{\lambda ^2}{(1-\beta\lambda )^2} + \frac{1}{2}\sum_{s=3}^{\infty}\sum_{\lambda \in \mathrm{Spec}(W)} \frac{\lambda ^s}{(1-\beta\lambda )^s} \frac{t^s}{s} . 
    \label{eq:logMGFinfchisq}
\end{align}
Moreover, once again using the relation $\sum_{\lambda \in \mathrm{Spec}(W)}\lambda ^2 = \|W\|_2^2$ gives,
\begin{align}\label{eq:sigmaWalternate}
    \sigma_W^2 = \frac{1}{2\theta}\int_{[0, 1]^2} W(x,y)(1-\theta W(x,y))\mathrm{d} x\mathrm{d} y = \frac{1}{2\theta}\int_{[0, 1]^2} W(x,y)\mathrm{d} x\mathrm{d} y - \frac{1}{2}\sum_{\lambda \in \mathrm{Spec}(W)}\lambda ^2.
\end{align}
Now, substituting expressions from \eqref{eq:logMGFinfchisq} and \eqref{eq:sigmaWalternate} in \eqref{eq:logMGFUbeta} gives,
\begin{align*}
    \log M_\beta(t)  & = \sum_{s=1}^\infty a_s(\beta, W) t^s , 
\end{align*}
where 
\begin{align}\label{eq:coefficientW} 
\begin{aligned}
a_1(\beta, W) & := \frac{\beta}{2\theta}\int_{[0, 1]^2} W(x,y)\mathrm{d} x\mathrm{d} y + \frac{1}{2}\int_0^1 W(x,x)\mathrm{d} x + \frac{\beta^2}{2}\sum_{\lambda \in \mathrm{Spec}(W)}\frac{\lambda ^3}{1-\beta\lambda } , \\ 
a_2(\beta, W) & := \frac{1}{4\theta}\int_{[0, 1]^2} W(x,y)\mathrm{d} x\mathrm{d} y + \frac{\beta}{2}\sum_{\lambda \in \mathrm{Spec}(W)}\frac{\lambda ^3}{(1-\beta\lambda )^2} - \frac{\beta^2}{4}\sum_{\lambda \in \mathrm{Spec}(W)}\frac{\lambda ^4}{(1-\beta\lambda )^2} , 
\end{aligned}
\end{align} 
and $a_s(\beta, W) = \frac{1}{2 s} \sum_{\lambda \in \mathrm{Spec}(W)} \frac{\lambda ^s}{(1-\beta\lambda )^s}$, for $s \geq 3$. To complete the proof, recall the definition of $\phi(t)$ from \eqref{eq:mgfU} and note that it is well defined by Lemma \ref{lemma:eigenvalue_summability} (b). Moreover, applying Lemma \ref{lemma:eigenvalue_summability} (b) there exists $\tau>0$ such that we can take logarithm on both sides of \eqref{eq:mgfU} to get
\begin{align*}
    \log\phi(t) = \frac{t^2}{2}\sigma_W^2 + t\beta\sigma_W^2 + \frac{t}{2}\int W(x,x)\mathrm{d}x + \sum_{\lambda\in \mathrm{Spec}(W)} \left[-\frac{t\lambda}{2} - \frac{1}{2}\log \left(1-\frac{t\lambda}{1-\beta\lambda}\right)\right] , 
\end{align*}
for all $|t|<\tau$. From Assumption \ref{assumption} recall that $\beta\|W\|_{\rm op}<1$ and,  hence, we can expand the logarithm to get,
\begin{align*}
    & \sum_{\lambda\in \mathrm{Spec}(W)}\frac{t\lambda}{2} + \frac{1}{2}\log \left(1-\frac{t\lambda}{1-\beta\lambda}\right)  \nonumber \\
    & = \sum_{\lambda\in \mathrm{Spec}(W)} \frac{t\lambda}{2} - \frac{1}{2}\sum_{s=1}^{\infty}\frac{1}{s}\left(\frac{t\lambda}{1-\beta\lambda}\right)^s\\
    & = \sum_{\lambda\in \mathrm{Spec}(W)}\frac{t\lambda}{2} - \frac{1}{2}\frac{t\lambda}{1-\beta\lambda} - \frac{1}{4}\left(\frac{t\lambda}{1-\beta\lambda}\right)^2 - \sum_{s=3}^{\infty}\frac{1}{2s}\left(\frac{t\lambda}{1-\beta\lambda}\right)^s\\
    & = \sum_{\lambda\in \mathrm{Spec}(W)}-\frac{\beta t\lambda^2}{2(1-\beta\lambda)} - \frac{1}{4}\left(\frac{t\lambda}{1-\beta\lambda}\right)^2 - \sum_{s=3}^{\infty}\frac{1}{2s}\left(\frac{t\lambda}{1-\beta\lambda}\right)^s.
\end{align*}
Now, once again applying Fubini's theorem (similar to arguments in the proof of Lemma \ref{lemma:eigenvalue_summability} (a)), using the expansion from \eqref{eq:sigmaWalternate} and collecting the coefficients of $t^s$, for $s\geq 1$, shows that $\log M_\beta(t) = \log \phi(t)$ for all $|t|<\tau$. 
\end{proof}

\section{Proofs from Section \ref{sec:estimation}}

We begin with the proof of Theorem \ref{mle} in Appendix \ref{sec:mlepr8}. Theorem \ref{clttilde} is proved in Appendix \ref{sec:clttildepf}.

\subsection{Proof of Theorem \ref{mle}}\label{sec:mlepr8}

Recall from \eqref{eq:mleeq7}, that the ML estimate $\hat{\beta}_N$ satisfies the equation 
\begin{align}\label{eq:mle_eq7}
    H_N(\bs)= \E_{\hat{\beta}_N} \left[H_N(\bs)\mid A_{G_N}\right] = - \psi_N'(\hat{\beta}_N) ,  
\end{align}
if it exists, where $\psi_N(\beta) = \log Z_N (\beta)$. Otherwise, $\hat{\beta}_N$ is defined to be $+\infty$. We begin by determining when the ML estimate is finite.

\begin{lemma} Suppose $\theta_N\to \theta\in [0,1]$. Then the following hold: 
\begin{enumerate}
            \item[$(1)$] (Dense regime) If $\theta \in (0,1]$, 
            $$\lim_{N \rightarrow \infty} \P(\hat{\beta}_N < \infty ) = \P\left( V_{\beta} > 0 \right) \in (0, 1) , $$
            where $V_{\beta}$ is defined in \eqref{eq:Vbeta}. 
            
             \item[$(2)$] (Sparse regime) If $\theta = 0$, 
            $$\lim_{N \rightarrow \infty} \P(\hat{\beta}_N < \infty ) = 1.$$
\end{enumerate}
\label{lem:mlesolution}
\end{lemma}

\begin{proof} 
To begin with, define  $ \sG_N:= \{\bt \in \{-1, 1\}^N: -H_N(\bt) = \frac{1}{2N\theta_N} \eta_N \}$. Note that $|H_N(\bs)| \le \frac{1}{2N\theta_N} \eta_N$. Now,
$\psi_N'(0) =  \frac{1}{2N\theta_N}\sum_{i=1}^N A_{G_N} (i, i) =  \frac{1}{2N\theta_N} \vartheta_N$ (recall \eqref{eq:ANij}) and 
we also have:
\begin{align*}
    \lim_{\beta\to \infty} \psi_N'(\beta) &= -\lim_{\beta \rightarrow\infty} \frac{\sum_{\bs \in \{-1,1\}^N} H_N(\bs) e^{-\beta H_N(\bs)}}{\sum_{\bs \in \{-1,1\}^N} e^{-\beta H_N(\bs)}}\\
&= -\lim_{\beta \rightarrow\infty} \frac{\sum_{\bs \in \{-1,1\}^N} H_N(\bs) e^{-\beta H_N(\bs) -\frac{\beta}{2N\theta_N} \eta_N}}{\sum_{\bs \in \{-1,1\}^N} e^{-\beta H_N(\bs) -\frac{\beta}{2N\theta_N} \eta_N}}\\
&= -\lim_{\beta \rightarrow\infty} \frac{-\frac{|\sG_N|\eta_N}{2N\theta_N} + \sum_{\bs \notin \sG_N} H_N(\bs) e^{-\beta H_N(\bs) -\frac{\beta}{2N\theta_N} \eta_N}}{|\sG_N| + \sum_{\bs \notin \sG_N} e^{-\beta H_N(\bs) -\frac{\beta}{2N\theta_N} \eta_N}} \\
    &= \frac{1}{2N\theta_N}\eta_N.
\end{align*} 
The monotonicity and continuity of $\psi_N'$ now imply that the MLE equation \eqref{eq:mle_eq7} has a solution if and only if: $$-H_N(\bs) \in \left[ \frac{1}{2N\theta_N}\vartheta_N , \frac{1}{2N\theta_N}\eta_N\right).$$
We now consider the two cases separately: 

\begin{enumerate}

\item[$(1)$] $\theta \in (0,1]$:  Note that
   \begin{align}\label{eq:HNprobability}
        \P(\mathscr{G}_N) &= \P\left(-\frac{1}{N\theta_N} H_N(\bs)=\frac{1}{2N^2\theta_N^2}\eta_N\right) \nonumber \\
        &= \P\left(\frac{1}{N\theta_N} H_N(\bs)+\frac{1}{2N^2\theta_N^2}\eta_N=0\right) \ll 1 , 
    \end{align}
    where the last step uses the following facts: $\frac{1}{N\theta_N} H_N(\bs) \xrightarrow{P} 0$ (recall \eqref{eq:HNdense}) and $$\frac{1}{N^2\theta_N^2}\eta_N \xrightarrow{P} \frac{1}{\theta}\int_{[0, 1]^2} W(x,y) \mathrm{d}x\mathrm{d}y > 0.$$ 
Also, note that $\frac{1}{2N\theta_N} \vartheta_N \xrightarrow{L^2} \frac{1}{2}\int_0^1 W(x,x) \mathrm{d}x$,
and from \eqref{eq:HNdense}, 
$$-H_N(\bs) \dto \frac{1}{2}\int_0^1 W(x,x) \mathrm{d}x + V_{\beta}.$$ Hence, 
\begin{align*}
\P(\hat{\beta}_N < \infty ) & = \P\left( -H_N(\bs) \geq \frac{1}{2N\theta_N}\vartheta_N  \text{ and } \bs \notin \mathscr{G}_N \right) 
\rightarrow \P\left( V_{\beta} > 0 \right) , 
\end{align*}
 by \eqref{eq:HNprobability} and \eqref{eq:HNdense}.

 \item[$(2)$] $\theta = 0$: In this case,  $-\theta_N H_N(\bs) \xrightarrow{P} \frac{\beta}{2}\int_{[0,1]^2} W(x,y) \mathrm{d}x  \mathrm{d}y$  (recall \eqref{eq:HNsparse}) and hence, $-H_N(\bs) \xrightarrow{P} \infty$ and $\frac{1}{N} H_N(\bs) \xrightarrow{P} 0$. Also, $\frac{1}{N^2\theta_N}\eta_N \xrightarrow{P} \int_{[0,1]^2} W(x,y) \mathrm{d}x  \mathrm{d}y > 0$. Hence, by \eqref{eq:HNprobability}, $\P_{\beta}(\mathscr{G}_N) \ll 1$. Also, as in the previous case, note that $\frac{1}{2N\theta_N} \vartheta_N \xrightarrow{L^2} \frac{1}{2}\int_0^1 W(x,x) \mathrm{d}x. $
Hence, 
\begin{align*}
\P(\hat{\beta}_N < \infty ) & = \P\left( -H_N(\bs) \geq \frac{1}{2N\theta_N}\vartheta_N  \text{ and } \bs \notin \mathscr{G}_N \right) \rightarrow 1 . 
\end{align*} 
 \end{enumerate}
This completes the proof of \eqref{lem:mlesolution}. 
\end{proof}

The above lemma shows that in the sparse regime the ML estimate is finite with probability tending to 1. On the other hand, in the dense regime there is a positive probability that the ML estimate is infinite. With this result, we now proceed with the proof of Theorem \ref{mle}. 

\subsubsection{Proof of Theorem \ref{mle} $(1)$}
Note that for any $t \in \R$,
\begin{align}\label{eq:mlesparse}
\P_\beta\left(\frac{1}{\sqrt{\theta_N}}\left(\hat{\beta}_N-\beta\right)\leq t \right) 
&=  \P_\beta\left(\hat{\beta}_N\leq t \sqrt{\theta_N}+\beta\right) \nonumber \\ 
&= \P_\beta\left( \psi_N'(\hat{\beta}_N) \leq \psi_N'\left(t \sqrt{\theta_N}+\beta\right)\right)                                                                                                                                      \nonumber \\ 
    &= \P_\beta\left(H_N(\bs)\geq \E_{\beta+t\sqrt{\theta_N}}\left[H_N(\bs) | A_{G_N} \right]\right)                                                                                                                                                                                                                                                                 \nonumber \\ 
        &= \P_\beta\left( T_N \geq \E_{\beta+t\sqrt{\theta_N}}\left[ T_N | A_{G_N} \right]\right)   , 
    \end{align}
    where 
    \begin{align}\label{eq:TN}
    T_N := \sqrt{\theta_N}\left(\frac{1}{2N\theta_N}\bs^\top A_{G_N} \bs-\frac{\beta}{2N^2\theta_N}\sum_{1 \leq i,j \leq N}W\left(\frac{i}{N},\frac{j}{N}\right) \right) . 
    \end{align}
   The next lemma establishes the distributional convergence of $T_N$ and its conditional mean and variance.

\begin{ppn}\label{mlelemd4y} 
Fix $0<\beta<\frac{1}{\|W\|_{\mathrm{op}}}$ and $t \in \R$, and suppose $\bs \sim \P_{\beta+t\sqrt{\theta_N}}$. Then, in the regime $N^{-\frac{2}{3}} \lesssim \theta_N \ll 1$ and for $T_N$ as defined in \eqref{eq:TN}, 
\begin{align}\label{eq:TNdistribution}
T_N \xrightarrow{D} \cN\left(\frac{t}{2}\int_{[0,1]^2} W(x,y)\mathrm{d}x\mathrm{d}y, \frac{1}{2} \int_{[0,1]^2} W(x,y)\mathrm{d}x\mathrm{d}y\right). 
\end{align} 
Further, 
\begin{align*}
\E\left[ T_N |A_{G_N}\right]\xrightarrow{P} \frac{t}{2}\int_{[0,1]^2} W(x,y)\mathrm{d}x\mathrm{d}y \quad \text{ and } \quad 
\mathrm{Var}\left[ T_N |A_{G_N}\right] \xrightarrow{P}  \frac{1}{2}\int_{[0,1]^2} W(x,y)\mathrm{d}x\mathrm{d}y,
\end{align*} 
and for every subsequence $\{N_s\}_{s \ge 1}$ of the natural numbers, there is a further subsequence $\{N_{s_a}\}_{a \ge 1}$ and a probability $1$ set $\mathcal{A} \in \sigma(\{A_{G_N}\}_{N\ge 1})$, on which the conditional distribution of $T_{N_{s_a}}$ given $A_{G_{N_{s_a}}}$ converges weakly to the RHS of \eqref{eq:TNdistribution}. 
\end{ppn}

Proposition \ref{mlelemd4y} is proved in Appendix \ref{sec:mlelemd4ypf}. Applying this result in \eqref{eq:mlesparse} gives 
    \begin{align*}
\P_\beta\left(\frac{1}{\sqrt{\theta_N}}\left(\hat{\beta}_N-\beta\right)\leq t \right) &= \P_\beta\left( T_N \geq \E_{\beta+t\sqrt{\theta_N}}\left[ T_N | A_{G_N} \right]\right)   \\  
    &\to\P_\beta\left(\cN\left(0 , ~ \frac{1}{2}\int_{[0, 1]^2} W(x,y)\mathrm{d}x\mathrm{d}y\right) \geq  \frac{t}{2}\int_{[0,1]^2} W(x,y)\mathrm{d}x\mathrm{d}y \right)\\
    &=\P\left(\cN\left(0,\frac{2}{\int_{[0, 1]^2} W(x,y)\mathrm{d}x\mathrm{d}y}\right)\leq t \right) . 
\end{align*}
This completes the proof of Theorem \ref{mle} (1).  \hfill $\Box$

\subsubsection{Proof of Theorem \ref{mle} $(2)$}

Recall from the proof of Theorem \ref{thm:HNsigma}, that for every $t\in \mathbb{R}$,
\begin{align*}
    \E_\beta\left[e^{ -tH_N(\bs)}|A_{G_N}\right] \xrightarrow{P} \E\left[e^{ tU_\beta} \right],
\end{align*}
where 
\begin{align*}
    U_\beta & := \frac{1}{2}\int_0^1 W(x,x) \mathrm{d}x + V_{\beta} \nonumber \\ 
    & = \beta \sigma_W^2 + \frac{1}{2}\int_{0}^{1}W(x,x)\mathrm{d}x + \sigma_WZ_0+\frac{1}{2} \sum_{\lambda \in \mathrm{Spec}(W)} \lambda \left(\frac{Z_{\lambda}^2}{1-\beta \lambda}-1\right) . 
\end{align*}
This implies, by Lemma \ref{l2ratio762}, 
\begin{align}\label{mgfsubseq34}
    \psi_N'(\beta) = -\E_\beta\left[\left.H_N(\bs)\right|A_{G_N}\right]\xrightarrow{P} \E[U_\beta] = \frac{1}{2}\int_0^1 W(x,x) \mathrm{d}x + \E[ V_{\beta} ] , 
\end{align} 
for $0<\beta<\frac{1}{\|W\|_{\mathrm{op}}}$. 

We now proceed to compute the limit of $\P_\beta(\hat{\beta}_N>t)$, for $t \in \R$. To begin with, note that $\hat{\beta}_N$ is supported on $(0,\infty]$, hence, 
\begin{align}\label{eq:betamle1} 
\P_\beta(\hat{\beta}_N>t)=1 , 
\end{align}
for all $t\le 0$. Next, suppose $0<t<\frac{1}{\|W\|_{\mathrm{op}}}$. Since the function $\psi_N'(\beta)$ is strictly increasing,  
   \begin{align}
        \P_\beta(\hat{\beta}_N>t) &= \P(\hat{\beta}_N=+\infty) +  \P_\beta ( \psi_N'(\hat{\beta}_N)> \psi_N'(t)) + o(1) \nonumber \\
        & = \P(\hat{\beta}_N=+\infty) +  \P_\beta ( - H_N(\bs) > \psi_N'(t)) + o(1)\nonumber \\ 
        &\to \P(V_\beta<0)+ \P_\beta(V_\beta>\E V_t) \tag*{(by \eqref{eq:HNdense} and \eqref{mgfsubseq34})} \nonumber \\
        &= \P(V_\beta<0) + \P_\beta(V_{\beta} > F(t)) \nonumber \\ 
        & = \P(V_\beta<0) + \P_\beta(F^{-1}(V_{\beta} \bm{1}\{V_\beta >0\})> t)  , 
        \label{eq:betamle2} 
    \end{align}
 where the last line follows from the fact that $F(t) > 0$ and is increasing for $0<t<\frac{1}{\|W\|_{\mathrm{op}}}$. 
    
    Now, suppose $t\geq \frac{1}{\|W\|_{\mathrm{op}}}$. Then for any $\delta >0$ small enough,     \begin{align*}
        \P_\beta(\hat{\beta}_N>t) &=   \P_\beta(-H_N(\bs)> \psi_N'(t) ) + \P(\hat{\beta}_N=\infty) \\ 
         & \leq \P_\beta\left (-H_N(\bs)> \psi_N'\left(\frac{1}{\|W\|_{\mathrm{op}}}-\delta\right) \right) + \P(\hat{\beta}_N=\infty) \\ 
        & \to \P_\beta\left(V_\beta^+> F\left(\frac{1}{\|W\|_{\mathrm{op}}}-\delta\right)\right)+\P(V_\beta<0)  , 
    \end{align*} 
    by \eqref{eq:HNdense} and \eqref{mgfsubseq34}. Hence,
    \begin{align*}
        \limsup_{N\to \infty} \P_\beta(\hat{\beta}_N>t) \le \P_\beta\left(V_\beta^+ > F\left(\frac{1}{\|W\|_{\mathrm{op}}}-\delta\right)\right)+\P(V_\beta<0).
    \end{align*}
As $\delta \ll 1$, $F(\frac{1}{\|W\|_{\mathrm{op}}}-\delta) \to \infty$ and, hence, $ \limsup_{N\to \infty} \P_\beta(\hat{\beta}_N>t)\leq\P(V_\beta<0)$.  Also, $\P_\beta(\hat{\beta}_N>t) \geq \P(\hat{\beta}_N=\infty) \rightarrow  \P(V_\beta<0)$.  Hence, for $t\geq \frac{1}{\|W\|_{\mathrm{op}}}$, 
    \begin{align}
        \P_\beta(\hat{\beta}_N>t) \rightarrow  \P(V_\beta<0). 
        \label{eq:betamle3} 
    \end{align}
    Combining \eqref{eq:betamle1}, \eqref{eq:betamle2}, and \eqref{eq:betamle3} gives, 
    \begin{align*}
        \lim_{N\to \infty} \P_\beta(\hat{\beta}_N>t) = \begin{cases}
            1                                                       & \text{ if } t\leq 0 , \\
            \P_\beta(F^{-1}(V_\beta^+)>t)+\P(V_\beta<0) & \text{ if } 0<t<
            \frac{1}{\|W\|_{\mathrm{op}}}       ,                                                         \\
            \P(V_\beta<0)                                           & \text{ if } t\geq
            \frac{1}{\|W\|_{\mathrm{op}}} . 
        \end{cases}
    \end{align*}    
This completes the proof of Theorem \ref{mle} (2).   \hfill $\Box$

\subsubsection{Proof of Proposition \ref{mlelemd4y}}
\label{sec:mlelemd4ypf}

    Let $\beta_N=\beta + t \sqrt{\theta_N}$. Then, by \eqref{eq:Z_Nratio}, for $s$ with small enough absolute value, we have:
    \begin{align*}
        \frac{Z_N(\beta_N+ s)}{Z_N(\beta_N)} = \left(1 + o_{P}\left(1\right)\right)\frac{\E\hat{Z}_N(\beta_N+ s)}{\E\hat{Z}_N(\beta_N)} e^{ \textstyle \frac{s}{2N}\sum_{i=1}^{N}W\left(\frac{i}{N},\frac{i}{N}\right) + \frac{s^2+2\beta_N s}{4N^2\theta_N}\sum_{1 \leq i,j \leq N}W\left(\frac{i}{N},\frac{j}{N}\right) } . 
    \end{align*} 
    Next, define $$T_N' := \sqrt{\theta_N} \left(-H_N(\bs)-\frac{\beta_N}{2N^2\theta_N}\sum_{1 \leq i,j \leq N}W\left(\frac{i}{N},\frac{j}{N}\right) \right).$$ Then,  
    \begin{align*}
        &\E_{\beta_N} \left[e^{ s T_N'}\Big| A_{G_N}\right]\\ 
        &=\left(1 + o_{P}\left(1\right)\right)\frac{\E\hat{Z}_N(\beta_N+ s\sqrt{\theta_N})}{\E\hat{Z}_N(\beta_N)} e^{ \textstyle \frac{ s \sqrt{\theta_N}}{2N}\sum_{i=1}^{N}W\left(\frac{i}{N},\frac{i}{N}\right) + \frac{s^2}{4N^2}\sum_{1 \leq i,j \leq N}W\left(\frac{i}{N},\frac{j}{N}\right) } . 
    \end{align*}
    By Lemma \ref{thetanperturb6},  
    \begin{align*}
        \frac{\E\hat{Z}_N(\beta_N+ s\sqrt{\theta_N})}{\E\hat{Z}_N(\beta_N)} \to 1.
    \end{align*}
    Also, observe  that 
    \begin{align*}
        e^{ \textstyle \frac{ s \sqrt{\theta_N}}{2N}\sum_{i=1}^{N}W\left(\frac{i}{N},\frac{i}{N}\right) + \frac{s^2}{4N^2}\sum_{1 \leq i,j \leq N}W\left(\frac{i}{N},\frac{j}{N}\right) } \to  e^{ \textstyle  \frac{ s^2}{4} \int_{[0, 1]^2} W(x,y)\mathrm{d}x\mathrm{d}y }  . 
    \end{align*}
   Combining the above gives, 
    \begin{align}\label{eq:TNmgf}
        \E_{\beta_N}\left[e^{ sT_N'}\Big| A_{G_N}\right]\xrightarrow{P} e^{ c s^2},
    \end{align}
    where $c:=\frac{1}{4} \int_{[0, 1]^2} W(x,y)\mathrm{d}x\mathrm{d}y$. Also, recalling \eqref{eq:TN}, observe that
        \begin{align}\label{eq:differenceTN}
        T_N- T_N' & = \frac{t}{2N^2}\sum_{1 \leq i,j \leq N} W\left(\frac{i}{N},\frac{j}{N}\right) \rightarrow \frac{t}{2}\int_{[0, 1]^2} W(x,y) \mathrm{d}x \mathrm{d}y.
    \end{align}
    Combining \eqref{eq:TNmgf} and \eqref{eq:differenceTN} gives, 
    $$\E_{\beta_N}\left[e^{s T_N} | A_{G_N}\right]\xrightarrow{P} e^{2c s t + cs ^2} , $$ 
    for $s$ in a neighborhood of zero. Proposition \ref{mlelemd4y} now follows from Lemma \ref{l2ratio762} and its proof, on noting that the RHS of the above convergence is the moment-generating function of $\cN(2c t ,2c)$. \hfill $\Box$

\subsection{Proof of Theorem \ref{clttilde}} 
\label{sec:clttildepf}

To begin with, define:
\begin{align}\label{eq:dn}
D_N = D_N(\beta) := \frac{1}{\sqrt{\theta_N}}\left(\frac{\beta}{2N^2\theta_N}\sum_{1 \leq i,j \leq N} A_{G_N}(i,j) -  \frac{\beta}{2N^2}\sum_{1 \leq i,j \leq N} W\left(\frac{i}{N}, \frac{j}{N}\right)\right) . 
\end{align}
 By a straightforward algebra, we can re-write
\begin{align}\label{dnexp98}
 D_N = \frac{\beta}{2N\theta_N} & \frac{\sum_{1 \leq i,j \leq N} (A_{G_N}(i,j) - \theta_N W(\frac{i}{N},\frac{j}{N}))}{\sqrt{\theta_N \sum_{1 \leq i,j \leq N}W(\frac{i}{N},\frac{j}{N})(1-\theta_N W(\frac{i}{N},\frac{j}{N}))}} \nonumber \\ 
 & \hspace{0.95in} \sqrt{\frac{\sum_{1 \leq i,j \leq N}W(\frac{i}{N},\frac{j}{N})(1-\theta_N W(\frac{i}{N},\frac{j}{N}))}{N^2}}.
\end{align}
Note that $\frac{\beta}{2N\theta_N}  \ll 1$, since $N\theta_N \gg 1$. By the Lindeberg-Feller CLT, 
$$\frac{\sum_{1 \leq i,j \leq N} (A_{G_N}(i,j) - \theta_N W(\frac{i}{N},\frac{j}{N}))}{\sqrt{\theta_N \sum_{1 \leq i,j \leq N}W(\frac{i}{N},\frac{j}{N})(1-\theta_N W(\frac{i}{N},\frac{j}{N}))}}  \dto \cN(0,2).$$ 
Furthermore, 
$$\sqrt{\frac{\sum_{1 \leq i,j \leq N}W(\frac{i}{N},\frac{j}{N})(1-\theta_N W(\frac{i}{N},\frac{j}{N}))}{N^2}}  \pto \sqrt{\int_{[0,1]^2} W(x,y) \mathrm{d}x \mathrm{d}y}.$$ Combining all these, we get:
\begin{equation}\label{dnequalop1}
  D_N = o_P(1).  
\end{equation}
Therefore, Theorem \ref{thm:HNsigma} (1) gives,  
\begin{align}\label{modthmham8}
     \frac{1}{\sqrt{\theta_N}}\left(\theta_N H_N(\bs)+\frac{\beta}{2N^2\theta_N}\sum_{1 \leq i,j \leq N} A_{G_N}(i,j) \right)\xrightarrow{D}\cN\left(0,\frac{1}{2}\int_{[0,1]^2} W(x,y)\mathrm{d}x\mathrm{d}y\right).
\end{align}
    Hence, plugging the identity (recall \eqref{newestimate}) 
    $$\theta_N H_N(\bs) = -\frac{\wbtu}{2N^2\theta_N}\sum_{1 \leq i,j \leq N} A_{G_N}(i,j)$$ in the LHS of \eqref{modthmham8} gives, 
\begin{align}\label{retiosl78}
    \frac{\sum_{1 \leq i,j \leq N}A_{G_N}(i,j)}{2N^2\theta_N^{\frac{3}{2}}} (\wbtu -\beta) \xrightarrow{D} \cN\left(0,\frac{1}{2}\int_{[0,1]^2} W(x,y)\mathrm{d}x\mathrm{d}y\right).
\end{align}  
Next, note that 
\begin{align*}
    \var\left[\frac{\sum_{1 \leq i,j \leq N} A_{G_N}(i,j)}{\theta_N \sum_{1 \leq i,j \leq N} W(\frac{i}{N}, \frac{j}{N})}\right] = \frac{\sum_{1 \leq i<j \leq N} W(\frac{i}{N}, \frac{j}{N}) (1- \theta_N W(\frac{i}{N}, \frac{j}{N}))}{\theta_N (\sum_{1 \leq i<j \leq N} W(\frac{i}{N}, \frac{j}{N}))^2} = O\left(\frac{1}{N^2\theta_N}\right) \ll 1,
\end{align*}
which implies that 
\begin{align*}
 \frac{\sum_{1 \leq i,j \leq N} A_{G_N}(i,j)}{\theta_N \sum_{1 \leq i,j \leq N} W(\frac{i}{N}, \frac{j}{N})} \xrightarrow{P} 1.  
\end{align*}
This, together with \eqref{retiosl78} and an application of Slutsky's theorem gives, 
$$\frac{\sum_{1 \leq i,j \leq N} W(\frac{i}{N}, \frac{j}{N})}{2N^2\sqrt{\theta_N}}(\wbtu-\beta) \xrightarrow{D} \cN\left(0,\frac{1}{2}\int_{[0,1]^2} W(x,y)\mathrm{d}x\mathrm{d}y\right),$$ and hence,
$$\frac{\int_{[0,1]^2} W(x,y) \mathrm{d}x \mathrm{d}y}{2\sqrt{\theta_N}}(\wbtu-\beta) \xrightarrow{D} \cN\left(0,\frac{1}{2}\int_{[0,1]^2} W(x,y)\mathrm{d}x\mathrm{d}y\right).$$  
   The proof of Theorem \ref{clttilde} is now complete.

\subsection{Proof of Theorem \ref{locminest}}\label{sec:proof_locmintest}

We first prove \eqref{eq:locmin2} in Appendix \ref{sec:locmin2pf}. Then we prove   \eqref{eq:locmin1} in Appendix \ref{sec:locmin1pf}.

\subsubsection{Proof of \eqref{eq:locmin2} } 
\label{sec:locmin2pf}

 To begin with, define:
$$B_N:=\frac{1}{\sqrt{\theta_N}}\left(\theta_N H_N(\bs)+\frac{\beta}{2N^2}\sum_{1 \leq i,j \leq N}W\left(\frac{i}{N},\frac{j}{N}\right)\right) + \frac{h}{2N^2} \sum_{1 \leq i,j \leq N}W\left(\frac{i}{N},\frac{j}{N}\right).$$
Also, recall the definition of $D_N$ from \eqref{eq:dn}:
$$D_N := \frac{1}{\sqrt{\theta_N}}\left(\frac{\beta}{2N^2\theta_N}\sum_{1 \leq i,j \leq N} A_{G_N}(i,j) -  \frac{\beta}{2N^2}\sum_{1 \leq i,j \leq N} W\left(\frac{i}{N}, \frac{j}{N}\right)\right).$$
Finally, for $h \in \R$, define:
$$E_N := \frac{1}{\sqrt{\theta_N}}\left(\frac{h\sqrt{\theta_N}}{2N^2\theta_N}\sum_{1 \leq i,j \leq N} A_{G_N}(i,j) -  \frac{h\sqrt{\theta_N}}{2N^2}\sum_{1 \leq i,j \leq N} W\left(\frac{i}{N}, \frac{j}{N}\right)\right).$$ 
With these notations,  
\begin{align}\label{l2riskalt6}
    &\E_{\beta+h\sqrt{\theta_N}} \left[\left(\frac{1}{\sqrt{\theta_N}}(\wbtu - \beta) -h\right)^2 \Bigg| A_{G_N}\right]\nonumber\\ 
    &= \frac{1}{\theta_N}\E_{\beta+h\sqrt{\theta_N}} \left[\left(\frac{2N^2\theta_N^2}{\sum_{1 \leq i,j \leq N} A_{G_N}(i,j)} H_N(\bs) + \beta +h\sqrt{\theta_N}\right)^2 \Bigg| A_{G_N}\right]\nonumber\\ 
    &= \left(\frac{2N^2\theta_N}{\sum_{1 \leq i,j \leq N}A_{G_N}(i,j)}\right)^2 \E_{\beta+h\sqrt{\theta_N}}\left[( B_N+D_N + E_N)^2 \Bigg| A_{G_N}\right].
\end{align}
By \eqref{dnexp98}, both $D_N$ and $E_N$ are $o_P(1)$. Also, by Proposition \ref{mlelemd4y}, $$\E_{\beta+h\sqrt{\theta_N}} [B_N^2|A_{G_N}] \xrightarrow{P} \frac{1}{2}\int_{[0, 1]^2} W(x,y) \mathrm{d}x \mathrm{d}y,$$ under $\P_{\beta+h\sqrt{\theta_N}}$. Hence,
\begin{align}\label{gden4}
    \E_{\beta+h\sqrt{\theta_N}}\left[( B_N+D_N + E_N)^2 \Bigg| A_{G_N}\right] = \frac{1}{2}\int_{[0, 1]^2} W(x,y) \mathrm{d}x \mathrm{d}y +o_P(1).
\end{align}
Now, define $$F_N := \frac{\sum_{1 \leq i,j \leq N}A_{G_N}(i,j)}{\theta_N\sum_{1 \leq i,j \leq N} W(\frac{i}{N}, \frac{j}{N})}. $$ By Hoeffding's inequality, we have for any $\varepsilon > 0$,
\begin{align*}
\P\left(|F_N-1| >\varepsilon\right) & = \P\left(\left|\sum_{1 \leq i,j \leq N} \left(A_{G_N}(i,j) -\theta_N W\left(\frac{i}{N},\frac{j}{N}\right)\right)\right| > \varepsilon\theta_N \sum_{1 \leq i,j \leq N} W\left(\frac{i}{N},\frac{j}{N}\right)\right) \nonumber \\ 
& \le 2e^{-C \varepsilon^2 (N\theta_N)^2} , 
\end{align*}
for some universal constant $C>0$. This shows that $F_N = 1+ o_P(1)$. Hence,
\begin{align}\label{F5first6}
    \frac{2N^2\theta_N}{\sum_{1 \leq i,j \leq N}A_{G_N}(i,j)} \xrightarrow{P} \frac{2}{\int_{[0,1]^2} W(x,y) \mathrm{d}x \mathrm{d}y} .  
\end{align}
Combining \eqref{l2riskalt6}, \eqref{gden4} and \eqref{F5first6} gives, 
$$\E_{\beta+h\sqrt{\theta_N}} \left[\left(\frac{1}{\sqrt{\theta_N}}(\wbtu - \beta) -h\right)^2 \Bigg| A_{G_N}\right] \xrightarrow{P} \frac{2}{\int_{[0,1]^2} W(x,y) \mathrm{d}x \mathrm{d}y} , 
$$
for any $h \in \R$. Hence, for any finite set $I \subset \R$, 
$$\sup_{h \in I} \E_{\beta+h\sqrt{\theta_N}} \left[\left(\frac{1}{\sqrt{\theta_N}}(\wbtu - \beta) -h\right)^2 \Bigg| A_{G_N}\right] \xrightarrow{P} \frac{2}{\int_{[0,1]^2} W(x,y) \mathrm{d}x \mathrm{d}y} , 
$$
which proves the result in \eqref{eq:locmin2}. \hfill $\Box$

\subsubsection{Proof of \eqref{eq:locmin1}} 
\label{sec:locmin1pf}

To prove this result, we will follow the recipe outlined in the proof of \cite[Theorem 8.11]{Vaart_1998}. Throughout we will assume that $(\bar{\beta}_N - \beta)/\sqrt{\theta_N}$ is uniformly tight under $\p_\beta$ (this assumption suffices, since the  general case can be proved similarly with the help of a compactification tool to induce tightness, as noted in \cite[Page 118]{Vaart_1998} and \cite[Chapter 3.11]{van1996weak}). 
A key ingredient of the proof of \eqref{eq:locmin1} is the convergence of the likelihood ratio $\mathrm{d}\p_{\beta+h\sqrt{\theta_N}}/\mathrm{d}\p_{\beta}$, under $\p_\beta$. This is established in the following proposition, which is proved later in Appendix \ref{sec:limexp1pf}.

\begin{ppn}\label{limexp1} Suppose Assumption \ref{assumption} holds and $\theta_N \ll 1$. Then the following hold: 
    \begin{enumerate}
        \item[$(1)$] (Sparse regime) If $\theta=0$, then for any finite set $I\subset \mathbb{R}$, 
        $$\left(\frac{\mathrm{d} \P_{\beta_0+h\sqrt{\theta_N}}}{\mathrm{d} \P_{\beta_0+ h_0\sqrt{\theta_N}}}(\bs)\right)_{h\in I} \xrightarrow[\bs \sim \P_{\beta_0+ h_0\sqrt{\theta_N}}]{D} \left( e^{ \textstyle (h-h_0) Z + \frac{1}{4}(h_0^2-h^2)\int_{[0,1]^2} W(x,y) \mathrm{d}x \mathrm{d}y } \right)_{h\in I}, $$
    where 
    \begin{align}\label{eq:Zh}
    Z \sim \cN\left(\frac{h_0}{2}\int_{[0,1]^2} W(x,y) \mathrm{d}x \mathrm{d}y, \frac{1}{2}\int_{[0,1]^2} W(x,y) \mathrm{d}x \mathrm{d}y\right). 
    \end{align}
    
    \item[$(2)$] (Dense regime) If $\theta > 0$, then for any finite set $I \subset (0,\frac{1}{\|W\|_{\mathrm{op}}} )$, 
    \begin{align*}
        \left(\frac{\mathrm{d} \P_{\beta}}{\mathrm{d} \P_{\beta_0}} (\bs)\right)_{\beta\in I} \xrightarrow[\bs \sim \P_{\beta_0}]{D} \left( e^{ \textstyle \frac{\beta_0^2-\beta^2}{2}\sigma_W^2 -(\beta_0-\beta) V_{\beta_0} } \prod_{i=1}^\infty e^{ \textstyle \frac{(\beta-\beta_0)\lambda }{2} } \sqrt{\frac{1-\beta\lambda }{1-\beta_0\lambda }}\right)_{\beta\in I} , 
    \end{align*}
   where $V_{\beta_0}$ is as in \eqref{eq:Vbeta} and $\sigma_W^2:=\frac{1}{2\theta} \int_{[0,1]^2} W(x,y)(1-\theta W(x,y))\mathrm{d}x\mathrm{d}y$. 
    \end{enumerate}    
\end{ppn}

\begin{remark} 
Note that Proposition \ref{limexp1} computes the asymptotic distribution of the likelihood ratio in both the sparse and the dense cases. For the proof of Proposition \ref{limexp1}, we only require the result in the sparse case. The dense case will be used later in the proof of Theorem \ref{th:limexp} in Appendix \ref{proof_limexp8}.  
\end{remark}

We now complete the proof of \eqref{eq:locmin1} using Proposition \ref{limexp1}. Towards this, note by Proposition \ref{limexp1} that the likelihood ratio $\mathrm{d}\p_{\beta+h\sqrt{\theta_N}}/\mathrm{d}\p_{\beta}$ is weakly convergent, hence tight under $\p_\beta$. Hence, the random vector
$$\bm V_N := \left(\frac{\bar{\beta}_N-\beta}{\sqrt{\theta_N}}, ~\frac{\mathrm{d}\p_{\beta+h\sqrt{\theta_N}}}{\mathrm{d}\p_{\beta}}\right)$$ is tight under $\p_\beta$. Hence, by Prohorov's theorem, every subsequence of $\N$ has a further subsequence along which 
$\bm V_N$ is weakly convergent under $\P_\beta$. Moreover, by Proposition \ref{limexp1} and Le Cam's first lemma \cite[Lemma 6.4]{Vaart_1998}, $\p_{\beta+h\sqrt{\theta_N}}$ and $\p_{\beta}$ are mutually contiguous. This, coupled with the observation above, shows that along the further subsequence, $(\bar{\beta}_N-\beta-h\sqrt{\theta_N})/\sqrt{\theta_N}$ is weakly convergent under $\P_{\beta+h\sqrt{\theta_N}}$, by Le Cam's third lemma \cite[Theorem 6.6]{Vaart_1998}. Then \cite[Theorem 9.3]{Vaart_1998} and Theorem \ref{th:limexp} imply that this weak limit is of the form $T-h$ for a (possibly randomized) statistic $T$ based on the limiting experiment: $$\cF_h :=\cN\left(h, \frac{2}{\int_{[0,1]^2} W(x,y) \mathrm{d}x \mathrm{d}y} \right).$$ Finally, by \cite[Proposition 8.6]{Vaart_1998}, we have:
\begin{align*}
    \sup_{h\in \mathbb{R}} \e_{\cF_h} \left(T-h\right)^2 \ge \e_{X\sim \cF_0} X^2 = \frac{2}{\int_{[0,1]^2} W(x,y) \mathrm{d}x \mathrm{d}y}.
\end{align*}
Now, following the arguments in the proof of  \cite[Theorem 8.11]{Vaart_1998}, the result in \eqref{eq:locmin1} follows. \hfill $\Box$

\subsubsection{Proof of Proposition \ref{limexp1}} 
\label{sec:limexp1pf}

We begin with the proof of (1). For this, suppose $I = \{h_1,\ldots,h_K\} \subset \R$, for some $K \geq 1$. Fix $\bm \alpha = (\alpha_1,\ldots,\alpha_K) \in \R^K$. Then, by \eqref{ref2707}, 
    \begin{align*}
        \sum_{i=1}^{K} \alpha_i\log \frac{\mathrm{d} \P_{\beta_0+h_i\sqrt{\theta_N}}}{\mathrm{d} \P_{\beta_0+ h_0\sqrt{\theta_N}}}(\bs) & = \sum_{i=1}^{K} \alpha_i(h_0-h_i) \sqrt{\theta_N} H_N(\bs) + \sum_{i=1}^{K} \alpha_i \log \frac{Z_N(\beta_0 + h_0\sqrt{\theta_N})}{Z_N(\beta_0 + h_i\sqrt{\theta_N})}\\
        & = T_{n, \bm \alpha} +  \frac{1}{4}\int_{[0,1]^2} W(x,y) \mathrm{d}x \mathrm{d}y \sum_{i=1}^{K} \alpha_i (h_0^2-h_i^2) + o_P(1) ,   
        \end{align*}
        where 
        $$T_{n, \bm \alpha} := \sum_{i=1}^{K} \alpha_i (h_0-h_i) \left[\sqrt{\theta_N} H_N(\bs) + \frac{\beta_0}{2N^2\sqrt{\theta_N}} \sum_{1 \leq u,v \leq N}W\left(\frac{u}{N},\frac{v}{N}\right)\right]. $$ 
    By Proposition \ref{mlelemd4y}, as $N \rightarrow \infty$, 
    $$T_{n, \bm \alpha} \xrightarrow[\bs \sim \P_{\beta_0+ h_0\sqrt{\theta_N}}]{D} \sum_{i=1}^K \alpha_i (h_0-h_i) Z,$$
where $ Z $ as in \eqref{eq:Zh}. Hence,
$$\sum_{i=1}^K \alpha_i\log \frac{\mathrm{d} \P_{\beta_0+h_i\sqrt{\theta_N}}}{\mathrm{d} \P_{\beta_0+ h_0\sqrt{\theta_N}}} (\bs) \xrightarrow[\bs \sim \P_{\beta_0+ h_0\sqrt{\theta_N}}]{D}  \sum_{i=1}^K \alpha_i \left[(h_0-h_i) Z + \frac{1}{4}(h_0^2-h_i^2)\int_{[0,1]^2} W(x,y) \mathrm{d}x \mathrm{d}y \right].$$ 
By the Cram\'er-Wold device and the continuous mapping theorem, the result in Proposition \ref{limexp1} (1) now follows.

Next, we prove Proposition \ref{limexp1} (2). For this suppose $I = \{\beta_1,\ldots,\beta_K\} \subset (0, \frac{1}{\|W\|_{\mathrm{op}}})$, for some $K \geq 1$. Fix $\bm \alpha = (\alpha_1,\ldots,\alpha_K) \in \R^K$. Then 
$$\sum_{i=1}^K \alpha_i\log \frac{\mathrm{d} \P_{\beta_i}}{\mathrm{d} \P_{\beta_0}}(\bs) = \sum_{i=1}^K \alpha_i(\beta_0-\beta_i)  H_N(\bs) + \sum_{i=1}^K \alpha_i \log \frac{Z_N(\beta_0)}{Z_N(\beta_i)}.$$
It follows from \eqref{eq:Z_Nratio} and \eqref{eq:EZ_Nratio} that, 
\begin{align*}
    & \log \frac{Z_N(\beta_0)}{Z_N(\beta_i)} \\ 
    &= \frac{\beta_0-\beta_i}{2}\int_0^1 W(x,x) \mathrm{d}x + \frac{(\beta_0^2-\beta_i^2) \sigma_W^2}{2} + \frac{1}{2}\sum_{\lambda \in \mathrm{Spec}(W)} \left[\log \left(\frac{1-\beta_i\lambda}{1-\beta_0\lambda}\right) + (\beta_i-\beta_0) \lambda\right] + o_P(1) , 
\end{align*} 
where $\sigma_W^2$ is as defined in the statement of Proposition \ref{limexp1}. Also, from \eqref{eq:HNdense} we know that $H_N(\bs) \dto -U_{\beta_0}$, for $\bs \sim \P_{\beta_0}$, where 
\begin{align*}
    U_{\beta_0} :=\beta_0 \sigma_W^2 + \frac{1}{2}\int_{0}^{1}W(x,x)\mathrm{d}x + \sigma_WZ_0+\sum_{\lambda \in \mathrm{Spec}(W)} \frac{\lambda}{2}\left(\frac{Z_\lambda^2}{1-\beta_0 \lambda}-1\right).
\end{align*}
Hence, 
\begin{align*}
    &\sum_{i=1}^K \alpha_i\log \frac{\mathrm{d} \P_{\beta_i}}{\mathrm{d} \P_{\beta_0}}(\bs)\\ &\xrightarrow[\bs \sim \P_{\beta_0}]{D} \sum_{i=1}^K \alpha_i\left(\frac{ (\beta_0^2-\beta_i^2) \sigma_W^2 }{2}\right) + \frac{1}{2}\sum_{i=1}^K \alpha_i \sum_{\lambda \in \mathrm{Spec}(W)} \left[\log \left(\frac{1-\beta_i\lambda}{1-\beta_0\lambda}\right) + (\beta_i-\beta_0) \lambda\right]\\
    & \hspace{0.95in} - \sum_{i=1}^K \alpha_i (\beta_0-\beta_i)\left[\beta_0 \sigma_W^2 + \sigma_WZ_0+\sum_{\lambda \in \mathrm{Spec}(W)} \frac{\lambda}{2}\left(\frac{Z_\lambda^2}{1-\beta_0 \lambda}-1\right)\right] . 
\end{align*}
By the Cram\'er-Wold device and the continuous mapping theorem, the result in Proposition \ref{limexp1} (2) now follows. \hfill $\Box$

\section{Proofs from Section \ref{sec:gof} } 

This section is organized as follows: We begin with the proof of 
Theorem \ref{locpower} in Appendix \ref{sec:lppf}. Theorem \ref{thm:mb} and Corollary \ref{cor:testgraphon} are proved in Appendix \ref{sec:mbpf} and Appendix \ref{sec:prooftestgraphon}, respectively. The proof of  Theorem \ref{testminimax} is given in Appendix \ref{proofminimax}.

\subsection{Proof of Theorem \ref{locpower}}
\label{sec:lppf}

To begin with, define $\gamma := \frac{1}{2}\int_{[0,1]^2} W(x,y) \mathrm{d}x \mathrm{d}y$. By Proposition \ref{mlelemd4y}, it follows that under the model $\P_{\beta_0+t\sqrt{\theta_N}}$, 
\begin{equation*}
   \sqrt{\theta_N} H_N(\bs) + \frac{\beta_0}{2N^2\sqrt{\theta_N}} \sum_{1 \leq i,j \leq N} W\left(\frac{i}{N}, \frac{j}{N}\right) \xrightarrow{D} \cN(-t\gamma, \gamma), 
\end{equation*} 
and in fact, the above weak convergence also happens conditionally on $A_{G_{N}}$ for every further subsequence $N_2$ of a chosen subsequence $N_1$ of the natural numbers, almost surely. 
 Next, define:
$$c_\alpha := -z_{\alpha/2}\sqrt{\frac{1}{2\theta_N}\iint_{[0,1]^2}W(x,y) dx dy} - \frac{\beta_0}{2N^2 \theta_N} \sum_{1 \leq i,j \leq N}W\left(\frac{i}{N},\frac{j}{N}\right) $$
and 
$$d_\alpha := z_{\alpha/2}\sqrt{\frac{1}{2\theta_N}\iint_{[0,1]^2}W(x,y) dx dy} - \frac{\beta_0}{2N^2 \theta_N} \sum_{1 \leq i,j \leq N}W\left(\frac{i}{N},\frac{j}{N}\right).$$
The following argument is meant to be on this further subsequence $N_2$, but for notational convenience, we simply denote $N_2$ by $N$. Specifically, we have almost surely,
\begin{align}\label{hnnaivepower4}
    &\P_{\beta_0+t\sqrt{\theta_N}}(H_N(\bs) < c_\alpha| A_{G_N}) + \P_{\beta_0+t\sqrt{\theta_N}}(H_N(\bs) > d_\alpha| A_{G_N})\notag\\
    &= \Phi\left(- z_{\alpha/2} + t\sqrt{\gamma}\right) + \Phi\left(- z_{\alpha/2} - t\sqrt{\gamma}\right) + o(1).
\end{align}

We now claim that $\sqrt{\theta_N}(\hat{c}_\alpha - c_\alpha) = o_P(1)$ and $\sqrt{\theta_N}(\hat{d}_\alpha - d_\alpha) = o_P(1)$. We will only prove the result for $\sqrt{\theta_N}(\hat{c}_\alpha - c_\alpha)$, as the other result will follow similarly. Towards this, note that:
$$\sqrt{\theta_N}(\hat{c}_\alpha - c_\alpha) = S_1+ S_2,$$ where
$$S_1 := -z_{\alpha/2}\left(\sqrt{\frac{1}{2N^2\theta_N}\sum_{1\le i,j\le N} A_{i,j}(G_N)} - \sqrt{\gamma}\right) \quad \text{and}\quad S_2 := -D_N(\beta_0),$$
where $D_N(\beta_0)$ is as defined in \eqref{eq:dn} (with $\beta$ replaced by $\beta_0$). It follows from \eqref{dnequalop1} that $S_2 = o_P(1)$. Further, we have $S_1 := S_1' + S_1''$, where 
\begin{align*}
    &S_1' := -z_{\alpha/2} \left(\sqrt{\frac{1}{2N^2\theta_N}\sum_{1\le i,j\le N} A_{i,j}(G_N)} - \sqrt{\frac{1}{2N^2} \sum_{1\le i,j\le N} W\left(\frac{i}{N}, \frac{j}{N}\right)}\right) 
    \end{align*} 
       and 
       \begin{align*} 
       S_1'' := -z_{\alpha/2} \left(\sqrt{\frac{1}{2N^2} \sum_{1\le i,j\le N} W\left(\frac{i}{N}, \frac{j}{N}\right)} - \sqrt{\gamma}\right).
\end{align*}
First, note that 
\begin{align}\label{t1panalysis108}
    S_1' &= -z_{\alpha/2}  \frac{{\frac{1}{2N^2\theta_N}\sum_{1\le i,j\le N} A_{i,j}(G_N)} - {\frac{1}{2N^2} \sum_{1\le i,j\le N} W\left(\frac{i}{N}, \frac{j}{N}\right)}}{\sqrt{\frac{1}{2N^2\theta_N}\sum_{1\le i,j\le N} A_{i,j}(G_N)} + \sqrt{\frac{1}{2N^2} \sum_{1\le i,j\le N} W\left(\frac{i}{N}, \frac{j}{N}\right)}}\notag\\ &= -z_{\alpha/2} \frac{\sqrt{\theta_N} D_N}{\beta \left(\sqrt{\frac{1}{2N^2\theta_N}\sum_{1\le i,j\le N} A_{i,j}(G_N)} + \sqrt{\frac{1}{2N^2} \sum_{1\le i,j\le N} W\left(\frac{i}{N}, \frac{j}{N}\right)}\right)}\\&= o_P(1),\notag
\end{align}
since the numerator of \eqref{t1panalysis108} is $o_P(1)$ by \eqref{dnequalop1}, and the denominator is larger than $\sqrt{\frac{\gamma}{2}}$ for all large $N$. Also, since $W(x,y)$ is Riemann integrable, one has $S_1'' = o(1)$. Hence, $S_1 = o_P(1)$, thereby proving the claim.

Returning to the main proof, there exists a further subsequence $N \equiv N_3$ of $N_2$, along which both $\sqrt{\theta_N}(\hat{c}_\alpha - c_\alpha)$ and $\sqrt{\theta_N}(\hat{d}_\alpha - d_\alpha)$ converge to $0$ almost surely. Since weak convergence to a continuous distribution implies uniform convergence of the underlying distribution functions, it follows from \eqref{hnnaivepower4} and our claim, that along the subsequence $N\equiv N_3$, we have almost surely,
\begin{align*}
    & \P_{\beta_0+t\sqrt{\theta_N}}(H_N(\bs) < \hat{c}_\alpha| A_{G_N}) + \P_{\beta_0+t\sqrt{\theta_N}}(H_N(\bs) > \hat{d}_\alpha| A_{G_N}) \nonumber \\  
    & \rightarrow \Phi\left(- z_{\alpha/2} + t\sqrt{\gamma}\right) + \Phi\left(- z_{\alpha/2} - t\sqrt{\gamma}\right).
\end{align*}
Taking expectations on both sides above, we have by the dominated convergence theorem,
\begin{equation}\label{hu778}
    \rho_{N_3}(\beta_0+t\sqrt{\theta_{N_3}}) \rightarrow \Phi\left(- z_{\alpha/2} + t\sqrt{\gamma}\right) + \Phi\left(- z_{\alpha/2} - t\sqrt{\gamma}\right).
\end{equation}
Theorem \ref{locpower} now follows on noticing that \eqref{hu778} holds for some subsequence $N_3$ of any arbitrary sequence $N_1$ of natural numbers.

\subsection{Proof of Theorem \ref{thm:mb}}
\label{sec:mbpf}

Now recall that $\{\hat\lambda_i(G_N), 1\leq i\leq N\}$ are the eigenvalues of $\frac{1}{N \theta_N} A_{G_N}$ arranged in non-increasing order. Then,
\begin{align}\label{eq:hateigensum}
    \frac{1}{N^2\theta_N^2}\sum_{1 \leq i \ne j \leq N} A_{G_N}(i,j) = \frac{1}{N^2\theta_N^2}\sum_{1 \leq i \ne j \leq N} A_{G_N}(i,j)^2 = \sum_{i=1}^{N}\hat\lambda_i(G_N)^2,
\end{align} 
where the first equality follows, since $A_{G_N}(i,j)\in \{0,1\}$ for all $1\leq i,j\leq N$. By Lemma \ref{lemma:A_{G_N}_convg_W} and \cite[Theorem 11.54]{lovasz2012large} we know that there exists an almost sure set $\cA_1$, such that on $\cA_1$, $\|\frac{1}{N \theta_N} A_{G_N}\|\ra \|W\|_{\mathrm{op}}$. The assumption $\beta\left\|W\right\|<1$ implies there exists a $\gamma>0$ small enough such that $|\beta|\left(\|W\|_{\mathrm{op}}+\gamma\right)<1$. Hence on the set $\cA_1$, for large enough $N$, 
\begin{align}\label{eq:lambda_hat_bdd_1}
    | \beta\hat\lambda_i(G_N) |\leq |\beta| \left\|\frac{1}{N \theta_N} A_{G_N}\right\|\leq |\beta|\left(\|W\|_{\mathrm{op}} + \gamma\right)<1 ,  
\end{align} 
for all $i\geq 1$. Then on $\cA_1$, we can find constant $C(W,\beta)$ such that 
\begin{align}\label{eq:1_beta_hat_lambda_bdd}
    \frac{1}{|1-\beta\hat\lambda_i(G_N)|} \leq C(W,\beta),
\end{align}
for all $N$ large enough. Now, recalling the definition of $ \hat U_\beta$ from \begin{align*}
    \hat U_\beta := \frac{\beta}{2N^2\theta_N^2}\sum_{1 \leq i \ne j \leq N} A_{G_N}(i,j) & +  \frac{1}{2N\theta_N}\sum_{i=1}^N A_{G_N}(i,i) \nonumber \\ 
    & + \frac{\beta^2}{2}\sum_{i=1}^{N}\frac{\hat\lambda_i(G_N)^3}{1-\beta\hat\lambda_i(G_N)} + \frac{1}{2}\sum_{i=1}^{N}\frac{\hat\lambda_i(G_N)}{1-\beta\hat\lambda_i(G_N)}\left(\eta_i^2-1\right) . 
\end{align*}
Then conditional on $A_{G_N}$, by \cite[Proposition 7.1]{bhattacharya2017universal} the moment generating function of $\hat U_\beta$ is given by 
\begin{align}\label{eq:logMGFUbetahat}
    \log \E\left[ e^{ t\hat U_\beta } \mid A_{G_N}\right] = \sum_{s=1}^\infty a_s(\beta, G_N) t^s , 
\end{align}
where 
\begin{align*} 
a_1(\beta, G_N) & := \frac{\beta}{2N^2\theta_N^2}\sum_{1 \leq i \ne j \leq N} A_{G_N}(i,j) + \frac{1}{2N\theta_N}\sum_{i=1}^N A_{G_N}(i,i) + \frac{\beta^2}{2}\sum_{i=1}^{N}\frac{\hat\lambda_i(G_N)^3}{1-\beta\hat\lambda_i(G_N)} , \nonumber \\ 
a_s(\beta, G_N) & := \frac{1}{2 s }\sum_{i=1}^{N}\frac{\hat\lambda_i(G_N)^s}{(1-\beta\hat\lambda_i(G_N))^s}  , 
\end{align*} 
for $s \geq 2$, and 
%
%
for all $|t|<\frac{1}{8}\left( \sum_{i=1}^{N}\frac{\hat\lambda_i(G_N)^2}{(1-\beta\hat\lambda_i(G_N))^2} \right)^{-1}$. Notice that on the set $\cA_1$,
\begin{align*}
    \sum_{i=1}^{N}\frac{\hat\lambda_i(G_N)^2}{1-\beta\hat\lambda_i(G_N)^2}\leq C(W,\beta)^2 \sum_{i=1}^{N}\hat\lambda_i(G_N)^2 \leq \frac{4C(W,\beta)^2}{\theta^2} , 
\end{align*}
for all $N$ large enough. Hence, on $\cA_1$, for all $N$ large enough, the conditional moment generating function of $\hat U_\beta$ exists for all $|t|<\delta$, for some $\delta<\min\{\frac{1}{2}, \frac{1}{8}\|W\|_2, \frac{\theta^2}{32C(W,\beta)^2}\}$. Also, by Hoeffding's concentration inequality there exists an almost sure set $\cA_2$ such that,
\begin{align}\label{eq:Hoeffdin_convgsedges}
    \frac{1}{N^2\theta_N^2}\sum_{1 \leq i \ne j \leq N} A_{G_N}(i,j)\ra \frac{1}{\theta}\int_{[0, 1]^2} W(x,y)\mathrm{d} x\mathrm{d} y 
    \end{align} 
    and 
    \begin{align}\label{eq:Hoeffdin_convgsdiagonal}
    \frac{1}{N\theta_N}\sum_{i=1}^N A_{G_N}(i,i)\ra \int_0^1 W(x,x)\mathrm{d} x . 
\end{align}
Now, to complete the proof of Theorem \ref{thm:mb} we need the following lemma, which computes the limit of the terms appearing in $a_s(\beta, G_N)$.

\begin{lemma}\label{lemma:eigen_convg}
    Under the assumptions of  Theorem \ref{thm:mb} there exists an almost sure set $\cA_0$ such that,
    \begin{align}\label{eq:eigen_convg}
        \sum_{i=1}^{N}\frac{\hat\lambda_i(G_N)^{s}}{\left(1-\beta\hat\lambda_i(G_N)\right)^{t}}\ra\sum_{ \lambda \in \mathrm{Spec}(W) } \frac{\lambda ^{s}}{\left(1-\beta\lambda \right)^{t}} ,     \end{align}
     as $N\ra\infty$, for all $s\geq 3$ and $t\geq 1$ on $\cA_0$.
\end{lemma}

The proof of Lemma \ref{lemma:eigen_convg} is deferred to Appendix  \ref{sec:proofof_eigen_convg}. To complete the proof of Theorem \ref{thm:mb}, define $\cA = \cA_0 \cap  \cA_1 \cap \cA_2$, where $\cA_0, \cA_1, \cA_2$ are as defined above. By \eqref{eq:Hoeffdin_convgsedges}, \eqref{eq:Hoeffdin_convgsdiagonal}, and Lemma \ref{lemma:eigen_convg}, 
\begin{align}\label{eq:aGN}
   a_1(\beta, G_N) \ra  a_1(\beta, W) ,
\end{align}
on $\mathcal A$, where $a_1(\beta, W)$ is as defined in \eqref{eq:coefficientW}. Again, by Lemma \ref{lemma:eigen_convg}, 
\begin{align}\label{eq:aGNt}
   a_2(\beta, G_N) & :=  \frac{1}{4}\sum_{i=1}^{N}\frac{\hat\lambda_i(G_N)^2}{1-\beta\hat\lambda_i(G_N)^2}
    \nonumber \\ 
    & = \frac{1}{4}\left[\sum_{i=1}^{\infty}\hat\lambda_i(G_N)^2 + \sum_{i=1}^{\infty} \left( \frac{\hat\lambda_i(G_N)^2}{(1-\beta\hat\lambda_i(G_N))^2} - \hat\lambda_i(G_N)^2 \right) \right] \nonumber \\
    & = \frac{1}{4}\left[\sum_{i=1}^{\infty}\hat\lambda_i(G_N)^2 + 2\beta\sum_{i=1}^{\infty}\frac{\hat\lambda_i(G_N)^3}{(1-\beta\hat\lambda_i(G_N))^2} - \beta^2\sum_{i=1}^{\infty}\frac{\hat\lambda_i(G_N)^4}{(1-\beta\hat\lambda_i(G_N))^2}\right] \nonumber \\
    & = \frac{1}{4}\left[\frac{1}{N^2\theta_N^2}\sum_{1 \leq i \ne j \leq N} A_{G_N}(i,j) + 2\beta\sum_{i=1}^{\infty}\frac{\hat\lambda_i(G_N)^3}{(1-\beta\hat\lambda_i(G_N))^2} - \beta^2\sum_{i=1}^{\infty}\frac{\hat\lambda_i(G_N)^4}{(1-\beta\hat\lambda_i(G_N))^2}\right]  \nonumber \\ 
   & \rightarrow a_2(\beta, W) ,
    \end{align} 
    on $\mathcal A$, where $a_2(\beta, W)$ is as defined in \eqref{eq:coefficientW}. Also, by Lemma \ref{lemma:eigen_convg}, 
    \begin{align}\label{eq:aGNz}
    a_s(\beta, G_N) \rightarrow \frac{1}{2 s} \sum_{i=1}^{\infty} \frac{\lambda ^{s}}{\left(1-\beta\lambda \right)^{s}} := a_s(\beta, W)  , 
    \end{align} 
    on $\mathcal A$, for $s \geq 3$. Now to complete the proof, from \eqref{eq:hateigensum} notice that for large enough $N$, $|\hat\lambda_i(G_N)|\leq \frac{2}{\theta\sqrt{i}}$ for all $i\geq 1$. Hence on the set $\cA$, recalling \eqref{eq:logMGFUbetahat} and combining the convergences from \eqref{eq:aGN}, \eqref{eq:aGNt}, and \eqref{eq:aGNz} with the bound from \eqref{eq:1_beta_hat_lambda_bdd}, we can apply the Dominated Convergence Theorem to show, 
\begin{align}\label{eq:logMGFUbetahat_convg}
    \log \E\left[ e^{ t\hat U_\beta } \mid A_{G_N}\right] = \sum_{s=1}^\infty a_s(\beta, G_N) t^s \rightarrow \sum_{s=1}^\infty a_s(\beta, W) t^s = \log \E\left[ e^{ t U_\beta } \right] , 
\end{align}
for $|t|< \tau$, for some small $\tau<\delta< \min\{\frac{1}{2},\frac{1}{8}\|W\|_2, \frac{\theta^2}{32C(W,\beta)^2}\}$. Note that the last equality in 
 \eqref{eq:logMGFUbetahat_convg} is by Lemma \ref{lm:HNmgf}. This completes the proof of \eqref{eq:Ubetaconvergence}.  


\subsubsection{Proof of Lemma \ref{lemma:eigen_convg}}\label{sec:proofof_eigen_convg}

 Since $\beta\left\|W\right\|<1$ by assumption, we can expand the RHS of \eqref{eq:eigen_convg} as,
\begin{align}\label{eq:RHS_expand}
    \sum_{\lambda \in \mathrm{Spec}(W)} \frac{\lambda ^s}{\left(1-\beta\lambda \right)^t} 
    & = \sum_{\lambda \in \mathrm{Spec}(W)} \sum_{a=0}^{\infty}{\binom{t+a-1}{a}}\beta^a\lambda ^{a+s}\nonumber\\
    & = \sum_{a=0}^{\infty}{\binom{t+a-1}{a}}\beta^a\sum_{\lambda \in \mathrm{Spec}(W)} \lambda ^{a+s} , 
\end{align}
where the final equality follows by Fubini's theorem in conjunction with Lemma \ref{lemma:eigenvalue_summability} (d). Recalling \eqref{eq:cycle}, we can rewrite \eqref{eq:RHS_expand} as,
\begin{align}\label{eq:lambda_sum_tC}
    \sum_{\lambda \in \mathrm{Spec}(W)}\frac{\lambda ^s}{\left(1-\beta\lambda \right)^t} = \sum_{a=0}^{\infty}{\binom{t+a-1}{a}}\beta^at\left(C_{a+s},W\right) . 
\end{align}

Next, consider the LHS of \eqref{eq:eigen_convg}. From \eqref{eq:lambda_hat_bdd_1} recall that on $\mathcal A_1$, $| \beta\hat\lambda_i(G_N) |\leq |\beta|\left(\|W\|_{\mathrm{op}} + \gamma\right)<1$, for all $i\geq 1$. Now, defining $\hat\lambda_i(G_N) = 0$, for all $i> N$, we can expand the LHS of \eqref{eq:eigen_convg} on $\cA_1$ as follows: 
\begin{align}\label{eq:eigenvalueconvergence}
    \sum_{i=1}^{N}\frac{\hat\lambda_i(G_N)^s}{\left(1-\beta\hat\lambda_i(G_N)\right)^t} = \sum_{i=1}^{\infty}\frac{\hat\lambda_i(G_N)^s}{\left(1-\beta\hat\lambda_i(G_N)\right)^t} = \sum_{i=1}^{\infty}\sum_{a=0}^{\infty}{\binom{t+a-1}{a}}\beta^a\hat\lambda_i(G_N)^{a+s}.
\end{align}
Note that on the set $\cA_1$ using the bound in \eqref{eq:lambda_hat_bdd_1} gives, 
\begin{align*}
    \sum_{i=1}^{\infty}\sum_{a=0}^{\infty}{\binom{t+a-1}{a}}|\beta|^a|\hat\lambda_i(G_N) |^{a+s}
    &\leq \sum_{i=1}^{\infty}\sum_{a=0}^{\infty}{\binom{t+a-1}{a}}|\beta|^a\left(\|W\|_{\mathrm{op}} + \gamma\right)^a|\hat\lambda_i(G_N)|^s\\
    &\leq \frac{1}{\left(1-\beta(\|W\|_{\mathrm{op}} + \gamma)\right)^t}\sum_{i=1}^{\infty}|\hat\lambda_i(G_N)|^s.
\end{align*}
Also, from \eqref{eq:hateigensum}, for large enough $N$, $|\hat\lambda_i(G_N)|\leq \frac{2}{\theta\sqrt{i}}$, for all $i\geq 1$. 
Hence,
\begin{align*}
    \sum_{i=1}^{\infty}\sum_{a=0}^{\infty}{\binom{t+a-1}{a}}\left|\beta\right|^a |\hat\lambda_i(G_N) |^{a+s}\leq \frac{(2/\theta)^s}{\left(1-\beta(\|W\|_{\mathrm{op}} + \gamma)\right)^t}\sum_{i=1}^{\infty}\frac{1}{i^{\frac{s}{2}}}<\infty.
\end{align*} 
Then, by Fubini's Theorem, on the set $\cA_1$ and for all large enough $N$, the order of the summation in \eqref{eq:eigenvalueconvergence} can be interchanged. This gives, 
\begin{align*}
    \sum_{i=1}^{N}\frac{\hat\lambda_i(G_N)^s}{\left(1-\beta\hat\lambda_i(G_N)\right)^t} = \sum_{a=0}^{\infty}{\binom{t+a-1}{a}}\beta^a\sum_{i=1}^{\infty}\hat\lambda_i(G_N)^{a+s} = \sum_{a=0}^{\infty}{\binom{t+a-1}{a}}\frac{\beta^a}{\theta_N^{a+s}} t\left(C_{a+s},W^{G_N}\right), 
\end{align*} 
where $W^{G_N}$ is the empirical graphon as defined in \eqref{eq:WGN}. By Lemma \ref{lemma:A_{G_N}_convg_W} and \cite[Theorem 11.5]{lovasz2012large} there exists an almost sure set $\mathcal A_3$ such that on $\mathcal A_3$, $t\left(C_{a+s},W^{G_N}\right) \rightarrow t(C_{a+s}, W)$, for all $a\geq 0$ and $s\geq 3$, as $N \rightarrow \infty$. Now, to complete the proof we need to apply dominated convergence theorem to exchange the limit and the sum. To that end notice that on $\mathcal A_0 = \mathcal A_1\cap \mathcal A_3$, for $t\geq 1$, $a\geq 0$, and $s\geq 3$,
\begin{align*}
    {\binom{t+a-1}{a}} \frac{\left|\beta^at\left(C_{a+s},W^{G_N}\right)\right|}{\theta_N^{a+s}}
    & \leq {\binom{t+a-1}{a}}\beta^a\sum_{i=1}^{\infty}\left|\hat \lambda_i(G_N)\right|^{a+s}\\
    & \leq {\binom{t+a-1}{a}}\beta^a\left(\|W\|_{\rm op} + \gamma\right)^{a}\sum_{i=1}^{\infty}\frac{(2/\theta)^{s}}{i^{\frac{s}{2}}}.
\end{align*}
Moreover,
\begin{align*}
    \sum_{a=0}^{\infty}{\binom{t+a-1}{a}}\beta^a\left(\|W\|_{\rm op} + \gamma\right)^{a}\sum_{i=1}^{\infty}\frac{(2/\theta)^{s}}{i^{\frac{s}{2}}} = \frac{1}{\left(1-\beta\left(\|W\|_{\rm op}+\gamma\right)\right)^t}\sum_{i=1}^{\infty}\frac{(2/\theta)^{s}}{i^{\frac{s}{2}}}<\infty,
\end{align*}
where the first equality follows by recalling that $\gamma>0$ is chosen such that $\beta\left(\|W\|_{\rm op} + \gamma\right)<1$. 
Hence, by the dominated convergence theorem we conclude that on $\cA_0$,
\begin{align*}
    \sum_{i=1}^{N}\frac{\hat\lambda_i(G_N)^s}{\left(1-\beta\hat\lambda_i(G_N)\right)^t}\ra \sum_{a=0}^{\infty}{\binom{t+a-1}{a}}\beta^at\left(C_{a+s},W\right) = \sum_{ \lambda \in \mathrm{Spec}(W) }\frac{\lambda ^s}{\left(1-\beta\lambda \right)^t} , 
\end{align*} 
as $N \rightarrow \infty$, where the final equality is from \eqref{eq:lambda_sum_tC}. \hfill $\Box$

\subsection{Proof of Corollary \ref{cor:testgraphon}}\label{sec:prooftestgraphon}
To prove Corollary \ref{cor:testgraphon}, we begin by noting that for all $\beta\in (0, \frac{1}{\|W\|_{\rm op}})$, $U_{\beta}$ admits a density with respect to the Lebesgue measure on $\R$ (recall \eqref{eq:def_Ubeta}). Fix $\zeta\in (0,1)$. Then, by the distributional convergence established in Theorem \ref{thm:mb}, 
\begin{align*}
    \hat q_{\zeta, \beta_0, G_N}\ra q_{\zeta, \beta_0, W}, \text{ almost surely},
\end{align*}
where $q_{\zeta, \beta_0, W}$ denotes the $\zeta^{th}$ quantile of $U_{\beta_0}$. Furthermore, for any $\beta\in (0, \frac{1}{\|W\|_{\rm op}})$, Theorem \ref{thm:HNsigma} (2), together with Slutsky's lemma, implies that
\begin{align*}
    -H_N(\bs) - \hat q_{\zeta, \beta_0, G_N}\dto U_{\beta} - q_{\zeta,\beta_0, W},
\end{align*}
where $\bs$ is sampled from the model $\P_\beta$ in \eqref{model_def}. Since $U_\beta$ is absolutely continuous, it follows immediately that
\begin{align}\label{eq:convg_HN_qhat_beta}
    \P_\beta\left(-H_N(\bs)\leq \hat q_{\zeta, \beta_0, G_N}\right)\ra \P_\beta(U_\beta\leq q_{\zeta, \beta_0, W}).
\end{align}
Observe that the convergence in \eqref{eq:convg_HN_qhat_beta} holds for every $\zeta\in (0,1)$. Therefore, under any $\beta$,
\begin{align*}
    \lim_{N\ra\infty}\E_\beta(\phi_N^+) 
    & = \lim_{N\ra\infty}\P_\beta\left(-H_N(\bs)\leq \hat q_{1-\alpha/2, \beta_0, G_N}\text{ or }-H_N(\bs)>\hat q_{\alpha/2, \beta_0, G_N}\right)\\
    & = \lim_{N\ra\infty}\P_{\beta}\left(-H_N(\bs)\leq \hat q_{1-\alpha/2, \beta_0, G_N}\right)
    + \P_{\beta}\left(-H_N(\bs)>\hat q_{\alpha/2, \beta_0, G_N}\right)\\
    & = \P_{\beta}(U_{\beta}\leq q_{1-\alpha/2, \beta_0, W})
    + \P_{\beta}(U_{\beta}> q_{\alpha/2, \beta_0, W}).
\end{align*}
The proof is completed by observing that under $H_0:\beta=\beta_0$,
\begin{align*}
    \P_{\beta_0}(U_{\beta_0}\leq q_{1-\alpha/2, \beta_0, W})
    + \P_{\beta_0}(U_{\beta_0}> q_{\alpha/2, \beta_0, W})
    = \alpha,
\end{align*}
while for $\beta\neq \beta_0$,
\begin{align*}
    \P_{\beta}(U_{\beta}\leq q_{1-\alpha/2, \beta_0, W})
    + \P_{\beta}(U_{\beta}> q_{\alpha/2, \beta_0, W})
    = 1 + F_\beta(q_{1-\alpha/2, \beta_0, W}) - F_\beta(q_{\alpha/2, \beta_0, W}).
\end{align*}

\subsection{Proof of Theorem \ref{testminimax}}\label{proofminimax}

\subsubsection{Proof of Theorem \ref{testminimax} $(1)$} 

Fix a sequence $\delta_N \ll \sqrt{\theta_N}$ and let $\beta_1 := \beta_{0} + \delta_N$. Denote by $L_N$ the likelihood ratio of the joint measures of $(\bs, A_{G_N})$ at $\beta_{1}$ and $\beta_0$. Then 
$$L_N(\bs) := \frac{Z_N(\beta_0)}{Z_N(\beta_{1})} e^{-(\beta_{1}-\beta_0) H_N(\bs)}.$$
Hence, 
\begin{align}\label{l2exp998}
    \E_{\beta_0} [L_N^2(\bs)| A_{G_N}] =\frac{Z_N^2(\beta_0)}{Z_N^2(\beta_0 + \delta_N)} \cdot \frac{Z_N(\beta_0 +2 \delta_N)}{Z_N(\beta_0)}.
\end{align} 
The following lemma gives the asymptotics of the ratio of the partition functions in the regime 
$\delta_N \ll \sqrt{\theta_N}$.

\begin{lemma}\label{znanint6} Suppose $\beta \in (0, \frac{1}{\|W\|_{\mathrm{op}}})$ and $N^{-\frac{2}{3}}  \lesssim \theta_N \ll 1$. Then for any sequence $a_N \ll \sqrt{\theta_N}$, 
\begin{align*}
    \frac{Z_N(\beta+ a_N)}{Z_N(\beta)} e^{ \textstyle -\frac{\beta a_N}{2N^2 \theta_N} \sum_{1 \leq i,j \leq N} W\left(\frac{i}{N}, \frac{j}{N}\right) } \xrightarrow{P} 1 . 
\end{align*}
\end{lemma}

\begin{proof} 
Note that 
\begin{align}\label{eq:znanint6pf}
\frac{Z_N(\beta+ a_N)}{Z_N(\beta)} e^{ \textstyle -\frac{\beta a_N}{2N^2 \theta_N} \sum_{1 \leq i,j \leq N} W\left(\frac{i}{N}, \frac{j}{N}\right) } = \E_\beta \left [ e^{ \frac{a_N}{\sqrt{\theta_N}} B_N(\bs) } |A_{G_N} \right ], 
\end{align}
where 
$$B_N(\bs) :=  -\sqrt{\theta_N} H_N(\bs) - \frac{\beta}{2N^2\sqrt{\theta_N}} \sum_{1 \leq i,j \leq N} W\left(\frac{i}{N}, \frac{j}{N}\right).$$ From the proofs of Theorem \ref{thm:HNsigma} (1) and Lemma \ref{l2ratio762} it follows that $B_N(\bs)$ converges weakly under the measure $\P_\beta(\cdot|A_{G_N})$. This means, since $a_N \ll \sqrt{\theta_N}$, 
$$e^{ \frac{a_N}{\sqrt{\theta_N}} B_N(\bs) } \pto 1 , $$ 
under  $\P_{\beta}(\cdot|A_{G_N})$.  Furthermore, from the proof of Theorem \ref{thm:HNsigma} (1), it follows by an application of Jensen's inequality on the concave function $x \mapsto x^{2a_N/\sqrt{\theta_N}}$, that 
$$\E_\beta\left[e^{ \frac{2a_N}{\sqrt{\theta_N}} B_N(\bs)}\Big|A_{G_N}\right] \le \left(\E_\beta \left[e^{B_N(\bs)}\Big|A_{G_N}\right]\right)^{\frac{2a_N}{\sqrt{\theta_N}}} = 1+o_P(1).$$ 
Hence, by uniform integrability, the RHS of \eqref{eq:znanint6pf} converges to 1. 
\end{proof}

Since $\beta_0 \in (0, \frac{1}{\|W\|_{\mathrm{op}}})$, applying Lemma \ref{znanint6} to the terms in \eqref{l2exp998} gives, 
$$\frac{Z_N^2(\beta_0)}{Z_N^2(\beta_0 + \delta_N)} e^{ \textstyle \frac{\beta_0 \delta_N}{N^2 \theta_N} \sum_{1 \leq i,j \leq N} W\left(\frac{i}{N}, \frac{j}{N}\right) } \xrightarrow{P} 1$$
and 
$$\frac{Z_N(\beta_0 +2\delta_N)}{Z_N(\beta_0)} e^{ \textstyle -\frac{\beta_0 \delta_N}{N^2 \theta_N} \sum_{1 \leq i,j \leq N} W\left(\frac{i}{N}, \frac{j}{N}\right)} \xrightarrow{P} 1.$$ 
This implies, 
\begin{align}\label{l2lncv6}
   \E_{\beta_0} [ L_N^2(\bs)| A_{G_N} ]  \xrightarrow{P} 1.
\end{align}
Now, the risk of any test function $\phi_N$ (recall \eqref{eq:RphiN}) can be easily bounded below as: 
\begin{align*}
    \mathcal{R}(\phi_N) & \geq \E_{\beta_0} [\phi_N(\bs)|A_{G_N}] + \E_{\beta_{1}}[1-\phi_N(\bs)|A_{G_N}]\\
    &= \E_{\beta_0} [\phi_N(\bs)|A_{G_N}] + \E_{\beta_{0}}[(1-\phi_N(\bs)) L_N(\bs)|A_{G_N}]\\
    &= 1+ \E_{\beta_{0}}[\phi_N(\bs)(1- L_N(\bs))|A_{G_N}].
\end{align*} 
By the Cauchy-Schwarz inequality, 
$$|\E_{\beta_{0}}[\phi_N(\bs)(1- L_N(\bs))|A_{G_N}]| \le \sqrt{\E_{\beta_0} [(1-L_N(\bs))^2|A_{G_N}]} = \sqrt{\E_{\beta_0} [L_N^2(\bs)|A_{G_N}] -1}. $$
Therefore, almost surely, 
$$ 1- \sqrt{\E_{\beta_0} [L_N^2(\bs)|A_{G_N}] -1} \leq \inf_{\phi_N} \mathcal{R}(\phi_N) \leq 1 , $$ 
where the upper bound follows by taking the trivial test function that always rejects $H_0$ (that is, $\phi_N = 1$ everywhere). Hence, by \eqref{l2lncv6}, 
$$\inf_{\phi_N} \mathcal{R}(\phi_N) \pto 1,$$ 
when $\delta_N \ll \sqrt{\theta_N}$. This completes the proof of Theorem \ref{testminimax} (1).

\subsubsection{Proof of Theorem \ref{testminimax} $(2)$} 

 Note that the test \eqref{eq:testcd} can be written as:
$$\phi^\alpha_N = \bm 1\left\{ |\tilde{\beta}_N -  \beta_0 | \geq \hat {\sigma}_{N, \alpha} \sqrt{\theta_N} \right\}, $$ 
where $\tilde{\beta}_N$ is the estimate in \eqref{newestimate} and   
\begin{align}\label{eq:sigmaestimate}
        \hat {\sigma}_{N, \alpha} = z_{\alpha/2} \sqrt{\frac{2 \theta_N N^2}{\sum_{1 \leq i,j \leq N} A_{G_N}(i,j)}} \pto z_{\alpha/2} \sqrt{ \frac{ 2 }{\int_{[0, 1]^2} W(x, y) \mathrm d x \mathrm d y} } := \sigma_{\alpha}. 
      \end{align} 
Then we have (see Corollary \ref{31cor266} and the proof of Lemma \ref{l2ratio762}): 
\begin{align}\label{eq:TypeI}
\P_{H_0}(\phi^\alpha_N = 1 | A_{G_N}) =  \P_{H_0}(  | \tilde{\beta}_N -\beta_0 |  \geq \hat {\sigma}_{N, \alpha} \sqrt{\theta_N}  | A_{G_N} )\pto \alpha . 
\end{align}

Towards this, first note that $\P_{\beta}(\cdot|A_{G_N})$ has monotone likelihood ratio in $-H_N(\bs)$ (and hence, in $\tilde{\beta}_N$). This means, for any $t \in \R$ and $0< \beta < \beta' < \frac{1}{\|W\|_{\mathrm{op}}}$, 
\begin{align}\label{eq:sd}
\P_{\beta'}(\tilde{\beta}_N < t | A_{G_N}) \leq \P_{\beta}(\tilde{\beta}_N < t | A_{G_N} ).
\end{align} 
Now, suppose $\beta_{1, N} \in \Theta(\beta_0, \delta_N)$ (as defined in \eqref{eq:RphiN}). First, we consider the case $\beta_{1, N} \geq \beta_0 + \delta_N$. Since $\delta_N \gg \sqrt {\theta_N}$, for any $M > 1$ there exists $N_0(M)$ such that $\delta_N \geq M \sigma_{\alpha} \sqrt{\theta_N}$, for all $N \geq N_0(M)$. Hence, $\beta_{1, N} \geq \beta_0 + \delta_N \geq \beta_0 + M \sigma_{\alpha} \sqrt{\theta_N}$, when $N \geq N_0(M)$. Then using \eqref{eq:sd} gives, 
\begin{align*}
\P_{\beta_{1, N}}(\phi^\alpha_N = 0 | A_{G_N}) & =  \P_{\beta_{1, N} }(  |\tilde{\beta}_N -\beta_0 | <  \hat {\sigma}_{N, \alpha} \sqrt{\theta_N} | A_{G_N}) \nonumber \\ 
& \leq  \P_{\beta_{1, N} }(  \tilde{\beta}_N  <  \beta_0 + \hat {\sigma}_{N, \alpha} \sqrt{\theta_N} | A_{G_N}) \nonumber \\ 
& \leq  \P_{\beta_0 + M \sigma_{\alpha} \sqrt{\theta_N} }(  \tilde{\beta}_N  <  \beta_0 + \hat {\sigma}_{N, \alpha} \sqrt{\theta_N} | A_{G_N}) \nonumber \\
 & =  \P_{\beta_0 + M \sigma_{\alpha} \sqrt{\theta_N} }(  \tilde{\beta}_N  - ( \beta_0 + M \sigma_{\alpha} \sqrt{\theta_N} ) < - ( M \sigma_{\alpha} \sqrt{\theta_N} -  \hat {\sigma}_{N, \alpha} \sqrt{\theta_N} )  | A_{G_N}) \nonumber \\  
& :=  A_{N,M} .  
\end{align*}
Next, suppose  $\beta_{1, N} \leq \beta_0 - \delta_N \leq \beta_0 - M \sigma_{\alpha} \sqrt{\theta_N}$, when $N \geq N_0(M)$. Then using \eqref{eq:sd} gives, 
\begin{align*}
\P_{\beta_{1, N}}(\phi^\alpha_N = 0 | A_{G_N}) & =  \P_{\beta_{1, N} }(  |\tilde{\beta}_N -\beta_0 | <  \hat {\sigma}_{N, \alpha} \sqrt{\theta_N} | A_{G_N}) \nonumber \\ 
& \leq  \P_{\beta_{1, N} }(  \tilde{\beta}_N  >  \beta_0 - \hat {\sigma}_{N, \alpha} \sqrt{\theta_N} | A_{G_N}) \nonumber \\ 
& \leq  \P_{\beta_0 - M \sigma_{\alpha} \sqrt{\theta_N} }(  \tilde{\beta}_N  >  \beta_0 - \hat {\sigma}_{N, \alpha} \sqrt{\theta_N} | A_{G_N}) \nonumber \\
 & =  \P_{\beta_0 - M \sigma_{\alpha} \sqrt{\theta_N} }(  \tilde{\beta}_N  - ( \beta_0 - M \sigma_{\alpha} \sqrt{\theta_N} ) >  M \sigma_{\alpha} \sqrt{\theta_N} -  \hat {\sigma}_{N, \alpha} \sqrt{\theta_N}   | A_{G_N}) \nonumber \\ 
& := B_{N,M} 
\end{align*} 
Hence, for every $M>1$, we have the following for all $N \ge N_0(M)$:
\begin{equation}\label{riskambm78}
    \sup_{\beta_{1, N} \in \Theta(\beta_0, \delta_N)} \P_{\beta_{1, N}}(\phi^\alpha_N = 0 | A_{G_N}) \leq \max\{A_{N,M}, B_{N,M}\}.
\end{equation}
Note that by Proposition \ref{mlelemd4y} and \eqref{eq:sigmaestimate}, we have the following, for every $M > 1$ and as $N\rightarrow \infty$, 
$$A_{N,M} \pto \Phi(-(M-1)z_{\alpha/2})\quad\text{and}\quad B_{N,M} \pto 1-\Phi((M-1)z_{\alpha/2})$$
and hence,
$$ \max\{A_{N,M}, B_{N,M} \} \pto c_M := \max\left\{\Phi(-(M-1)z_{\alpha/2}), 1-\Phi((M- 1)z_{\alpha/2})\right\}.$$
Now, fix $\varepsilon >0$ and choose $M$ large enough, such that $c_M < \frac{\varepsilon}{2}$. Then we have from \eqref{riskambm78}
\begin{align*}
    \limsup_{N\rightarrow \infty}~\p\left(\sup_{\beta_{1, N} \in \Theta(\beta_0, \delta_N)} \P_{\beta_{1, N}}(\phi^\alpha_N = 0 | A_{G_N}) > \varepsilon\right) & \le \limsup_{N\rightarrow \infty} ~\p\left( \max\{A_{N,M}, B_{N,M}\} > \varepsilon\right)\\
    &\le \limsup_{N\rightarrow \infty}~\p\left(\max\{A_{N,M}, B_{N,M}\} - c_M > \frac{\varepsilon}{2}\right) \nonumber \\ 
    & = 0,
\end{align*}
thereby showing that
$$\sup_{\beta_{1, N} \in \Theta(\beta_0, \delta_N)} \P_{\beta_{1, N}}(\phi^\alpha_N = 0 | A_{G_N}) \pto 0.$$
Combining this with \eqref{eq:TypeI} shows that $\mathcal{R}(\phi_N^\alpha) \pto \alpha$. \hfill $\Box$

\section{Proof of Theorem \ref{th:limexp}}\label{proof_limexp8}

\subsection{Proof of Theorem \ref{th:limexp} (1)}
For $h \in \R$, denote by $Q_{\infty,h}$ the normal probability distribution with mean $\frac{h}{2}\int_{[0,1]^2} W(x,y) \mathrm{d}x \mathrm{d}y$ and variance $\frac{1}{2}\int_{[0,1]^2} W(x,y) \mathrm{d}x \mathrm{d}y$. Then it is straightforward to verify that:
  \begin{align}\label{expinfinite6}
      \frac{\mathrm{d} Q_{\infty,h}}{\mathrm{d} Q_{\infty,h_0}} (x) = e^{\textstyle (h-h_0)x + \frac{1}{4} (h_0^2-h^2)\int_{[0,1]^2} W(x,y) \mathrm{d}x \mathrm{d}y } . 
  \end{align}
  By Proposition \ref{limexp1} (1), the RHS of \eqref{expinfinite6} evaluated at $x=Z \sim Q_{\infty,h_0}$, is precisely the weak limit of $\mathrm{d} \P_{\beta_0+h\sqrt{\theta_N}}/\mathrm{d} \P_{\beta_0+ h_0\sqrt{\theta_N}}(\bs)$, for $\bs \sim \P_{\beta_0+ h_0\sqrt{\theta_N}}$. Moreover, this weak convergence happens jointly for every finite collection of the parameter $h$. Hence, the experimental limit in this case is $\{Q_{\infty,h}\}_{h\in \mathbb{R}}$ (recall Definition \ref{definition:experiment}). Now, it is straightforward to verify that if $Q_{\infty,h}'$ denotes the normal probability distribution with mean $h$ and variance $\frac{2}{\int_{[0,1]^2} W(x,y) \mathrm{d}x \mathrm{d}y}$, then 
  $$\frac{\mathrm{d} Q_{\infty,h}}{\mathrm{d} Q_{\infty,h_0}} (Z) \stackrel{D}{=} \frac{\mathrm{d} Q_{\infty,h}'}{\mathrm{d} Q_{\infty,h_0}'} (Z') \sim \cN\left(-\frac{\sigma^2}{2},\sigma^2\right)$$ where $Z\sim Q_{\infty,h_0}$, $Z'\sim Q_{\infty, h_0}'$, and $$\sigma^2 := \frac{(h-h_0)^2}{2}\int_{[0,1]^2} W(x,y) \mathrm{d}x \mathrm{d}y .$$ This completes the proof of part (1) of Theorem \ref{th:limexp}. \hfill $\Box$

\subsection{Proof of Theorem \ref{th:limexp} (2)}

In this case, by Proposition \ref{limexp1} (2) we know that for any finite set $I \subset (0,\frac{1}{\|W\|_{\mathrm{op}}} )$, 
    \begin{align*}
        \left(\frac{\mathrm{d} \P_{\beta}}{\mathrm{d} \P_{\beta_0}} (\bs)\right)_{\beta\in I} \xrightarrow[\bs \sim \P_{\beta_0}]{D} \left( e^{ \textstyle \frac{\beta_0^2-\beta^2}{2}\sigma_W^2 -(\beta_0-\beta) V_{\beta_0} } \prod_{\lambda \in \mathrm{Spec}(W)} e^{ \textstyle \frac{(\beta-\beta_0)\lambda }{2} } \sqrt{\frac{1-\beta\lambda }{1-\beta_0\lambda }}\right)_{\beta\in I} , 
    \end{align*}
   where $V_{\beta_0}$ is as in \eqref{eq:Vbeta} and $\sigma_W^2:=\frac{1}{2\theta} \int_{[0,1]^2} W(x,y)(1-\theta W(x,y))\mathrm{d}x\mathrm{d}y$. The result in Theorem \ref{th:limexp} (2) now follows from the lemma below.

\begin{lemma}\label{frinv}
    Let $f_{V_\beta}$ denote the density of the random variable $V_\beta$ defined in \eqref{eq:Vbeta}, with respect to the Lebesgue measure. Then, for all $\beta_0,\beta\in (0,\frac{1}{\|W\|_{\mathrm{op}}})$ and all $x\in \mathbb{R}$,    
    $$\frac{f_{V_\beta}(x)}{f_{V_{\beta_0}}(x)} = e^{ \textstyle \frac{\beta_0^2-\beta^2}{2}\sigma_W^2 -(\beta_0-\beta)x } \prod_{\lambda \in \mathrm{Spec}(W)} e^{\textstyle \frac{(\beta-\beta_0)\lambda}{2} } \sqrt{\frac{1-\beta\lambda}{1-\beta_0\lambda }}.$$
\end{lemma}

\begin{proof} From \eqref{eq:UV} and Lemma \ref{lm:HNmgf} the characteristic function of $V_\beta$ can be computed as follows: 
\begin{align*}
  \E e^{it V_\beta} &= e^{\textstyle -\frac{ (t^2 -2it\beta) \sigma_W^2 }{2} } \prod_{\lambda \in \mathrm{Spec}(W)} e^{\textstyle -\frac{it\lambda}{2}} \sqrt{\frac{1-\beta\lambda}{1-(\beta+it)\lambda }}\\ 
  &= \left[e^{\textstyle -\frac{\beta^2\sigma_W^2}{2}}\prod_{\lambda \in \mathrm{Spec}(W)} e^{\textstyle \frac{\beta \lambda}{2}} \sqrt{1-\beta\lambda}\right] g(\beta+it),  
\end{align*}
where $$g(y) := e^{\textstyle \frac{y^2\sigma_W^2}{2}} \prod_{ \lambda \in \mathrm{Spec}(W) }  \frac{e^{\textstyle -\frac{y\lambda}{2}}}{\sqrt{1-y\lambda}}.$$
Therefore, by the Fourier inversion formula,
\begin{align}\label{eq:fbeta}
    \frac{f_{V_\beta}(x)}{f_{V_{\beta_0}}(x)} &= \frac{\int_{-\infty}^\infty e^{ -itx} \E e^{ it Y_\beta} \mathrm{d} t}{\int_{-\infty}^\infty e^{ -itx} \E e^{ it Y_{\beta_0}} \mathrm{d} t} \nonumber \\ 
    &= \left[e^{\textstyle -\frac{(\beta^2-\beta_0^2)\sigma_W^2}{2}}\prod_{\lambda \in \mathrm{Spec}(W)} e^{\textstyle \frac{(\beta-\beta_0) \lambda}{2}} \sqrt{\frac{1-\beta\lambda}{1-\beta_0\lambda}}\right]\frac{\int_{-\infty}^\infty e^{ -itx} g(\beta+it) \mathrm{d} t}{\int_{-\infty}^\infty e^{ -itx} g(\beta_0+it) \mathrm{d} t} \nonumber \\
    &= \left[e^{\textstyle (\beta-\beta_0)x -\frac{(\beta^2-\beta_0^2)\sigma_W^2}{2}}\prod_{\lambda \in \mathrm{Spec}(W)} e^{\textstyle \frac{(\beta-\beta_0) \lambda}{2}} \sqrt{\frac{1-\beta\lambda}{1-\beta_0\lambda}}\right]\frac{\int_{-\infty}^\infty h(\beta+it) \mathrm{d} t}{\int_{-\infty}^\infty h(\beta_0+it) \mathrm{d} t} , 
\end{align}
where $h(y) := e^{-yx}g(y)$. Note that $g(z) \ll 1$, as $|\mathrm{Im}(z)| \rightarrow \infty$. Hence, 
$$\lim_{\gamma \rightarrow \infty} \int_{\beta_0+i \gamma}^{\beta+i \gamma} h(z) \mathrm{d}z = \lim_{\gamma \rightarrow \infty} \int_{\beta-i \gamma}^{\beta_0-i \gamma} h(z) \mathrm{d}z = 0.$$ Further, since $h$ is analytic on the region bounded by the closed rectangle formed by the vertices $\beta_0\pm i \gamma$ and $\beta\pm i \gamma$, its contour integral along the edges of the rectangle is $0$ (by Cauchy's integral theorem). This shows that:
$$\lim_{\gamma \rightarrow\infty} \int_{\beta_0-i \gamma}^{\beta_0+i \gamma} h(z) \mathrm{d} z + \lim_{\gamma \rightarrow\infty} \int_{\beta +i \gamma}^{\beta-i \gamma} h(z) \mathrm{d} z = 0\quad  \Rightarrow \quad \int_{-\infty}^\infty h(\beta+it) \mathrm{d}t = \int_{-\infty}^\infty h(\beta_0+it) \mathrm{d}t.$$ 
Applying this to \eqref{eq:fbeta} completes the proof of Lemma \ref{frinv}. 
\end{proof}

\section{Additional Technical Results}

In this section we collect the proofs of various technical lemmas that are used in the proofs of the main results.

\subsection{Asymptotic Expansions for Expectations and Covariances}
In this section, we obtain the asymptotic expansions for the expectations and covariance of $T(\bs)$ and $T(\bt)$ (recall \eqref{Tstatdef}) for two spin configurations $\bs$ and $\bt \in \{-1,+1\}^N$, under the assumption $\theta_N \gtrsim N^{-\frac{2}{3}}$. Throughout this section $\E$ will denote the expectation with respect to the randomness of the adjacency matrix $A_{G_N}$. 

\begin{lemma}\label{lem:expvariance}
 Suppose $\theta_N  \gtrsim N^{-\frac{2}{3}}$ and let $Q_N$ and $R_N$ be as defined in \eqref{eq:QN} and \eqref{eq:RN}, respectively. Then for all $\bs \in \{-1,+1\}^N$,     \begin{align}\label{eq:expectationsigma}
        \E T\left(\bs\right) = e^{\textstyle  \mathcal{E}_N(\beta, W)+\frac{\beta}{N}\sum_{1 \leq i < j \leq N}W\!\left(\frac{i}{N},\frac{j}{N}\right)\sigma_i\sigma_j
+o\!\left(\frac1{N\theta_N}\right)+O\!\left(Q_N(\boldsymbol{\sigma})\right)} , 
    \end{align}
   where 
   \begin{align}\label{eq:JN}
   & \mathcal{E}_N(\beta, W) \nonumber \\ 
   & := -\frac{\beta^2}{4N^2\theta_N}\sum_{i=1}^{N}W\left(\frac{i}{N},\frac{i}{N}\right) - \frac{\beta^2}{2N^2}\sum_{1 \leq i < j \leq N}W^2\left(\frac{i}{N},\frac{j}{N}\right)-\frac{\beta^4}{12 N^4\theta_N^3}\sum_{1 \leq i < j \leq N} W\left(\frac{i}{N},\frac{j}{N}\right). 
   \end{align} 
        Furthermore, for all $\bs, \bt \in \{-1,+1\}^N$,     
        \begin{align}\label{eq:variancesigma}
        & \E T\left(\bs\right) T\left(\bt\right) \nonumber \\ 
        & = e^{\textstyle 2 \mathcal{E}_N(\beta, W) + \frac{\beta}{N}\sum_{1 \leq i < j \leq N}W\left(\frac{i}{N},\frac{j}{N}\right)(\sigma_i \sigma_j + \tau_i \tau_j)    +      
        R_N(\bs, \bt) +  o\left(\frac 1{N\theta_N}\right) + O\left(Q_N^{(2)}(\bs, \bt) \right) } . 
    \end{align} 
    where $ Q_N^{(2)}(\bs, \bt) := Q_N(\bs) + Q_N(\bt) +  \frac{1}{ N \theta_N } Q_N(\bs\bt)$. 
\end{lemma}

\begin{proof} Recalling \eqref{Tstatdef}, note that 
    \begin{align}\label{tsgexp}
        T\left(\bs\right)= e^{\textstyle  -\gamma^2 \sum_{i=1}^N A_{G_N}\left(i,i\right) + \sum_{1 \leq i < j \leq N} ( 2\gamma \sigma_i \sigma_j - 2\gamma^2 ) A_{G_N}\left(i,j\right) } .
    \end{align} 
    For $1 \leq i, j \leq N$, define 
    \begin{align*}
        f^{(1)}_i & := \log\left(1-\theta_NW\left(\frac{i}{N},\frac{i}{N}\right)+\theta_N W\left(\frac{i}{N},\frac{i}{N}\right) e^{ -\gamma^2 }\right) , \nonumber \\ 
    f^{(2)}_{ij}\left(x\right)& :=\log\left(1-\theta_N W\left(\frac{i}{N},\frac{j}{N}\right)+\theta_N W\left(\frac{i}{N},\frac{j}{N}\right)e^{ 2 \gamma x-2\gamma^2}\right) ,
    \end{align*} 
    for $x \in \R$. 
    In these notations, it follows from \eqref{tsgexp}, 
    \begin{align}\label{wr81}
        \E T\left(\bs\right)= e^{\textstyle  \sum_{i=1}^N f^{(1)}_i + \sum_{1 \leq i < j \leq N} f^{(2)}_{ij}\left(\sigma_i \sigma_j\right) } .
    \end{align}
    Expanding the exponential and the logarithm in $f_i^1$ and using the fact $N\theta_N \gg 1$ gives, 
    \begin{align}\label{wr92}
        \sum_{i=1}^{N} f^{(1)}_i & = -\sum_{i=1}^{N}\theta_N W\left(\frac{i}{N},\frac{i}{N}\right)\gamma^2 +o\left(\frac{1}{N\theta_N}\right)  \nonumber \\ 
        & = -\frac{\beta^2}{4N^2\theta_N}\sum_{i=1}^{N}W\left(\frac{i}{N},\frac{i}{N}\right)+o\left(\frac{1}{N\theta_N}\right).
    \end{align}
    To analyze the second term in exponent in \eqref{wr81}, note that since $\sigma_i\sigma_j$ can only take the values $\pm 1$, we can write: 
    \begin{align}\label{eq:f2sigma}
        f^{(2)}_{ij}\left(\sigma_i\sigma_j\right)=a_0(i,j)+a_1(i,j) \sigma_i \sigma_j,
    \end{align}
    where $a_0(i,j) := \frac {f^{(2)}_{ij}\left(1\right)+f^{(2)}_{ij}\left(-1\right)}{2}$ and $a_1(i,j):= \frac {f^{(2)}_{ij}\left(1\right)-f^{(2)}_{ij}\left(-1\right)}{2}$. Now, define the function $F_1 : \{ -1, 1\} \times \mathbb{C}^2 \rightarrow \mathbb{C}$
    as:
    \begin{align}\label{eq:F1}
        F_1(x,p, z) := \log \left(1 - p + pe^{xz- \frac{z^2}{2}}\right).
\end{align}
  Then,  $$f^{(2)}_{ij}\left(\pm 1\right)=F_1\left(\pm 1,\theta_N W\left(\frac{i}{N},\frac{j}{N}\right),2\gamma\right).$$ 
  The following result gives power series expansions for linear combinations of $F_1$. The proof follows by a direct Taylor expansion in the variables $p$ and $z$. The details are omitted. (A similar result appears in \cite[Lemma 2.1]{kabluchko2021fluctuations}.)   
  
\begin{lemma}\label{lm:F1} Let $F_1$ be as defined in \eqref{eq:F1}. For each fixed $x$, the function $F_1(x,p,z)$ is analytic on the set $$\mathcal D = \{(p,z)\in\C^2\colon |p| < 2, |z|<z_0\}. $$
    Further, we have the following power series expansions around the point $(0,0)$:
    \begin{align*}
        \frac {F_1(1, p, z) + F_1(-1, p, z)}2 = & -\frac 12 p^2z^2 - \frac 1{12} pz^4 +  \mathcal{O}(p^2z^4) +  \mathcal{O}(pz^6), \\
        \frac {F_1(1, p, z) - F_1(-1, p, z)}2 = & pz +  \mathcal{O}(pz^3).
    \end{align*}
\end{lemma}

Lemma \ref{lm:F1} gives, 
\begin{align*} 
a_0(i,j) &:= \frac {f^{(2)}_{ij}\left(1\right)+f^{(2)}_{ij}\left(-1\right)}{2} \nonumber \\ 
&=  -2\gamma^2 \theta_N^2 W^2\left(\frac{i}{N},\frac{j}{N}\right)-\frac 4{3} \theta_N W\left(\frac{i}{N},\frac{j}{N}\right)\gamma^4 +O\left(\theta_N^2 \gamma^4\right)+O\left(\gamma^6 \theta_N \right).
\end{align*} 
    The assumption $\theta_N \gtrsim N^{-\frac{2}{3}} $ now implies that:
    $$
        \sum_{1 \leq i < j \leq N} a_0(i,j)= -\frac{\beta^2}{2N^2}\sum_{1 \leq i < j \leq N}W^2\left(\frac{i}{N},\frac{j}{N}\right)-\frac{\beta^4}{12 N^4\theta_N^3}\sum_{1 \leq i < j \leq N} W\left(\frac{i}{N},\frac{j}{N}\right)+o\left(\frac 1{N\theta_N}\right).
    $$
Similarly, 
\begin{align*}
        a_1(i,j) & =  \frac {f^{(2)}_{ij}\left(1\right)-f^{(2)}_{ij}\left(-1\right)}{2} \nonumber \\ 
        & =  2\gamma \theta_N W\left(\frac{i}{N},\frac{j}{N}\right) + O\left(\theta_N W\left(\frac{i}{N},\frac{j}{N}\right)\gamma^3\right).
\end{align*} 
   Hence, 
    $$
        \sum_{1 \leq i < j \leq N}a_1(i,j)\sigma_i \sigma_j=\frac{\beta}{N}\sum_{1 \leq i < j \leq N}W\left(\frac{i}{N},\frac{j}{N}\right)\sigma_i\sigma_j+O\left(\frac{1}{N^3 \theta_N^2}\left|\sum_{1 \leq i < j \leq N}W\left(\frac{i}{N},\frac{j}{N}\right)\sigma_i \sigma_j\right| \right).
    $$
    Therefore, recalling \eqref{eq:f2sigma}, 
    \begin{align}\label{wr93}
        & \sum_{1 \leq i < j \leq N} f_{ij}^{(2)}(\sigma_i \sigma_j) \nonumber \\ 
        &=  -\frac{\beta^2}{2N^2}\sum_{1 \leq i < j \leq N}W^2\left(\frac{i}{N},\frac{j}{N}\right)-\frac{\beta^4}{12 N^4\theta_N^3}\sum_{1 \leq i < j \leq N} W\left(\frac{i}{N},\frac{j}{N}\right)+o\left(\frac 1{N\theta_N}\right)\nonumber\\
        & \hspace{0.35in} + \frac{\beta}{N}\sum_{1 \leq i < j \leq N}W\left(\frac{i}{N},\frac{j}{N}\right)\sigma_i\sigma_j+O\left(\frac{1}{N^3 \theta_N^2}\left|\sum_{1 \leq i < j \leq N}W\left(\frac{i}{N},\frac{j}{N}\right)\sigma_i \sigma_j\right| \right).
    \end{align}
    Combining \eqref{wr81}, \eqref{wr92}, and \eqref{wr93}, the result in \eqref{eq:expectationsigma} follows.

Next, we prove \eqref{eq:variancesigma}. To begin with, note that 
    \begin{align}\label{tsgexp41}
        T\left(\bs\right) T(\bt)= e^{\textstyle -2\gamma^2 \sum_{i=1}^N A_{G_N}\left(i,i\right) + \sum_{1 \leq i < j \leq N} \left( 2\gamma \sigma_i \sigma_j + 2\gamma \tau_i \tau_j - 4\gamma^2 \right) A_{G_N}\left(i,j\right) } . 
    \end{align}
    Next, define 
    \begin{align*}
        f^{(3)}_i & := \log\left(1-\theta_N W\left(\frac{i}{N},\frac{i}{N}\right)+\theta_N W\left(\frac{i}{N},\frac{i}{N}\right) e^{ -2\gamma^2 }\right) , \nonumber \\   
        f^{(4)}_{ij}\left(x\right) & := \log\left(1-\theta_N W\left(\frac{i}{N},\frac{j}{N}\right)+\theta_N W\left(\frac{i}{N},\frac{j}{N}\right)e^{ 2\gamma x-4\gamma^2 }\right) , 
    \end{align*} 
    for $x \in \R$. In these notations, we have the following from \eqref{tsgexp41}, 
    \begin{align}\label{eq:variancesigma1} 
        \E T\left(\bs\right) T\left(\bt\right)= e^{\textstyle \sum_i f^{(3)}_i + \sum_{1 \leq i < j \leq N} f^{(4)}_{ij}\left(\sigma_i \sigma_j+\tau_i\tau_j\right) } . 
    \end{align}
    An exactly similar computation as in the proof of \eqref{wr92} gives, 
    \begin{align}\label{eq:variancesigma2} 
        \sum_i f^{(3)}_i= -\frac{\beta^2}{2N^2\theta_N}\sum_{i=1}^{N}W\left(\frac{i}{N},\frac{i}{N}\right)+o\left(\frac 1{N\theta_N}\right).
    \end{align}
    Next, since $\sigma_i \sigma_j+\tau_i\tau_j \in \{-2,0,2\}$ for all possible choices of $\bs$ and $\bt$, we can write:
    \begin{align}\label{eq:variancesigma3} 
        f^{(4)}_{ij}\left(\sigma_i\sigma_j +\tau_i\tau_j\right)= b_0(i,j)+b_1(i,j)\sigma_i\sigma_j  +b_2(i,j) \tau_i\tau_j +b_{12}(i,j) \sigma_i\sigma_j\tau_i\tau_j , 
    \end{align} 
    with coefficients $b_0, b_1,b_2,b_{12}(i,j)$ given by:
    $$
        b_0(i,j)= \frac{f^{(4)}_{ij}\left(2\right)+f^{(4)}_{ij}\left(-2\right)+2f^{(4)}_{ij}\left(0\right)}4, \quad b_1(i,j)=b_2(i,j)=\frac{f^{(4)}_{ij}\left(2\right)-f^{(4)}_{ij}\left(-2\right)}4,
    $$
    and
    $$
        b_{12}(i,j)= \frac{f^{(4)}_{ij}\left(2\right)+f^{(4)}_{ij}\left(-2\right)-2f^{(4)}_{ij}\left(0\right)}4.
    $$
    Now, define the function $F_2 : \{ -2, 0, 2\} \times \mathbb{C}^2 \rightarrow \mathbb{C}$
    as:
    \begin{align}\label{eq:F2}
        F_2(x,p, z) := \log \left(1 - p + pe^{ xz- z^2 }\right).
\end{align}
Then, 
$$f^{(4)}_{ij}\left(\pm 2\right)=F_2\left(\pm 2, \theta_N W\left(\frac{i}{N},\frac{j}{N}\right), 2\gamma\right)\quad\text{and}\quad f^{(4)}_{ij}\left(0\right)= F_2\left(0,\theta_N W\left(\frac{i}{N},\frac{j}{N}\right), 2\gamma\right).$$ 
Now, similar to Lemma \ref{lm:F1}, we have the following power series expansions for linear combinations of $F_2$.  

\begin{lemma}\label{lm:F2} 
  For each fixed $x$, the function $F_m(x,p,z)$, as defined in \eqref{eq:F2}, is analytic on the set
    $$
        \mathcal D = \{(p,z)\in\C^2\colon |p| < 2, |z|<z_0\} . 
    $$ 
    Further, we have the following power series expansions around the point $(0,0)$:
      \begin{align*}
        \frac {F_2(2, p, z) + F_2(-2, p, z)+2F_2(0, p, z)}4
         & =-p^2z^2 - \frac 1{6} pz^4 +  \mathcal{O}(p^2z^4) +  \mathcal{O}(pz^6),    \\
        \frac {F_2(2, p, z) + F_2(-2, p, z)-2F_2(0, p, z)}4
         & = p(1-p)z^2 - \frac 2{3} pz^4 +  \mathcal{O}(p^2z^4) +  \mathcal{O}(pz^6), \\
        \frac {F_2(2, p, z) - F_2(-2, p, z)}4
         & =
        pz + O(pz^3).
    \end{align*}
\end{lemma}

Lemma \ref{lm:F2} now gives, 
    $$
        \sum_{1 \leq i < j \leq N} b_0(i,j)= -\frac{\beta^2}{N^2}\sum_{i<j}W^2\left(\frac{i}{N},\frac{j}{N}\right)-\frac{\beta^4}{6 N^4\theta_N^3}\sum_{i<j}W\left(\frac{i}{N},\frac{j}{N}\right)+o\left(\frac 1{N\theta_N}\right),
    $$
  Similarly,   
    \begin{multline*}
        \sum_{1 \leq i < j \leq N} b_{12}(i,j)\sigma_i \sigma_j \tau_i\tau_j = \frac{\beta^2}{\theta_N N^2} \sum_{1 \leq i < j \leq N}W\left(\frac{i}{N},\frac{j}{N}\right)\left(1-\theta_N W\left(\frac{i}{N},\frac{j}{N}\right)\right)\sigma_i\sigma_j\tau_i\tau_j+ o\left(\frac 1{N\theta_N}\right)
        \\ + O\left(\frac{1}{\theta_N^3 N^4}\sum_{1 \leq i < j \leq N}W\left(\frac{i}{N},\frac{j}{N}\right)\sigma_i\sigma_j\tau_i\tau_j\right),
    \end{multline*}
    and
    $$
        b_1(i,j)= b_2(i,j)=\frac{\beta}{N}W\left(\frac{i}{N},\frac{j}{N}\right)+ O\left(\frac{1}{\theta_N^2 N^3} W\left(\frac{i}{N},\frac{j}{N}\right)\right).
    $$
    Combining the above with \eqref{eq:variancesigma1},  \eqref{eq:variancesigma2}, and  \eqref{eq:variancesigma3} gives the result in \eqref{eq:variancesigma}. 
\end{proof}

\subsection{Joint Convergence of Rademacher Quadratic Forms }\label{sec:jointconv99}

In this section we prove a result about the  joint weak convergence of quadratic forms in Rademacher random variables.

\begin{ppn}
    \label{quadJointAsym}
    Let $W, W'$ be Riemann integrable graphons. Suppose $\{\sigma_i\}_{1 \leq i \leq N}$ and $\{\tau_i\}_{1 \leq i \leq N}$ are i.i.d. Rademacher random variables with mean $0$. Then,
    \begin{align*} 
        \begin{pmatrix} 
          \displaystyle 
            \frac{1}{2 N} \sum_{1 \leq i \neq j \leq N}W\left(\frac{i}{N}, \frac{j}{N}\right) \left(\sigma_{i} \sigma_{j} + \tau_{i} \tau_{j}\right) \\
              \displaystyle 
            \frac{1}{2 N} \sum_{1 \leq i \neq j \leq N}W'\left(\frac{i}{N}, \frac{j}{N}\right) \sigma_{i}\tau_i \sigma_{j}\tau_j.
        \end{pmatrix}
        \xrightarrow{D}  
        \begin{pmatrix} 
          \displaystyle  
           \frac{1}{2} \sum_{\lambda \in \mathrm{Spec}(W)} \lambda  \left(X_\lambda^2-1\right) + \frac{1}{2} \sum_{\lambda \in \mathrm{Spec}(W)} \lambda  \left(Y_\lambda^2-1\right) \\
               \displaystyle 
            \frac{1}{2} \sum_{\lambda \in \mathrm{Spec}(W')} \lambda \left(Z_\lambda^2-1\right)
        \end{pmatrix} .
    \end{align*}
    where $\{X_i\}_{1 \leq i \leq N}, \{Y_i\}_{1 \leq i \leq N}, \{Z_i\}_{1 \leq i \leq N}$ are i.i.d. collections of $\mathcal N(0, 1)$ random variables.
\end{ppn}

\begin{proof}
    Fix $a, b \in \R$ and define, 
    \begin{align}\label{eq:XNab}
        X_{N, a, b}
         & :=\frac{a}{N}\sum_{i < j}W\left(\frac{i}{N}, \frac{j}{N}\right) \left(\sigma_{i} \sigma_{j} + \tau_{i} \tau_{j}\right) + \frac{b}{N}\sum_{i< j} W^\prime\left(\frac{i}{N},\frac{j}{N}\right) \sigma_{i}\tau_i \sigma_{j}\tau_j . 
    \end{align}
    For notational convenience, define 
    \begin{align*}
        U_N(i, j) = a W\left(\frac{i}{N}, \frac{j}{N}\right) \left(\sigma_{i} \sigma_{j} + \tau_{i} \tau_{j}\right) + b W^\prime\left(\frac{i}{N}, \frac{j}{N}\right) \sigma_{i}\tau_i \sigma_{j}\tau_j.
    \end{align*}
    Then from the proof of \cite[Corollary 2.2 ]{o1993asymptotic} it follows that $X_{N, a, b}$ converges in distribution to a limit with $r$-th cumulant $\frac{1}{2} (r-1)! w_r$, for $r \geq 2$, where 
    \begin{align}\label{wkexpression8}
        w_r := \lim_{N\ra\infty}\frac{1}{N^r}\sum_{1 \leq i_1 \ne i_2 \ne \ldots i_r \leq N } \E\left[ U_N(i_1, i_2)U_N(i_2, i_3)\cdots U_N(i_r, i_1)\right] . 
    \end{align}  
    To identify the asymptotic distribution, note that for each $1 \leq i_1 \ne i_2 \ne \ldots i_r \leq N$,
   \begin{align}\label{eq:circle_moment_expand}
        & \E\left[U_N(i_1, i_2)U_N(i_2, i_3)\cdots U_N(i_r, i_1)\right] \nonumber \\ 
        & = \sum_{(J_1,\ldots,J_r)\in \{0,1\}^r} \prod_{s = 1}^{r} a^{J_s}b^{1-J_s}W\left(\frac{i_s}{N}, \frac{i_{s+1}}{N}\right)^{J_s}W^\prime\left(\frac{i_s}{N}, \frac{i_{s+1}}{N}\right)^{1-J_s}\nonumber\\
        & \hspace{1.25in} \E\left[\prod_{s = 1}^{r}(\sigma_{i_s}\sigma_{i_{s+1}} + \tau_{i_s}\tau_{i_{s+1}})^{J_s}(\sigma_{i_s}\sigma_{i_{s+1}}\tau_{i_s}\tau_{i_{s+1}})^{1-J_s}\right] , 
    \end{align} 
    where $i_{r+1} = i_1$. To compute the inner expectation, first suppose that $(J_1,\ldots,J_r)$ is such that $J_{+} := \sum_{s=1}^{r}J_s \not\in \{0, r\}$, that is, not all entries of this vector are the same. Then, 
    \begin{align*}
        & \E
        \left[\prod_{s = 1}^{r}(\sigma_{i_s}\sigma_{i_{s+1}} + \tau_{i_s}\tau_{i_{s+1}})^{J_s}(\sigma_{i_s}\sigma_{i_{s+1}}\tau_{i_s}\tau_{i_{s+1}})^{1-J_s}\right]\\
        & = \frac{1}{2^{r - J_{+} }}\E\left[\prod_{s = 1}^{r}\left(\sigma_{i_s}\sigma_{i_{s+1}}(\tau_{i_s}\tau_{i_{s+1}})^{1-J_s} + (\sigma_{i_s}\sigma_{i_{s+1}})^{1-J_s}\tau_{i_s}\tau_{i_{s+1}}\right)\right]\\
        & = \frac{1}{2^{r - J_{+} }} \sum_{(\tilde J_1,\ldots, \tilde J_r)\in \{0,1\}^r}\E\left[\prod_{s = 1}^{r}(\sigma_{i_s}\sigma_{i_{s+1}})^{\tilde{J}_s}(\tau_{i_s}\tau_{i_{s+1}})^{\tilde{J}_s(1-J_s)} (\sigma_{i_s}\sigma_{i_{s+1}})^{(1-\tilde{J}_s)(1-J_s)}(\tau_{i_s}\tau_{i_{s+1}})^{1-\tilde{J}_s}\right]\\
        & = \frac{1}{2^{r - J_{+} }} \sum_{(\tilde J_1,\ldots, \tilde J_r)\in \{0,1\}^r} \E \left[\prod_{s = 1}^{r}\sigma_{i_{s+1}}^{\tilde{J}_s + (1-\tilde{J}_s)(1-J_s) + \tilde{J}_{s+1} + (1-\tilde{J}_{s+1})(1-\tilde{J}_{s+1})} \tau_{i_{s+1}}^{\tilde{J}_s(1-J_s) + (1-\tilde{J}_s) + \tilde{J}_{s+1}(1-\tilde{J}_{s+1}) + (1-\tilde{J}_{s+1})} \right]\\
        & = \frac{1}{2^{r - J_{+} }} \sum_{(\tilde J_1,\ldots, \tilde J_r)\in \{0,1\}^r} \prod_{s = 1}^{r}\E\left[\sigma_{i_{s+1}}^{2-J_s(1-\tilde{J}_s) - \tilde{J}_{s+1}(1-\tilde{J}_{s+1})}\tau_{i_{s+1}}^{2-J_s\tilde{J}_s - \tilde{J}_{s+1}\tilde{J}_{s+1}}\right] .
    \end{align*} 
    Now, suppose $J_1 = 0$. Then since $\sum_{s=1}^{r} J_s \neq 0$, there exists $1\leq s'<r$ such that $J_{s'} = 0$ and $J_{s'+1} = 1$, and 
    \begin{align*}
        \E\left[\sigma_{i_{s'+1}}^{2-J_{s'} (1-J_{s'} ) - J_{s'+1}(1-J_{s'+1})}\tau_{i_{s'+1}}^{2-J_{s'} J_{s'}  - J_{s'+1}J_{s'+1}}\right] = \E\left[\tau_{i_{s'+1}}\right] = 0.
    \end{align*}
    Similarly, if $J_{1} = 1$, then since $\sum_{s=1}^{r} J_s \neq r$, there exists $1\leq s'<r$, such that 
    \begin{align*}
        \E\left[\sigma_{i_{s'+1}}^{2-J_{s'} (1-J_{s'} ) - J_{s'+1}(1-J_{s'+1})}\tau_{i_{s'+1}}^{2-J_{s'} J_{s'}  - J_{s'+1}J_{s'+1}}\right] = 0.
    \end{align*}
    Hence, combining the above gives, 
    \begin{align*}
        \E
         & \left[\prod_{s = 1}^{r}(\sigma_{i_s}\sigma_{i_{s+1}} + \tau_{i_s}\tau_{i_{s+1}})^{J_s}(\sigma_{i_s}\sigma_{i_{s+1}}\tau_{i_s}\tau_{i_{s+1}})^{1-J_s}\right] = 0 , 
    \end{align*}
    whenever $\sum_{s=1}^{r} J_r\notin \{0, r\}$. It now follows from \eqref{eq:circle_moment_expand}, that:
    \begin{align*}
        &\E
        \left[U_N(i_1, i_2)U_N(i_2, i_3)\cdots U_N(i_r, i_1)\right]\\
        & = a ^r \prod_{s = 1}^{r}W\left(\frac{i_s}{N}, \frac{i_{s+1}}{N}\right)\E\left[\prod_{s = 1}^{r}\left(\sigma_{i_s}\sigma_{i_{s+1}} + \tau_{i_s}\tau_{i_{s+1}}\right)\right] + b^r\prod_{s = 1}^{r}W^\prime\left(\frac{i_s}{N}, \frac{i_{s+1}}{N}\right)\E\left[\prod_{s = 1}^{r}\sigma_{i_s}\sigma_{i_{s+1}}\tau_{i_s}\tau_{i_{s+1}}\right] .  
    \end{align*}
    \normalsize
    Now, it is straightforward to verify that 
    \begin{align*}
        \E\left[\prod_{s = 1}^{r}\left(\sigma_{i_s}\sigma_{i_{s+1}} + \tau_{i_s}\tau_{i_{s+1}}\right)\right]= 2 \quad\text{and}\quad \E\left[\prod_{s = 1}^{r}\sigma_{i_s}\sigma_{i_{s+1}}\tau_{i_s}\tau_{i_{s+1}}\right] = 1.
    \end{align*}
Thus, 
    \begin{align*}
        \E
         & \left[U_N(i_1, i_2)U_N(i_2, i_3)\cdots U_N(i_r, i_1)\right] = 2a^r\prod_{s = 1}^{r}W\left(\frac{i_s}{N}, \frac{i_{s+1}}{N}\right) + b^r \prod_{s = 1}^{r}W^\prime\left(\frac{i_s}{N}, \frac{i_{s+1}}{N}\right) 
    \end{align*}
    Taking limits on both sides and recalling \eqref{wkexpression8} and \eqref{eq:cycle} now gives, 
    \begin{align}\label{eq:wrab}
        w_r = 2a^r t\left(C_r, W\right) + b^r t\left(C_r, W^\prime\right) . 
    \end{align}
Note that for any graphon $W$ and $r \geq 3$, $t(C_r, W)$ is as defined in \eqref{eq:cycle}. Further, for $r=2$, we define $t(C_2, W) := \int_{[0, 1]^2}  W(x, y)^2 \mathrm d x \mathrm dy = \|W\|_2^2$.
 
    Now, for $a, b \in \R$, define 
    \begin{align}\label{eq:Xab}
    X_{a, b}:=  a \left(\frac{1}{2} \sum_{\lambda \in \mathrm{Spec}(W)}  \lambda  \left(X_\lambda^2-1\right) + \frac{1}{2} \sum_{\lambda \in \mathrm{Spec}(W)}  \lambda  \left(Y_\lambda^2-1\right)\right) + \frac{b}{2} \sum_{\lambda \in \mathrm{Spec}(W')}  \lambda  \left(Z_\lambda^2-1\right). 
    \end{align}
  If we show that the $r$-th cumulant of $X_{a,b}$ equals $\frac{1}{2}(r-1)! w_r$, for $r \ge 2$, then this would imply that $X_{N,a,b} \dto X_{a,b}$ (recall \eqref{eq:XNab}) and, hence, the result in Lemma~\ref{quadJointAsym}, by the Cram\'{e}r-Wold device. To this end, define 
      $$V := \frac{1}{2}\sum_{\lambda \in \mathrm{Spec}(W)} \lambda  (X_\lambda^2-1),$$ where $\{X_i\}_{i \geq 1}$ are i.i.d. $\mathcal N(0, 1)$ random variables. Note that $V$ is a well-defined real-valued random variable by \eqref{eq:eigenvalue_l2_sum} and Kolmogorov's convergence criterion (see Theorem 5.2 in Chapter 6 from \cite{gut2006probability}). Further, from \cite[Proposition 7.1]{bhattacharya2017universal}, 
    \begin{align}
        \E\left[e^{t V} \right]=\prod_{ \lambda \in \mathrm{Spec}(W) } \frac{ e^{-\frac{1}{2}\lambda t} }{\sqrt{1-\lambda t}} , \nonumber 
    \end{align} 
 for $|t|\leq \frac{1}{4\|W\|_{2}}$. Note that: 
    \begin{align}\label{eq:mgfseries}
            \log \E[e^{t V}] & =- \sum_{\lambda\in \mathrm{Spec}(W)}\frac{t}{2}\lambda +\frac{1}{2}\sum_{\ell=1}^{\infty} \frac{\left(\lambda  t\right)^{\ell}}{\ell} =\frac{1}{2} \sum_{\lambda \in \mathrm{Spec}(W)}  \sum_{\ell=2}^{\infty} \frac{\left(\lambda  t\right)^{\ell}}{\ell}. 
            \end{align}
    The above series is absolutely summable by Lemma \ref{lemma:eigenvalue_summability} (c) for all $|t|<\frac{1}{4\|W\|_{2}}$. Hence, interchanging the order of the summations  in \eqref{eq:mgfseries} and recalling \eqref{eq:cycle} gives, 
    \begin{align*}
        \log \E[e^{t V}] =\frac{1}{2} \sum_{\ell=2}^{\infty} \sum_{\lambda \in \mathrm{Spec}(W)}  \frac{\left(\lambda  t\right)^{\ell}}{\ell}   & =\sum_{\ell=2}^{\infty} \frac{t\left(C_{\ell}, W\right)}{2 \ell} t^{\ell} ,  
    \end{align*} 
  for all $|t|< \frac{1}{4\|W\|_{2}}$. Recalling \eqref{eq:Xab}, this implies, 
    $$\log \E\left[e^{tX_{a ,b}}\right] = \sum_{\ell=2}^{\infty} \frac{t\left(C_{\ell}, W\right)}{\ell} (a  t)^{\ell} + \sum_{\ell=2}^\infty  \frac{t\left(C_{\ell}, W'\right)}{2 \ell} (b t)^{\ell}, $$
     for all $|t|< \frac{1}{4\|W\|_{2}}$. Hence, the cumulants of $X_{a ,b}$ are given by:
    $$\eta_r := (r-1)! \left(a^r t(C_r, W) + \frac{b^r}{2} t(C_r,W')\right) = \frac{(r-1)! w_r}{2}, $$
    where the last step uses \eqref{eq:wrab}. This completes the proof of Proposition \ref{quadJointAsym}. 
\end{proof}

\subsection{Finiteness of Moment Generating Function}

In this section we show that the moment generating function (MGF) of quadratic forms in Rademacher variables (as considered in Appendix \ref{sec:jointconv99}) is finite in the high-temperature regime.

\begin{lemma}
    \label{expUniInt} 
    Suppose $\{\sigma_i\}_{1 \leq i \leq N}$ are i.i.d. Rademacher random variables with mean $0$. Let $0<\beta <\frac{1}{\|W\|_{\mathrm{op}}} $. Then there exists $\delta >0$, such that:
    \begin{align*}
        \limsup_{N \to \infty} \E \left[ e^{\textstyle \frac{\beta+\delta}{2N}\sum_{1 \leq i\neq j \leq N}W\left(\frac{i}{N},\frac{j}{N}\right)\sigma_i \sigma_j} \right] < \infty.
    \end{align*}
    Consequently, $e^{ \textstyle \frac{\beta}{2 N} \sum_{1 \leq i \neq j \leq N}W\left(\frac{i}{N}, \frac{j}{N}\right) \sigma_{i} \sigma_j }$ is uniformly integrable. 
\end{lemma}

\begin{proof}  
    Choose $\delta>0$ such that $\gamma :=\beta+\delta<\frac{1}{\|W\|_{\mathrm{op}}}$. Now, define the matrix $\bm{W}_N$ as follows:
    \begin{align*}
        \bm{W}_N(i,j)=\begin{cases}
            W\left(\frac{i}{N}, \frac{j}{N}\right) & \text{ if } i\neq j \\
            0                                      & \text{ otherwise}.
        \end{cases}
    \end{align*}
    Let $\left\{\lambda_i(\bm{W}_N)\right\}_{i=1}^N$ denote the eigenvalues of $\bm{W}_N$. 
    Since the entries of $\bm{W}_N$ are non-negative, from the proof of \cite[Lemma 7.1]{bhattacharya2018inference} it follows that 
    $$
        \E\left[ e^{ \frac{\beta+\delta}{2N} \sum_{1 \leq i \neq j \leq N}W\left(\frac{i}{N}, \frac{j}{N}\right) \sigma_{i} \sigma_{j} } \right] \leq \E\left[ e^{\frac{\gamma}{2N}  \bm Z^{\top} \bm{W}_N \bm Z } \right], 
    $$
    where $\bm Z \sim \mathcal{N}(\bm 0, \bm I_N)$. Consider the spectral decomposition $\bm{W}_N=\bm U \bm \Lambda \bm U^{\top}$, where $\bm U$ is an orthogonal matrix and $\bm \Lambda$ is a diagonal matrix with diagonal entries $\left\{\lambda_i(\bm{W}_N)\right\}_{i=1}^N$. Then, 
    \begin{align*}
        \E\left[ e^{\textstyle \frac{\gamma}{2N}  \bm Z^{\top} \bm{W}_N \bm Z } \right] = \E\left[ e^{\textstyle \frac{\gamma}{2N} \sum_{i=1}^{N}\lambda_i(\bm{W}_N)\tilde{Z}^2_i } \right] , 
    \end{align*} 
    where $\tilde{\bm Z} = (\tilde{Z}_{1} \ldots \tilde{Z}_N)^{\top} := \bm U^{\top} \bm Z$.     
    Clearly, $\tilde{\bm Z} \sim \mathcal N( \bm 0, \bm I_N)$. Also, by \cite[Theorem 11.54]{lovasz2012large}, we know that $\left\|\frac{1}{N} \bm{W}_N \right\|_{\mathrm{op}} \to\|W\|_{\mathrm{op}}$.\footnote{For a symmetric matrix $\bm A$, we denote its operator norm by $\|\bm{A}\|_{\mathrm{op}} = \sup_{\| \bm x \|_2 = 1} \| \bm A \bm x \|_2$.} Hence,
    \begin{align*}
        \frac{\gamma}{N} \|\bm{W}_N\|_{\mathrm{op}} \rightarrow \gamma \|W\|_{\mathrm{op}} < 1 \quad \implies \quad  \frac{\gamma}{N} \|\bm{W}_N\|_{\mathrm{op}} < 1-\varepsilon,~\text{for some}~\varepsilon>0~\text{and all large}~N.
    \end{align*}  
    Now, using \cite[Lemma D.1]{bhattacharya2018inference} and the observation that $\sum_i{\lambda_i(\bm{W}_N)} = 0$ gives, 
    \begin{align}\label{logexpan788}
        \log \E\left[ e^{ \textstyle \frac{\gamma}{2N}\sum_{i=1}^{N}\lambda_i(\bm{W}_N)\tilde{Z}_i^2 } \right]=\frac{1}{2}\sum_{i=1}^{N}\left(-\log \left(1-\frac{\gamma \lambda_{i}\left(\bm{W}_N\right)}{N}\right) - \frac{\gamma \lambda_i(\bm{W}_N)}{N}\right) . 
    \end{align}
    Now, it is easy to check that for every $\varepsilon>0$, there exists  $M>0$ (depending on $\varepsilon$) such that $$-\log (1-x)-x \leq M x^2\quad \text{ whenever } ~|x| < 1-\varepsilon.$$ Applying this to the RHS of \eqref{logexpan788} gives, 
    \begin{align*}
        \log \E\left[ e^{\textstyle \frac{\gamma}{2N}\sum_{i=1}^{N}\lambda_i(\bm{W}_N)\tilde{Z}_i^2 } \right] \leq \frac{M\gamma^2}{2N^2} \mathrm{tr}(\bm{W}_N^2) \leq \frac{M\gamma^2}{2} \lesssim 1 . 
    \end{align*}
This completes the proof of Lemma \ref{expUniInt}. 
\end{proof}

\subsection{Convergence of Riemann Integrable Graphons}

Given a graphon $W: [0, 1]^2 \rightarrow [0, 1]$, its {\it cut norm} is defined as: 
\begin{align*}
\|W \|_{\square} :=\sup_{S,T\subset [0,1]}\left|\int_{S\times T} W(x,y) \mathrm dx\mathrm dy\right| . 
\end{align*} 
To obtain a metric that is invariant to labeling, one defines the  {\it cut distance} between two graphons $U, W$ as follows: 
\begin{align}\label{eq:delta_box_graphon}
\delta_\square(U, W):=\inf_{\sigma \in \mathcal{F}} \| U - W^{\sigma}\|_{\square} , 
\end{align}  
where $\mathcal{F}$ is the set of measure-preserving bijections from $[0,1]$ to $[0,1]$ and $W^{\sigma}(x, y):=W(\sigma(x), \sigma(y))$, for $\sigma \in \mathcal{F}$.

Given a graphon $W$, suppose $G_N \sim G(N, \theta_N, W)$ with adjacency matrix $A_{G_N}$ as in \eqref{eq:sufficientstatistics}. Define the empirical graphon associated with $G_N$ as follows: 
 \begin{align}\label{eq:WGN}
       W^{G_N}(x,y) := \sum_{1 \leq u, v \leq N} A_{G_N}(u, v) \one\left\{(x,y)\in I_u \times I_v \right\} \text{ for all }(x,y)\in [0,1]^2,
    \end{align}
    where $I_1 = \left(0, \frac{1}{N}\right]$ and $I_u = \left(\frac{u-1}{N}, \frac{u}{N}\right]$ for all $2\leq u \leq N$. The following lemma establishes the convergence of $W^{G_N}$ in the cut distance for Riemann integrable graphons.

\begin{lemma}\label{lemma:A_{G_N}_convg_W}
 Suppose $W:[0,1]^2\ra[0,1]$ is a Riemann integrable graphon and $\theta_N\ra\theta\in (0,1]$. Then for $W^{G_N}$ as defined above, 
    \begin{align*}
        \delta_\Box\left(W^{G_N}, \theta W\right)\ra 0 , 
    \end{align*}
almost surely. 
\end{lemma} 

\begin{proof}
   To begin with, define the matrix $\bm W_N  = \frac{1}{N} \left(\left(W(\frac{u}{N}, \frac{v}{N})\right)\right)_{1 \leq u, v \leq N}$ and consider the function,
    \begin{align*}
        W_N(x,y) = \sum_{1 \leq u, v \leq N} W\left(\frac{u}{N}, \frac{v}{N} \right)\one\left\{(x,y)\in I_u\times I_v\right\}, \text{ for all }(x,y)\in [0,1]^2.
    \end{align*}
    Then by \cite[Corollary 4.4.8]{vershynin2018high} we know that 
    \begin{align}\label{eq:AGn_Wn_convg}
        \left\|\frac{1}{N \theta_N} A_{G_N} - \bm W_N\right\|_{\mathrm{op}} \ra 0 \text{  almost surely.}
    \end{align}
    Moreover, by definition it is easy to conclude that 
    $$\left\|\frac{1}{N \theta_N} A_{G_N} - \bm W_N\right\|_{\mathrm{op}} = \left\|\frac{W^{G_N}}{\theta_N} - W_N\right\|_{\text{op}} =\frac{1}{\theta_N}\left\|W^{G_N} - \theta_NW_N\right\|_{\rm op}.$$
    Consequently, using the convergence from \eqref{eq:AGn_Wn_convg} and recalling $\theta_N\ra\theta\in (0,1]$ we get,
    \begin{align*}
        \left\|W^{G_N} - \theta_NW_N\right\|_{\rm op}\ra0 \text{ almost surely.}
    \end{align*}
    Then similar to \cite[Lemma 8.11]{lovasz2012large} we can conclude that $\delta_\Box\left(W^{G_N}, \theta_NW_N\right)\ra 0$ almost surely. To complete the proof, it is enough to show that $\delta_\Box\left(\theta_NW_N, \theta W\right)\ra 0$. To that end by triangle inequality and definition of $\delta_\Box$ from \eqref{eq:delta_box_graphon} we have the upper bound,
    \begin{align*}
        \delta_\Box\left(\theta_NW_N, \theta W\right)\leq \theta_N\delta_{\Box}\left(W_N,W\right) + |\theta_N - \theta| \cdot \|W\|_{\Box}
    \end{align*}
    Once again leveraging the convergence $\theta_N\ra\theta\in (0,1]$ it is now enough to show that $\delta_\Box\left(W_N, W\right)\ra 0$. Indeed, in the following we will show,
    \begin{align*}
        \|W_N - W\|_1\ra 0
    \end{align*}
    which will complete the proof by using the inequality $\delta_\Box\left(W_N, W\right)\leq \|W_N-W\|_1$, where $\| \cdot \|_1$ is the $L^1$ norm (see \cite[Chapter 8]{lovasz2012large}). Towards that define,
    \begin{align*}
        U_N(W) := \frac{1}{N^2} \sum_{1 \leq u, v \leq N}  \sup_{(x,y)\in I_u\times I_v}W(x,y) \text{ and } L_N(W) := \frac{1}{N^2} \sum_{1 \leq u, v \leq N} \inf_{(x,y)\in I_u\times I_v}W(x,y) .
    \end{align*}
    Then we can upper bound $\|W_N-W\|$ as,
    \begin{align}\label{eq:L1_bdd_UL}
        \|W_N-W\|_1 = \int \left|W_N(x,y)-W(x,y)\right| \mathrm{d} x\mathrm{d} y \leq U_N(W) - L_N(W).
    \end{align}
    Since $W$ is Riemann-integrable, for every $\vep>0$ there is $N_0$ large enough, such that for all $N>N_0$,
    \begin{align*}
        U_N(W) - L_N(W)\leq \vep . 
    \end{align*}
     This combined with \eqref{eq:L1_bdd_UL} completes the proof of Lemma \ref{lemma:A_{G_N}_convg_W}.
\end{proof}
\subsection{From Conditional to Joint Convergence}

The following result allows us to establish joint (unconditional) convergence from the convergence of the conditional moment generating function.

\begin{lemma}\label{l2ratio762} 
   Suppose that $\{X_N\}_{N\ge 1}, \{Y_N\}_{N\ge 1}$ are sequences of random variables all defined on a common probability space $\Omega$. Assume that for some $\delta >0$ and for every $t \in (-\delta,\delta)$, as $N \rightarrow \infty$, $$\E \left[e^{tX_N}\Big |Y_N\right] \xrightarrow{P} \E e^{tV},$$ for some random variable $V$. Then,
   \begin{enumerate}
       \item[$(1)$] $X_N \xrightarrow{D} V$,
       \item[$(2)$] For every $r \ge 1$, $\E [X_N^r|Y_N] \xrightarrow{P} \E V^r$. 
   \end{enumerate}
\end{lemma}

\begin{proof}
   Enumerate the set of rational numbers in $(-\delta,\delta)$ as $\mathbb Q \bigcap (-\delta,\delta) = \{q_1,q_2,\ldots\}$. Let $\{N^{(0)}_s\}_{s \geq 1}$ be any subsequence  of positive integers. Then, since $\E [ e^{q_1X_{N_s}^{(0)}}|Y_{N_s^{(0)}} ] \rightarrow \E e^{q_1V}$ converges in probability, there exists a subsequence $\{N^{(1)}_s\}_{s \geq 1}\subseteq \{N^{(0)}_s\}_{s \geq 1}$ and an $A^{(1)}\subseteq \Omega$ with $\P(A^{(1)})=1$, such that 
   $$\E \left[ e^{q_1X_{N_s}^{(1)}}|Y_{N_s^{(1)}} \right] \rightarrow \E e^{q_1V} \text{ on } A^{(1)}.$$ Again, since $\E [e^{q_2X_{{N_s}^{(1)}}}|Y_{{N_s}^{(1)}}] \xrightarrow{P} \E e^{q_2V}$ converges in probability, there exists a subsequence $\{ N_s^{(2)} \}_{s \geq 1} \subseteq \{N_s^{(1)} \}_{s \geq 2}$ and an $A_2\subseteq \Omega$ with $\P(A_2)=1$, such that 
   $$\E\left[e^{q_2X_{N_s^{(2)}}}|Y_{N_s^{(2)}} \right] \rightarrow \E e^{q_2V} \text{ on } A_2.$$ 
   In this way, for every $t \ge 1$, we get a subsequence $\{N_s^{(t)}\}_{s \geq 1}\subseteq \{N_s^{(t-1)}\}_{s \geq 1}$ and a probability $1$ set $A_s$, such that 
   $$\E \left[e^{q_s X_{N_s^{(t)}}} |Y_{N_s^{(t)}} \right] \rightarrow \E e^{q_s V} \text{ on }A_s. $$ 
Define $\{M_t\}_{t \geq 1}$ to be the diagonal subsequence of the sequences $\{ \{N_s^{(t)}\}_{s \geq 1} : t \geq 1\}$, that  is, $M_t := N_t^{(t)}$, for all $t\ge 1$. Then by Cantor's diagonalization argument, on the probability $1$ set $A := \bigcap_{t \ge 1} A_t$, $$\E \left[e^{qX_{M}}\Big |Y_{M}\right] \rightarrow \E e^{qV}, \quad \text{for all}~q\in \mathbb{Q}\bigcap(-\delta,\delta). $$ 
  
  Next, note that for each $M$, by H\"older's inequality, the functions
\[
m_M(t) := \E\big(e^{tX_M}\mid Y_M\big)\quad\text{and}\quad m(t) := \E e^{tV}
\]
are log-convex in $t$, and hence continuous on $(-\delta,\delta)$. Since $m_M(q) \to m(q)$ for every $q \in \mathbb{Q} \bigcap (-\delta,\delta)$, it follows by the log-convexity of $m_M$, continuity of $m$ and a sandwiching argument, that $m_M(t) \to m(t)$ for all $t \in (-\delta,\delta)$. Hence, on the set $A$, 
   $\P(X_M \le t| Y_M) \rightarrow \P(V\le t)$ for all continuity points $t$ of the distribution of $V$, and also, $\E(X_M^r|Y_M)\rightarrow \E V^r$, for all $r \ge 1$. Now, recall that $\{N_s^{(0)}\}_{s\geq 1}$ was an arbitrarily chosen sequence and $\{ M_s\}_{s \geq 1} \subseteq \{N_s^{(0)}\}_{s\geq 1}$. Hence,
   for every fixed $t$ in the set of continuity points of the distribution of $V$, 
   $$\P(X_N\le t| Y_N) \xrightarrow{P} \P(V\le t)\quad\text{and}\quad \E[X_N^r |Y_N]\rightarrow \E V^r.$$
   Lemma \ref{l2ratio762} now follows by an application of the dominated convergence theorem. 
\end{proof}

\subsection{Absolute Convergence of Eigenvalue Series}

In this section we collect a few results about the absolute convergence of series indexed by eigenvalues of the graphon $W$. Towards this, recall the operator $T_W$ from \eqref{eq:TW}. The countable multiset of eigenvalues of $W$ is denoted by $\mathrm{Spec}(W)$. Enumerate the eigenvalues  in decreasing order of magnitude such that $|\lambda_1(W)|\geq |\lambda_2(W)|\geq\cdots$.  
\begin{lemma}\label{lemma:eigenvalue_summability}
Let $W$ be a graphon as above and assume $0<\beta< \frac{1}{\|W\|_{\rm op}}$. Then there exists $\tau>0$ such that the following hold: 
    \begin{enumerate}
        \item [$(a)$] For $|t|<\tau$, $$\sum_{i=1}^{\infty}\sum_{s=2}^{\infty}\left|\frac{\lambda_i^s(W)}{\left(1-\beta\lambda_i(W)\right)^s}\frac{t^s}{s}\right|<\infty.$$
        \item [$(b)$] For $|t|<\tau$, 
        $$\sum_{i=1}^{\infty}\left|\log\left(e^{\textstyle -\frac{t \lambda_i(W)}{2} } \sqrt{\frac{1-\beta \lambda_i(W)}{1-(\beta+t) \lambda_i(W) }}\right)\right|<\infty.$$
       \item [$(c)$] For $|t|<\frac{1}{4\|W\|_{2}}$, $$\sum_{i=1}^{\infty}\sum_{\ell=2}^{\infty}\frac{\left|\lambda_i(W) t\right|^\ell}{\ell}<\infty.$$ 
        \item [$(d)$] For all $s\geq 3$ and $t\geq 1$, 
        $$\sum_{i=1}^{\infty}\sum_{a=0}^{\infty}{\binom{t+a-1}{a}}\beta^a|\lambda_i^{a+s}(W)|<\infty .$$
    \end{enumerate}
\end{lemma}
\begin{proof}
    Given the enumeration of the eigenvalues of $T_W$ according to decreasing order of magnitude, Equation (7.20) from \cite{lovasz2012large} shows that
    \begin{align}\label{eq:eigenvalue_bdd}
        \left|\lambda_i(W)\right|\leq \frac{\|W\|_{2}}{\sqrt{i}} , 
    \end{align} 
    for all $i\geq 1$. To prove (a) recall that $\beta<\frac{1}{\|W\|_{\rm op}}$ and hence note that for all $i\geq 1$,  \begin{align}\label{eq:def_C0betaW}
        0< \frac{1}{1-\beta\lambda_i(W)} \leq \frac{1}{ 1-\beta\|W\|_{\rm op} } =: C_0(\beta,W).
    \end{align}
    Then using the upper bound from \eqref{eq:eigenvalue_bdd}, 
    \begin{align*}
    \sum_{i=1}^{\infty}\sum_{s=2}^{\infty}\left|\frac{\lambda_i^s(W)}{\left(1-\beta\lambda_i(W)\right)^s}\frac{t^s}{s}\right|
        & \leq \sum_{i=1}^{\infty}\frac{\lambda_i^2(W)t^2}{2C_0(\beta,W)^2} + \sum_{i=1}^{\infty}\sum_{s=3}^{\infty}\frac{|t|^s}{s}\frac{\|W\|_{2}^s}{i^{\frac{s}{2}}C_0(\beta,W)^s}\\
        & \lesssim t^2\sum_{i=1}^{\infty}\lambda_i^2(W) + \sum_{s=3}^{\infty}\frac{|t|^s\|W\|_{2}^s}{C_0(\beta,W)^s}\sum_{i=1}^{\infty}\frac{1}{i^{\frac{3}{2}}} , 
    \end{align*}
    where the first inequality follows using Fubini's theorem. Now, we can choose $\tau>0$ small enough such that $\tau\|W\|_{2} < C_0(\beta,W)$. Then for all $|t|<\tau$ we get,
    \begin{align}\label{eq:part_a_finite}
    \sum_{i=1}^{\infty}\sum_{s=2}^{\infty}\left|\frac{\lambda_i^s(W)}{\left(1-\beta\lambda_i(W)\right)^s}\frac{t^s}{s}\right|\lesssim \sum_{i=1}^{\infty}\lambda_i^2(W) + \sum_{s=3}^{\infty}\left|\frac{\tau\|W\|_{2}}{C_0(\beta,W)}\right|^s<\infty,
    \end{align}
    where the finiteness follows by recalling that $\sum_{i=1}^{\infty}\lambda_i(W)^2 = \|W\|_2^2$ from \eqref{eq:eigenvalue_l2_sum}. This completes the proof of (a).\\
    To prove (b) first recall that $\beta\|W\|_{\rm op}< 1$. Hence, we can take $\tau>0$ small enough such that for all $|t|<\tau$, $|(\beta+t)\lambda_i(W)|<1$, for all $i\geq 1$. Then we can expand the logarithm as,
    \begin{align*}
        \sum_{i=1}^{\infty}\left|\log\left(e^{\textstyle -\frac{t \lambda_i(W)}{2} } \sqrt{\frac{1-\beta \lambda_i(W)}{1-(\beta+t) \lambda_i(W) }}\right)\right| 
        & =\frac{1}{2} \sum_{i=1}^{\infty}\left|-t\lambda_i(W) - \log\left(1-\frac{t\lambda_i(W)}{1-\beta\lambda_i(W)}\right)\right|\\
        & = \frac{1}{2}\sum_{i=1}^{\infty}\left|t\lambda_i(W) - \sum_{s=1}^{\infty}\left(\frac{t\lambda_i(W)}{1-\beta\lambda_i(W)}\right)^s\right|\\
        & = \frac{1}{2}\sum_{i=1}^{\infty}\left|t\lambda_i(W) - \frac{t\lambda_i(W)}{1-\beta\lambda_i(W)} - \sum_{s=2}^{\infty}\frac{1}{s}\left(\frac{t\lambda_i(W)}{1-\beta\lambda_i(W)}\right)^s\right| . 
    \end{align*}
    Now, to further simplify the summands we expand $(1-\beta\lambda_i(W))^{-1}$ and observe,
    \begin{align}\label{eq:b_bd_1}
        \sum_{i=1}^{\infty}
        & \left|\log\left(e^{\textstyle -\frac{t \lambda_i(W)}{2} } \sqrt{\frac{1-\beta \lambda_i(W)}{1-(\beta+t) \lambda_i(W) }}\right)\right|\nonumber\\
        & = \frac{1}{2}\sum_{i=1}^{\infty}\left|t\lambda_i(W) - t\lambda_i(W)\sum_{s=0}^{\infty}(\beta\lambda_i(W))^s - \sum_{s=2}^{\infty}\frac{1}{s}\left(\frac{t\lambda_i(W)}{1-\beta\lambda_i(W)}\right)^s\right|\nonumber\\
        &\lesssim \sum_{i=1}^{\infty}\left|t\beta\lambda_i(W)^2\right| + \sum_{i=1}^{\infty}\sum_{s=2}^{\infty}\left|t\beta^s\lambda_i^{s+1}(W)\right| + \sum_{i=1}^{\infty}\sum_{s=2}^{\infty}\left|\frac{1}{s}\frac{t\lambda_i(W)}{1-\beta\lambda_i(W)}\right|^s\nonumber\\
        &\lesssim \|W\|_2^2 + \sum_{i=1}^{\infty}\sum_{s=2}^{\infty}\left|\beta^s\lambda_i^{s+1}(W)\right| + \sum_{i=1}^{\infty}\sum_{s=2}^{\infty}\left|\frac{1}{s}\frac{t\lambda_i(W)}{1-\beta\lambda_i(W)}\right|^s.
\end{align}
where the final bound follows from $|t|<\tau$ and recalling that $\sum_{i=1}^{\infty}\lambda_i(W)^2 = \|W\|_2^2$ from \eqref{eq:eigenvalue_l2_sum}. To further bound the right hand side of \eqref{eq:b_bd_1}, in the following we begin with the second term:
\begin{align}\label{eq:b_second_bd}
    \sum_{i=1}^{\infty}\sum_{s=2}^{\infty}\left|\beta^s\lambda_i^{s+1}(W)\right| = \sum_{i=1}^{\infty}\sum_{s=3}^{\infty}\left|\beta^{s-1}\lambda_i(W)^s\right| 
    & = \beta\sum_{i=1}^{\infty}\lambda_i(W)^2\sum_{s=3}^{\infty}\left|\beta\lambda_i(W)\right|^{s-2}\nonumber\\
    & \lesssim \|W\|_2^2\sum_{s=3}^{\infty}\left|\beta\|W\|_{\rm op}\right|^{s-2}<\infty,
\end{align}
where the final bound follows by recalling $\beta\|W\|_{\rm op}<1$ and invoking the identity from \eqref{eq:eigenvalue_l2_sum}. To complete the proof of (b) we choose $\tau$ small enough such that $\tau\|W\|_2<C_0(\beta,W)$, where $C_0(\beta,W)$ is defined in \eqref{eq:def_C0betaW}. Now, combining the bounds from \eqref{eq:b_bd_1}, \eqref{eq:b_second_bd} and \eqref{eq:part_a_finite} we conclude,
\begin{align*}
    \sum_{i=1}^{\infty}
        & \left|\log\left(e^{\textstyle -\frac{t \lambda_i(W)}{2} } \sqrt{\frac{1-\beta \lambda_i(W)}{1-(\beta+t) \lambda_i(W) }}\right)\right|\\
        &\lesssim \|W\|_2^2 + \|W\|_2^2\sum_{s=3}^{\infty}\left|\beta\|W\|_{\rm op}\right|^{s-2} + \sum_{i=1}^{\infty}\sum_{s=2}^{\infty}\left|\frac{1}{s}\frac{t\lambda_i(W)}{1-\beta\lambda_i(W)}\right|^s<\infty.
\end{align*}
This completes the proof of (b).\\
    To prove part (c) note that,
    \begin{align*}
        \sum_{i=1}^{\infty}\sum_{\ell=2}^{\infty}\frac{\left|\lambda_i(W) t\right|^\ell}{\ell} = \sum_{i=1}^{\infty}\frac{t^2}{2}\lambda_i^2(W) + \sum_{i=1}^{\infty}\sum_{\ell=3}^{\infty}\frac{\left|\lambda_i(W) t\right|^\ell}{\ell}
        & \lesssim\|W\|_2^2 + \sum_{i=1}^{\infty}\sum_{\ell=3}^{\infty}\frac{1}{4^\ell i^{\frac{\ell}{2}}\ell}\\
        & \lesssim\|W\|_2^2 + \sum_{\ell=3}^{\infty}\frac{1}{4^\ell}\sum_{i=1}^{\infty}\frac{1}{i^{\frac{3}{2}}}<\infty , 
    \end{align*}
    where the penultimate inequality follows by recalling that $|t|<\frac{1}{4\|W\|_{2}}$. This proves (c).\\
To prove part (d) we begin with the following decomposition.
\begin{align*}
    \sum_{i=1}^{\infty}\sum_{a=0}^{\infty}{\binom{t+a-1}{a}}\beta^a|\lambda_i^{a+s}(W)| = \sum_{i=1}^{\infty}|\lambda_i(W)|^s\sum_{a=0}^{\infty}{t+a-1\choose a}\beta^a|\lambda_i(W)|^a.
\end{align*}
Now using the bounds $|\lambda_i(W)|\leq \|W\|_{\rm op}$ and \eqref{eq:eigenvalue_bdd} we get,
\begin{align*}
    \sum_{i=1}^{\infty}\sum_{a=0}^{\infty}{\binom{t+a-1}{a}}\beta^a|\lambda_i^{a+s}(W)|\leq \sum_{i=1}^{\infty}\frac{\|W\|_2^s}{i^{\frac{s}{2}}}\sum_{a = 0}^{\infty}{t+a-1\choose a}(\beta\|W\|_{\rm op})^a.
\end{align*}
Moreover by definition $\|W\|_2\leq 1$. Hence we can simplify the above bound as,
\begin{align*}
    \sum_{i=1}^{\infty}\sum_{a=0}^{\infty}{\binom{t+a-1}{a}}\beta^a|\lambda_i^{a+s}(W)|\leq\frac{1}{(1-\beta\|W\|_{\rm op})^t}\sum_{i=1}^{\infty}\frac{1}{i^{\frac{s}{2}}}<\infty,
\end{align*}
where the final bound follows by recalling that $s\geq 3$ and $\beta\|W\|_{\rm op}<1$. This completes the proof of part (d).
\end{proof}

\end{document}